\documentclass[pra,nofootinbib,reprint,showkeys]{revtex4-1}
\usepackage{amssymb}
\usepackage{amsmath}
\usepackage{amsthm}
\usepackage{color}
\usepackage{comment}
\usepackage{enumitem}
\usepackage{graphicx}
\usepackage[table]{xcolor}  

\makeatletter \renewcommand{\p@subsection}{} \makeatother				

\let\originalleft\left
\let\originalright\right
\renewcommand{\left}{\mathopen{}\mathclose\bgroup\originalleft}
\renewcommand{\right}{\aftergroup\egroup\originalright}

\def\cF{\mathcal{F}}
\def\cG{\mathcal{G}}
\def\cM{\mathcal{M}}

\def\co{o}
\def\cO{O}

\def\manualEndProof{\hfill$\Box$}
\def\rD{{\rm D}}
\def\re{{\rm e}}
\def\ri{{\rm i}}

\newtheorem{theorem}{Theorem}[section]
\newtheorem{lemma}[theorem]{Lemma}
\theoremstyle{definition}
\newtheorem{definition}{Definition}[section]

\begin{document}

\title{Twenty Hopf-like bifurcations in piecewise-smooth dynamical systems.}
\author{D.J.W.~Simpson}
\affiliation{ 
School of Fundamental Sciences,
Massey University,
Palmerston North, 4410,
New Zealand
}
\date{\today}

\begin{abstract}
For many physical systems the transition from a stationary solution to sustained small amplitude oscillations corresponds to a Hopf bifurcation. For systems involving impacts, thresholds, switches, or other abrupt events, however, this transition can be achieved in fundamentally different ways. This paper reviews 20 such `Hopf-like' bifurcations for two-dimensional ODE systems with state-dependent switching rules. The bifurcations include boundary equilibrium bifurcations, the collision or change of stability of equilibria or folds on switching manifolds, and limit cycle creation via hysteresis or time delay. In each case a stationary solution changes stability and possibly form, and emits one limit cycle. Each bifurcation is analysed quantitatively in a general setting: we identify quantities that govern the onset, criticality, and genericity of the bifurcation, and determine scaling laws for the period and amplitude of the resulting limit cycle. Complete derivations based on asymptotic expansions of Poincar\'e maps are provided. Many of these are new, done previously only for piecewise-linear systems. The bifurcations are collated and compared so that dynamical observations can be matched to geometric mechanisms responsible for the creation of a limit cycle. The results are illustrated with impact oscillators, relay control, automated balancing control, predator-prey systems, ocean circulation, and the McKean and Wilson-Cowan neuron models.
\end{abstract}

\maketitle

\onecolumngrid
\tableofcontents
\vspace{3mm}
\hrule
\vspace{3mm}
\twocolumngrid

\section{Introduction}
\setcounter{equation}{0}
\setcounter{figure}{0}
\setcounter{table}{0}
\label{sec:intro}

Sustained oscillations are central to many dynamical systems.
Ocean and atmospheric oscillations drive global weather patterns,
and seasonal variations direct annual cycles such as the Arctic sea ice extent \cite{Di13,KaEn13}.
An understanding of these oscillations may lead to better predictions regarding anthropogenic climate change.
Biological systems involve oscillations on diverse scales,
from circadian rhythms and day/night cycles to intra-cellular processes and neural firing \cite{Fo17,KeSn98,Wi01}.
Detailed knowledge of such oscillations is crucial to preventing and treating improper functioning such as cardiac arrhythmia.
Also mechanical systems oscillate both by design and as unwanted vibrations \cite{Di10,DuSr12}.

Mathematically an oscillation is a loop in {\em phase space}:
the abstract space representing all possible states of a system.
The loop is traced by the state of the system as time evolves.
For systems of ODEs (ordinary differential equations) the loop is a periodic orbit
and called a limit cycle if it is isolated from other periodic orbits.

For many systems limit cycles only exist over some range of parameter values.
At the endpoints of the parameter range the limit cycle is created or destroyed in a bifurcation,
the simplest of which is a Hopf bifurcation.
At a Hopf bifurcation an equilibrium loses stability
(more generally its stable and unstable manifolds change dimension)
and a small amplitude limit cycle is created.
The bifurcation occurs when a complex conjugate pair
of eigenvalues associated with the equilibrium crosses the imaginary axis in the complex plane.
Hopf bifurcations are important in
fluid dynamics \cite{MaTu95}, laser physics \cite{NiHa90}, and ecology \cite{FuEl00};
for further examples refer to the classic texts \cite{HaKa81,MaMc76}.

This paper concerns piecewise-smooth systems for which phase space
contains one or more switching manifolds where a map is applied or
the functional form of the equations of motion changes (as in Fig.~\ref{fig:schemIntro}).
Piecewise-smooth ODEs form natural mathematical models of
relay control systems \cite{BaVe01,Jo03,ZhMo03},
mechanical systems with impacts \cite{AwLa03,Br99,WiDe00},
and phenomena in other areas
including climate science \cite{Wi13}, neuroscience \cite{KeHo81}, and ecology \cite{DeGr07}.

\begin{figure}[b!]
\begin{center}
\includegraphics[width=5.6cm]{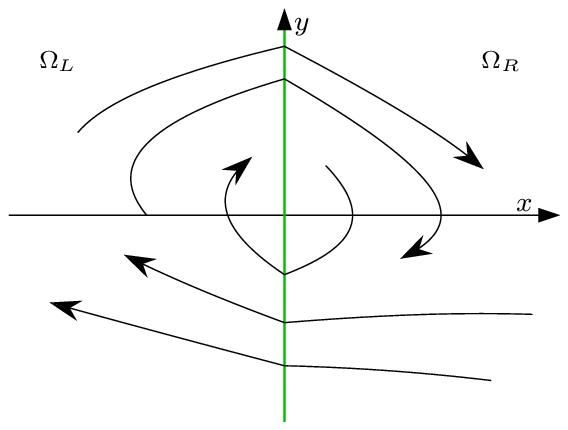}
\caption{
A sketch of the phase space of a two-dimensional piecewise-smooth ODE system
with switching manifold, $x = 0$.
The switching manifold divides phase space into two regions,
$\Omega_L$ and $\Omega_R$, where the ODEs are smooth.
Parts of typical orbits are shown.
\label{fig:schemIntro}
} 
\end{center}
\end{figure}

\begin{table*}
\begin{center}
\rowcolors{2}{gray!15}{white}
\setlength\arrayrulewidth{.3mm}
\begin{tabular}{|@{}c@{}|@{}c@{}|@{}c@{}|@{}c@{}|@{}c@{}|@{}c@{}|@{}c@{}|}
\noalign{\hrule height .3mm}
\rowcolor{gray!15}
\rule{0pt}{5mm}				
\parbox{4mm}{\centering \#~~}
& \parbox{40mm}{\centering description}
& \parbox{26mm}{\centering bifurcation onset}
& \parbox{28mm}{\centering non-degeneracy coefficient, $\alpha$}
& \parbox{21mm}{\centering amplitude exponent, $a$}
& \parbox{21mm}{\centering period exponent, $b$}
& \parbox{18mm}{\centering key \phantom{,}reference(s)\phantom{,}}
\\[2mm] 							
\noalign{\hrule height .3mm}
- & Hopf & $\omega = 0$ & \eqref{eq:hopfNondegCond} & $1/2$ & $0$ & \cite{AnWi30,Ho42,Po92} \\
$1$ & focus/focus BEB & BEB & $\frac{\lambda_L}{\omega_L} + \frac{\lambda_R}{\omega_R}$ & $1$ & $0$ & \cite{FrPo97,SiMe07} \\
$2$ & focus/node BEB & BEB & - & $1$ & $0$ & \cite{FrPo98} \\
$3$ & generic BEB & BEB & - & $1$ & $0$ & \cite{Fi88,	KuRi03} \\
$4$ & degenerate BEB & BEB & $\frac{\lambda_L}{\omega_L} + \frac{\lambda_R}{\omega_R}$ & $1$ & $0$ & \cite{Si18f} \\
$5$ & slipping foci & collision & $\frac{\lambda_L}{\omega_L} + \frac{\lambda_R}{\omega_R}$ & $1$ & $0$ & \cite{CaLl17} \\
$6$ & slipping focus/fold & collision & $\lambda_L$ & $1$ & $0$ & \cite{CaLl17} \\
$7$ & slipping folds & collision & $\sigma_{{\rm fold},L} - \sigma_{{\rm fold},R}$ & $1/2$ & $1/2$ & \cite{CaLl17,KuRi03} \\
$8$ & fixed foci & $\frac{\lambda_L}{\omega_L} + \frac{\lambda_R}{\omega_R} = 0$ & $\chi_{{\rm focus},L} - \chi_{{\rm focus},R}$ & $1$ & $0$ & \cite{CoGa01,ZoKu06} \\
$9$ & fixed focus/fold & $\lambda_L = 0$ & $\chi_{{\rm focus},L} - \frac{\sigma_{{\rm fold},R}}{3}$ & $1$ & $0$ & \cite{CoGa01} \\
$10$ & fixed folds & $\sigma_{{\rm fold},L} = \sigma_{{\rm fold},R}$ & $\chi_{{\rm fold},L} - \chi_{{\rm fold},R}$ & $1/2$ & $1/2$ & \cite{LiHu14} \\
$11$ & impacting admissible focus & BEB & \eqref{eq:impactingNondegCond} & $1$ & $0$ & \cite{DiNo08} \\
$12$ & impacting virtual focus & BEB & \eqref{eq:impactingNondegCond} & $1$ & $0$ & \cite{DiNo08} \\
$13$ & impacting virtual node & BEB & - & $1$ & $0$ & \cite{DiNo08} \\
$14$ & impulsive & \eqref{eq:impulsivexi} & \eqref{eq:impulsiveNondegCond} & $1$ & $0$ & \cite{Ak05} \\
$15$ & hysteretic pseudo-eq. & hysteresis & \eqref{eq:hysteresisNondeg} & $1$ & $1$ & \cite{MaHa17} \\
$16$ & time delayed pseudo-eq. & time delay & \eqref{eq:hysteresisNondeg} & $1$ & $1$ & \cite{LiYu08} \\
$17$ & hysteretic two-fold & hysteresis & $\sigma_{{\rm fold},L} - \sigma_{{\rm fold},R}$ & $1/3$ & $1/3$ & \cite{Ma17} \\
$18$ & time delayed two-fold & time delay & $\sigma_{{\rm fold},L} - \sigma_{{\rm fold},R}$ & $1/2$ & $1/2$ & \cite{Ko17,LiYu13} \\
$19$ & intersecting sw.~manifolds & \eqref{eq:fourPieceLambda} & \eqref{eq:fourPieceNondegCond} & $1$ & $1$ & \cite{HaEr15} \\
$20$ & square-root singularity & BEB & - & $1$ & $0$ & \cite{NiCa16} \\
\noalign{\hrule height .3mm}
\end{tabular}
\caption{
Attributes of Hopf bifurcations and $20$ Hopf-like bifurcations (HLBs).
The HLBs are numbered 1--20, as indicated in the first column; the second column gives a brief description
(BEB abbreviates {\em boundary equilibrium bifurcation}).
The exponents $a$ and $b$ refer to the scaling laws \eqref{eq:scalingLaws}
and for most HLBs the stability of the limit cycle (i.e.~the criticality of the bifurcation)
is determined by the sign of $\alpha$.
Full details including formulas for $\alpha$ and additional references are provided in later sections.
Note, $\alpha$ is not specified for four HLBs where
criticality is governed a combination of algebraic constraints.
\label{tb:bifs}
}
\end{center}
\end{table*}

Piecewise-smooth systems admit several types of stationary solutions (where the state of the system is constant in time).
These include regular equilibria (a zero of a smooth component of the vector field)
and certain points on switching manifolds.
Interactions between stationary solutions and switching manifolds
induce a wide variety of bifurcations that are unique to piecewise-smooth systems \cite{DiBu08}.
This paper concerns such bifurcations that are `Hopf-like'
in the sense that an isolated stationary solution
changes stability and possibly form and emits one limit cycle locally \cite{OlAn04}.
While there have been numerous studies of Hopf-like bifurcations (HLBs),
many focus on one particular scenario, consider only piecewise-linear systems, or provide only qualitative results.
The purpose of this paper is provide general quantitative results and compare different HLBs.
This extends the summary presented in \cite{Si18c};
below we provide formal statements for each HLB
in a general setting. 

So that the geometric mechanisms underpinning the HLBs can be realised with minimal complexity,
only two-dimensional systems are treated.
In more than two dimensions HLBs can occur in an essentially two-dimensional fashion \cite{KuHo13,SiKo09},
but other dynamics (including chaos) may be involved \cite{Gl18,Si16c,Si18d}.
HLBs have been described in three-dimensional systems
\cite{CaFe11,CaFe12,CaFr05b,FrPo05,FrPo07,HuYa16}
and some limited forms of dimension reduction have been achieved \cite{KoGl11,Ku08,KuHo13,PrTe16,WeKu12},
but in general higher dimensions cannot be accommodated via a centre manifold analysis
(the standard approach used for smooth systems \cite{Ku04,Me07})
because a centre manifold simply may not exist.
Much of the bifurcation theory of piecewise-smooth systems is dimension specific.
In $n$ dimensions bifurcations can be inextricably $n$-dimensional \cite{GlJe15}.

The HLBs are labelled $1$--$20$ and summarised by Table \ref{tb:bifs}.
The first two columns provide the number and a brief description.
The next two columns indicate the cause or onset of the bifurcation
and give a formula for the non-degeneracy coefficient $\alpha$.
In most cases a stable limit cycle is created if $\alpha < 0$
and unstable limit cycle is created if $\alpha > 0$.
For example, a (classical) Hopf bifurcation occurs when the imaginary part $\omega$ of the associated eigenvalues is zero,
and $\alpha$ is a combination of quadratic and cubic terms in a Taylor expansion of the ODEs, see \eqref{eq:hopfNondegCond}.
Several HLBs are boundary equilibrium bifurcations (BEBs)
that occur when a regular equilibrium collides with a switching manifold.

The next two columns provide scaling laws for the amplitude of the limit cycle
(as measured by its diameter in phase space) and its period.
If a HLB occurs at $\mu = 0$, where $\mu \in \mathbb{R}$ is a parameter,
then the limit cycle exists for small $\mu < 0$ or small $\mu > 0$ and satisfies
\begin{equation}
\begin{split}
\text{amplitude} &\sim k_1 |\mu|^a, \\
\text{period} &\sim k_2 |\mu|^b.
\end{split}
\label{eq:scalingLaws}
\end{equation}
The exponents $a > 0$ and $b \ge 0$ are determined by the type of HLB;
the coefficients $k_1 > 0$ and $k_2 > 0$ are system specific.
For Hopf bifurcations, $a = \frac{1}{2}$ (square-root growth) and $b = 0$
(the period limits to $\frac{2 \pi}{\omega}$ at the bifurcation).
For a physical system, $a$ and $b$ could be estimated experimentally
in which case the results here could assist model selection in that models
involving HLBs with incorrect values of $a$ and $b$ would be discarded.

The HLBs described here are local bifurcations.
For piecewise-smooth systems, limit cycles can be created in global bifurcations
such as `canard super-explosions' where an equilibrium transitions instantaneously
to a large amplitude limit cycle \cite{DeFr13,RoGl14,RoCo12,WaWi16},
For piecewise-linear systems, a limit cycle may be created at infinity \cite{LlPo99},
and if there are two switching manifolds
a limit cycle intersecting both manifolds may appear suddenly
\cite{LlOr13,LlPo15,PoRo15,ToGe03}.

Other bifurcations not considered here include
those at which two equilibria and a local limit cycle emanate from a single point on a switching manifold.
Such bifurcations combine the characteristic features of saddle-node and Hopf bifurcations
and arise in continuous systems \cite{SiMe12}, discontinuous systems \cite{KuRi03,PoRo18}, and impacting systems \cite{DiNo08}.
The bifurcation of an equilibrium into two limit cycles has been described
in switched control systems \cite{SiKu12b},
impacting systems \cite{MoBu14},
and abstract piecewise-linear systems \cite{ArLl14,DiEl14}.
Remarkably, the bifurcation of an equilibrium into three limit cycles is possible for
two-dimensional piecewise-linear discontinuous systems, see \cite{BrMe13,FrPo14,HuYa12b,Li14,LlNo15}.
For codimension-two HLBs refer to \cite{DeDe11,DiPa08,GiPl01,GuSe11,SiMe08}
and for non-autonomous perturbations (e.g.~periodic forcing) see \cite{FePo16}.

The remainder of this paper is organised as follows.
We begin in \S\ref{sec:smooth} by reviewing Hopf bifurcations
and stating the Hopf bifurcation theorem in a way that is later
mimicked, as much as possible, with the HLBs.
In \S\ref{sec:background} we review geometric aspects of piecewise-smooth systems (particularly Filippov systems)
such as folds, pseudo-equilibria, and sliding motion.
Section \ref{sec:lemmas} provides some technical results to support the proofs of the main theorems.
Each proof is based on the construction and analysis of a locally valid Poincar\'e map. 
An isolated fixed point of a Poincar\'e map corresponds to limit cycle.
Most of the Poincar\'e maps are constructed by combining two maps,
one for the system on each side of a switching manifold.
These two maps are return maps, capturing the effect of evolution from, and back to, a switching manifold.
In \S\ref{sec:lemmas} we provide asymptotic expansions for return maps
about an invisible fold, a boundary focus, and a regular focus in an affine system
(Lemmas \ref{le:focus}--\ref{le:affine}).
We also review our key analytical tool: the implicit function theorem (IFT),
and discuss the degree of smoothness of orbits and Poincar\'e maps.

The HLBs are examined in Sections \ref{sec:pwsc}--\ref{sec:other} in order.
For each HLB we provide a theorem, an example, and a proof (given in an appendix).
Most theorems have the following features.
There exists an isolated stationary solution for all values of $\mu$ in a neighbourhood of $0$.
There is a transversality condition $\beta \ne 0$
(to ensure that $\mu$ unfolds the bifurcation in a generic fashion).
Where appropriate $\beta$ is defined so that $\beta > 0$
(chosen without loss of generality in view of the substitution $\mu \mapsto -\mu$)
implies the stationary solution is stable for $\mu < 0$ and unstable for $\mu > 0$.
When $\mu = 0$ the stationary solution is stable if $\alpha < 0$ and unstable if $\alpha > 0$.
If $\alpha < 0$ then the bifurcation is {\em supercritical}
(a stable limit cycle exists for small $\mu > 0$),
while if $\alpha > 0$ then the bifurcation is {\em subcritical}
(an unstable limit cycle exists for small $\mu < 0$).
Then \S\ref{sec:conc} provides concluding remarks.

The theorems are proved in appendices by identifying a unique fixed point of a Poincar\'e map $P$ via the IFT
(an alternative is to use the intermediate value theorem for monotone functions).
Many of these involve a spatial scaling to `blow up' the dynamics about the origin,
or a brute-force evaluation of terms (using Lemmas \ref{le:focus} and \ref{le:fold}),
but several have unique technical complexities.
For instance, HLB 4 requires a careful analysis of nonlinear functions,
HLB 17 uses a non-constructive derivation,
and HLB 20 requires a multiple time-scales analysis.
The IFT is usually applied to a displacement function $D(r;\mu)$ whose zeros correspond to fixed points of $P(r;\mu)$.
For example, for HLBs $8$, $9$, $14$, and $19$ we use $D(r;\mu) = \frac{1}{r} \left( P(r;\mu) - r \right)$
so that the $r=0$ fixed point (corresponding to a stationary solution at the origin) is removed.
Typically $D$ is only defined for, say, $r > 0$, so $D$ must be smoothly extended
to a neighbourhood of $(r;\mu) = (0;0)$ before the IFT is applied.

\section{Hopf bifurcations}
\setcounter{equation}{0}
\setcounter{figure}{0}
\setcounter{table}{0}
\label{sec:smooth}

The mathematicians most closely affiliated with Hopf bifurcations are Poincar\'e, Andronov, and Hopf.
Poincar\'e studied Hopf bifurcations in the late $19^{\rm th}$ century \cite{Po92},
Andronov and coworkers later proved the Hopf bifurcation theorem
for two-dimensional systems of ODEs \cite{AnWi30},
and this was extended by Hopf to systems of arbitrary dimension in \cite{Ho42}.
Hopf bifurcations generalise to other equations, such as
partial differential equations \cite{HaKa81,MaMc76,Se94},
delay differential equations \cite{Er09,HaLu93},
and stochastic differential equations \cite{ArSr96}.

Here we state the Hopf bifurcation theorem, following \cite{Gl99}, for a two-dimensional system,
\begin{equation}
\begin{bmatrix} \dot{x} \\ \dot{y} \end{bmatrix} = F(x,y;\mu) =
\begin{bmatrix} f(x,y;\mu) \\ g(x,y;\mu) \end{bmatrix},
\label{eq:smoothODE}
\end{equation}
where $\mu \in \mathbb{R}$ is a parameter.
Throughout this paper dots represents differentiation with respect to time, $t$.
We first make two assumptions that allow us to state the theorem concisely.
For a general system
these assumptions can be imposed by performing an appropriate change of variables.

We first assume that the origin is an equilibrium for all values of $\mu$ in a neighbourhood of $0$
(this simplifies the formula for $\beta$ given below).
That is,
\begin{equation}
f(0,0;\mu) = g(0,0;\mu) = 0,
\label{eq:hopfEqCond}
\end{equation}
for small $\mu \in \mathbb{R}$.
Second we assume
\begin{equation}
\rD F(0,0;0) = \begin{bmatrix} 0 & -\omega \\ \omega & 0 \end{bmatrix},
\label{eq:hopfEigCond}
\end{equation}
for some $\omega \ne 0$.
That is, when $\mu = 0$ the eigenvalues associated with the origin are $\pm \ri \omega$,
and the Jacobian matrix is in real Jordan form.

In order for $\mu$ to unfold the Hopf bifurcation in a generic fashion,
the real part of the eigenvalues associated with the equilibrium
must change linearly, to leading order, with respect to $\mu$.
A simple calculation reveals that the real part of the eigenvalues is
$\frac{\beta \mu}{2} + \cO \left( \mu^2 \right)$, where
\begin{equation}
\beta = \left( \frac{\partial^2 f}{\partial \mu \partial x} + \frac{\partial^2 g}{\partial \mu \partial y}
\right) \bigg|_{x = y = \mu = 0}.
\label{eq:hopfTransCond}
\end{equation}
Consequently, $\beta \ne 0$ is the transversality condition for the Hopf bifurcation theorem.

We also require the quadratic and cubic terms in the Taylor expansion of $F(x,y;0)$ (centred about the origin)
to be such that a limit cycle is created in a non-degenerate fashion.
To this end, we define
\begin{align}
\alpha &= \left( \frac{\partial^3 f}{\partial x^3} +
\frac{\partial^3 g}{\partial x^2 \partial y} +
\frac{\partial^3 f}{\partial x \partial y^2} +
\frac{\partial^3 g}{\partial y^3} \right) \bigg|_{x = y = \mu = 0} \nonumber \\
&\quad+ \frac{1}{\omega} \bigg[ \frac{\partial^2 f}{\partial x \partial y}
\left( \frac{\partial^2 f}{\partial x^2} + \frac{\partial^2 f}{\partial y^2} \right) -
\frac{\partial^2 f}{\partial x^2} \frac{\partial^2 g}{\partial x^2} \nonumber \\
&\quad- \frac{\partial^2 g}{\partial x \partial y}
\left( \frac{\partial^2 g}{\partial x^2} + \frac{\partial^2 g}{\partial y^2} \right) +
\frac{\partial^2 f}{\partial y^2} \frac{\partial^2 g}{\partial y^2} \bigg] \bigg|_{x = y = \mu = 0}.
\label{eq:hopfNondegCond}
\end{align}
At the bifurcation (i.e.~with $\mu = 0$),
nearby orbits converge to the origin as $t \to \infty$ if $\alpha < 0$,
and as $t \to -\infty$ if $\alpha > 0$.
For this reason, $\alpha \ne 0$ is the non-degeneracy condition for the Hopf bifurcation theorem.

\begin{theorem}[Hopf bifurcation theorem]
Consider \eqref{eq:smoothODE} where $F$ is $C^3$.
Suppose \eqref{eq:hopfEqCond} and \eqref{eq:hopfEigCond} are satisfied and $\beta > 0$.
In a neighbourhood of $(x,y;\mu) = (0,0;0)$,
\begin{enumerate}
\item
the origin is the unique equilibrium 
and is stable for $\mu < 0$ and unstable for $\mu > 0$,
\item
if $\alpha < 0$ [$\alpha > 0$] there exists a unique stable [unstable] limit cycle
for $\mu > 0$ [$\mu < 0$], and no limit cycle for $\mu < 0$ [$\mu > 0$].
\end{enumerate}
The minimum and maximum $x$ and $y$-values of the limit cycle are asymptotically proportional to $\sqrt{|\mu|}$,
and its period is
$T = \frac{2 \pi}{|\omega|} + \cO(\mu)$.
\end{theorem}

The Hopf bifurcation theorem can be proved by performing a sequence of near-identity transformations
so that in polar coordinates, or with a single complex-valued variable, the ODEs take a simple form.
The proof can then be completed by using the Poincar\'e-Bendixson theorem, as in \cite{Gl99},
or applying the implicit function theorem to a Poincar\'e map, as in \cite{HaKa81,MaMc76}.

\begin{figure*}
\begin{center}
\includegraphics[width=16.6cm]{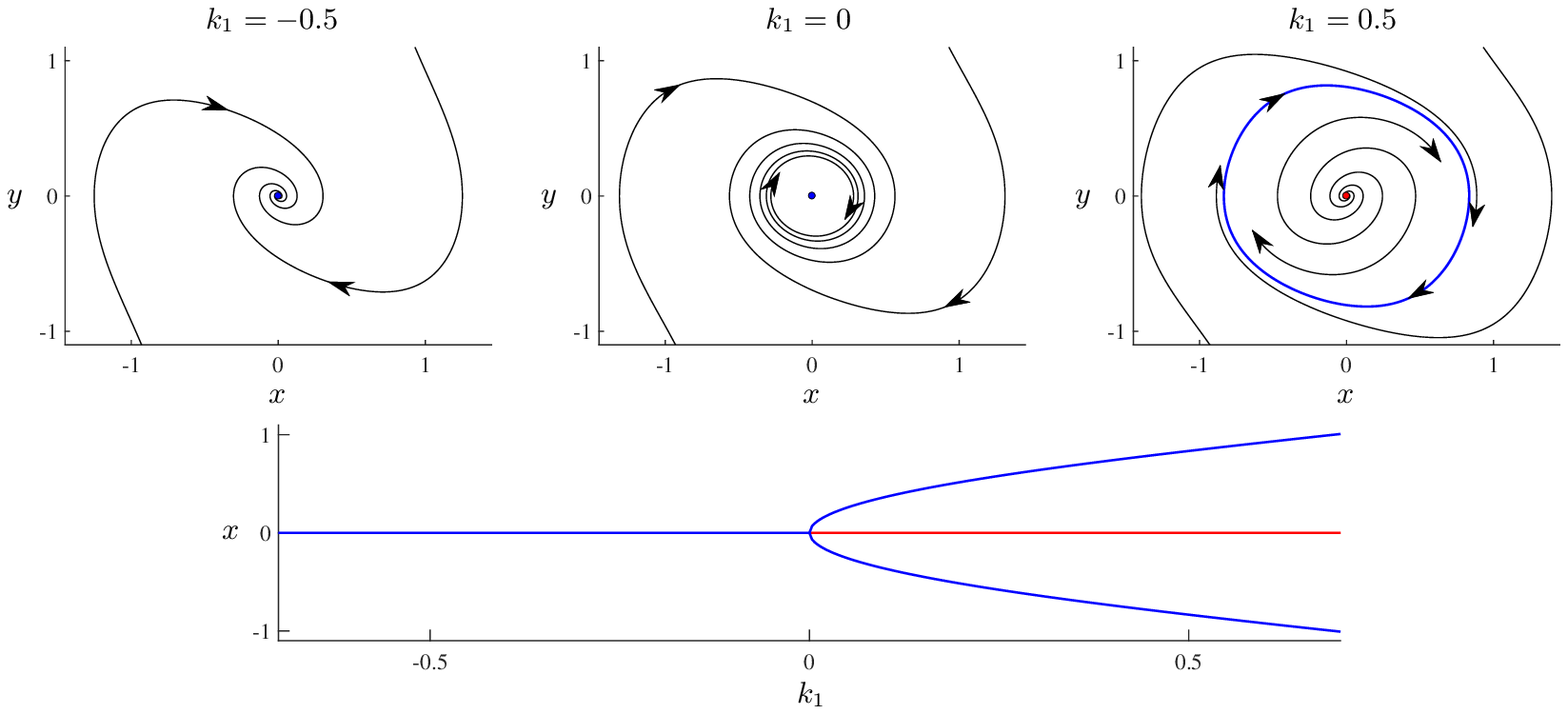}
\caption{
An illustration of the Hopf bifurcation for the van der Pol oscillator \eqref{eq:vdp} with $k_2 = 1$.
The bifurcation diagram shows the equilibrium
and the maximum and minimum $x$-values of the limit cycle.
Stable solutions are coloured blue; unstable solutions are coloured red.
These conventions are used in all subsequent bifurcation diagrams.
\label{fig:allVdP}
} 
\end{center}
\end{figure*}

The Hopf bifurcation theorem can be illustrated with the van der Pol oscillator \cite{Va26}.
Following \cite{Me07}, we consider
\begin{equation}
\begin{split}
\frac{d v}{d t} &= i, \\
\frac{d i}{d t} &= -v + k_1 i - k_2 i^3,
\end{split}
\label{eq:vdp}
\end{equation}
where $i(t)$ represents the non-dimensionalised current
in a circuit with a capacitor, an inductor, and an additional element
(vacuum tube or tunnel diode) with voltage drop $-k_1 i + k_2 i^3$,
and $v(t)$ represents the non-dimensionalised voltage
at a particular location in the circuit.

The origin, $(v,i) = (0,0)$, is an equilibrium of \eqref{eq:vdp}
and has purely imaginary eigenvalues when $k_1 = 0$.
Indeed the Hopf bifurcation theorem, as stated above, can be applied directly using $(x,y;\mu) = (v,i;k_1)$.
The transversality coefficient is $\beta = 1$ and
the non-degeneracy coefficient is $\alpha = -6 k_2$.
Fig.~\ref{fig:allVdP} shows a bifurcation diagram and representative phase portraits using $k_2 = 1$.
Here $\alpha < 0$, thus a stable limit cycle exists for $k_1 > 0$ as shown.
Notice the amplitude of the limit cycle is proportional to $\sqrt{k_1}$, to leading order.

\section{Fundamentals of Filippov systems}
\setcounter{equation}{0}
\setcounter{figure}{0}
\setcounter{table}{0}
\label{sec:background}

Filippov systems are discontinuous, piecewise-smooth ODE systems
with an additional rule defining evolution on switching manifolds, known as {\em sliding motion}.
Here we review the basic principles of Filippov systems
for two-dimensional systems with a single switching manifold.
More general expositions can be found in \cite{CoDi12,DiBu08,Fi88,Je18b}.

We consider ODE systems on $\mathbb{R}^2$ of the form
\begin{equation}
\begin{bmatrix} \dot{x} \\ \dot{y} \end{bmatrix} =
\begin{cases}
F_L(x,y), & x < 0, \\
F_R(x,y), & x > 0. \end{cases}
\label{eq:FilippovBackg}
\end{equation}
Here the switching manifold is $x=0$.
For systems with a more general switching manifold, say $H(x,y) = 0$,
the idea is that one may perform a change of variables to put the system in the form \eqref{eq:FilippovBackg}.
The change of variables may only be valid locally, but this is sufficient
for the local bifurcation theory of this paper.

We let
\begin{equation}
\begin{split}
\Omega_L &= \left\{ (x,y) \,\middle|\, x < 0, y \in \mathbb{R} \right\}, \\
\Omega_R &= \left\{ (x,y) \,\middle|\, x > 0, y \in \mathbb{R} \right\},
\end{split}
\label{eq:OmegaLR}
\end{equation}
denote the left and right half-planes.
We refer to $(\dot{x},\dot{y}) = F_L(x,y)$
and $(\dot{x},\dot{y}) = F_R(x,y)$
as the {\em left} and {\em right half-systems} of \eqref{eq:FilippovBackg}, respectively.
Orbits of \eqref{eq:FilippovBackg} are governed by the left half-system while in $\Omega_L$,
and by the right half-system while in $\Omega_R$.
Orbits may also slide on $x = 0$, as defined below.

To ensure sliding motion is well-defined,
we assume $F_J$ is $C^1$ on the closure of $\Omega_J$, for each $J \in \{ L,R \}$.
In mathematical models each $F_J$ is often well-defined (or extends analytically) beyond $\Omega_J$,
in which case we can consider the dynamics of $(\dot{x},\dot{y}) = F_J(x,y)$ outside $\Omega_J$ --- such dynamics is termed {\em virtual}.
Virtual dynamics is not exhibited by \eqref{eq:FilippovBackg},
but sometimes helps us understand the behaviour of \eqref{eq:FilippovBackg}.
We use the term {\em admissible} to describe dynamics
of $(\dot{x},\dot{y}) = F_J(x,y)$ in $\Omega_J$.

\begin{definition}
A {\em regular equilibrium} of \eqref{eq:FilippovBackg} is a point $(x,y)$ for which
$F_L(x,y) = (0,0)$ or $F_R(x,y) = (0,0)$.
\label{df:eq}
\end{definition}

If $F_L(x,y) = (0,0)$, then the equilibrium is admissible if $x < 0$ and virtual if $x > 0$
(and vice-versa if $F_R(x,y) = (0,0)$).
As parameters are varied, an admissible equilibrium becomes virtual
when it collides with $x=0$, where it is termed a {\em boundary equilibrium}.
This is known as a {\em boundary equilibrium bifurcation} (BEB) \cite{DiBu08}.
Many of the HLBs below are BEBs.

Next we distinguish different parts of the switching manifold, and to do this write
\begin{equation}
F_J(x,y) = \begin{bmatrix} f_J(x,y) \\ g_J(x,y) \end{bmatrix},
\label{eq:FJBackg}
\end{equation}
for each $J \in \{ L,R \}$.

\begin{definition}
A subset $\left\{ (0,y) \,\middle|\, y \in (a,b) \right\}$ of $\Sigma$ is said to be
\begin{enumerate}
\item
a {\em crossing region} if $f_L(0,y) f_R(0,y) > 0$ for all $y \in (a,b)$,
\item
an {\em attracting sliding region} if $f_L(0,y) > 0$ and $f_R(0,y) < 0$ for all $y \in (a,b)$,
\item
and a {\em repelling sliding region} if $f_L(0,y) < 0$ and $f_R(0,y) > 0$ for all $y \in (a,b)$.
\end{enumerate}
\label{df:swManDivision}
\end{definition}

Orbits of \eqref{eq:FilippovBackg} are directed into attracting sliding regions,
away from repelling sliding regions, and through crossing regions (see already Fig.~\ref{fig:schemSwManDivision}).
At endpoints of these regions we have $f_J(0,y) = 0$, for some $J \in \{ L,R \}$.
If $f_J(0,y) = 0$ and $g_J(0,y) \ne 0$,
then $F_J(0,y)$ is tangent to $x = 0$.
If also $\frac{\partial f_J}{\partial y}(0,y) \ne 0$, then
the orbit of the associated half-system that passes through $(0,y)$ has a quadratic tangency with $x = 0$ at this point.
In this case the tangent point is called a {\em fold},
and said to be visible if the orbit is admissible and invisible if it is virtual:

\begin{definition}
A point $(0,y)$ for which $f_L(0,y) = 0$ [$f_R(0,y) = 0$] is said to be
\begin{enumerate}
\item
a {\em visible fold} if
$\frac{\partial f_L}{\partial y}(0,y) g_L(0,y) < 0$ \\
$\left[ \frac{\partial f_R}{\partial y}(0,y) g_R(0,y) > 0 \right]$,
\item
and an {\em invisible fold} if
$\frac{\partial f_L}{\partial y}(0,y) g_L(0,y) > 0$ \\
$\left[ \frac{\partial f_R}{\partial y}(0,y) g_R(0,y) < 0 \right]$.
\end{enumerate}
\label{df:fold}
\end{definition}

For example, the system
\begin{equation}
\begin{bmatrix} \dot{x} \\ \dot{y} \end{bmatrix} =
\begin{cases}
\begin{bmatrix} y \\ 1 \end{bmatrix}, & x < 0, \\
\begin{bmatrix} -2(y-1) \\ -1 \end{bmatrix}, & x > 0,
\end{cases}
\label{eq:schemSwManDivision}
\end{equation}
depicted in Fig.~\ref{fig:schemSwManDivision},
has a visible fold at $(0,1)$ and an invisible fold at $(0,0)$.

{\em Two-folds} are points at which $f_L(0,y) = f_R(0,y) = 0$,
and are naturally classified as visible-visible,
visible-invisible, invisible-invisible, or degenerate.
For the two-dimensional system \eqref{eq:FilippovBackg},
two-folds are codimension-one phenomena and can correspond to the creation of a limit cycle (HLB 7).
In three-dimensional systems, two-folds occur generically
and may serve as a hub for important dynamics \cite{CoJe11,CrCa17,DiCo11,Te90}.

On sliding regions (both attracting and repelling)
we can define orbits by introducing a one-dimensional vector field:
\begin{equation}
\dot{y} = g_{\rm slide}(y).
\label{eq:gslide}
\end{equation}
To motivate the definition of $g_{\rm slide}$
we first present the general notion of a Filippov solution.

\begin{definition}
An absolutely continuous function $\phi_t : (a,b) \to \mathbb{R}^2$
is said to be a {\em Filippov solution} to \eqref{eq:FilippovBackg}
if it satisfies $\dot{\phi}_t \in \cF(\phi_t)$ for almost all $t \in (a,b)$,
where $\cF$ is the set-valued function
\begin{widetext}
\begin{equation}
\cF(x,y) = \begin{cases}
\{ F_L(x,y) \}, & x < 0, \\
\left\{ (1-s) F_L(0,y) + s F_R(0,y) \,\big|\, s \in [0,1] \right\}, & x = 0, \\
\{ F_R(x,y) \}, & x > 0.
\end{cases}
\label{eq:FSetValued}
\end{equation}
\end{widetext}
\label{df:FilippovSolution}
\end{definition}

The system \eqref{eq:FilippovBackg} is a {\em Filippov system} if we equate orbits to Filippov solutions.
Notice $\cF(0,y)$ is defined as the convex hull of $F_L(0,y)$ and $F_R(0,y)$.
For more general systems $\cF$ is defined as the convex hull of all
smooth components of the system associated with each point.

\begin{figure}[b!]
\begin{center}
\includegraphics[width=6cm]{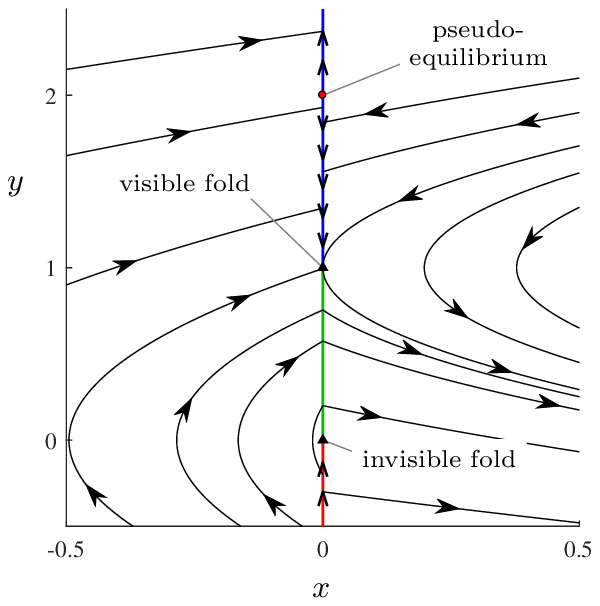}
\caption{
A phase portrait of the Filippov system \eqref{eq:schemSwManDivision}.
The switching manifold $x=0$ is a crossing region for $0 < y < 1$ (green),
an attracting sliding region for $y > 1$ (blue),
and a repelling sliding region for $y < 0$ (red).
On the sliding regions arrows indicate the direction of sliding motion.
\label{fig:schemSwManDivision}
} 
\end{center}
\end{figure}

Now consider a Filippov solution $\phi_t$
that is constrained to a sliding region for some interval of time.
In order to express the $y$-component of $\phi_t$ as a solution to \eqref{eq:gslide},
we use Definition \ref{df:FilippovSolution} to construct $g_{\rm slide}$.
The condition $\dot{\phi}_t \in \cF(\phi_t)$ implies
that the $x$-component of
$(1-s) F_L(0,y) + s F_R(0,y) = 0$ is zero.
Thus $s = \frac{f_L}{f_L - f_R} \big|_{x = 0}$ and
notice $s \in (0,1)$ because we are considering a sliding region.
By then defining $g_{\rm slide}$ as the $y$-component of $(1-s) F_L(0,y) + s F_R(0,y)$:
\begin{equation}
g_{\rm slide} = \left. \frac{f_L g_R - f_R g_L}{f_L - f_R} \right|_{x = 0} \,,
\label{eq:gslide2}
\end{equation}
the $y$-component of $\phi_t$ is a solution to \eqref{eq:gslide} as desired.

\begin{definition}
A point $(0,y)$ is said to be a {\em pseudo-equilibrium} of \eqref{eq:FilippovBackg} if $g_{\rm slide}(y) = 0$.
\end{definition}

The pseudo-equilibrium is admissible if $f_L(0,y) f_R(0,y) < 0$
(i.e.~it belongs to a sliding region)
and virtual if $f_L(0,y) f_R(0,y) > 0$.
As an example, again consider \eqref{eq:schemSwManDivision}.
Here $g_{\rm slide}(y) = \frac{y - 2}{3 y - 2}$,
thus there is a unique pseudo-equilibrium at $(0,2)$.
This pseudo-equilibrium is unstable
because $\frac{d g_{\rm slide}}{d y} \big|_{y=2} > 0$.

To summarise, orbits of the Filippov system \eqref{eq:FilippovBackg} are piecewise:
governed by the left half-system in $\Omega_L$,
by the right half-system in $\Omega_R$,
and by \eqref{eq:gslide} with \eqref{eq:gslide2} on sliding regions.
The forward orbit of a point can be non-unique.
For instance, the forward orbit of a point on a repelling sliding region
can immediately enter $\Omega_L$, immediately enter $\Omega_R$,
or slide along $x=0$ for some time before entering $\Omega_L$ or $\Omega_R$.

\section{Return maps and smoothness}
\setcounter{equation}{0}
\setcounter{figure}{0}
\setcounter{table}{0}
\label{sec:lemmas}

To study limit cycles in systems of the form \eqref{eq:FilippovBackg}
we construct Poincar\'e maps by using the switching manifold $x=0$ as a Poincar\'e section.
Each Poincar\'e map is a composition $P = P_L \circ P_R$,
where $P_R$ is the return map on $x=0$ for orbits of the right half-system,
and $P_L$ is the return map on $x=0$ for orbits of the left half-system, see Fig.~\ref{fig:schemPoinComposition}.
In this section we first clarify the smoothness of orbits and return maps, \S\ref{sub:analysis}.
We then derive three return maps on $x = 0$ for evolution in $x > 0$ for a smooth system
\begin{equation}
\begin{bmatrix} \dot{x} \\ \dot{y} \end{bmatrix} = F(x,y) =
\begin{bmatrix} f(x,y) \\ g(x,y) \end{bmatrix}.
\label{eq:smoothLemmasODE}
\end{equation}
These maps are used below for $P_R$ and also $P_L$ through a change of variables such as $(x,y) \mapsto (-x,-y)$.
For the first map the origin is a focus, \S\ref{sub:focus};
for the second map the origin is an invisible fold
(treating \eqref{eq:smoothLemmasODE} as the right half-system of \eqref{eq:FilippovBackg}), \S\ref{sub:fold}.
For these two maps we derive the first few terms in a series expansion about the origin.
For the third map the ODE system is affine and we provide an exact solution, \S\ref{sub:affine}.

\begin{figure}[b!]
\begin{center}
\includegraphics[width=5.6cm]{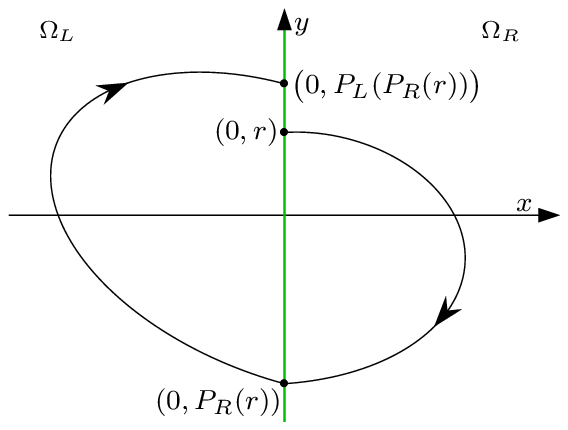}
\caption{
An illustration of the Poincar\'e map $P = P_L \circ P_R$
(defined more precisely in later sections subject to different assumptions)
for a Filippov system of the form \eqref{eq:FilippovBackg}.
\label{fig:schemPoinComposition}
} 
\end{center}
\end{figure}

\subsection{Key principles from analysis}
\label{sub:analysis}

The Picard-Lindel\"of theorem \cite{Me07} tells us that if $F$ is Lipschitz then
for any initial point $(x,y)$, \eqref{eq:smoothLemmasODE} has a unique solution $\phi_t(x,y)$
on some time interval $(a,b)$ containing $0$.
That is, $\phi_0(x,y) = (x,y)$ and $\dot{\phi}_t(x,y) = F(\phi_t(x,y))$ for all $t \in (a,b)$.
If $F$ is $C^k$ then this smoothness is exhibited by $\phi_t$ \cite{Ha02}:

\begin{lemma}
If $F$ is a $C^k$ ($k \ge 1$) function of $x$ and $y$,
then $\phi_t(x,y)$ is a $C^k$ function of $x$, $y$, and $t$.
\label{le:flowSmoothness}
\end{lemma}

Actually an extra time derivative is available because $\phi_t$
can be written as an integral of $F$ over time, but we do not utilise this here.
Return maps share the same degree of differentiability as $\phi_t$,
as long as orbits intersect the Poincar\'e section transversally.
This is a simple consequence of the implicit function theorem,
here stated for a real-valued function $h(u,v)$.

\begin{theorem}[Implicit function theorem (IFT)]
Suppose $h : \mathbb{R} \times \mathbb{R}^m \to \mathbb{R}$ is $C^k$,
with $h(0,0) = 0$ and $\frac{\partial h}{\partial u}(0,0) \ne 0$.
Then there exists a neighbourhood $V \subset \mathbb{R}^m$ of $0$
and a unique $C^k$ function $\xi : V \to \mathbb{R}$ such that
$h(\xi(v),v) = 0$ for all $v \in V$.
\end{theorem}

The IFT tells us we can solve $h(u,v) = 0$ for $u$,
the solution being $u = \xi(v)$.
As indicated in Fig.~\ref{fig:schemPoinComposition},
given $r \in \mathbb{R}$ let $P_R(r)$ be the $y$-value of the next intersection of the forward orbit of $(0,r)$ with $x=0$
(if the orbit never returns to $x=0$ leave $P_R(r)$ undefined).
Also write $\phi_t = (\varphi_t,\psi_t)$ (so that $\varphi_t$ and $\psi_t$ denote the $x$ and $y$-components of $\phi_t$ respectively). 
Then $P_R(r) = \psi_t(0,r)$, where $t > 0$ is defined via $\varphi_t(0,r) = 0$.
The following result follows from the smoothness of the flow and the IFT
(the conditions $f(0,r) \ne 0$ and $f(0,P_R(r)) \ne 0$ ensure the orbit intersects $x=0$ transversally).

\begin{lemma}
Suppose $F$ is $C^k$ ($k \ge 1$).
If $P_R(r)$ is well-defined, $f(0,r) \ne 0$, and $f(0,P_R(r)) \ne 0$, then $P_R(r)$ is $C^k$ at $r$.
\label{le:mapSmoothness}
\end{lemma}

Throughout this paper we use big-O and little-o notation
to describe higher order terms in series expansions \cite{De81}.
For functions $G, H : \mathbb{R}^m \to \mathbb{R}$ these are defined as follows:
if ${\displaystyle \limsup_{z \to 0} \left| \tfrac{G(z)}{H(z)} \right|}$ is bounded then $G(z) = \cO(H(z))$,
and if ${\displaystyle \lim_{z \to 0} \tfrac{G(z)}{H(z)} = 0}$ then $G(z) = \co(H(z))$.
For example,
\begin{align}
\re^z = 1 + z + \tfrac{1}{2} z^2 + \cO \left( z^3 \right)
= 1 + z + \tfrac{1}{2} z^2 + \co \left( z^2 \right).
\nonumber
\end{align}
For multi-variable series ($m \ge 2$) we usually
use some power of the $1$-norm of $z$ for the function $H$.

\subsection{The return map about a focus}
\label{sub:focus}

Here we consider \eqref{eq:smoothLemmasODE} assuming
\begin{equation}
f(0,0) = g(0,0) = 0,
\label{eq:focusEqCond}
\end{equation}
so that the origin is an equilibrium.
We suppose that the eigenvalues associated with the origin are complex:
\begin{equation}
{\rm eig}(\rD F(0,0)) = \lambda \pm \ri \omega, \quad
{\rm with~} \lambda \in \mathbb{R}, \omega > 0.
\label{eq:focusEigCond}
\end{equation}
If $\lambda \ne 0$ then the origin is a focus,
but here we also allow $\lambda = 0$.

\begin{figure}
\begin{center}
\includegraphics[width=8.6cm]{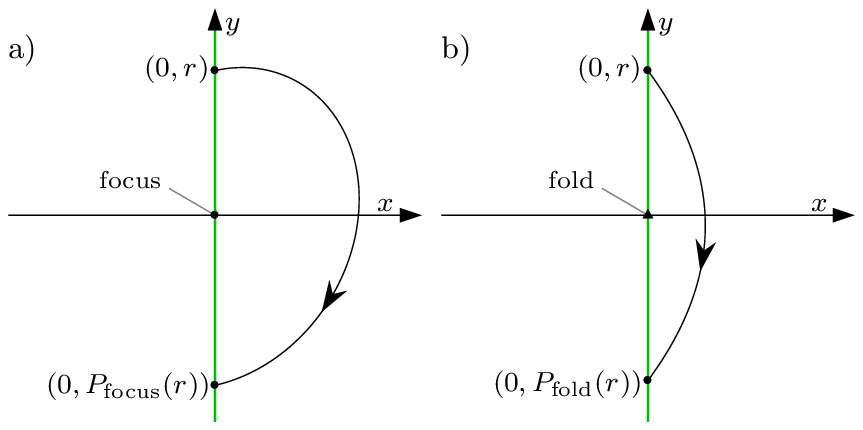}
\caption{
Typical orbits associated with the return maps $P_{\rm focus}$ \eqref{eq:focusP} (where the origin is a focus)
and $P_{\rm fold}$ \eqref{eq:foldP} (where the origin is an invisible fold).
\label{fig:schemPoinFocusFold}
} 
\end{center}
\end{figure}

For any $r > 0$, let $P_{\rm focus}(r)$ be the $y$-value
of the next intersection of the forward orbit of $(0,r)$ with $x = 0$, see Fig.~\ref{fig:schemPoinFocusFold}-a.
This is well-defined for small values of $r$.
Also let $T_{\rm focus}(r)$ denote the corresponding evolution time.
Thus the flow $\phi_t(x,y)$ satisfies
\begin{equation}
\phi_{T_{\rm focus}(r)}(0,r) = \left( 0, P_{\rm focus}(r) \right).
\nonumber
\end{equation}
We could assume orbits rotate clockwise,
so that the corresponding orbits evolve in $x > 0$ as discussed above,
but this assumption is not required here.
To state $P_{\rm focus}$ to second order we write
\begin{equation}
\rD F(0,0) = \begin{bmatrix} a_1 & a_2 \\ b_1 & b_2 \end{bmatrix},
\label{eq:focusDF}
\end{equation}
and let
\begin{align}
\chi_{\rm focus} &= \frac{1}
{\left( \lambda^2 + \omega^2 \right) \left( \lambda^2 + 9 \omega^2 \right)}
\bigg( k_1 \frac{\partial^2 f}{\partial x^2}
+ k_2 \frac{\partial^2 f}{\partial x \partial y} \nonumber \\
&\quad+ k_3 \frac{\partial^2 f}{\partial y^2}
+ \ell_1 \frac{\partial^2 g}{\partial x^2}
+ \ell_2 \frac{\partial^2 g}{\partial x \partial y}
+ \ell_3 \frac{\partial^2 g}{\partial y^2} \bigg) \bigg|_{x=y=0} \,,
\label{eq:focustau}
\end{align}
where
\begin{equation}
\begin{split}
k_1 &= -a_2 \left( 2 a_1 b_2 - 3 a_2 b_1 - b_2^2 \right), \\
k_2 &= -a_2 b_1 (4 a_1 + b_2) + a_1 b_2 (2 a_1 - b_2), \\
k_3 &= \frac{b_1 (2 a_1^2 + 7 a_1 b_2 - 3 a_2 b_1)}{2} -
\frac{b_2^2 (2 a_1^2 - a_1 b_2 + a_2 b_1)}{2 a_2}, \\
\ell_1 &= -a_2^2 (a_1 + b_2), \\
\ell_2 &= a_2 (2 a_1^2 - a_1 b_2 + 3 a_2 b_1), \\
\ell_3 &= -\frac{a_1 \left( 2 a_1^2 - 3 a_1 b_2 + b_2^2 \right)}{2} - \frac{a_2 b_1 \left(5 a_1 - b_2 \right)}{2}.
\end{split}
\label{eq:focusk1tol3}
\end{equation}

\begin{lemma}
Consider \eqref{eq:smoothLemmasODE} where $F$ is $C^2$.
Suppose \eqref{eq:focusEqCond} and \eqref{eq:focusEigCond} are satisfied.
Then
\begin{align}
P_{\rm focus}(r) &= -\re^{\frac{\lambda \pi}{\omega}} r + 
\re^{\frac{\lambda \pi}{\omega}} \left( \re^{\frac{\lambda \pi}{\omega}} + 1 \right) \chi_{\rm focus} r^2 \nonumber \\
&\quad+ \co \left( r^2 \right),
\label{eq:focusP} \\
T_{\rm focus}(r) &= \frac{\pi}{\omega} + \cO(r).
\label{eq:focusT}
\end{align}
\label{le:focus}
\end{lemma}

Lemma \ref{le:focus} is proved in Appendix \ref{app:lemmas} via brute-force asymptotic expansions of the flow.
The result is used in many places later in the paper,
but the $r^2$-term in \eqref{eq:focusP} is only required for HLBs 8 and 9.

\begin{figure}[b!]
\begin{center}
\includegraphics[width=8cm]{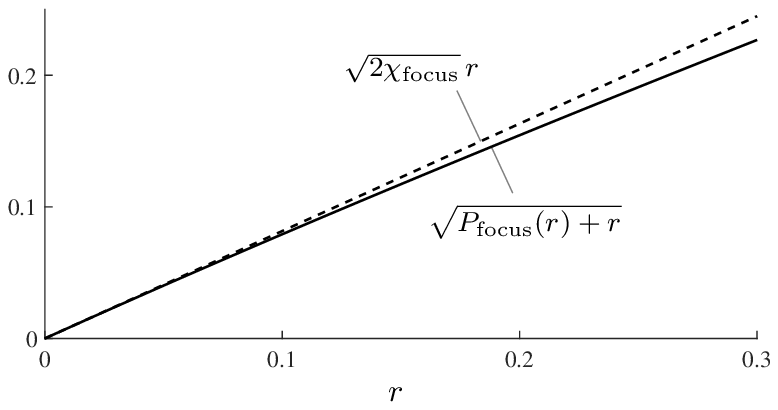}
\caption{
A numerical verification of Lemma \ref{le:focus} for the system \eqref{eq:focusExample}.
\label{fig:focusExample}
} 
\end{center}
\end{figure}

As a simple demonstration of Lemma \ref{le:focus}, consider the system
\begin{equation}
\begin{split}
\dot{x} &= y - y^2, \\
\dot{y} &= -x,
\end{split}
\label{eq:focusExample}
\end{equation}
for which $\lambda = 0$ and $\omega = 1$.
Here $\chi_{\rm focus} = -\frac{2 k_3}{9}$, where $k_3 = -\frac{3}{2}$.
By Lemma \ref{eq:focustau}, $P_{\rm focus}(r) = -r + 2 \chi_{\rm focus} r^2 + \co \left( r^2 \right)$.
Fig.~\ref{fig:focusExample} shows $\sqrt{P_{\rm focus}(r) + r}$
(computed by numerically simulating \eqref{eq:focusExample})
and we observe that its slope at $r=0$ indeed appears to be $\sqrt{2 \chi_{\rm focus}}$.

\vspace{2mm}		
\subsection{The return map about an invisible fold}
\label{sub:fold}

Here we suppose 
\begin{align}
f(0,0) &= 0, &
\frac{\partial f}{\partial y}(0,0) &> 0, &
g(0,0) &< 0.
\label{eq:foldCond}
\end{align}
If we consider \eqref{eq:smoothLemmasODE} only in $x > 0$ and treat $x=0$ as a switching manifold,
then the origin is an invisible fold (see Definition \ref{df:fold})
about which orbits rotate clockwise.

For any $r > 0$, let $P_{\rm fold}(r)$ be the $y$-value
of the next intersection of the forward orbit of $(0,r)$ with $x = 0$, see Fig.~\ref{fig:schemPoinFocusFold}-b,
and let $T_{\rm fold}(r)$ denote the corresponding evolution time.
These are well-defined for sufficiently small values of $r > 0$.
A brief calculation reveals that maximum $x$-value attained by the orbit
travelling from $(0,r)$ to $\left( 0, P_{\rm fold}(r) \right)$ is
$\frac{\frac{\partial f}{\partial y}(0,0)}{2 |g(0,0)|} \,y^2 + \cO \left( y^3 \right)$.
For this reason it is helpful to perform asymptotic expansions in $\sqrt{x}$ and $y$,
and so we write
\begin{equation}
\begin{split}
f(x,y) &= a_1 x + a_2 y + a_5 y^2 + \cO \left( \left( \sqrt{x} + |y| \right)^3 \right), \\
g(x,y) &= b_0 + b_2 y + \cO \left( \left( \sqrt{x} + |y| \right)^2 \right),
\end{split}
\label{eq:foldfg}
\end{equation}
where the coefficients have been labelled in a way that is consistent with other expressions in this paper.
Notice $a_2 > 0$ and $b_0 < 0$ by \eqref{eq:foldCond}.
Let
\begin{equation}
\sigma_{\rm fold} = \frac{a_1}{b_0} + \frac{b_2}{b_0} - \frac{a_5}{a_2},
\label{eq:foldsigma}
\end{equation}
and
\begin{widetext}
\begin{align}
\chi_{\rm fold} &= -\frac{a_1 (a_1 + b_2)(a_1 + 2 b_2)}{b_0^3}
+ \frac{(a_1 + b_2)(4 a_1 + 3 b_2)}{b_0^2} \,\sigma_{\rm fold}
- \frac{5 (a_1 + b_2)}{b_0} \,\sigma_{\rm fold}^2
+ \frac{40}{9} \,\sigma_{\rm fold}^3 \nonumber \\
&\quad+ \bigg[ \frac{1}{b_0} \left( \frac{a_1}{b_0} - \frac{a_5}{a_2} \right) \frac{\partial^2 f}{\partial x \partial y}
- \frac{1}{6 a_2} \left( \frac{2 a_1}{b_0} + \frac{2 b_2}{b_0} - \frac{5 a_5}{a_2} \right) \frac{\partial^3 f}{\partial y^3}
+ \frac{a_2}{b_0^2} \left( \frac{a_1}{b_0} + \frac{b_2}{b_0} \right) \frac{\partial g}{\partial x} \nonumber \\
&\quad+ \frac{1}{2 b_0} \left( \frac{a_1}{b_0} - \frac{b_2}{b_0} - \frac{2 a_5}{a_2} \right) \frac{\partial^2 g}{\partial y^2}
- \frac{a_2}{b_0^2} \,\frac{\partial^2 f}{\partial x^2}
+ \frac{1}{2 b_0} \,\frac{\partial^3 f}{\partial x \partial y^2}
- \frac{1}{8 a_2} \,\frac{\partial^4 f}{\partial y^4}
- \frac{a_2}{b_0^2} \,\frac{\partial^2 g}{\partial x \partial y}
+ \frac{1}{2 b_0} \,\frac{\partial^3 g}{\partial y^3} \bigg] \bigg|_{x = y = 0} \,.
\label{eq:foldtau}
\end{align}
\end{widetext}

\begin{lemma}
Consider \eqref{eq:smoothLemmasODE} where $F$ is $C^5$
and suppose \eqref{eq:foldCond} is satisfied.
Then
\begin{align}
P_{\rm fold}(r) &= -r + \frac{2 \sigma_{\rm fold}}{3} \,r^2 - \frac{4 \sigma_{\rm fold}^2}{9} \,r^3
+ \frac{2 \chi_{\rm fold}}{15} \,r^4 \nonumber \\
&\quad+ \co \left( r^4 \right),
\label{eq:foldP} \\
T_{\rm fold}(r) &= \frac{-2}{b_0} \,r + \cO \left( r^2 \right).
\label{eq:foldT}
\end{align}
\label{le:fold}
\end{lemma}

Lemma \ref{le:fold} is proved in Appendix \ref{app:lemmas} via brute-force
asymptotic expansions, except, for brevity, we have only stated enough terms to obtain $P_{\rm fold}$ to second order.
The full expression \eqref{eq:foldP} is obtained by simply including more terms in the expansions.
Indeed, while Lemma \ref{le:fold} is used below in several places,
the $r^4$-term in \eqref{eq:foldP} is only required for HLB 10.

It is interesting to note that quantity $\frac{P_{\rm fold}(r) + r}{r}$ has the same form
as the Poincar\'e map associated with a Hopf bifurcation, or a perturbed centre more generally \cite{AnLe71,DuLl06}.
For \eqref{eq:foldP} the $\cO \left( \left( \sqrt{x} + |y| \right)^2 \right)$ terms in $f$
and $\cO \left( \left( \sqrt{x} + |y| \right)^1 \right)$ terms in $g$
behave like linear terms and determine the value of $\sigma_{\rm fold}$.
If $\sigma_{\rm fold} = 0$, then the $r^2$ and $r^3$ terms of \eqref{eq:foldP} are both zero
in the same way that the first non-zero Lyapunov constant of a perturbed centre corresponds to an odd power of $r$.
For this reason the next two orders of terms contribute to the value of $\chi_{\rm fold}$,
and so, in particular, both $\frac{\partial g}{\partial x}$ and $\frac{\partial^4 f}{\partial y^4}$ appear in \eqref{eq:foldtau}.

\begin{figure}[b!]
\begin{center}
\includegraphics[width=8cm]{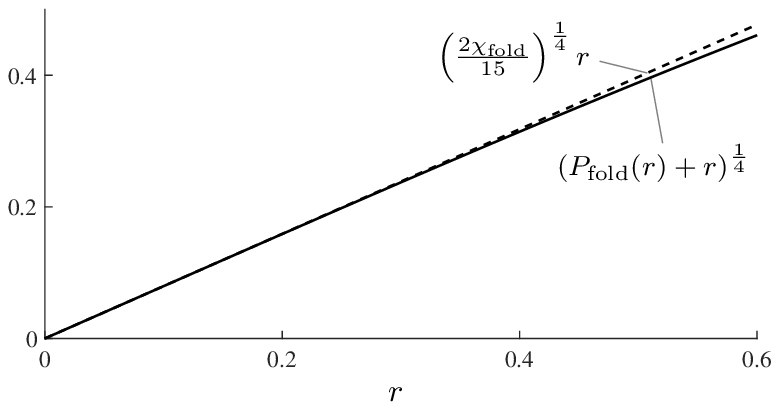}
\caption{
A numerical verification of Lemma \ref{le:fold} for the system \eqref{eq:foldExample}.
\label{fig:foldExample}
} 
\end{center}
\end{figure}

As a minimal example, consider
\begin{equation}
\begin{split}
\dot{x} &= y - y^4, \\
\dot{y} &= -1,
\end{split}
\label{eq:foldExample}
\end{equation}
for which $\sigma_{\rm fold} = 0$ and $\chi_{\rm fold} = 3$.
By Lemma \ref{le:fold}, $P_{\rm fold}(r) = -r + \frac{2 \chi_{\rm fold}}{15} \,r^4 + \co \left( r^4 \right)$.
Fig.~\ref{fig:foldExample} shows $\left( P_{\rm fold}(r)+r \right)^{\frac{1}{4}}$
(obtained by numerical simulation) and its linear appromixation.

\subsection{The return map for an affine system}
\label{sub:affine}

Here we suppose \eqref{eq:smoothLemmasODE} has the form
\begin{equation}
\begin{bmatrix} \dot{x} \\ \dot{y} \end{bmatrix} =
\begin{bmatrix} a_1 x + a_2 y \\ b_0 + b_1 x + b_2 y \end{bmatrix},
\label{eq:affineODE}
\end{equation}
and derive a return map on $x=0$ assuming that the Jacobian matrix
\begin{equation}
A = \begin{bmatrix} a_1 & a_2 \\ b_1 & b_2 \end{bmatrix},
\end{equation}
has complex eigenvalues:
\begin{equation}
{\rm eig}(A) = \lambda \pm \ri \omega, \quad
{\rm with~} \lambda \in \mathbb{R}, \omega > 0.
\label{eq:affineEigCond}
\end{equation}
The formulas obtained here are used in several places below.
This is because the dynamics near a BEB are well-approximated by a piecewise-linear system,
meaning each smooth component is affine (linear plus a constant term).
For the affine systems below, we are always able to find a coordinate change
that removes the constant term in the $\dot{x}$ equation, as in \eqref{eq:affineODE}.
For HLBs 2 and 13 we require the eigenvalues of $A$ to be real-valued.
This is remedied by using a purely imaginary value for $\omega$, but here we assume $\omega > 0$.

The assumption $\omega > 0$ ensures $A$ is invertible, and so \eqref{eq:affineODE} has the unique equilibrium
\begin{equation}
\begin{bmatrix} x^* \\ y^* \end{bmatrix} =
\frac{\kappa}{\omega} \begin{bmatrix} a_2 \\ -a_1 \end{bmatrix},
\label{eq:affineEq}
\end{equation}
where
\begin{equation}
\kappa = \frac{b_0 \omega}{\lambda^2 + \omega^2}.
\label{eq:affinekappa}
\end{equation}
Since the $\dot{x}$ equation has no constant term and $a_2 \ne 0$ (a consequence of $\omega > 0$),
if $b_0 \ne 0$ then
the system has a unique fold at the origin (treating $x=0$ as a switching manifold).

We now assume $a_2 > 0$, so that for any $r > 0$ the forward orbit of $(0,r)$ immediately enters $x > 0$.
For any $r > 0$, we let $P_{\rm affine}(r)$ be the $y$-value
of the next intersection of the forward orbit of $(0,r)$ with $x = 0$.
If this orbit does not reintersect $x=0$ we say $P_{\rm affine}(r)$ is undefined.
Also we let $T_{\rm affine}(r)$ denote the corresponding evolution time.
The functions $P_{\rm affine}$ and $T_{\rm affine}$ depend on
$a_1$, $a_2$, $b_0$, $b_1$, and $b_2$,
but, as we will see,
they are completely determined by the values of $\lambda$, $\omega$, and $b_0$.
Accordingly we write $P_{\rm affine}(r;\lambda,\omega,b_0)$ and $T_{\rm affine}(r;\lambda,\omega,b_0)$.

\begin{figure}[b!]
\begin{center}
\includegraphics[width=8cm]{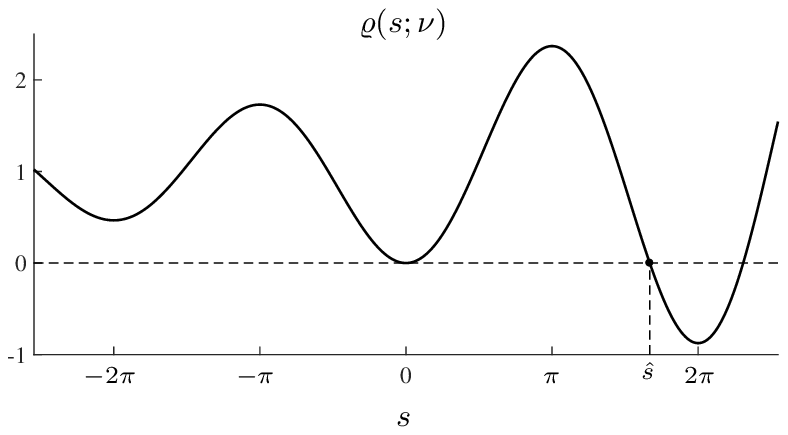}
\caption{
\label{fig:schemAuxFunc}
The auxiliary function \eqref{eq:auxFunc} with $\nu = 0.1$.
} 
\end{center}
\end{figure}

We cannot obtain an explicit closed form expression for $T_{\rm affine}$,
so write $P_{\rm affine}$ and $T_{\rm affine}$ implicitly via the auxiliary function
\begin{equation}
\varrho(s;\nu) = 1 - \re^{\nu s} \left( \cos(s) - \nu \sin(s) \right),
\label{eq:auxFunc}
\end{equation}
which was used effectively in Chapter VIII of \cite{AnVi66}.
Fig.~\ref{fig:schemAuxFunc} provides a plot of $\varrho(s;\nu)$ for a typical value of $\nu > 0$.
The function $\varrho$ attains its local maximum and minimum values at integer multiples of $\pi$ because
\begin{equation}
\frac{\partial \varrho}{\partial s} = \left( 1 + \nu^2 \right) \re^{\nu s} \sin(s).
\label{eq:auxFuncDeriv}
\end{equation}
For any $\nu > 0$, we have $\varrho(s;\nu) > 0$ for all nonzero $s \le \pi$,
and there exists unique $\hat{s} = \hat{s}(\nu) \in (\pi,2 \pi)$ such that
\begin{equation}
\varrho(\hat{s};\nu) = 0.
\label{eq:shat}
\end{equation}
The case $\nu < 0$ can be understood through the symmetry relation $\varrho(-s,-\nu) = \varrho(s,\nu)$.

\begin{lemma}
Consider \eqref{eq:affineODE} with $a_2 > 0$ and $b_0 \ne 0$,
and suppose \eqref{eq:affineEigCond} is satisfied.
For all $r > 0$, $T = T_{\rm affine}(r;\lambda,\omega,b_0)$ and $P = P_{\rm affine}(r;\lambda,\omega,b_0)$ satisfy
\begin{align}
r &= \frac{-\kappa \re^{-\lambda T} \varrho \left( \omega T; \frac{\lambda}{\omega} \right)}{\sin(\omega T)}, \label{eq:affineT2} \\
P &= \frac{\kappa \re^{\lambda T} \varrho \left( \omega T; -\frac{\lambda}{\omega} \right)}{\sin(\omega T)}. \label{eq:affineP2}
\end{align}
If $b_0 < 0$ [$b_0 > 0$] then $\frac{d T}{d r} > 0$ [$\frac{d T}{d r} < 0$] for all $r > 0$; also
\begin{equation}
\frac{d P}{d r} = \frac{r}{P} \,\re^{2 \lambda T}.
\label{eq:PaffineDeriv2}
\end{equation}
As $r \to \infty$, $T \to \frac{\pi}{\omega}$
and $\frac{P}{r} \to -\re^{\frac{\lambda \pi}{\omega}}$.
\begin{enumerate}[label=\Roman{*}),ref=\Roman{*}]
\item
\label{it:b0neg}
If $b_0 < 0$, then as $r \to 0$, $T \to 0$ and $P \to 0$.
\item
\label{it:b0poslambdaneg}
If $b_0 > 0$ and $\lambda < 0$,
let $\hat{r} = -\frac{b_0}{\omega} \,\re^{\nu \hat{s}(\nu)} \sin \left( \hat{s}(\nu) \right)$,
where $\nu = -\frac{\lambda}{\omega}$.
Then $T(\hat{r}) = \frac{\hat{s}(\nu)}{\omega}$, $P(\hat{r}) = 0$,
and $P$ is undefined for $r \in (0,\hat{r})$ ($P$ is defined in all other cases).
\item
\label{it:b0poslambdapos}
If $b_0 > 0$ and $\lambda > 0$,
then as $r \to 0$, $T \to \frac{\hat{s}(\nu)}{\omega}$
and $P \to \frac{b_0}{\omega} \,\re^{\nu \hat{s}(\nu)} \sin \left( \hat{s}(\nu) \right)$,
where $\nu = \frac{\lambda}{\omega}$.
\end{enumerate}
\label{le:affine}
\end{lemma}

\begin{figure*}
\begin{center}
\includegraphics[width=16.6cm]{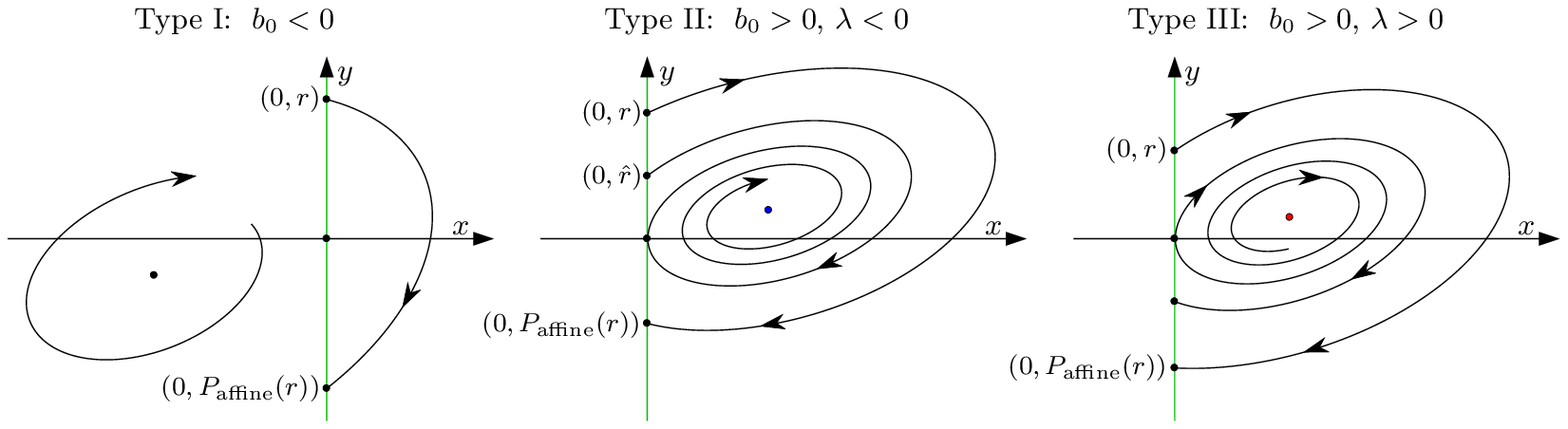}
\caption{
Phase portraits of the affine system \eqref{eq:affineODE}.
The return map $P_{\rm affine}$ conforms to one of three types, see Lemma \ref{le:affine}.
\label{fig:schemPoinAffine}
} 
\end{center}
\end{figure*}

Lemma \ref{le:affine} is proved in Appendix \ref{app:lemmas} via direct calculations.
There are three cases for the qualitative nature of $T_{\rm affine}$ and $P_{\rm affine}$,
see Fig.~\ref{fig:schemPoinAffine} (this ignores the special case $\lambda = 0$).
If $b_0 < 0$, we say $T_{\rm affine}$ and $P_{\rm affine}$ are of Type I.
The equilibrium $(x^*,y^*)$ lies in $x<0$,
so $(x^*,y^*)$ is virtual and the origin is an invisible fold
(in the context of the Filippov system \eqref{eq:FilippovBackg}).
If $b_0 > 0$ and $\lambda < 0$, we say $T_{\rm affine}$ and $P_{\rm affine}$ are of Type II.
Here $(x^*,y^*)$ is an admissible stable focus and the origin is a visible fold.
For all $r \in (0,\hat{r})$, $P_{\rm affine}$ is undefined because the forward orbit of $(0,r)$ converges
to $(x^*,y^*)$ without reintersecting $x=0$.
Finally if $b_0 > 0$ and $\lambda > 0$, we say $T_{\rm affine}$ and $P_{\rm affine}$ are of Type III.
Here $(x^*,y^*)$ is an admissible unstable focus and the origin is a visible fold.

Lemma \ref{le:affine} assumes $b_0 \ne 0$.
If $b_0 = 0$, then $P_{\rm affine}$ and $T_{\rm affine}$
are well-defined, but the formulas \eqref{eq:affineT2} and \eqref{eq:affineP2}
are ill-posed because they are of the form $\frac{0}{0}$.
In this case we can use Lemma \ref{le:focus}.
Taking $r \to \infty$ is analogous to taking $b_0 \to 0$,
and for this reason the $r \to \infty$ asymptotics in Lemma \ref{le:affine}
match the leading order terms of \eqref{eq:focusP} and \eqref{eq:focusT}.

\section{Boundary equilibrium bifurcations in continuous systems}
\setcounter{equation}{0}
\setcounter{figure}{0}
\setcounter{table}{0}
\label{sec:pwsc}

Here we consider piecewise-smooth systems of the form
\begin{equation}
\begin{bmatrix} \dot{x} \\ \dot{y} \end{bmatrix} =
\begin{cases} F_L(x,y;\mu), & x \le 0, \\ F_R(x,y;\mu), & x \ge 0. \end{cases}
\label{eq:pwscODE}
\end{equation}
The non-strict inequalities mean that
$F_L(0,y;\mu) = F_R(0,y;\mu)$, for all values of $y$ and $\mu$.
That is \eqref{eq:pwscODE} is continuous on $x=0$ and does not exhibit sliding motion.

As the value of $\mu$ is varied, a BEB (boundary equilibrium bifurcation) occurs
when a regular equilibrium collides with $x=0$.
For clarity, let us suppose the bifurcation occurs at the origin when $\mu = 0$.
For the continuous system \eqref{eq:pwscODE},
at $\mu = 0$ the origin is a regular equilibrium
of both left and right half-systems.
Assuming certain genericity conditions are satisfied,
each half-system has a unique regular equilibrium locally.
Each equilibrium is admissible for either small $\mu < 0$ or small $\mu > 0$
(we say the equilibrium is admissible for `one sign of $\mu$').
Consequently there are two cases:
the equilibria may be admissible for different signs of $\mu$
(called {\em persistence} because a single equilibrium appears to persist)
or the same sign of $\mu$
(called a {\em nonsmooth-fold} as it resembles a saddle-node bifurcation) \cite{DiBu08,Si10}.
The distance of each equilibrium from $x=0$ is asymptotically proportional to $|\mu|$.

Other invariant sets can be created in BEBs \cite{Le06,LeNi04}.
These include limit cycles (as discussed below)
and, for systems of three or more dimensions, chaotic sets \cite{DiNo08,Si16c}.
The diameter of all bounded invariant sets (except equilibria)
is asymptotically proportional to $|\mu|$.

For two-dimensional systems, equilibria and limit cycles are the only possible bounded invariant sets
and there are ten topologically distinct BEBs (assuming genericity).
These can be grouped into the following five scenarios:
(i) persistence with the equilibria being of the same stability,
(ii) persistence with different stability,
(iii) persistence with different stability and a limit cycle,
(iv) nonsmooth-fold, and
(v) nonsmooth-fold with a limit cycle, see \cite{Si10,SiMe12}.
We view scenario (iii) as Hopf-like,
for which there are two topologically distinct cases.
Either both equilibria are foci, HLB 1,
or one is a focus and the other is a node, HLB 2.
In both cases one equilibrium is attracting and the other is repelling.
For HLB 1, the stability of the limit cycle is determined by the sign of
$\frac{\lambda_L}{\omega_L} + \frac{\lambda_R}{\omega_R}$,
where the eigenvalues associated with the equilibria are
$\lambda_L \pm \ri \omega_L$ and $\lambda_R \pm \ri \omega_R$.
For HLB 2, the stability of the limit cycle is the same as that of the node.
In both cases the stability of the limit cycle is the same as the stability of the origin when $\mu = 0$
(an important stability problem for many switched control systems \cite{CaHe03,DiCa05}).
	
In a neighbourhood of a BEB, the system is approximately piecewise-linear.
The {\em local} behaviour of the BEB is in part governed by
{\em global} properties of the piecewise-linear approximation.
As a result, in the limit $\mu \to 0$ the limit cycle for HLBs 1 and 2
is a rather irregular union of orbit segments of the left and right half-systems.
Also the formula for the period of the limit cycle is not simple and can only be stated implicitly, see \eqref{eq:periodpwscFocusFocus}.

The earliest rigorous analyses of HLBs appear to be for piecewise-linear systems.
In Chapter VIII of \cite{AnVi66}
a limit cycle is identified in a piecewise-linear model of a valve generator
by constructing a Poincar\'e map (there called a point transformation)
and using the auxiliary function $\varrho$ \eqref{eq:auxFunc}.
This model is actually Filippov; it is given below as an example in \S\ref{sub:degenerateBEB}
and	described also in \cite{Ye86}.

Motivated by circuit systems \cite{Ch94,ChDe86},
Lum and Chua \cite{LuCh91} performed an extensive analysis of two-dimensional piecewise-linear continuous ODE systems.
For certain scenarios they were able to prove the existence of an attracting annulus
and conjectured that within the annulus there exists a unique stable limit cycle.
This was subsequently proved by Friere and coworkers in \cite{FrPo98}
by reducing the number of parameters in the equations,
using the function $\varrho$,
and meticulously considering all possible cases for the nature of the equilibria. 
For piecewise-linear systems they established HLB 1 in \cite{FrPo97} and HLB 2 in \cite{FrPo98}.

If \eqref{eq:pwscODE} is not piecewise-linear,
then robust dynamics of its piecewise-linear approximation are
exhibited by \eqref{eq:pwscODE} near the BEB \cite{DiBu08,Si10}.
Intuitively, hyperbolic limit cycles are robust phenomena, but this needs to be proved.
This was achieved in \cite{SiMe07} by demonstrating that the limit cycle
can be expressed as a hyperbolic fixed point of a smooth Poincar\'e map,
except their proof lacks the necessary scaling to apply the IFT correctly
(this is fixed in the proof in Appendix \ref{app:beb}).

\subsection{The focus/focus case --- HLB 1}
\label{sub:focusFocusBEB}

For each $J \in \{ L,R \}$, we write
\begin{equation}
F_J(x,y;\mu) = \begin{bmatrix} f_J(x,y;\mu) \\ g_J(x,y;\mu) \end{bmatrix},
\label{eq:FJpwsc}
\end{equation}
and suppose
\begin{equation}
f_J(0,0;0) = g_J(0,0;0) = 0,
\label{eq:pwscEqCond}
\end{equation}
so that the origin is a boundary equilibrium when $\mu = 0$.
We suppose the origin is an unstable focus of the left half-system
and a stable focus of the right half-system, that is
\begin{equation}
\begin{split}
{\rm eig}(\rD F_L(0,0;0)) &= \lambda_L \pm \ri \omega_L \,, \quad
{\rm with~} \lambda_L > 0, \omega_L > 0, \\
{\rm eig}(\rD F_R(0,0;0)) &= \lambda_R \pm \ri \omega_R \,, \quad
{\rm with~} \lambda_R < 0, \omega_R > 0. \\
\end{split}
\label{eq:pwscEigCondFocusFocus}
\end{equation}
Locally each half-system has a unique equilibrium with an $x$-value
equal to $\frac{-\beta \mu}{\lambda_J^2 + \omega_J^2} + \cO \left( \mu^2 \right)$, where
\begin{equation}
\beta = \left( \frac{\partial f_L}{\partial \mu} \frac{\partial g_L}{\partial y}
- \frac{\partial g_L}{\partial \mu} \frac{\partial f_L}{\partial y} \right)
\bigg|_{x = y = \mu = 0}.
\label{eq:pwscTransCond}
\end{equation}
Here we have used the continuity of \eqref{eq:pwscODE}
(the derivatives of $F_L$ in \eqref{eq:pwscTransCond} are the same for $F_R$).
Consequently $\beta \ne 0$ is the transversality condition for HLB 1.
Below we assume $\beta > 0$ without loss of generality.
As mentioned above the stability of the limit cycle emanating from the origin
is determined by the sign of
\begin{equation}
\alpha = \frac{\lambda_L}{\omega_L} + \frac{\lambda_R}{\omega_R}.
\label{eq:pwscNondegCond}
\end{equation}
Recall $\Omega_L$ and $\Omega_R$ denote the left and right half-planes, see \eqref{eq:OmegaLR}.

\begin{theorem}[HLB 1]
Consider \eqref{eq:pwscODE} where $F_L$ and $F_R$ are $C^2$.
Suppose \eqref{eq:pwscEqCond} and \eqref{eq:pwscEigCondFocusFocus} are satisfied and $\beta > 0$.
In a neighbourhood of $(x,y;\mu) = (0,0;0)$,
\begin{enumerate}
\item
there exists a unique stationary solution:
a stable focus in $\Omega_R$ for $\mu < 0$, and an unstable focus in $\Omega_L$ for $\mu > 0$,
\item
if $\alpha < 0$ [$\alpha > 0$] there exists a unique stable [unstable] limit cycle
for $\mu > 0$ [$\mu < 0$], and no limit cycle for $\mu < 0$ [$\mu > 0$].
\end{enumerate}
The minimum and maximum $x$ and $y$-values of the limit cycle are asymptotically proportional to $|\mu|$,
and its period is $T = T_L + T_R + \cO(\mu)$ where $T_L$ and $T_R$ satisfy
\begin{equation}
\begin{split}
\frac{\omega_L G \left( \omega_L T_L; -\frac{\lambda_L}{\omega_L} \right)}
{\lambda_L^2 + \omega_L^2} +
\frac{\omega_R G \left( \omega_R T_R; \frac{\lambda_R}{\omega_R} \right)}
{\lambda_R^2 + \omega_R^2} &= 0, \\
\frac{\omega_L G \left( \omega_L T_L; \frac{\lambda_L}{\omega_L} \right)}
{\lambda_L^2 + \omega_L^2} +
\frac{\omega_R G \left( \omega_R T_R; -\frac{\lambda_R}{\omega_R} \right)}
{\lambda_R^2 + \omega_R^2} &= 0,
\end{split}
\label{eq:periodpwscFocusFocus}
\end{equation}
and $G(s;\nu) = \frac{\re^{-\nu s} \varrho(s;\nu)}{\sin(s)}$,
where $\varrho$ is the auxiliary function \eqref{eq:auxFunc}.
\label{th:pwscFocusFocus}
\end{theorem}

HLB 1 is exhibited by the McKean neuron model:
a piecewise-linear version of the Fitzhugh-Nagumo model for an excitable neuron \cite{Mc70}.
Following \cite{Co08} we write the model as
\begin{equation}
\begin{split}
\dot{v} &= q(v) - w + I, \\
\dot{w} &= b (v - c w),
\end{split}
\label{eq:McKean}
\end{equation}
where $v$ represents membrane voltage, $w$ is an effective gating variable, and
\begin{equation}
q(v) = \begin{cases}
-v, & v \le \frac{a}{2}, \\
v-a, & \frac{a}{2} \le v \le \frac{a+1}{2}, \\
1-v, & v \ge \frac{a+1}{2}.
\end{cases}
\label{eq:McKeanF}
\end{equation}
The function $q$ mimics the cubic caricature of the Fitzhugh-Nagumo model.
The McKean model is one of many piecewise-linear models of neuron dynamics.
As a general rule the dynamics of these models may not match experimentally observed dynamics
with as much quantitative accuracy as more sophisticated high-dimensional models,
but they often fit qualitatively and are amenable to an exact analysis \cite{DeGu16,RoCo12,ToGe03}.
This is particularly helpful for understanding neural networks \cite{Co08,CoLa18,CoTh16}.

\begin{figure*}
\begin{center}
\includegraphics[width=16.6cm]{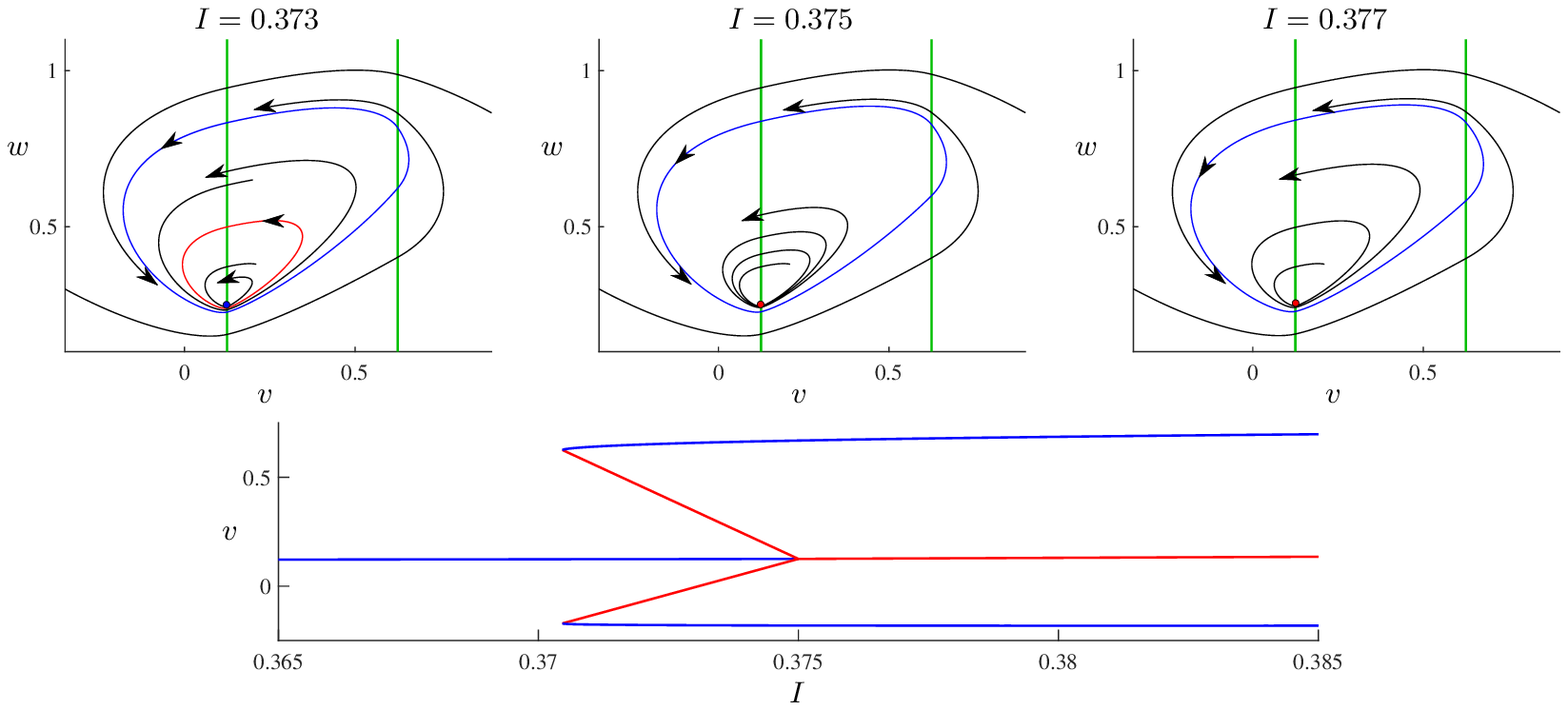}
\caption{
An illustration of HLB 1 for the McKean neuron model \eqref{eq:McKean} with \eqref{eq:McKeanParam}.
By decreasing the value of $I$
an unstable focus collides with the switching manifold $v = \frac{a}{2}$ when $I = 0.375$
where it turns into a stable focus and emits an unstable limit cycle.
The limit cycle subsequently undergoes a saddle-node bifurcation and becomes a stable relaxation oscillation.
\label{fig:allMcKean_5}
} 
\end{center}
\end{figure*}

Here we treat the external stimulus $I$ as a bifurcation parameter.
With the remaining parameter values fixed at
\begin{equation}
a = 0.25, \qquad
b = 0.5, \qquad
c = 0.5,
\label{eq:McKeanParam}
\end{equation}
a subcritical HLB 1 occurs at $I = 0.375$, see Fig.~\ref{fig:allMcKean_5}.
As the value of $I$ is decreased through $I = 0.375$
an unstable focus in $\frac{a}{2} < v < \frac{a+1}{2}$
collides with the switching manifold $v = \frac{a}{2}$ when $I = 0.375$
and transitions to a stable focus in $v < \frac{a}{2}$.
The eigenvalues associated with the foci are
$0.3750 \pm 0.3307 \ri$ and $-0.6250 \pm 0.5995 \ri$ (to four decimal places).
Thus $\alpha \approx 0.0913 > 0$ and so an unstable limit cycle is created. 
As the value of $I$ is decreased further,
the unstable limit cycle undergoes a grazing bifurcation with the
other switching manifold, $v = \frac{a+1}{2}$,
and is very shortly followed by a saddle-node bifurcation
at which the limit cycle collides and annihilates with a coexisting stable limit cycle. 
In summary, a subcritical HLB 1 and a saddle-node bifurcation of limit cycles
bound a small interval of bistability.
With smaller values of $b$ the McKean model exhibits HLB 2.

\subsection{The focus/node case --- HLB 2}
\label{sub:focusNodeBEB}

We now suppose that the origin is a stable node of the right half-system when $\mu = 0$.
We write the eigenvalues as
\begin{equation}
\begin{split}
{\rm eig}(\rD F_L(0,0;0)) &= \lambda_L \pm \ri \omega_L \,, \quad
{\rm with~} \lambda_L > 0, \omega_L > 0, \\
{\rm eig}(\rD F_R(0,0;0)) &= \lambda_R \pm \eta_R \,, \quad
{\rm with~} 0 < \eta_R < -\lambda_R \,.
\end{split}
\label{eq:pwscEigCondFocusNode}
\end{equation}

\begin{theorem}[HLB 2]
Consider \eqref{eq:pwscODE} where $F_L$ and $F_R$ are $C^2$.
Suppose \eqref{eq:pwscEqCond} and \eqref{eq:pwscEigCondFocusNode} are satisfied and $\beta > 0$.
In a neighbourhood of $(x,y;\mu) = (0,0;0)$,
\begin{enumerate}
\item
there exists a unique stationary solution:
a stable node in $\Omega_R$ for $\mu < 0$, and an unstable focus in $\Omega_L$ for $\mu > 0$,
\item
there exists a unique stable limit cycle for $\mu > 0$,
and no limit cycle for $\mu < 0$.
\end{enumerate}
The minimum and maximum $x$ and $y$-values of the limit cycle are asymptotically proportional to $\mu$,
and its period is $T = T_L + T_R + \cO(\mu)$ where $T_L$ and $T_R$ satisfy
\eqref{eq:periodpwscFocusFocus} using $\omega_R = \ri \eta_R$.
\label{th:pwscFocusNode}
\end{theorem}

Notice how Theorem \ref{th:pwscFocusNode} contains no criticality constant $\alpha$.
The stability of the limit cycle is determined from
our assumption that the focus is unstable and the node is stable.
The case of a stable focus and an unstable node (giving an unstable limit cycle) can be treated via time-reversal.

To illustrate HLB 2 we consider the reduced ocean circulation model of \cite{RoSa17}:
\begin{equation}
\begin{split}
\dot{\bar{y}} &= \bar{\mu} - \bar{y} - A \bar{y} |\bar{y} - 1|, \\
\dot{\bar{\mu}} &= \delta_0 (\lambda - \bar{y}).
\end{split}
\label{eq:RoSa16}	
\end{equation}
This is equation (18) of \cite{RoSa17}
except bars have been added to the variables to avoid confusion with the present notation.
The variable $\bar{y}$ represents the difference in the salinity of the ocean
near the equator from that near the poles, and $\bar{\mu}$ is a forcing ratio.
The parameters $\lambda$, $A$, and $\delta_0$ each take positive real values.

The model is a reduced version of a box model of ocean circulation developed by Stommel \cite{St61}.
The circulation is assumed to occur between two regions (boxes): one equatorial and one polar.
The regions are assumed to be well-mixed,
so the dynamics in the regions is assumed to depend on the magnitude of the circulation, but not its direction.
This manifests as $|\bar{y}-1|$ in \eqref{eq:RoSa16} making the model piecewise-smooth.

\begin{figure*}
\begin{center}
\includegraphics[width=16.6cm]{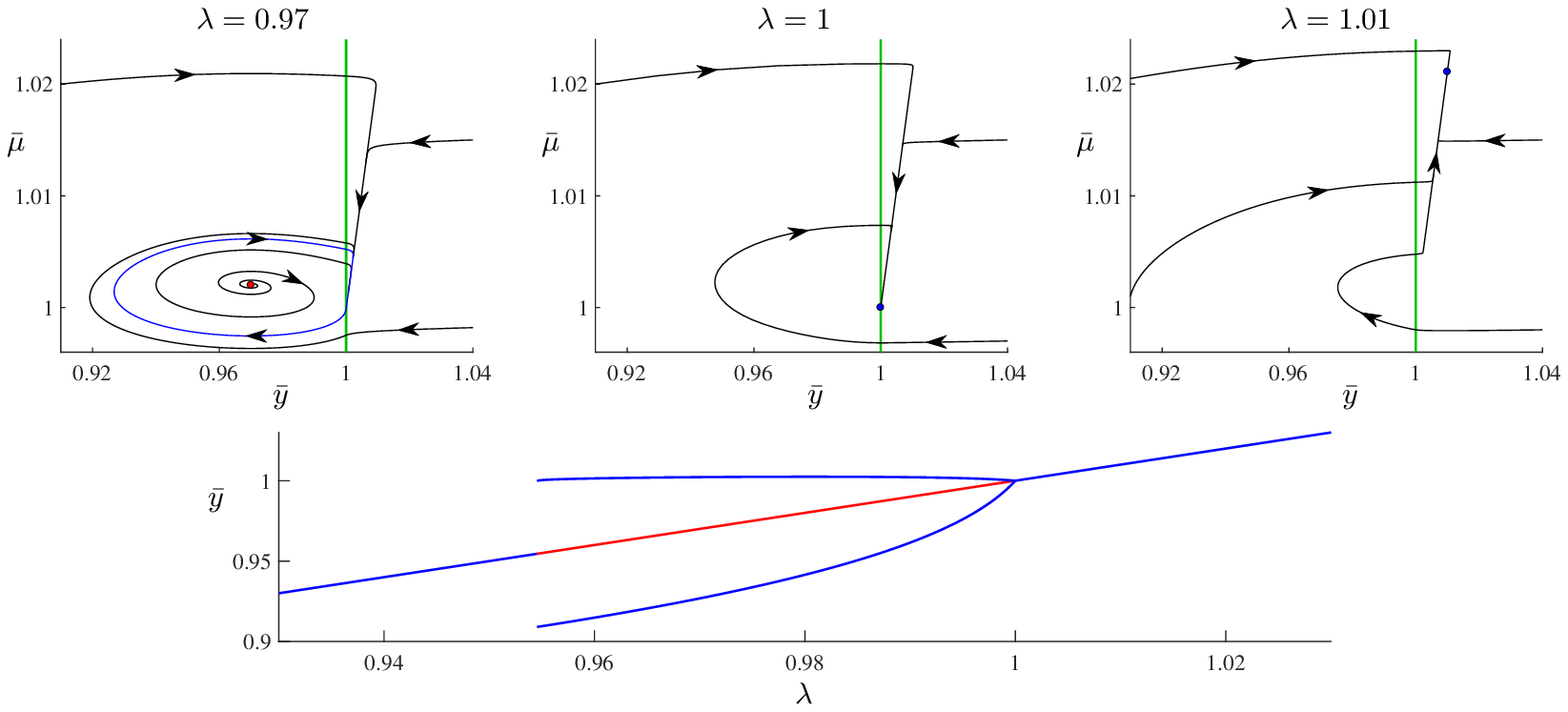}
\caption{
An illustration of HLB 2 for ocean circulation model 
\eqref{eq:RoSa16} with $A = 1.1$ and $\delta_0 = 0.01$.
By decreasing the value of $\lambda$ a stable node collides with the switching manifold $\bar{y} = 1$
when $\lambda = 1$ where it turns into an unstable focus and emits a stable limit cycle.
\label{fig:allRoSa16}
} 
\end{center}
\end{figure*}

The system exhibits a BEB when $\lambda = 1$, see Fig.~\ref{fig:allRoSa16}.
With $\lambda > 1$ there is a stable node in $\bar{y} > 1$,
and if $1 < A < 1 + 2 \sqrt{\delta_0}$, then with $\lambda < 1$ there is an unstable focus in $\bar{y} < 1$.
In this case an instance of HLB 2 occurs at $\lambda = 1$
and so a stable limit cycle is created.
For the parameter values of Fig.~\ref{fig:allRoSa16},
the focus subsequently changes stability at $\lambda = \frac{21}{22} \approx 0.9545$.
Here the limit cycle is destroyed in a degenerate fashion
(at $\lambda = \frac{21}{22}$ there exists an uncountable family of periodic orbits).
If the value of $A$ is increased past $1 + 2 \sqrt{\delta_0}$
the unstable focus changes to an unstable node and a large amplitude limit cycle is created
in the BEB at $\lambda = 1$.
Related to this, HLB 2 has been well-documented in PWL systems in the context of canards
\cite{DeFr13,FeDe16,PoRo15,SiKu11}.

\section{Boundary equilibrium bifurcations in Filippov systems}
\setcounter{equation}{0}
\setcounter{figure}{0}
\setcounter{table}{0}
\label{sec:beb}

Here we consider Filippov systems of the form
\begin{equation}
\begin{bmatrix} \dot{x} \\ \dot{y} \end{bmatrix} =
\begin{cases} F_L(x,y;\mu), & x < 0, \\ F_R(x,y;\mu), & x > 0, \end{cases}
\label{eq:FilippovBEB}
\end{equation}
and write
\begin{equation}
F_J(x,y;\mu) = \begin{bmatrix} f_J(x,y;\mu) \\ g_J(x,y;\mu) \end{bmatrix},
\label{eq:FJ2}
\end{equation}
for each $J \in \{ L,R \}$.

Let us suppose an equilibrium of the left half-system
collides with $x=0$ at the origin when $\mu = 0$.
Since $F_L$ and $F_R$ are essentially independent, we would expect that the origin
is not an equilibrium of the right half-system when $\mu = 0$.
In this way, generic BEBs of Filippov systems are different
to those of continuous systems (described in \S\ref{sec:pwsc}).
But as with continuous systems locally there are exactly two equilibria:
the regular equilibrium and a pseudo-equilibrium, and these coincide at the bifurcation.
The equilibria are either admissible for different signs of $\mu$ ({\em persistence})
or the same sign of $\mu$ ({\em nonsmooth-fold}) \cite{DiBu08,Si18d}.
The distance of each equilibrium from the origin,
and the diameter of any other bounded invariant set created in the bifurcation,
is asymptotically proportional to $|\mu|$.

For two-dimensional systems there are $12$ topologically distinct BEBs
(assuming genericity) \cite{Gl16d,HoHo16}.
Of these, only one is Hopf-like in that it corresponds to persistence and creates a local limit cycle.
This is HLB 3, called ${\rm BF}_3$ in \cite{KuRi03}, see \S\ref{sub:genericBEB}.
In \S\ref{sub:degenerateBEB} we consider the degenerate case
that the origin is also an equilibrium of the right half-system at $\mu = 0$ (HLB 4)
as this occurs in some applications.

\subsection{The generic case --- HLB 3}
\label{sub:genericBEB}

We suppose that the origin is an unstable focus of 
the left half-system of \eqref{eq:FilippovBEB} when $\mu = 0$, that is
\begin{equation}
f_L(0,0;0) = g_L(0,0;0) = 0,
\label{eq:FilippovEqCond}
\end{equation}
and
\begin{equation}
{\rm eig}(\rD F_L(0,0;0)) = \lambda_L \pm \ri \omega_L \,, \quad
{\rm with~} \lambda_L > 0, \omega_L > 0.
\label{eq:FilippovEigCondFocus}
\end{equation}
Locally the left half-system has a unique equilibrium
with an $x$-value of $\frac{-\beta \mu}{\lambda_L^2 + \omega_L^2} + \cO \left( \mu^2 \right)$, where
\begin{equation}
\beta = \left( \frac{\partial f_L}{\partial \mu} \frac{\partial g_L}{\partial y}
- \frac{\partial g_L}{\partial \mu} \frac{\partial f_L}{\partial y} \right)
\bigg|_{x = y = \mu = 0},
\label{eq:FilippovTransCond}
\end{equation}
which is identical to \eqref{eq:pwscTransCond}.

Since $\lambda_L > 0$, for a limit cycle to be created at $\mu = 0$ we require $a_{0R} < 0$ where
\begin{equation}
a_{0R} = f_R(0,0;0).
\end{equation}
Then, near the origin, the switching manifold
is the union of an attracting sliding region, a crossing region, and a fold
(or boundary focus when $\mu = 0$).
By evaluating the sliding vector field \eqref{eq:gslide2} when $\mu = 0$, we obtain
$g_{\rm slide}(y;0) = \frac{\gamma}{-a_{0R}} \,y + \cO \left( y^2 \right)$, where
\begin{equation}
\gamma = \left( \frac{\partial f_L}{\partial y} g_R
- \frac{\partial g_L}{\partial y} f_R \right)
\bigg|_{x = y = \mu = 0}.
\label{eq:FilippovStabilityCond}
\end{equation}
In the following theorem we assume $\gamma < 0$ so that the BEB corresponds to persistence.
If instead $\gamma > 0$ then the BEB is a nonsmooth-fold and a limit cycle may be created, this is ${\rm BF}_1$ in \cite{KuRi03}.
In some sense the case $\gamma > 0$ occurs more naturally:
if $\rD F_L(0,0;0)$ is in real Jordan form 
and orbits in $\Omega_R$ approach $x=0$ at right angles to $x=0$, then $\gamma > 0$.
To have $\gamma < 0$ the vector field $F_R$ needs to be directed so that it overcomes
the instability produced by the unstable focus.

\begin{theorem}[HLB 3]
Consider \eqref{eq:FilippovBEB} where $F_L$ and $F_R$ are $C^2$.
Suppose \eqref{eq:FilippovEqCond} and \eqref{eq:FilippovEigCondFocus} are satisfied,
$\beta > 0$, $f_R(0,0;0) < 0$, and $\gamma < 0$.
In a neighbourhood of $(x,y;\mu) = (0,0;0)$,
\begin{enumerate}
\item
there exists a unique stationary solution:
a stable pseudo-equilibrium for $\mu < 0$,
and an unstable focus in $\Omega_L$ for $\mu > 0$,
\item
there exists a unique stable limit cycle for $\mu > 0$, and no limit cycle for $\mu < 0$.
\end{enumerate}
The minimum $x$-value and minimum and maximum $y$-values
of the limit cycle are asymptotically proportional to $\mu$,
and its period is $T = T_L + T_{\rm slide} + \cO(\mu)$ where
\begin{equation}
\begin{split}
T_L &= \frac{1}{\omega_L} \,\hat{s} \left( \frac{\lambda_L}{\omega_L} \right), \\
T_{\rm slide} &= \frac{a_{0R}}{\gamma} \,\ln \left( 1 - \frac{\gamma}{a_{0R} \omega_L} \,\re^{\frac{\lambda_L}{\omega_L}
\,\hat{s} \left( \frac{\lambda_L}{\omega_L} \right)} \sin \left( \hat{s} \left( \frac{\lambda_L}{\omega_L} \right) \right) \right),
\end{split}
\label{eq:periodFilippovGen}
\end{equation}
and $\hat{s}$ is defined in \S\ref{sub:affine}.
\label{th:FilippovGeneric}
\end{theorem}

While HLB 3 has been well understood for some time through qualitative studies such as \cite{KuRi03},
the proof in Appendix \ref{app:beb}
is possibly the first quantitative analysis of HLB 3
that accommodates nonlinear terms in $F_L$ and $F_R$.
The limit cycle involves motion in $\Omega_L$ of duration $t_L(\mu) = T_L + \cO(\mu)$
and a segment of sliding motion of duration $t_{\rm slide}(\mu) = T_{\rm slide} + \cO(\mu)$.
If \eqref{eq:FilippovBEB} is piecewise-linear then $t_L(\mu)$ is independent of $\mu$
but $t_{\rm slide}(\mu)$ is, in general, not.
This is because the sliding vector field \eqref{eq:gslide2} is nonlinear.

To illustrate HLB 3, we consider the predator-prey model of \cite{YaTa13}:
\begin{equation}
\begin{split}
\dot{x} &= r x \left( 1 - \frac{x}{K} \right) - q(x) y, \\
\dot{y} &= \left( k q(x) - \delta \right) y,
\end{split}
\label{eq:YaTa13}
\end{equation}
where $x(t)$ and $y(t)$ denote the prey and predator populations respectively.
Also
\begin{equation}
q(x) = \begin{cases}
0, & x < R_c \,, \\
\frac{b x}{1 + b h x}, & x > R_c \,,
\end{cases}
\label{eq:YaTa13auxFunc}
\end{equation}
and all other quantities are positive constants.
In the limit $K \to \infty$ ($K$ is the {\em carrying capacity} of the prey)
the system reduces to the Gause predator-prey model \cite{GaSm36} for which
the prey is a yeast and the predator is a certain single-celled organism.
The Gause model is Lotka-Volterra with a Holling type II functional response \cite{BrCa01}.
It assumes that when the yeast population is below a threshold $R_c$
the yeast forms a sediment and cannot be consumed by the organisms.
For this reason the model is Filippov.
Interestingly its conception and initial study
predates the fundamental work of Filippov \cite{Fi60}.
For modern studies of the Gause model refer to \cite{Kr11,TaLi13}.

\begin{figure*}
\begin{center}
\includegraphics[width=16.6cm]{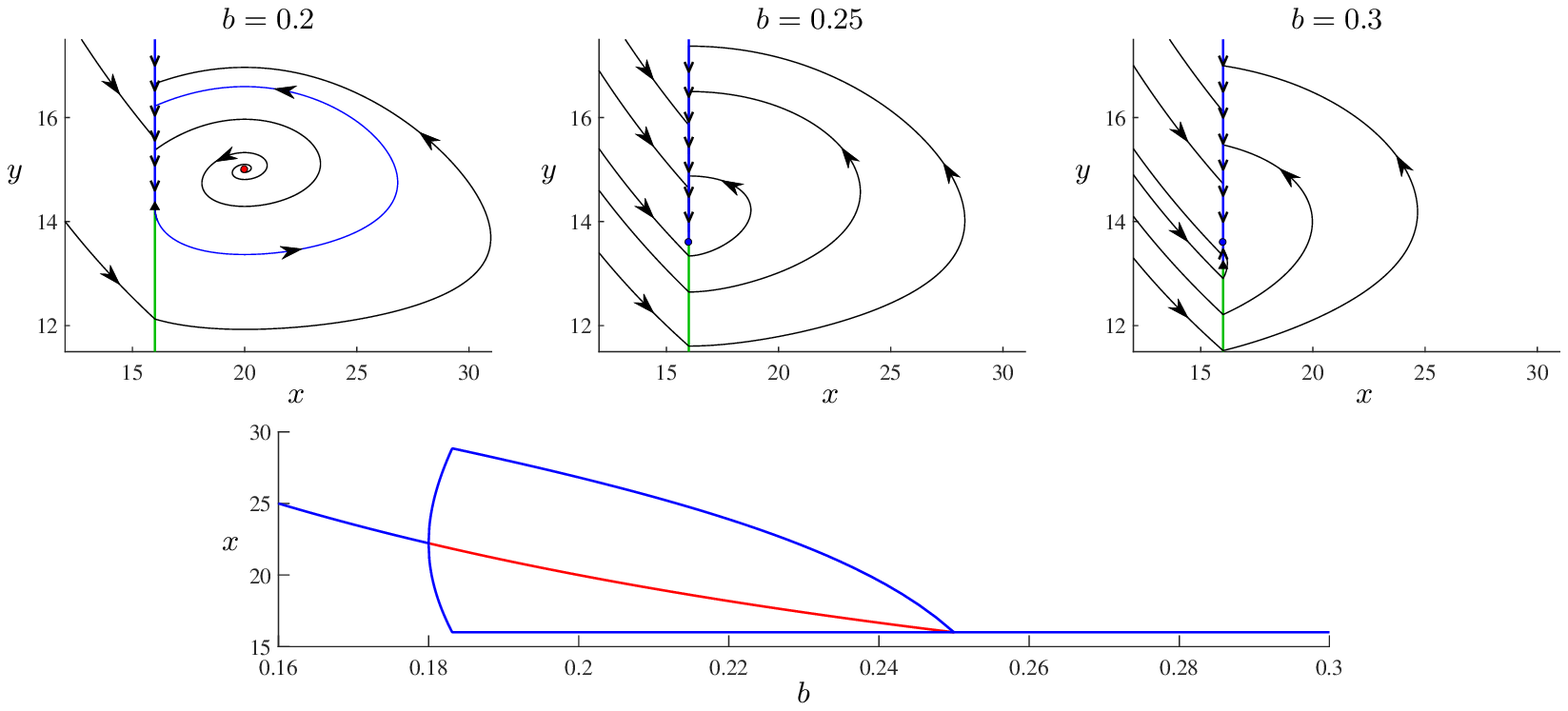}
\caption{
An illustration of HLB 3 for the Gause predator-prey model with finite prey carrying capacity,
\eqref{eq:YaTa13}--\eqref{eq:YaTa13Param}.
By decreasing the value of $b$ a stable pseudo-equilibrium
reaches a boundary of its sliding region when $b = 0.25$ where
it turns into an unstable focus and emits a stable limit cycle
involving one sliding segment.
\label{fig:allYaTa13}
} 
\end{center}
\end{figure*}

Fig.~\ref{fig:allYaTa13} illustrates the dynamics of \eqref{eq:YaTa13} using
\begin{equation}
\begin{aligned}
R_c &= 16, & r &= 1, & h &= 1, \\
k &= 0.45, & \delta &= 0.36, & K &= 50,
\end{aligned}
\label{eq:YaTa13Param}
\end{equation}
and $b$ as a bifurcation parameter.
As the value of $b$ is decreased,
a stable pseudo-equilibrium turns into an unstable focus at $b = 0.25$ (an instance of HLB 3).
The resulting limit cycle involves one segment of sliding motion
until a {\em grazing-sliding bifurcation} \cite{DiKo02,DiKo03,JeHo11} occurs at $b \approx 0.1831$.
The limit cycle is subsequently destroyed in a supercritical Hopf bifurcation at $b = 0.18$.

\subsection{A degenerate case --- HLB 4}
\label{sub:degenerateBEB}

For \eqref{eq:FilippovBEB} with $\mu = 0$, we now suppose the origin
is an unstable focus for the left half-system
and a stable focus for the right half-system, that is
\begin{equation}
f_L(0,0;0) = g_L(0,0;0) = f_R(0,0;0) = g_R(0,0;0) = 0,
\label{eq:FilippovDegEqCond}
\end{equation}
and
\begin{equation}
\begin{split}
{\rm eig}(\rD F_L(0,0;0)) &= \lambda_L \pm \ri \omega_L \,, \quad
{\rm with~} \lambda_L > 0, \omega_L > 0, \\
{\rm eig}(\rD F_R(0,0;0)) &= \lambda_R \pm \ri \omega_R \,, \quad
{\rm with~} \lambda_R < 0, \omega_R > 0.
\end{split}
\label{eq:FilippovDegEigCondFocusFocus}
\end{equation}
Locally, each half-system has a unique equilibrium (a focus)
and its direction of rotation is determined by the sign of
\begin{equation}
a_{2J} = \frac{\partial f_J}{\partial y}(0,0;0).
\label{eq:}
\end{equation}
If $a_{2J} > 0$, rotation is clockwise;
if $a_{2J} < 0$, rotation is anticlockwise.
Note that $a_{2J} \ne 0$ is a consequence of \eqref{eq:FilippovDegEigCondFocusFocus}
and has the opposite sign to $\frac{\partial g_J}{\partial x}(0,0;0)$
(which could instead be used to characterise the direction of rotation).
In order to have a Hopf-like bifurcation
we assume $a_{2L} a_{2R} > 0$ so that the foci have the same direction of rotation
(for HLB 1 this property follows automatically from the continuity of the system on the switching manifold).

As in \S\ref{sub:focusFocusBEB}, let
\begin{equation}
\alpha = \frac{\lambda_L}{\omega_L} + \frac{\lambda_R}{\omega_R}.
\label{eq:FilippovDegNondegCond}
\end{equation}
Here we show that if $\alpha < 0$ (so that the origin is stable when $\mu = 0$)
and, locally, there is no attracting sliding region when the unstable focus is admissible,
then a unique stable limit cycle is created.
We call this HLB 4; it generalises HLB 1 to Filippov systems.
The analogous creation of an unstable limit cycle can be understood by reversing time.
The assumption of no attracting sliding region
is sufficient, but not necessary, to ensure uniqueness of the limit cycle.
Without this assumption up to three limit cycles can be created
simultaneously \cite{BrMe13,FrPo14,HuYa12b}, see also \cite{HaZh10,HuYa13,HuYa14}.

The $x$-values of the foci are
$\frac{-\beta_J \mu}{\lambda_J^2 + \omega_J^2} + \cO \left( \mu^2 \right)$, where
\begin{equation}
\beta_J = \left( \frac{\partial f_J}{\partial \mu} \frac{\partial g_J}{\partial y}
- \frac{\partial g_J}{\partial \mu} \frac{\partial f_J}{\partial y} \right)
\bigg|_{x = y = \mu = 0}.
\label{eq:FilippovDegTransCond}
\end{equation}
Thus the condition $\beta_L \beta_R > 0$ ensures
the foci are admissible for different signs of $\mu$.
In view of the replacement $\mu \mapsto -\mu$, we assume $\beta_L > 0$ and $\beta_R > 0$.
Then the unstable focus is admissible (in $\Omega_L$) when $\mu > 0$.

Locally the left half-system has a unique visible fold when $\mu > 0$.
At this fold $f_R(0,y;\mu) = \frac{\gamma}{a_{2L}} \,\mu + \cO \left( \mu^2 \right)$, where
\begin{equation}
\gamma = \left( \frac{\partial f_L}{\partial y} \frac{\partial f_R}{\partial \mu}
- \frac{\partial f_L}{\partial \mu} \frac{\partial f_R}{\partial y} \right)
\bigg|_{x = y = \mu = 0}.
\label{eq:FilippovDegSlidingCond}
\end{equation}
Hence if $a_{2L} \gamma > 0$ then the fold bounds a crossing region and a repelling sliding region.
Thus $a_{2L} \gamma > 0$ produces the desired assumption for the absence of an attracting sliding region.
This condition also accommodates the case $\alpha > 0$ and below we also allow $\gamma = 0$.

\begin{theorem}[HLB 4]
Consider \eqref{eq:FilippovBEB} where $F_L$ and $F_R$ are $C^2$.
Suppose \eqref{eq:FilippovDegEqCond} and \eqref{eq:FilippovDegEigCondFocusFocus} are satisfied,
$a_{2L} a_{2R} > 0$,
$\beta_L > 0$, $\beta_R > 0$, and $a_{2L} \gamma \ge 0$.
In a neighbourhood of $(x,y;\mu) = (0,0;0)$,
\begin{enumerate}
\item
there exists a stable focus in $\Omega_R$ for $\mu < 0$,
and an unstable focus in $\Omega_L$ for $\mu > 0$,
\item
if $\alpha < 0$ [$\alpha > 0$] there exists a unique stable [unstable] limit cycle for $\mu > 0$ [$\mu < 0$],
and no limit cycle for $\mu < 0$ [$\mu > 0$].
\end{enumerate}
The minimum and maximum $x$ and $y$-values of the limit cycle are asymptotically proportional to $|\mu|$,
and its period is $T = T_L + T_R + \cO(\mu)$ where $T_L$ and $T_R$ satisfy
\begin{equation}
\begin{split}
\frac{\gamma}{a_{2L} a_{2R}} &= \frac{\beta_L \omega_L}{a_{2L} \left( \lambda_L^2 + \omega_L^2 \right)}
\,G \left( \omega_L T_L; -\frac{\lambda_L}{\omega_L} \right) \\
&\quad+
\frac{\beta_R \omega_R}{a_{2R} \left( \lambda_R^2 + \omega_R^2 \right)}
\,G \left( \omega_R T_R; \frac{\lambda_R}{\omega_R} \right), \\
\frac{-\gamma}{a_{2L} a_{2R}} &= \frac{\beta_L \omega_L}{a_{2L} \left( \lambda_L^2 + \omega_L^2 \right)}
\,G \left( \omega_L T_L; \frac{\lambda_L}{\omega_L} \right) \\
&\quad+
\frac{\beta_R \omega_R}{a_{2R} \left( \lambda_R^2 + \omega_R^2 \right)}
\,G \left( \omega_R T_R; -\frac{\lambda_R}{\omega_R} \right),
\end{split}
\label{eq:periodFilippovDeg}
\end{equation}
where $G(s;\nu) = \frac{\re^{-\nu s} \varrho(s;\nu)}{\sin(s)}$, and
$\varrho$ is the auxiliary function \eqref{eq:auxFunc}.
\label{th:FilippovDeg}
\end{theorem}

\begin{figure*}
\begin{center}
\includegraphics[width=16.6cm]{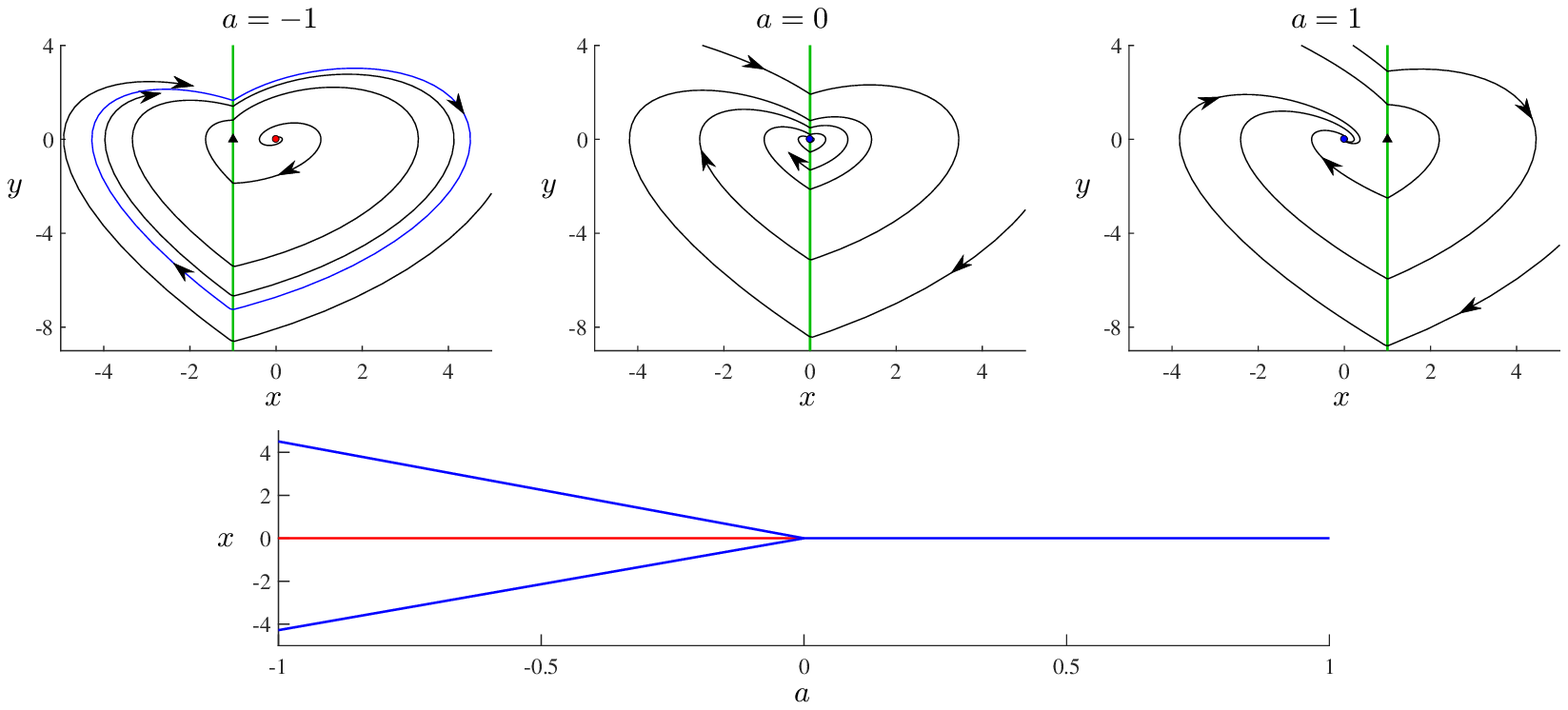}
\caption{
An illustration of HLB 4 for the valve generator model \eqref{eq:AnVi66} with $h_1 = 0.6$ and $h_2 = 0.3$.
By decreasing the value of $a$, a stable focus collides with the switching manifold $x = a$ when $a = 0$
where it turns into an unstable focus and emits a stable limit cycle.
\label{fig:allAnVi66}
} 
\end{center}
\end{figure*}

Theorem \ref{th:FilippovDeg} is proved for piecewise-linear systems in \cite{Si18f}
and in general in Appendix \ref{app:beb}.
As shown in \cite{Si18f},
there can exist pseudo-equilibria
but this only seems to occur over a relatively small fraction of parameter space.

The following equations model a valve generator
(a simple electrical circuit):
\begin{equation}
\begin{split}
\dot{x} &= y, \\
\dot{y} &= \begin{cases}
-x - 2 h_1 y, & x < a, \\
-x + 2 h_2 y, & x > a.
\end{cases}
\end{split}
\label{eq:AnVi66}
\end{equation}
With $a = -1$ this is equation (8.5) of \cite{AnVi66}
(a similar example is presented in Section 2.1 of \cite{MaLa12}).
The equations \eqref{eq:AnVi66} are non-dimensionalised,
where $x(t)$ represents voltage,
and $h_1$ and $h_2$ are combinations of certain circuit parameters.
In \cite{AnVi66} it is shown that \eqref{eq:AnVi66}
has a unique stable limit cycle when $a = -1$, $h_1 > h_2$, and $0 < h_2 < 1$.
Here we allow $a$ to vary so that we can observe the limit cycle
being created in an instance of HLB 4, see Fig.~\ref{fig:allAnVi66}.

Assuming $0 < h_1 < 1$ and $0 < h_2 < 1$ (so that the equilibria are foci of opposite stability),
\eqref{eq:AnVi66} has a BEB in accordance with Theorem \ref{th:FilippovDeg} when $a = 0$.
Here $\alpha = \frac{h_2}{\sqrt{1 - h_2^2}} - \frac{h_1}{\sqrt{1 - h_1^2}}$,
thus $h_1 > h_2$ implies $\alpha < 0$ and hence a stable limit cycle.
Also $\gamma = 0$ because \eqref{eq:AnVi66} has no sliding regions.

\section{Slipping foci and folds in Filippov systems}
\setcounter{equation}{0}
\setcounter{figure}{0}
\setcounter{table}{0}
\label{sec:slipping}

In this section we suppose that each half-system of \eqref{eq:FilippovBEB}
has either a focus or an invisible fold on $x=0$ for all values of $\mu$ in a neighbourhood of $0$
(see already Fig.~\ref{fig:schemSlippingFocusFocus2}).
So that a limit cycle may be created,
we suppose that orbits near these points involve the same direction of rotation.
We further suppose that these points `slip' along $x=0$ as $\mu$ is varied and collide when $\mu = 0$.
There are three cases:
both points are foci, HLB 5,
one point is a focus and the other is a fold, HLB 6,
and both points are folds, HLB 7.
The third case (called ${\rm II}_2$ in \cite{KuRi03})
is a generic codimension-one bifurcation for two-dimensional Filippov systems.
It occurs, for instance, with on/off PD control of an inverted pendulum \cite{Ko17} (see \S\ref{sub:slippingFolds}),
a similar automatic pilot model \cite{AnVi66}, 
and Welander's ocean convection model \cite{Le16}.

In each case, a sliding region shrinks to a point (at $\mu = 0$)
and changes from attracting to repelling.
The creation of a limit cycle in this fashion is called a {\em pseudo-Hopf bifurcation} in \cite{KuRi03}.
Transversality and non-degeneracy conditions for these bifurcations
were derived in \cite{CaLl17} for piecewise-linear systems.
Here we generalise these results to allow nonlinear terms in $F_L$ and $F_R$.
These terms only affect the qualitative behaviour of the bifurcation in the two-fold case.
For analogous bifurcations involving visible folds, see \cite{KuRi03,TaLi12}.

\subsection{The focus/focus case --- HLB 5}
\label{sub:slippingFoci}

\begin{figure}
\begin{center}
\includegraphics[width=4.2cm]{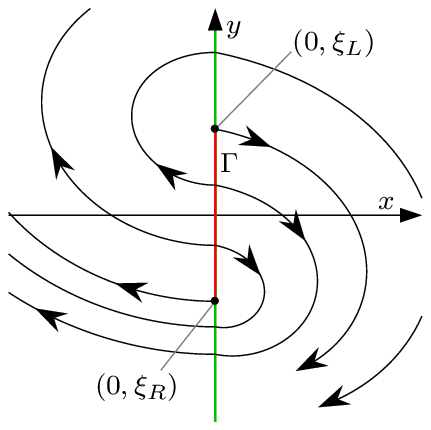}
\caption{
A typical phase portrait of the Filippov system \eqref{eq:FilippovBEB}
subject to the assumptions of Theorem \ref{th:slippingFocusFocus} with small $\mu > 0$.
The points $(0,\xi_L)$ and $(0,\xi_R)$ are boundary foci.
The direction of sliding motion on $\Gamma$ is governed by the sign of $\gamma$,
\eqref{eq:slippingFocusFocusgamma}.
\label{fig:schemSlippingFocusFocus2}
} 
\end{center}
\end{figure}

\begin{figure*}
\begin{center}
\includegraphics[width=16.6cm]{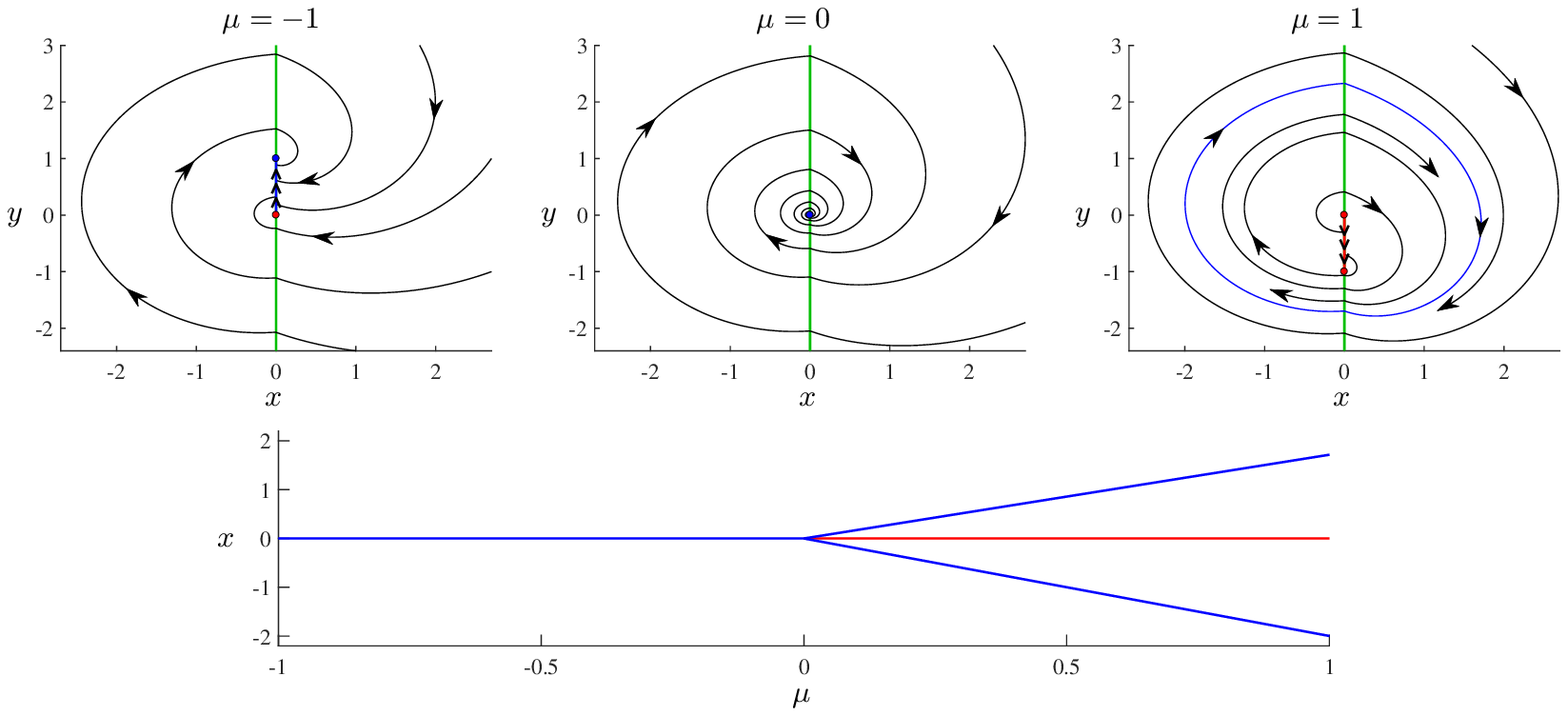}
\caption{
An illustration of HLB 5 for \eqref{eq:slippingFocusFocusExample} with $\lambda_L = 0.1$ and $\lambda_R = -0.5$.
By increasing the value of $\mu$ two boundary foci collide
causing an attracting sliding region to shrink to a point and reappear as a repelling sliding region
and producing a stable limit cycle.
\label{fig:allSlippingFocusFocus}
} 
\end{center}
\end{figure*}

Here we suppose that for all values of $\mu$ in a neighbourhood of $0$,
the left and right half-systems of \eqref{eq:FilippovBEB} have equilibria
at $(x,y) = (0,\xi_L(\mu))$ and $(x,y) = (0,\xi_R(\mu))$, respectively.
That is,
\begin{equation}
\begin{aligned}
f_L(0,\xi_L(\mu);\mu) &= 0, & f_R(0,\xi_R(\mu);\mu) &= 0, \\
g_L(0,\xi_L(\mu);\mu) &= 0, & g_R(0,\xi_R(\mu);\mu) &= 0,
\end{aligned}
\label{eq:slippingFocusFocusEqCond}
\end{equation}
for all sufficiently small $\mu \in \mathbb{R}$.
We suppose that the equilibria coincide at the origin when $\mu = 0$, that is
\begin{equation}
\xi_L(0) = \xi_R(0) = 0.
\label{eq:slippingFocusFocusPhiCond}
\end{equation}
We suppose $(0,\xi_L(\mu))$ is an unstable focus
and $(0,\xi_R(\mu))$ is a stable focus, that is
\begin{equation}
\begin{split}
{\rm eig}(\rD F_L(0,0;0)) &= \lambda_L \pm \ri \omega_L \,, \quad
{\rm with~} \lambda_L > 0, \omega_L > 0, \\
{\rm eig}(\rD F_R(0,0;0)) &= \lambda_R \pm \ri \omega_R \,, \quad
{\rm with~} \lambda_R < 0, \omega_R > 0,
\end{split}
\label{eq:slippingFocusFocusEigCond}
\end{equation}
and that the foci have the same direction of rotation.
To simplify the $\mu$-dependence
we assume (without loss of generality) that both foci have clockwise rotation,
Fig.~\ref{fig:schemSlippingFocusFocus2}.
That is, $\frac{\partial f_L}{\partial y}(0,0;0) > 0$ and $\frac{\partial f_R}{\partial y}(0,0;0) > 0$.
As with HLBs 1 and 4, the sign of
\begin{equation}
\alpha = \frac{\lambda_L}{\omega_L} + \frac{\lambda_R}{\omega_R},
\label{eq:slippingFocusFocusNondegCond}
\end{equation}
governs the stability of the origin when $\mu = 0$.

The distance between the foci is $|\beta \mu| + \cO \left( \mu^2 \right)$, where
\begin{equation}
\beta = \left( \frac{d \xi_L}{d \mu} - \frac{d \xi_R}{d \mu} \middle) \right|_{\mu = 0},
\label{eq:slippingFocusFocusTransCond}
\end{equation}
thus $\beta \ne 0$ is the transversality condition for HLB 5.
Let $\Gamma \subset \mathbb{R}^2$ be the subset of the switching manifold bounded by the foci,
see Fig.~\ref{fig:schemSlippingFocusFocus2}.
Notice $\Gamma$ is a sliding region on which motion is governed by \eqref{eq:gslide2}.
With the above assumptions the numerator of \eqref{eq:gslide2}
is $\gamma \left( y - \xi_L(\mu) \right) \left( y - \xi_R(\mu) \right)$,
plus higher order terms, where
\begin{equation}
\gamma = \left( \frac{\partial f_L}{\partial y} \frac{\partial g_R}{\partial y}
- \frac{\partial f_R}{\partial y} \frac{\partial g_L}{\partial y} \middle) \right|_{x = y = \mu = 0}.
\label{eq:slippingFocusFocusgamma}
\end{equation}

\begin{theorem}[HLB 5]
Consider \eqref{eq:FilippovBEB} where $F_L$ and $F_R$ are $C^1$.
Suppose \eqref{eq:slippingFocusFocusEqCond} and \eqref{eq:slippingFocusFocusPhiCond} are satisfied,
where $\xi_L$ and $\xi_R$ are $C^1$.
Suppose \eqref{eq:slippingFocusFocusEigCond} is satisfied,
$\frac{\partial f_L}{\partial y}(0,0;0) > 0$, $\frac{\partial f_R}{\partial y}(0,0;0) > 0$,
and $\beta > 0$.
In a neighbourhood of $(x,y;\mu) = (0,0;0)$,
\begin{enumerate}
\item
$\Gamma$ is an attracting sliding region for $\mu < 0$
(and if $\gamma < 0$ [$\gamma > 0$] then $(0,\xi_R(\mu))$ [$(0,\xi_L(\mu))$] is stable
while $(0,\xi_L(\mu))$ [$(0,\xi_R(\mu))$] is unstable)
and a repelling sliding region for $\mu > 0$,
\item
if $\alpha < 0$ [$\alpha > 0$] there exists a unique stable [unstable] limit cycle for $\mu > 0$ [$\mu < 0$],
and no limit cycle for $\mu < 0$ [$\mu > 0$].
\end{enumerate}
The minimum and maximum $x$ and $y$-values of the limit cycle are asymptotically proportional to $|\mu|$,
and the period limits to $\frac{\pi}{\omega_L} + \frac{\pi}{\omega_R}$ as $\mu \to 0$.
\label{th:slippingFocusFocus}
\end{theorem}

As an example consider
\begin{equation}
\begin{bmatrix} \dot{x} \\ \dot{y} \end{bmatrix} =
\begin{cases}
\begin{bmatrix} \lambda_L x + y \\ -x + \lambda_L y \end{bmatrix}, & x < 0, \\
\begin{bmatrix} \lambda_R x + y + \mu \\ -x + \lambda_R (y + \mu) \end{bmatrix}, & x > 0.
\end{cases}
\label{eq:slippingFocusFocusExample}
\end{equation}
which satisfies the conditions of Theorem \ref{th:slippingFocusFocus} when $\lambda_L > 0$ and $\lambda_R < 0$.
Specifically we have $\xi_L(\mu) = 0$ and $\xi_R(\mu) = -\mu$, thus $\beta = 1$.
Also $\alpha = \lambda_L + \lambda_R$ and $\gamma = \lambda_R - \lambda_L < 0$.
Fig.~\ref{fig:allSlippingFocusFocus} shows a bifurcation diagram and
representative phase portraits using values of $\lambda_L$ and $\lambda_R$ for which $\alpha < 0$.

\subsection{The focus/fold case --- HLB 6}
\label{sub:slippingFocusFold}

\begin{figure*}
\begin{center}
\includegraphics[width=16.6cm]{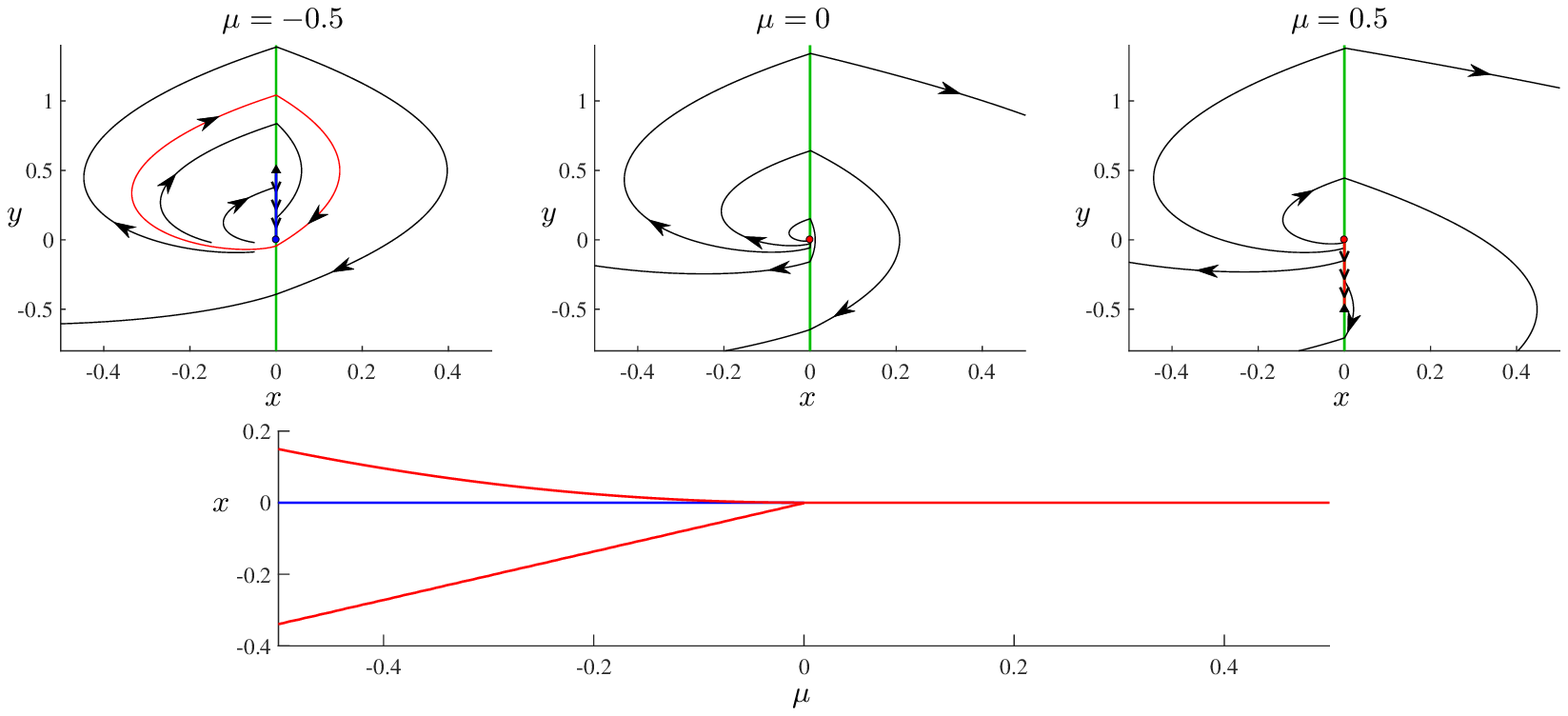}
\caption{
An illustration of HLB 6 for \eqref{eq:slippingFocusFoldExample} with $\lambda_L = 1$.
By decreasing the value of $\mu$ a boundary focus and invisible fold collide
causing an repelling sliding region to shrink to a point and reappear as an attracting sliding region
and producing an unstable limit cycle.
\label{fig:allSlippingFocusFold}
} 
\end{center}
\end{figure*}

Now suppose that for all values of $\mu$ in a neighbourhood of $0$
the left half-system has a focus at the origin.
That is,
\begin{align}
f_L(0,0;\mu) &= 0, & g_L(0,0;\mu) &= 0,
\label{eq:slippingFocusFoldEqCond}
\end{align}
and
\begin{equation}
{\rm eig}(\rD F_L(0,0;0)) = \lambda_L \pm \ri \omega_L \,, \quad
{\rm with~} \lambda_L \in \mathbb{R}, \omega_L > 0.
\label{eq:slippingFocusFoldEigCond}
\end{equation}
In order for the right half-system to have a fold at the origin when $\mu = 0$ we suppose
\begin{equation}
f_R(0,0;0) = 0.
\label{eq:slippingFocusFoldCond}
\end{equation}
As with HLB 5 we suppose $\frac{\partial f_L}{\partial y}(0,0;0) > 0$
and $\frac{\partial f_R}{\partial y}(0,0;0) > 0$
so that both half-systems involve clockwise rotation.
Then $g_R(0,0;0) < 0$ ensures the fold is invisible.
Locally the right half-system has a unique fold at $(x,y) = (0,\zeta_R(\mu))$, where
\begin{equation}
\frac{d \zeta_R}{d \mu}(0) = -\frac{\frac{\partial f_R}{\partial \mu}(0,0;0)}{\frac{\partial f_R}{\partial y}(0,0;0)}.
\label{eq:slippingFocusFoldxiR}
\end{equation}
It follows that the transversality condition for HLB 6 is $\beta \ne 0$, where
\begin{equation}
\beta = \frac{\partial f_R}{\partial \mu}(0,0;0).
\label{eq:slippingFocusFoldTransCond}
\end{equation}
Let $\Gamma \subset \mathbb{R}^2$ be the subset of the switching manifold
bounded by the focus $(0,0)$ and the fold $(0,\zeta_R(\mu))$.
When $\mu = 0$ the stability of the origin is governed by the sign of
\begin{equation}
\alpha = \lambda_L \,,
\label{eq:slippingFocusFoldNondegCond}
\end{equation}
as can be seen by composing the return maps \eqref{eq:focusP} and \eqref{eq:foldP}.

\begin{theorem}[HLB 6]
Consider \eqref{eq:FilippovBEB} where $F_L$ and $F_R$ are $C^1$.
Suppose \eqref{eq:slippingFocusFoldEqCond},
\eqref{eq:slippingFocusFoldEigCond}, and \eqref{eq:slippingFocusFoldCond} are satisfied,
$\frac{\partial f_L}{\partial y}(0,0;0) > 0$,
$\frac{\partial f_R}{\partial y}(0,0;0) > 0$,
$g_R(0,0;0) < 0$, and $\beta > 0$.
In a neighbourhood of $(x,y;\mu) = (0,0;0)$,
\begin{enumerate}
\item
the origin is the unique stationary solution
and is stable for $\mu < 0$ and unstable for $\mu > 0$
($\Gamma$ is an attracting sliding region for $\mu < 0$
and a repelling sliding region for $\mu > 0$),
\item
if $\alpha < 0$ [$\alpha > 0$] there exists a unique stable [unstable] limit cycle for $\mu > 0$ [$\mu < 0$],
and no limit cycle for $\mu < 0$ [$\mu > 0$].
\end{enumerate}
The maximum $x$-value of limit cycle is asymptotically proportional to $\mu^2$,
the minimum $x$-value and minimum and maximum $y$-values of the limit cycle are asymptotically proportional to $|\mu|$,
and the period limits to $\frac{\pi}{\omega_L}$ as $\mu \to 0$.
\label{th:slippingFocusFold}
\end{theorem}

As a minimal example consider
\begin{equation}
\begin{bmatrix} \dot{x} \\ \dot{y} \end{bmatrix} =
\begin{cases}
\begin{bmatrix} \lambda_L x + y \\ -x + \lambda_L y \end{bmatrix}, & x < 0, \\
\begin{bmatrix} y + \mu \\ -1 \end{bmatrix}, & x > 0,
\end{cases}
\label{eq:slippingFocusFoldExample}
\end{equation}
which satisfies the conditions of Theorem \ref{th:slippingFocusFold}.
Fig.~\ref{fig:allSlippingFocusFold} shows representative phase portraits and a bifurcation diagram using $\lambda_L = 1$
(here $\alpha > 0$ so an unstable limit cycle is created).

\subsection{The two-fold case --- HLB 7}
\label{sub:slippingFolds}

\begin{figure*}
\begin{center}
\includegraphics[width=16.6cm]{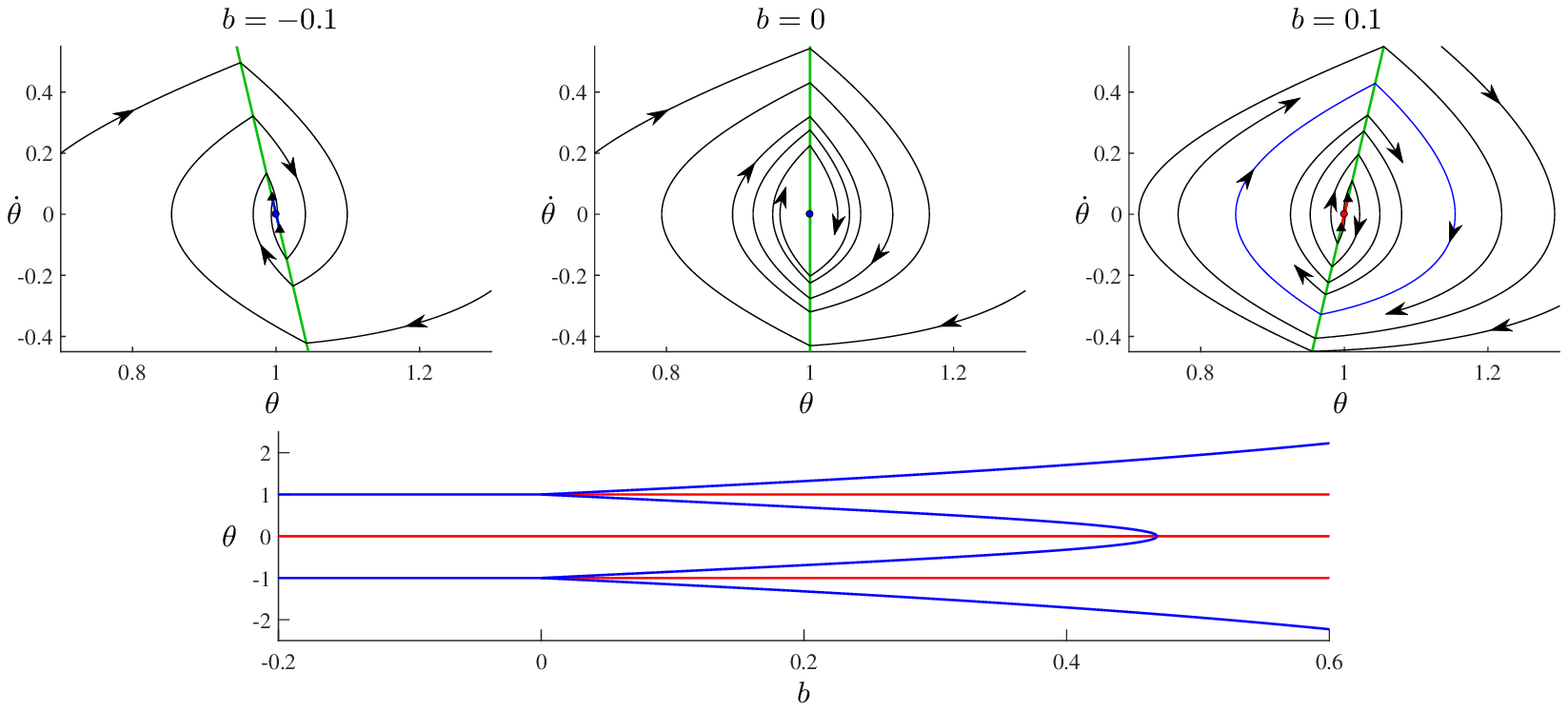}
\caption{
Phase portraits and a bifurcation diagram of the controlled pendulum model \eqref{eq:Ko17} with \eqref{eq:paramKo17}.
The system is symmetric and exhibits two instances of HLB 7 at $b = 0$.
For each instance, by increasing the value of $b$ two invisible folds collide
causing an attracting sliding region to shrink to a point and reappear as a repelling sliding region
and producing a stable limit cycle.
\label{fig:allKo17}
} 
\end{center}
\end{figure*}

Here we suppose both half systems of \eqref{eq:FilippovBEB} have a clockwise-rotating
invisible fold at the origin when $\mu = 0$.
That is,
\begin{equation}
f_L(0,0;0) = f_R(0,0;0) = 0,
\label{eq:slippingTwoFoldCond}
\end{equation}
and $\frac{\partial f_L}{\partial y}(0,0;0) > 0$,
$\frac{\partial f_R}{\partial y}(0,0;0) > 0$, $g_L(0,0;0) > 0$ and $g_R(0,0;0) < 0$.
The folds persist for small $\mu \ne 0$
and the distance between them is $|\beta \mu| + \cO \left( \mu^2 \right)$, where
\begin{equation}
\beta = \left( \frac{\frac{\partial f_R}{\partial \mu}}{\frac{\partial f_R}{\partial y}}
- \frac{\frac{\partial f_L}{\partial \mu}}{\frac{\partial f_L}{\partial y}} \right)
\bigg|_{x = y = \mu = 0}.
\label{eq:slippingTwoFoldTransCond}
\end{equation}
Let $\Gamma$ be the subset of $x=0$ bounded by the two folds.
Also let 
\begin{equation}
\alpha = \left( \sigma_{{\rm fold},L} - \sigma_{{\rm fold},R} \middle) \right|_{\mu = 0} \,,
\label{eq:slippingTwoFoldNondegCond}
\end{equation}
where \eqref{eq:foldsigma} is evaluated for both left and right half-systems as indicated
in \eqref{eq:slippingTwoFoldNondegCond} with subscripts.

\begin{theorem}[HLB 7]
Consider \eqref{eq:FilippovBEB} where $F_L$ and $F_R$ are $C^2$.
Suppose \eqref{eq:slippingTwoFoldCond} is satisfied,
$\frac{\partial f_L}{\partial y}(0,0;0) > 0$,
$\frac{\partial f_R}{\partial y}(0,0;0) > 0$, $g_L(0,0;0) > 0$, $g_R(0,0;0) < 0$, and $\beta > 0$.
In a neighbourhood of $(x,y;\mu) = (0,0;0)$,
\begin{enumerate}
\item
there exists a unique stationary solution:
a pseudo-equilibrium in $\Gamma$ that is stable for $\mu < 0$ and unstable for $\mu > 0$
($\Gamma$ is an attracting sliding region for $\mu < 0$ and a repelling sliding region for $\mu > 0$),
\item
if $\alpha < 0$ [$\alpha > 0$] there exists a unique stable [unstable] limit cycle for $\mu > 0$ [$\mu < 0$],
and no limit cycle for $\mu < 0$ [$\mu > 0$].
\end{enumerate}
The minimum and maximum $x$-values of the limit cycle are asymptotically proportional to $|\mu|$,
and the minimum and maximum $y$-values of the limit cycle are asymptotically proportional to $\sqrt{|\mu|}$,
and its period is
\begin{equation}
T = \left( \frac{2}{g_L(0,0;0)} - \frac{2}{g_R(0,0;0)} \right)
\sqrt{\frac{-3 \beta \mu}{\alpha}} + \co \left( \sqrt{\mu} \right).
\label{eq:slippingTwoFoldPeriod}
\end{equation}
\label{th:slippingTwoFold}
\end{theorem}

To illustrate HLB 7 we consider an inverted pendulum subject to on/off control.
There have been many studies of inverted pendulums with a control law that seeks to
maintain the pendulum in a roughly vertical position,
with applications to robotics \cite{BaMi10}
and human postural sway \cite{MiCa09}.
On/off control strategies are popular as they are simple,
easy to implement, and can be highly effective \cite{AsTa09,MiTo09,StIn06}.
Typically the time lag between when the controller makes measurements and is able to implement control
is an important factor, but this is not considered here, see \S\ref{sec:perturbations}.

Specifically we consider the model studied in \cite{Ko17}:
\begin{equation}
\ddot{\theta}  = \begin{cases}
a \theta, & |\theta - b \dot{\theta}| < \theta^*, \\
(a - K_p) \theta - K_d \dot{\theta}, & |\theta - b \dot{\theta}| > \theta^*.
\end{cases}
\label{eq:Ko17}
\end{equation}
Here $\theta(t)$ represents the angular displacement of the pendulum from vertical.
In the absence of control, $\theta(t)$ is assumed to vary according to $\ddot{\theta} = a \theta$
for some constant $a > 0$ (a reasonable assumption for small angles).
The applied control force is a linear combination of $\theta$ and $\dot{\theta}$ (PD control).
With $b = 0$, the control is applied when $|\theta|$ exceeds a threshold angle $\theta^*$,
and $b \ne 0$ incorporates dependency on $\dot{\theta}$.

Fig.~\ref{fig:allKo17} shows the dynamics near the switching manifold
$\theta = \theta^* + b \dot{\theta}$ using parameter values given in \cite{Ko17}:
\begin{align}
a &= 0.5, &
K_p &= 1, &
K_d &= 1, &
\theta^* &= 1.
\label{eq:paramKo17}
\end{align}
Two invisible folds collide at $(\theta,\dot{\theta}) = (\theta^*,0)$ when $b = 0$,
and a stable limit cycle exists for $b > 0$.
The system is symmetric and exhibits the same dynamics near the other switching manifold
$\theta = -\theta^* + b \dot{\theta}$.
Thus two stable limit cycles are created at $b = 0$ in symmetric instances of HLB 7.
At $b \approx 0.47$ the limit cycles become homoclinic to the origin
and merge into a single symmetric limit cycle.

Theorem \ref{th:slippingTwoFold} can be applied to this system via the coordinate change
\begin{equation}
\begin{split}
x &= \theta - (\theta^* + b \dot{\theta}), \\
y &= \dot{\theta}.
\end{split}
\nonumber
\end{equation}
It is a simple exercise to show that, with the parameter values \eqref{eq:paramKo17},
the conditions of Theorem \ref{th:slippingTwoFold} are met with $\alpha = \frac{-K_p}{(K_p - a) \theta^*} < 0$,
hence a stable limit cycle is created as described above.

\section{Fixed foci and folds in Filippov systems}
\setcounter{equation}{0}
\setcounter{figure}{0}
\setcounter{table}{0}
\label{sec:fixed}

In this section we suppose that each half-system of \eqref{eq:FilippovBEB}
has either a focus or an invisible fold at the origin for all values of $\mu$ in a neighbourhood of $0$.
We also suppose the direction of rotation is the same on both sides.
With these assumptions no sliding motion occurs locally.
The origin is a stationary solution and we suppose its stability changes at $\mu = 0$.
Generically a limit cycle is created and, as in the previous section, there are three cases:
(i) focus/focus, HLB 8,
(ii) focus/fold, HLB 9,
(iii) two-fold, HLB 10.

The focus/focus case occurs in a simple braking model for a car or bike
and called a {\em generalised Hopf bifurcation} in \cite{KuMo01,ZoKu06}, see also \cite{ChZo10,LlPo08}.
For physical reasons the nonlinear terms in this model are cubic;
consequently the amplitude of the limit cycle is asymptotically proportional to $\sqrt{|\mu|}$.
The absence of quadratic terms is a degeneracy in the context of general Filippov systems.
Below we show that non-degenerate quadratic terms produce asymptotically linear growth for the amplitude.
For the more general situation
that switching manifolds emanate from the origin at arbitrary angles
and the origin is a focus in each region bounded by these manifolds,
see \cite{AkKa17,AkAr09,ZoKu05}.

The two-fold case was described briefly by Filippov (see page 238 of \cite{Fi88}),
is analysed for piecewise-quadratic systems in \cite{LiHu14},
and occurs for a buck converter (a type of DC/DC power converter) with idealised switching \cite{FrPo12},
although here the model lacks nonlinear terms necessary for a limit cycle.

To analyse the HLB in each case we use a Poincar\'e map:
given $r > 0$ let $P(r)$ denote the $y$-value at which the forward orbit of $(x,y) = (0,r)$
next intersects the positive $y$-axis.
If there exists $n \ge 1$ such that
\begin{equation}
P(r) = r + V_n r^n + \cO \left( r^{n+1} \right),
\nonumber
\end{equation}
where $V_n \ne 0$, then $V_n$ is called the $n^{\rm th}$ {\em Lyapunov constant} \cite{Ku00}.
For the three scenarios considered here,
the bifurcation occurs when $V_n = 0$ (with $n=1$ for the focus/focus and focus/fold cases, and $n=2$ for the two-fold case).
The criticality of the bifurcation is determined by the sign of the next generically non-zero Lyapunov exponent $V_m$
(with $m=2$ for the focus/focus and focus/fold cases, and $m=4$ for the two-fold case).
For the focus/focus and focus/fold cases, Lyapunov constants were obtained in \cite{CoGa01}
by using differential equations with a single complex variable.
Higher order scenarios are described in \cite{GaTo03,LiHa09} (for the focus/focus case)
and in \cite{LiHa12} (for the focus/fold and two-fold cases);
see also \cite{YaHa11} for piecewise-Hamiltonian systems.

\subsection{The focus/focus case --- HLB 8}
\label{sub:fixedFoci}

Here we suppose
\begin{equation}
\begin{aligned}
f_L(0,0;\mu) &= 0, & f_R(0,0;\mu) &= 0, \\
g_L(0,0;\mu) &= 0, & g_R(0,0;\mu) &= 0,
\end{aligned}
\label{eq:fixedFocusFocusEqCond}
\end{equation}
for all values of $\mu$ in a neighbourhood of $0$, and
\begin{equation}
\begin{split}
{\rm eig}(\rD F_L(0,0;\mu)) &= \lambda_L(\mu) \pm \ri \omega_L(\mu), \\
&\quad~ {\rm with~} \lambda_L(\mu) > 0, \omega_L(\mu) > 0, \\
{\rm eig}(\rD F_R(0,0;\mu)) &= \lambda_R(\mu) \pm \ri \omega_R(\mu), \\
&\quad~ {\rm with~} \lambda_R(\mu) < 0, \omega_R(\mu) > 0.
\end{split}
\label{eq:fixedFocusFocusEigCond}
\end{equation}

Suppose the two foci involve the same direction of rotation
and let $P(r;\mu)$ denote the $y$-value of the next intersection of the forward orbit of $(x,y) = (0,r)$
with the positive $y$-axis.
By applying Lemma \ref{le:focus} to both left and right half-systems, we obtain
$P(r;\mu) = \re^{\Lambda(\mu) \pi} r + \cO \left( r^2 \right)$, where
\begin{equation}
\Lambda(\mu) = \frac{\lambda_L(\mu)}{\omega_L(\mu)} + \frac{\lambda_R(\mu)}{\omega_R(\mu)}.
\label{eq:fixedFocusFocusPsi}
\end{equation}
Thus the stability of the origin is determined by the sign of $\Lambda$.
For HLB 8 we suppose $\Lambda(0) = 0$ and $\beta \ne 0$ where
\begin{equation}
\beta = \frac{d \Lambda}{d \mu}(0).
\label{eq:fixedFocusFocusTransCond}
\end{equation}
The $r^2$-term in $P(r;0)$ governs the criticality of the bifurcation.
In Appendix \ref{app:fixed} we show that
\begin{equation}
\alpha = \left( \chi_{{\rm focus},L} - \chi_{{\rm focus},R} \middle) \right|_{\mu = 0} \,,
\label{eq:fixedFocusFocusNondegCond}
\end{equation}
is a factor in the coefficient of the $r^2$-term,
where \eqref{eq:focustau} is evaluated for both left and right half-systems
as indicated with subscripts.

\begin{figure*}
\begin{center}
\includegraphics[width=16.6cm]{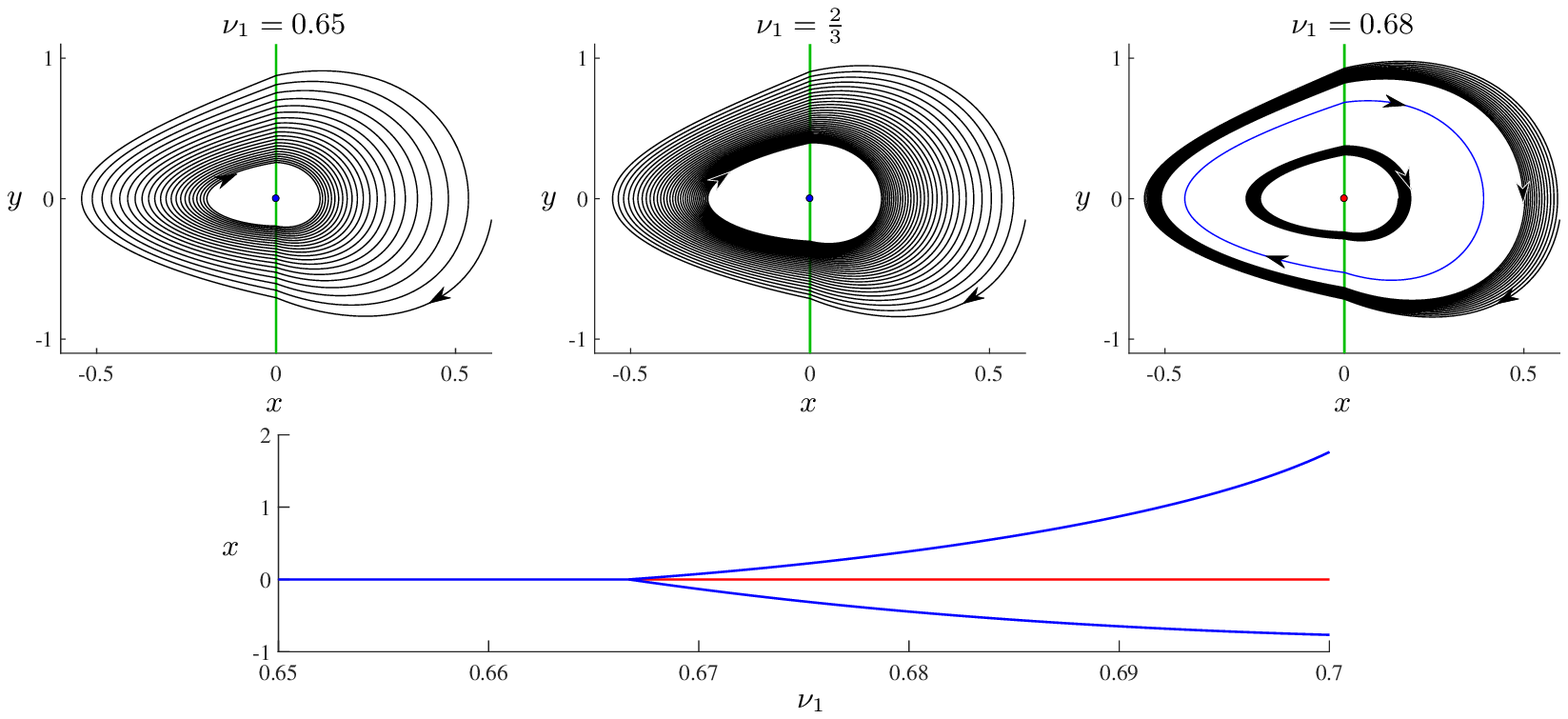}
\caption{
An illustration of HLB 8 for the bilinear oscillator
\eqref{eq:bilinearOsc} with \eqref{eq:bilinearOscFocusFocusParam},
$b = 0.5$, $\nu_2 = 1$, and $\hat{x} = 0$.
By increasing the value of $\nu_1$ the origin (a `merged focus') loses stability
when $\nu_1 = \frac{4 b}{3}$ and emits a stable limit cycle.
\label{fig:allFixedFocusFocus}
} 
\end{center}
\end{figure*}

\begin{theorem}[HLB 8]
Consider \eqref{eq:FilippovBEB} where $F_L$ and $F_R$ are $C^2$.
Suppose \eqref{eq:fixedFocusFocusEqCond} and
\eqref{eq:fixedFocusFocusEigCond} are satisfied,
$\frac{\partial f_L}{\partial y}(0,0;0) \frac{\partial f_R}{\partial y}(0,0;0) > 0$,
$\Lambda(0) = 0$, and $\beta > 0$.
In a neighbourhood of $(x,y;\mu) = (0,0;0)$,
\begin{enumerate}
\item
the origin is the unique stationary solution
and is stable for $\mu < 0$ and unstable for $\mu > 0$,
\item
if $\alpha < 0$ [$\alpha > 0$] there exists a unique stable [unstable] limit cycle for $\mu > 0$ [$\mu < 0$],
and no limit cycle for $\mu < 0$ [$\mu > 0$].
\end{enumerate}
The minimum and maximum $x$ and $y$-values of the limit cycle are asymptotically proportional to $|\mu|$,
and its period is
\begin{equation}
T = \frac{\pi}{\omega_L(0)} + \frac{\pi}{\omega_R(0)} + \cO(\mu).
\label{eq:fixedFocusFocusPeriod}
\end{equation}
\label{th:fixedFocusFocus}
\end{theorem}

As an example consider the bilinear oscillator:
\begin{equation}
\begin{split}
\dot{x} &= y, \\
\dot{y} &= \begin{cases}
-k_L x - b_L y + F_{\rm apply}(y), & x < 0, \\
-k_L x - k_R (x + \hat{x}) - (b_L + b_R) y + F_{\rm apply}(y), & x > 0.
\end{cases}
\end{split}
\label{eq:bilinearOsc}
\end{equation}
where
\begin{equation}
F_{\rm apply}(y) = \nu_1 y + \nu_2 y^2,
\label{eq:bilinearOscForcing}
\end{equation}
represents an applied force.
The variables $x(t)$ and $y(t)$ can be interpreted as the position and velocity of an oscillator
that undergoes compliant impacts at $x = 0$.
All parameters in \eqref{eq:bilinearOsc} are non-negative constants
(in particular $\hat{x} > 0$ is a prestress distance, see \cite{SiHo13}).
This model is motivated by experimental studies, particularly \cite{InPa08b,MaIn08},
but the parameter values used here are purely chosen to illustrate the HLBs.
For other studies of bilinear oscillators refer to \cite{DiBu08,PaYa08,ShHo83b}.

With $\hat{x} = 0$ the origin is an equilibrium of both half-systems.
By writing the associated eigenvalues in the form \eqref{eq:fixedFocusFocusEigCond} we obtain
\begin{equation}
\begin{aligned}
\lambda_L &= \frac{\nu_1 - b_L}{2}, \\
\lambda_R &= \frac{\nu_1 - b_L - b_R}{2}, \\
\omega_L &= \sqrt{k_L - \frac{(\nu_1 - b_L)^2}{4}}, \\
\omega_R &= \sqrt{k_L + k_R - \frac{(\nu_1 - b_L - b_R)^2}{4}}.
\end{aligned}
\label{eq:bilinearOsclambdaomega}
\end{equation}
Let us fix
\begin{equation}
\begin{aligned}
k_L &= 1, &
k_R &= 3, \\
b_L &= b, &
b_R &= b,
\end{aligned}
\label{eq:bilinearOscFocusFocusParam}
\end{equation}
for some $b > 0$.
Then $\Lambda$ is zero when $\nu_1 = \frac{4 b}{3}$,
thus to apply Theorem \ref{th:fixedFocusFocus} we define $\mu = \nu_1 - \frac{4 b}{3}$.
This gives
$\beta = \frac{d \Lambda}{d \nu_1} \big|_{\nu_1 = \frac{4 b}{3}}
= \frac{162}{\left( 36 - b^2 \right)^{\frac{3}{2}}}$
and $\alpha = \frac{-9 b \nu_2}{2 \left( 81 - 2 b^2 \right)}$.

Fig.~\ref{fig:allFixedFocusFocus} shows a bifurcation diagram using $b = 0.5$ and $\nu_2 = 1$.
Since $\beta > 0$ and $\alpha < 0$,
a stable limit cycle is created at $\nu_1 = \frac{4 b}{3}$.
The limit cycle grows quite quickly
because $|\alpha|$ is relatively small ($\alpha \approx -0.0280$)
and the amplitude of the limit cycle is inversely proportional to $|\alpha|$,
see \eqref{eq:fixedFocusFocusyStar}.
For the same reason the rate at which orbits converge to the limit cycle is relatively slow.

\begin{figure*}
\begin{center}
\includegraphics[width=16.6cm]{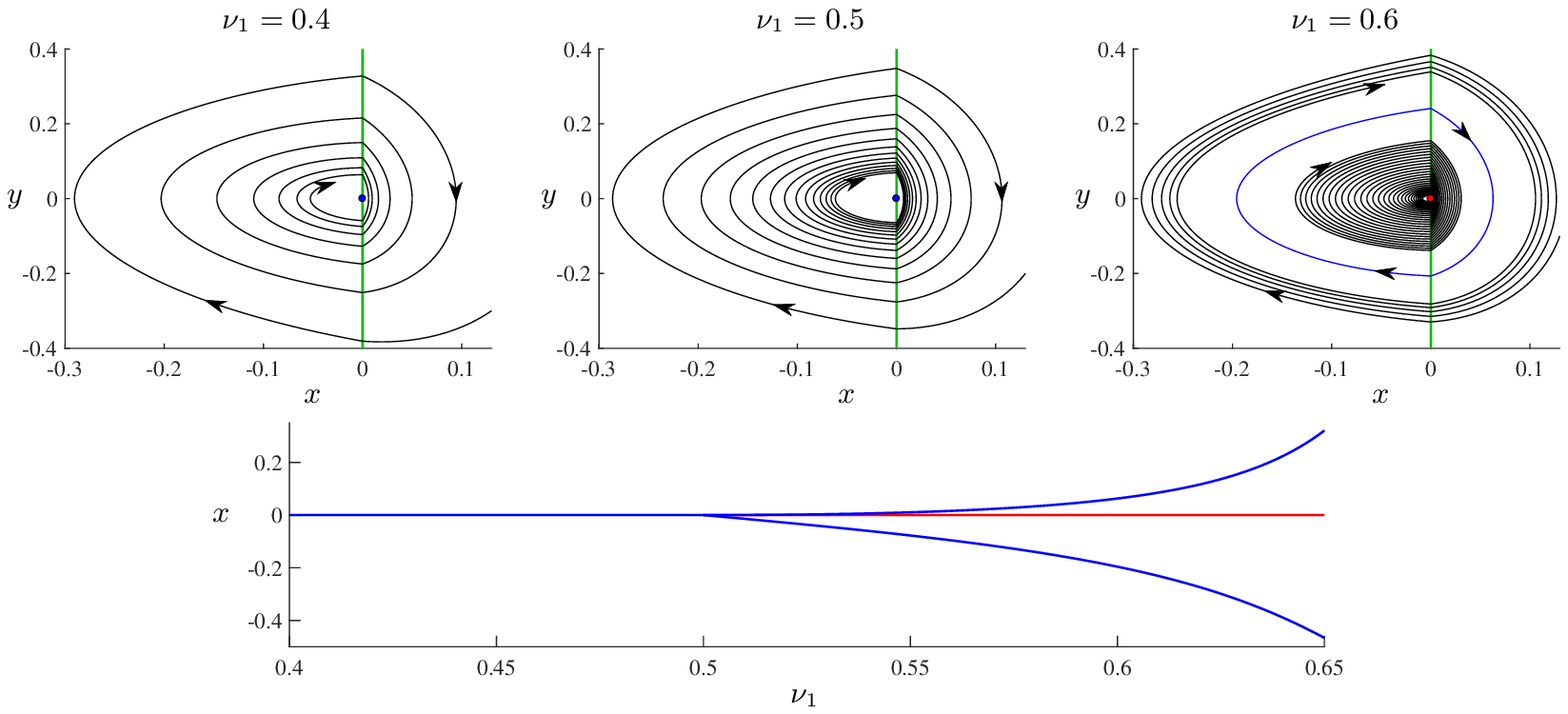}
\caption{
An illustration of HLB 9 for the bilinear oscillator
\eqref{eq:bilinearOsc} with \eqref{eq:bilinearOscFocusFocusParam},
$b = 0.5$, $\nu_2 = 1$, and $\hat{x} = 0.1$.
By increasing the value of $\nu_1$
the origin (a focus for the left half-system and an invisible fold for the right half-system)
loses stability when $\nu_1 = b$ and emits a limit cycle.
The limit cycle is stable because $\alpha = \frac{-b_R}{3 k_R \hat{x}} = -\frac{5}{9} < 0$.
\label{fig:allFixedFocusFold}
} 
\end{center}
\end{figure*}

\subsection{The focus/fold case --- HLB 9}
\label{sub:fixedFocusFold}

Here we suppose the left half-system has a focus at the origin, i.e.
\begin{align}
f_L(0,0;\mu) &= 0, & g_L(0,0;\mu) &= 0,
\label{eq:fixedFocusFoldEqCond}
\end{align}
and
\begin{equation}
\begin{split}
{\rm eig}(\rD F_L(0,0;\mu)) &= \lambda_L(\mu) \pm \ri \omega_L(\mu), \\
&\quad~ {\rm with~} \lambda_L(\mu) \in \mathbb{R}, \omega_L(\mu) > 0,
\end{split}
\label{eq:fixedFocusFoldEigCond}
\end{equation}
for all values of $\mu$ in a neighbourhood of $0$.
We suppose the right half-system has an invisible fold at the origin, thus
\begin{equation}
f_R(0,0;\mu) = 0,
\label{eq:fixedFocusFoldCond}
\end{equation}
and $\gamma < 0$ where
\begin{equation}
\gamma = \left( \frac{\partial f_R}{\partial y} \,g_R \middle) \right|_{x = y = \mu = 0}.
\label{eq:fixedFocusFoldInvisCond}
\end{equation}
Also suppose both half-systems involve the same direction of rotation.

In this situation the stability of the origin is simply determined by the sign of $\lambda_L(\mu)$.
Thus for HLB 9 we assume $\lambda_L(0) = 0$ and let
\begin{equation}
\beta = \frac{d \lambda_L}{d \mu}(0).
\label{eq:fixedFocusFoldTransCond}
\end{equation}
As shown below the criticality of the bifurcation is determined by the sign of
\begin{equation}
\alpha = \left( \chi_{{\rm focus},L} - \frac{\sigma_{{\rm fold},R}}{3}
\middle) \right|_{\mu = 0},
\label{eq:fixedFocusFoldNondegCond}
\end{equation}
where \eqref{eq:focustau} is evaluated for the left half-system
and \eqref{eq:foldsigma} is evaluated for the right half-system.

\begin{theorem}[HLB 9]
Consider \eqref{eq:FilippovBEB} where $F_L$ and $F_R$ are $C^2$.
Suppose \eqref{eq:fixedFocusFoldEqCond},
\eqref{eq:fixedFocusFoldEigCond}, and \eqref{eq:fixedFocusFoldCond} are satisfied,
$\frac{\partial f_L}{\partial y}(0,0;0) \frac{\partial f_R}{\partial y}(0,0;0) > 0$,
$\lambda_L(0) = 0$, $\beta > 0$, and $\gamma < 0$.
In a neighbourhood of $(x,y;\mu) = (0,0;0)$,
\begin{enumerate}
\item
the origin is the unique stationary solution and is stable for $\mu < 0$ and unstable for $\mu > 0$,
\item
if $\alpha < 0$ [$\alpha > 0$] there exists a unique stable [unstable] limit cycle for $\mu > 0$ [$\mu < 0$],
and no limit cycle for $\mu < 0$ [$\mu > 0$].
\end{enumerate}
The maximum $x$-value of the limit cycle is asymptotically proportional to $\mu^2$,
its minimum $x$-value and minimum and maximum $y$-values are asymptotically proportional to $|\mu|$,
and its period is
\begin{equation}
T = \frac{\pi}{\omega_L(0)} + \cO(\mu).
\label{eq:fixedFocusFoldPeriod}
\end{equation}
\label{th:fixedFocusFold}
\end{theorem}

To illustrate Theorem \eqref{th:fixedFocusFold} we consider
\eqref{eq:bilinearOsc} with $\hat{x} > 0$.
Again $\lambda_L$ and $\omega_L$ are given by \eqref{eq:bilinearOsclambdaomega}.
The bifurcation occurs when $\lambda_L = 0$,
thus here we use $\mu = \nu_1 - b_L$ for which $\beta = \frac{1}{2}$.

To evaluate $\alpha$ we first observe that 
the $\ell_3 \frac{\partial^2 g_L}{\partial y^2}$-term
is only potential contribution to $\chi_{{\rm focus},L}$, see \eqref{eq:focustau},
but $\ell_3 = \frac{k_L(\nu_1 - b_L)}{2}$ is zero at the bifurcation hence $\chi_{{\rm focus},L} = 0$.
From \eqref{eq:foldsigma} we obtain
$\sigma_{{\rm fold},R} = \frac{\nu_1 - b_L - b_R}{-k_R \hat{x}}$
and so $\alpha = \frac{-b_R}{3 k_R \hat{x}}$.
Fig.~\ref{fig:allFixedFocusFold} illustrates the dynamics near the bifurcation
with typical parameter values.

\begin{figure*}
\begin{center}
\includegraphics[width=16.6cm]{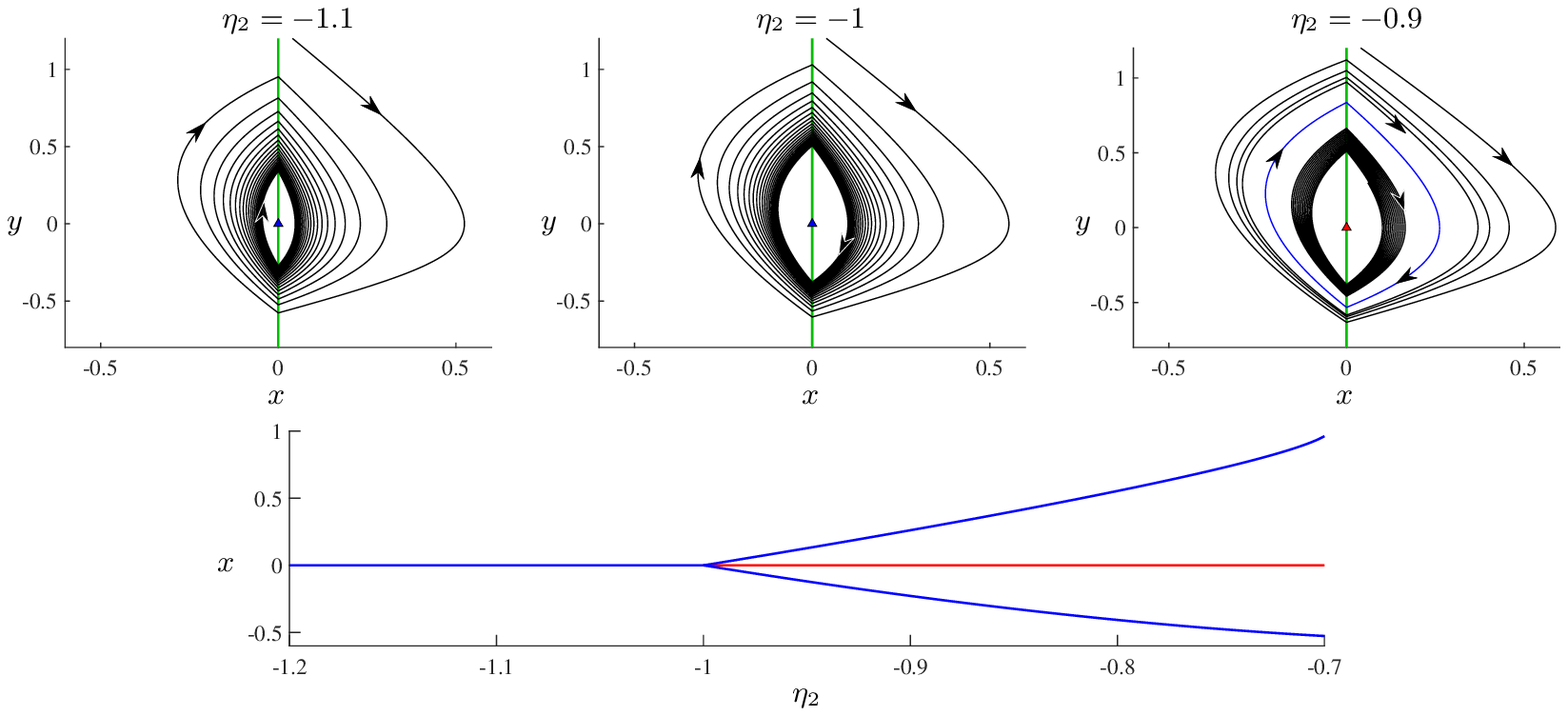}
\caption{
An illustration of HLB 10 for \eqref{eq:fixedTwoFoldODE} with $\eta_1 = 1$.
By increasing the value of $\eta_2$, the origin (an invisible-invisible two-fold) loses stability
when $\nu_2 = -1$ and emits a stable limit cycle.
Take care to note that the bifurcation diagram only
shows the asymptotically linear growth of the minimum and maximum
$x$-values of the limit cycle;
the amplitude (or diameter) of the limit cycle
is asymptotically proportional to $\sqrt{\eta_2 + 1}$ (see Theorem \ref{th:fixedTwoFold}).
\label{fig:allFixedTwoFold}
} 
\end{center}
\end{figure*}

\subsection{The two-fold case --- HLB 10}
\label{sub:fixedFolds}

We now suppose that for all values of $\mu$ in a neighbourhood of $0$,
the origin is an invisible fold for both half-systems of \eqref{eq:FilippovBEB}.
Thus
\begin{align}
f_L(0,0;\mu) &= 0, & f_R(0,0;\mu) &= 0,
\label{eq:fixedTwoFoldCond}
\end{align}
and $\gamma_L > 0$ and $\gamma_R < 0$ where
\begin{equation}
\gamma_J = \left( \frac{\partial f_J}{\partial y} \,g_J \middle) \right|_{x = y = \mu = 0},
\label{eq:fixedTwoFoldInvisCond}
\end{equation}
for each $J \in \{ L,R \}$.
Again suppose both half-systems involve the same direction of rotation.

Let $\sigma_{{\rm fold},J}(\mu)$ and $\chi_{{\rm fold},J}(\mu)$
be the result of evaluating \eqref{eq:foldsigma} and \eqref{eq:foldtau} with the vector field $F = F_J$.
As shown in Appendix \ref{app:fixed}, the stability of the origin is governed by the sign of
\begin{equation}
\Lambda(\mu) = \sigma_{{\rm fold},L}(\mu) - \sigma_{{\rm fold},R}(\mu).
\label{eq:fixedTwoFoldLambda}
\end{equation}
Thus for HLB 10 we assume $\Lambda(0) = 0$ and let
\begin{equation}
\beta = \frac{d \Lambda}{d \mu}(0).
\label{eq:fixedTwoFoldTransCond}
\end{equation}
We also let
\begin{equation}
\alpha = \chi_{{\rm fold},L}(0) - \chi_{{\rm fold},R}(0).
\label{eq:fixedTwoFoldNondegCond}
\end{equation}

\begin{theorem}[HLB 10]
Consider \eqref{eq:FilippovBEB} where $F_L$ and $F_R$ are $C^4$.
Suppose \eqref{eq:fixedTwoFoldCond} is satisfied,
$\frac{\partial f_L}{\partial y}(0,0;0) \frac{\partial f_R}{\partial y}(0,0;0) > 0$,
$\Lambda(0) = 0$,
$\beta > 0$, $\gamma_L > 0$, and $\gamma_R < 0$.
In a neighbourhood of $(x,y;\mu) = (0,0;0)$,
\begin{enumerate}
\item
the origin is the unique stationary solution and is stable for $\mu < 0$ and unstable for $\mu > 0$,
\item
if $\alpha < 0$ [$\alpha > 0$] there exists a unique stable [unstable] limit cycle for $\mu > 0$ [$\mu < 0$],
and no limit cycle for $\mu < 0$ [$\mu > 0$].
\end{enumerate}
The minimum and maximum $x$-values of the limit cycle are asymptotically proportional to $|\mu|$,
and the minimum and maximum $y$-values of the limit cycle are asymptotically proportional to $\sqrt{|\mu|}$,
and the period is
\begin{equation}
T = \left( \frac{2}{g_L(0,0;0)} - \frac{2}{g_R(0,0;0)} \right)
\sqrt{\frac{-5 \beta \mu}{\alpha}} + \co \left( \sqrt{|\mu|} \right).
\label{eq:fixedTwoFoldPeriod}
\end{equation}
\label{th:fixedTwoFold}
\end{theorem}

As an abstract example consider
\begin{equation}
\begin{bmatrix} \dot{x} \\ \dot{y} \end{bmatrix} =
\begin{cases}
\begin{bmatrix} x + y \\ 1 \end{bmatrix}, & x < 0, \\
\begin{bmatrix} y \\ -1 + \eta_1 x + \eta_2 y \end{bmatrix}, & x > 0,
\end{cases}
\label{eq:fixedTwoFoldODE}
\end{equation}
where $\eta_1, \eta_2 \in \mathbb{R}$ are parameters.
Here $\sigma_{{\rm fold},L} = 1$ and $\sigma_{{\rm fold},R} = -\eta_2$,
so in Theorem \ref{th:fixedTwoFold} we use $\mu = \eta_2 + 1$, for which $\beta = 1$.
Also $\chi_{{\rm fold},L} = \frac{22}{9}$
and $\chi_{{\rm fold},R} = \frac{22}{9} + \eta_1$, thus $\alpha = -\eta_1$.
Fig.~\ref{fig:allFixedTwoFold} shows phase portraits and a bifurcation diagram using $\eta_1 = 1$.

\section{Hybrid systems}
\setcounter{equation}{0}
\setcounter{figure}{0}
\setcounter{table}{0}
\label{sec:hybrid}

State variables of hybrid systems experience jumps between periods of continuous evolution.
The continuous evolution is governed by differential equations,
the jumps are defined by maps,
and the maps are applied when certain time-dependent or state-dependent conditions are met.
Hybrid systems naturally model mechanical systems with hard impacts,
ecological systems with periodic culling or addition of population,
and many control systems.
A smooth dynamical system becomes hybrid
when it is combined with control law that gives an occasional kick to the state of the system.
Such control laws are often both highly effective and efficient \cite{ChEl05}.

Hybrid systems are broad and this is reflected in
the theory that has been developed for them \cite{HaCh06,LeNi10,Pl10,VaSc00}.
Limit cycles involving jumps are often large amplitude,
and multi-jump limit cycles are often best analysed via graph-theoretic techniques \cite{MaSa00}.
This paper concerns small-amplitude limit cycles for which a local analysis is appropriate.

In \S\ref{sub:impacting} we study BEBs in hybrid systems involving a single map that leaves one variable unchanged.
Such systems arise when impacts are modelled as instantaneous events
(velocity is reversed but position is unchanged).
In \S\ref{sub:impulsive} we study hybrid systems with a single map
(most naturally interpreted as an impulse)
where the domain and range of the map intersect at a regular equilibrium.
Proofs are provided in Appendices \ref{app:impact} and \ref{app:impulse}.

\subsection{Impacting systems --- HLBs 11--13}
\label{sub:impacting}

Here we study systems of the form
\begin{equation}
\begin{split}
\begin{bmatrix} \dot{x} \\ \dot{y} \end{bmatrix} &= F(x,y;\mu), {\rm ~for~} x < 0, \\
y &\mapsto \phi(y;\mu), {\rm ~when~} x = 0,
\end{split}
\label{eq:impactingODE}
\end{equation}
and write
\begin{equation}
F(x,y;\mu) = \begin{bmatrix} f(x,y;\mu) \\ g(x,y;\mu) \end{bmatrix},
\nonumber
\end{equation}
where $f$, $g$, and $\phi$ are smooth functions.
Orbits evolve in $\Omega_L = \left\{ (x,y) \,\middle|\, x < 0, y \in \mathbb{R} \right\}$
following the vector field $F$ until reaching $x=0$
at which time the map $\phi$ is applied.
The variables $x(t)$ and $y(t)$ may be interpreted as the position and velocity of an object that undergoes
instantaneous impacts at $x=0$, where $\phi$ is the impact law.
Define
\begin{equation}
{\rm sgn}(a) = \begin{cases}
-1, & a < 0, \\
0, & a = 0, \\
1, & a > 0.
\end{cases}
\label{eq:sgn}
\end{equation}
We assume
\begin{equation}
{\rm sgn} \left( f(0,y;\mu) \right) = {\rm sgn}(y), \quad \text{for all}~ y, \mu \in \mathbb{R},
\label{eq:FLimpactingCond}
\end{equation}
so that orbits following $F$ depart the negative $y$-axis and arrive at the positive $y$-axis,
see Fig.~\ref{fig:schemPoinImpact}.
We also assume
\begin{equation}
{\rm sgn} \left( \phi(y;\mu) \right) = -{\rm sgn}(y),
\quad \text{for all}~ y \ge 0,\, \text{and all}~ \mu \in \mathbb{R}.
\label{eq:phiimpactingCond}
\end{equation}
This ensures that for any positive impact velocity $y$,
the rebound velocity $\phi(y;\mu)$ is negative and that impacts impart no change when $y = 0$.

\begin{figure}[b!]
\begin{center}
\includegraphics[width=4.2cm]{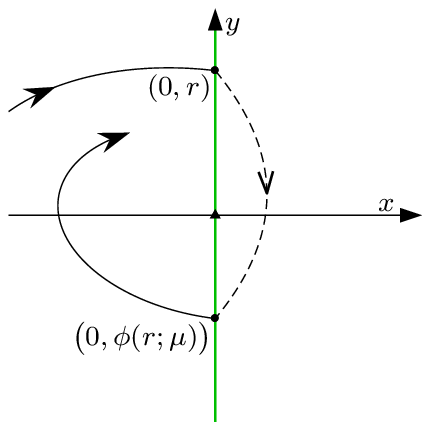}
\caption{
Part of a typical orbit of the impacting system \eqref{eq:impactingODE}.
\label{fig:schemPoinImpact}
} 
\end{center}
\end{figure}

The system \eqref{eq:impactingODE} experiences a BEB when a regular equilibrium collides with $x=0$.
At the bifurcation the equilibrium is necessarily located at $(x,y) = (0,0)$ by \eqref{eq:FLimpactingCond}. 
As detailed in Section 5.1.3 of \cite{DiNo08},
there are four generic topologically distinct scenarios
in which a local limit cycle is created in this type of BEB.
If the equilibrium is a focus,
a limit cycle may exist when the equilibrium is admissible (HLB 11)
or when the equilibrium is virtual (HLB 12).
If the equilibrium is a node,
a limit cycle may exist when the node is virtual (HLB 13).
If the equilibrium is a saddle, a limit cycle may be created
but the BEB is of nonsmooth-fold type (thus here it is not treated as Hopf-like).
As with HLBs 1--4, in generic situations
the local dynamics are determined by linear terms \cite{MaWa18}.

In \cite{DiNo08} this class of BEBs is analysed
by introducing a flow in $x > 0$ that simulates the action of the map $\phi$
and then applying known results for piecewise-smooth continuous systems
(such as Theorems \ref{th:pwscFocusFocus}--\ref{th:pwscFocusNode}).
This approach involves little work although corrections are needed near $x=0$
because the vector field cannot be made to be continuous here for all values of $\mu$.
The theorems presented below are instead proved by
directly analysing a Poincar\'e map
because the computationally intensive calculations are already covered by Lemma \ref{le:affine}.

We suppose the BEB occurs when $\mu = 0$, thus
\begin{equation}
g(0,0;0) = 0.
\label{eq:impactingEqCond}
\end{equation}
Notice $f(0,0;0) = 0$ is already implied by \eqref{eq:FLimpactingCond}.
If the Jacobian matrix $\rD F(0,0;0)$ is invertible
then, locally, \eqref{eq:impactingODE} has a unique regular equilibrium with an $x$-value
of $\frac{-\beta \mu}{\det \left( \rD F(0,0;0) \right)} + \co(\mu)$,
where
\begin{equation}
\beta = -\left( \frac{\partial g}{\partial \mu} \frac{\partial f}{\partial y}
\middle) \right|_{x = y = \mu = 0}.
\label{eq:impactingTransCond}
\end{equation}

Folds, sliding motion, and pseudo-equilibria defined above for Filippov systems
admit analogous definitions for impacting systems.
For details the reader is referred to \cite{DiBu08,DiNo08}.
Since $f(0,0;\mu) = 0$ by \eqref{eq:FLimpactingCond},
if $g(0,0;\mu) \ne 0$ then the origin is a visible fold
if $\beta \mu > 0$ and an invisible fold if $\beta \mu < 0$ (compare Definition \ref{df:fold}).
If the fold is visible, the orbit through the fold simply follows the tangent trajectory.
If the fold is invisible, it is a stationary solution that may be viewed as a pseudo-equilibrium.
If it is stable nearby orbits typically arrive at the pseudo-equilibrium
in {\em finite} time via an {\em infinite} sequence of impacts --- this is the Zeno phenomenon \cite{JoEg99,ZhJo01}.
If it is unstable this occurs in backwards time \cite{NoDa11}.

The following lemma (proved in Appendix \ref{app:impact}) clarifies the existence and stability
of the regular equilibrium and the pseudo-equilibrium (which is the origin when $\beta \mu < 0$).
Here we assume $\det(\rD F(0,0;0)) > 0$ so that these equilibria
are admissible on different sides of the BEB (i.e.~the bifurcation corresponds to persistence not a nonsmooth-fold).
This assumption also implies that the regular equilibrium is not a saddle.
Also we let
\begin{equation}
\gamma = -\frac{\partial \phi}{\partial y}(0;0),
\label{eq:impactingc}
\end{equation}
and notice $\gamma \ge 0$ by \eqref{eq:phiimpactingCond}.

\begin{lemma}
Consider \eqref{eq:impactingODE} where $F$ and $\phi$ are $C^1$.
Suppose \eqref{eq:FLimpactingCond},
\eqref{eq:phiimpactingCond}, and \eqref{eq:impactingEqCond} are satisfied,
$\beta > 0$, and $\det(\rD F(0,0;0)) > 0$.
Let $\lambda = \frac{1}{2} \,{\rm trace}(\rD F(0,0;0))$.
In a neighbourhood of $(x,y;\mu) = (0,0;0)$ there exists a unique stationary solution:
\begin{enumerate}
\item
a stable [unstable] regular equilibrium in $\Omega_L$
if $\lambda < 0$ [$\lambda > 0$] for $\mu > 0$,
\item
and a stable [unstable] pseudo-equilibrium if $\gamma < 1$ [$\gamma > 1$]
at the origin for $\mu < 0$.
\end{enumerate}
\label{le:impacting}
\end{lemma}

Now suppose
\begin{equation}
{\rm eig}(\rD F(0,0;0)) = \lambda \pm \ri \omega, \quad
{\rm with~} \lambda \in \mathbb{R}, \omega > 0,
\label{eq:impactingEigCond}
\end{equation}
so that the regular equilibrium is a focus.
Consider the forward orbit of a point $(0,r)$, with $r > 0$, when $\mu = 0$.
This point is mapped to
$\left( 0,-\gamma r + \co(r) \right)$,
then revolves to $\left( 0, \gamma \re^{\frac{\lambda \pi}{\omega}} r + \co(r) \right)$
by Lemma \ref{le:focus}.
For this reason the sign of
\begin{equation}
\alpha = \ln(\gamma) + \frac{\lambda \pi}{\omega},
\label{eq:impactingNondegCond}
\end{equation}
determines the stability of the origin when $\mu = 0$ (assuming $\gamma > 0$).

In order for a limit cycle to be created at $\mu = 0$
there needs to be competing actions of attraction and repulsion.
Specifically we require $\lambda \ln(\gamma) < 0$ so that (by Lemma \ref{le:impacting}) the regular equilibrium
and pseudo-equilibrium are of opposite stability.
There are two cases:
If $\lambda \alpha < 0$ then the stability of the regular equilibrium
differs from the stability of the origin when $\mu = 0$ (HLB 11).
Here a limit cycle exists when the regular equilibrium is admissible.
Notice the limit cycle completes more than half a revolution about the equilibrium
so its period $T$ satisfies $\frac{\pi}{\omega} < T < \frac{2 \pi}{\omega}$.
If $\lambda \alpha > 0$ then the stability of the regular equilibrium
is the same as the stability of the origin when $\mu = 0$ (HLB 12).
Here a limit cycle exists when the regular equilibrium is virtual
(and so $0 < T < \frac{\pi}{\omega}$).
In both cases the stability of the limit cycle
is the same as the stability of the origin when $\mu = 0$.

\begin{theorem}[HLB 11]
Consider \eqref{eq:impactingODE} where $F$ and $\phi$ are $C^2$.
Suppose \eqref{eq:FLimpactingCond},
\eqref{eq:phiimpactingCond}, \eqref{eq:impactingEqCond}, \eqref{eq:impactingEigCond} are satisfied,
$\beta > 0$,
$\lambda \ln(\gamma) < 0$,
and $\lambda \alpha < 0$.
In a neighbourhood of $(x,y;\mu) = (0,0;0)$,
if $\alpha < 0$ [$\alpha > 0$]
then there exists a unique stable [unstable] limit cycle
for $\mu > 0$, and no limit cycle for $\mu < 0$.
The minimum $x$-value and the minimum and maximum $y$-values
of the limit cycle are asymptotically proportional to $\mu$,
and its period is $T = T_0 + \cO(\mu)$, where $T_0$ satisfies
\begin{equation}
\gamma \re^{2 \lambda T_0} = \frac{\varrho \left( \omega T_0; \frac{\lambda}{\omega} \right)}
{\varrho \left( \omega T_0; -\frac{\lambda}{\omega} \right)},
\label{eq:periodImpactingA}
\end{equation}
and $\varrho$ is defined by \eqref{eq:auxFunc}.
\label{th:impactingA}
\end{theorem}

\begin{theorem}[HLB 12]
Consider \eqref{eq:impactingODE} where $F$ and $\phi$ are $C^2$.
Suppose \eqref{eq:FLimpactingCond},
\eqref{eq:phiimpactingCond}, \eqref{eq:impactingEqCond}, \eqref{eq:impactingEigCond} are satisfied, $\beta > 0$,
$\lambda \ln(\gamma) < 0$, and $\lambda \alpha > 0$.
In a neighbourhood of $(x,y;\mu) = (0,0;0)$,
if $\alpha < 0$ [$\alpha > 0$]
then there exists a unique stable [unstable] limit cycle
for $\mu < 0$, and no limit cycle for $\mu > 0$.
The minimum $x$-value and the minimum and maximum $y$-values
of the limit cycle are asymptotically proportional to $\mu$,
and its period is $T = T_0 + \cO(\mu)$, where $T_0$ satisfies \eqref{eq:periodImpactingA}.
\label{th:impactingB}
\end{theorem}

Now suppose the regular equilibrium is a node and write
\begin{equation}
{\rm eig}(\rD F(0,0;0)) = \lambda \pm \eta, \quad
{\rm with~} 0 < \eta < |\lambda|.
\label{eq:impactingEigCond2}
\end{equation}
In this case a local limit cycle cannot exist when the regular equilibrium is admissible.

\begin{theorem}[HLB 13]
Consider \eqref{eq:impactingODE} where $F$ and $\phi$ are $C^2$.
Suppose \eqref{eq:FLimpactingCond},
\eqref{eq:phiimpactingCond}, \eqref{eq:impactingEqCond}, \eqref{eq:impactingEigCond2} are satisfied, $\beta > 0$,
and $\lambda \ln(\gamma) < 0$.
In a neighbourhood of $(x,y;\mu) = (0,0;0)$,
if $\lambda < 0$ [$\lambda > 0$]
there exists an asymptotically stable [unstable] limit cycle
for small $\mu < 0$, and no limit cycle for small $\mu > 0$.
The minimum $x$-value and the minimum and maximum $y$-values
of the limit cycle are asymptotically proportional to $\mu$,
and its period is $T = T_0 + \cO(\mu)$, where $T_0$ satisfies \eqref{eq:periodImpactingA}
with $\omega = \ri \eta$.
\label{th:impactingC}
\end{theorem}

\begin{figure*}
\begin{center}
\includegraphics[width=16.6cm]{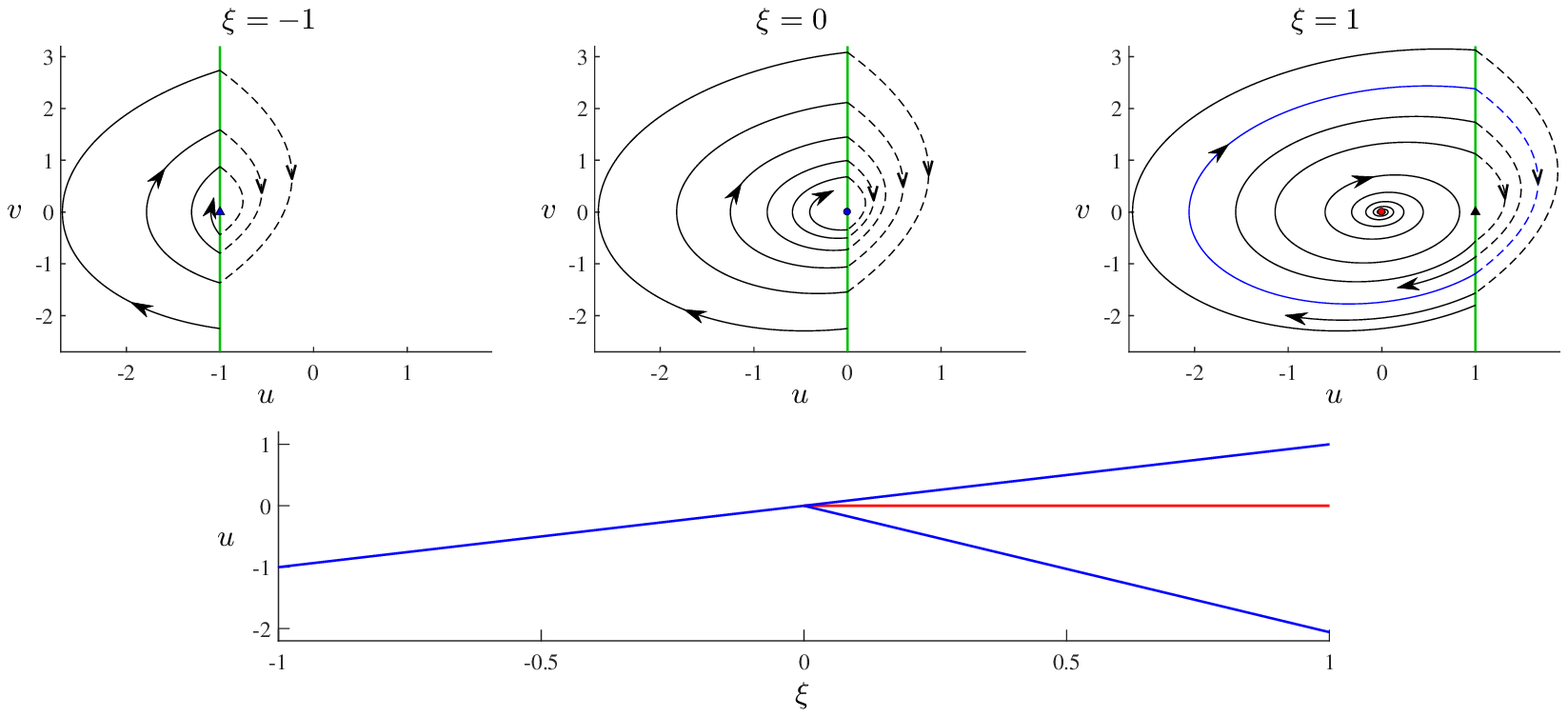}
\caption{
An illustration of HLB 11 for the impacting system
\eqref{eq:impactingExample} with $(\tau,\delta,r) = (0.2,1,0.5)$.
By increasing the value of $\xi$ a virtual unstable focus
becomes admissible and a stable limit cycle is created.
\label{fig:allImpactingA}
} 
\end{center}
\end{figure*}

\begin{figure*}
\begin{center}
\includegraphics[width=16.6cm]{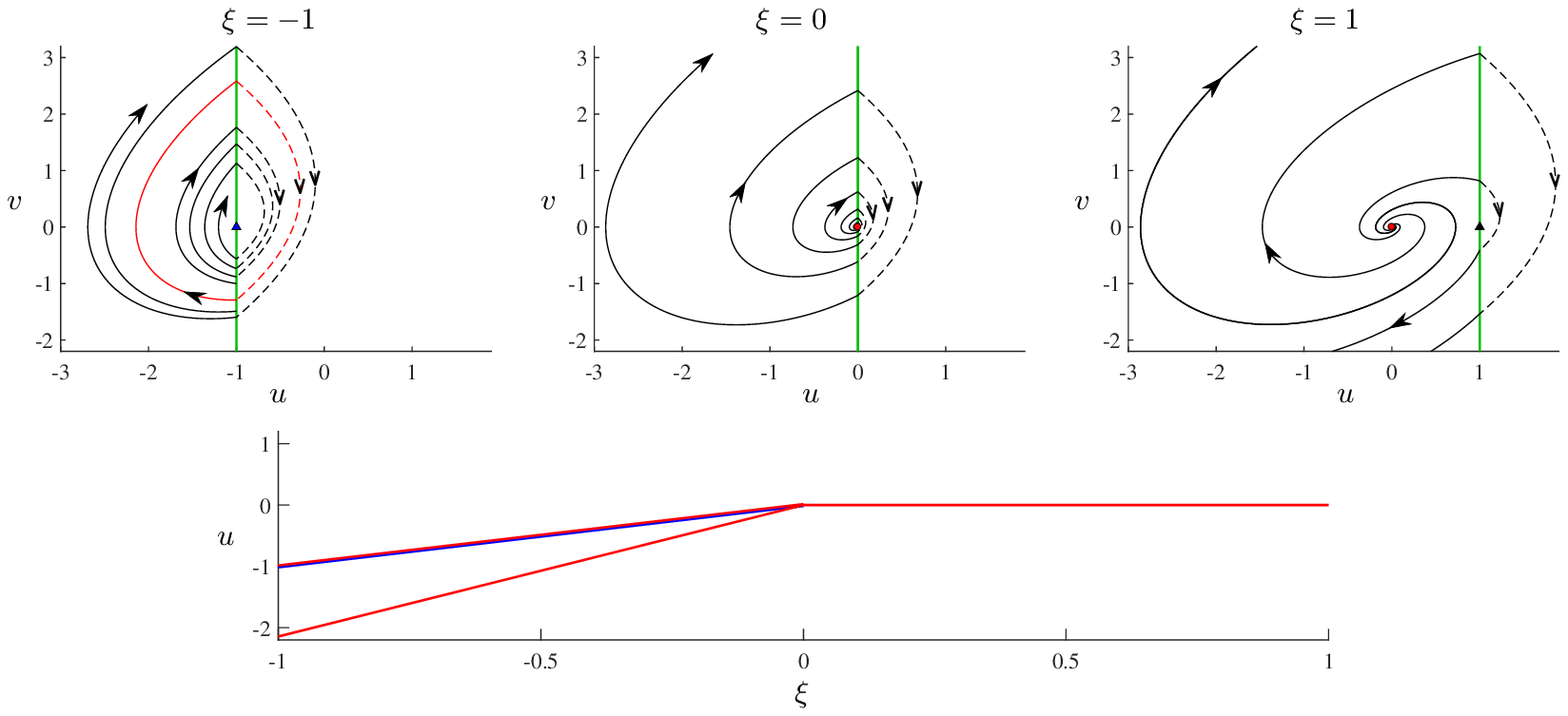}
\caption{
An illustration of HLB 12 for the impacting system
\eqref{eq:impactingExample} with $(\tau,\delta,r) = (0.8,1,0.5)$.
By decreasing the value of $\xi$
an admissible unstable focus undergoes a BEB with the impacting surface and an unstable limit cycle is created.
\label{fig:allImpactingB}
} 
\end{center}
\end{figure*}

\begin{figure*}
\begin{center}
\includegraphics[width=16.6cm]{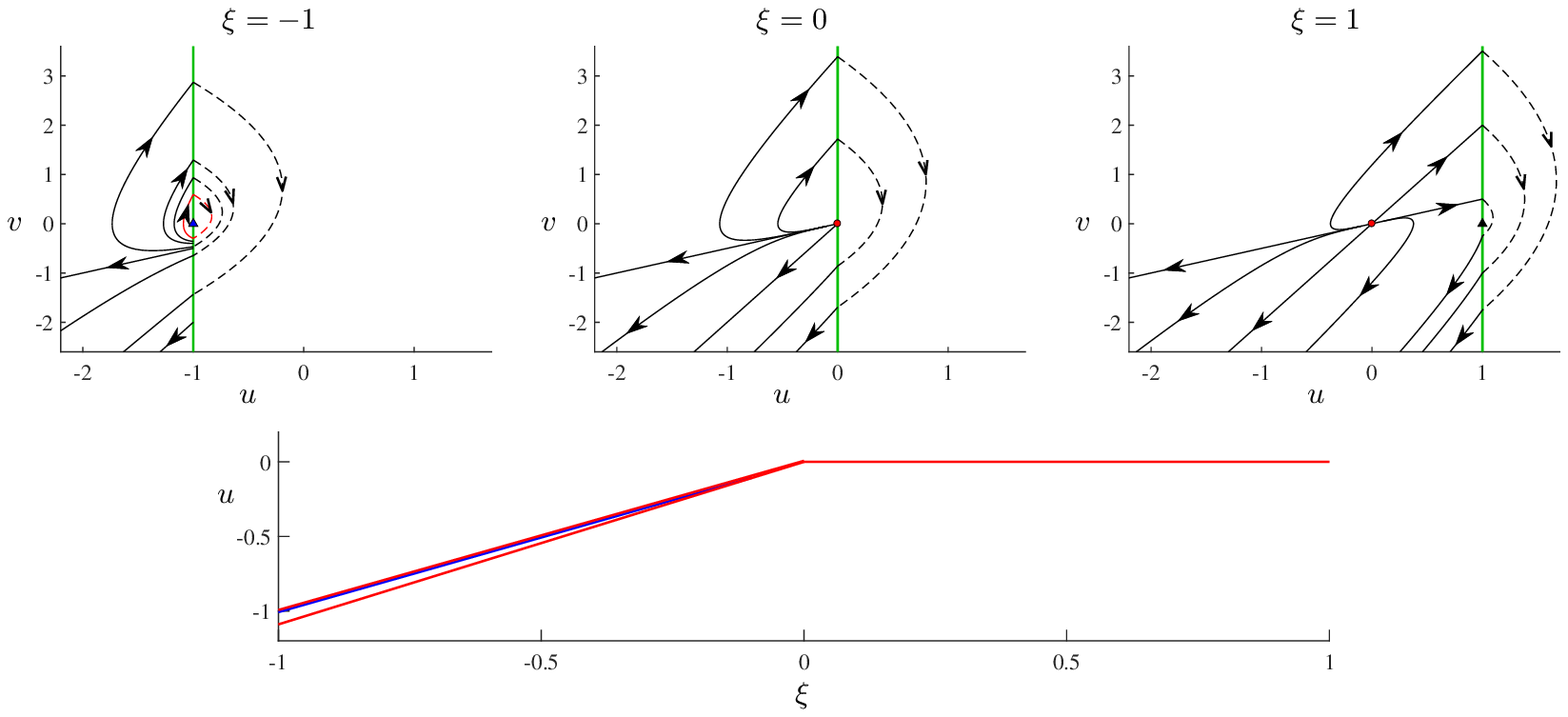}
\caption{
An illustration of HLB 13 for the impacting system
\eqref{eq:impactingExample} with $(\tau,\delta,r) = (2.5,1,0.5)$.
By decreasing the value of $\xi$
an admissible unstable node undergoes a BEB with the impacting surface and an unstable limit cycle is created.
\label{fig:allImpactingC}
} 
\end{center}
\end{figure*}

As an example we consider the linear impact oscillator
\begin{equation}
\begin{split}
\dot{u} &= v, \\
\dot{v} &= -\delta u + \tau v, \\
v &\mapsto -r v, {\rm ~when~} u = \xi.
\end{split}
\label{eq:impactingExample}
\end{equation}
where $\tau,\delta,\xi \in \mathbb{R}$ and $r \ge 0$ are constants.
Here $u(t)$ represents the displacement of the oscillator from its equilibrium position.
The oscillator undergoes instantaneous impacts with restitution coefficient $r$ whenever $u(t) = \xi$.

We treat $\xi$ as the main bifurcation parameter and apply 
Theorems \ref{th:impactingA}--\ref{th:impactingC} via the substitution $(x,y;\mu) = (u-\xi,v;\xi)$.
Notice $\tau$ and $\delta$ are the trace and determinant of $\rD F(0,0;0)$.
Also
\begin{align}
\alpha &= \ln{r} + \frac{\pi \tau}{\sqrt{4 \delta - \tau^2}}, &
\beta &= \delta, &
\gamma &= r.
\nonumber
\end{align}

Here we briefly consider parameter values
that illustrate HLBs 11--13.
Specifically we fix
\begin{align}
\delta &= 1, & r = 0.5,
\nonumber
\end{align}
and consider different values of $\tau$.
We study the change in the dynamics as the value of $\xi$ changes sign,
but \eqref{eq:impactingExample} is invariant when $u$, $v$, and $\xi$ are scaled by a positive constant,
hence it suffices to consider $\xi \in \{ -1, 0, 1 \}$.
With $\tau = 0.2$ the regular equilibrium is an unstable focus and $\alpha < 0$.
In accordance with Theorem \ref{th:impactingA}, a stable limit cycle
exists when the focus is admissible, Fig.~\ref{fig:allImpactingA}.
With $\tau = 0.8$, now $\alpha > 0$ and by Theorem \ref{th:impactingB} an unstable limit cycle 
exists when the focus is virtual, Fig.~\ref{fig:allImpactingB}.
Finally with $\tau = 2.5$ the regular equilibrium is an unstable node.
Again $\alpha > 0$ and so by Theorem \ref{th:impactingC} an unstable limit cycle 
exists when the node is virtual, Fig.~\ref{fig:allImpactingC}.

\subsection{Impulsive systems --- HLB 14}
\label{sub:impulsive}

\begin{figure}[b!]
\begin{center}
\includegraphics[width=5.6cm]{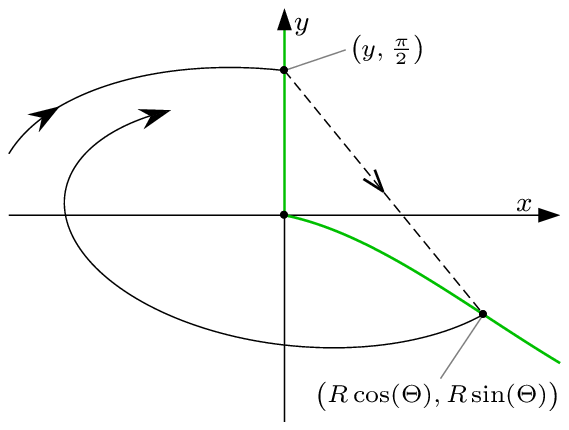}
\caption{
Part of a typical orbit of the impulsive system \eqref{eq:impulsiveODE} with \eqref{eq:impulsiveLaw}.
\label{fig:schemPoinImpulse}
} 
\end{center}
\end{figure}

Here we consider ODEs combined with an impulse law (defined below) that
is applied when orbits reach the positive $y$-axis,
see Fig.~\ref{fig:schemPoinImpulse}.
We suppose the origin is a focus equilibrium of the ODEs and that here the impulse is zero.
The origin is a stationary solution that
may change stability and emit a limit cycle under parameter variation; this is HLB 14.
The bifurcation can be generalised to allow several independent vector fields and impulse laws each acting 
in regions that share a vertex at the origin, \cite{AkKa17,Ak05,HuYa12}.

If the impulse law maps points to the negative $y$-axis we recover the impacting scenario of \S\ref{sub:impacting},
although here the regular equilibrium is fixed at the origin
whereas for HLBs 11--13 the bifurcation is triggered by the collision of a regular equilibrium with $x=0$ (i.e.~a BEB).
HLB 14 was described in \cite{HeMo17} for an abstract impacting system
although there the nonlinearity in the ODEs is cubic, not quadratic,
and consequently the amplitude of the limit cycle is asymptotically proportional to
square-root of the parameter change.

Let us clarify the class of impulsive systems under consideration.
We write the ODEs as
\begin{equation}
\begin{bmatrix} \dot{x} \\ \dot{y} \end{bmatrix} = F(x,y;\mu), {\rm ~until~} x=0,\, y>0,
\label{eq:impulsiveODE}
\end{equation}
with
\begin{equation}
F(x,y;\mu) = \begin{bmatrix} f(x,y;\mu) \\ g(x,y;\mu) \end{bmatrix},
\nonumber
\end{equation}
where $f$ and $g$ are smooth functions.
We assume
\begin{equation}
f(0,0;\mu) = g(0,0;\mu),
\label{eq:impulsiveCond}
\end{equation}
for all values of $\mu$ in a neighbourhood $0$
so that \eqref{eq:impulsiveODE} has a regular equilibrium fixed at the origin.
We assume this equilibrium is a focus and that its associated
Jacobian matrix is in real-Jordan form:
\begin{equation}
\rD F(0,0;\mu) = \begin{bmatrix} \lambda(\mu) & \omega(\mu) \\
-\omega(\mu) & \lambda(\mu) \end{bmatrix},
\label{eq:impulsiveDFCond}
\end{equation}
where $\lambda(\mu) \in \mathbb{R}$ and $\omega(\mu) > 0$.
This assumption is included so that Theorem \ref{th:impulsive}, given below, is reasonably succinct.
As with other results in this paper,
in order to apply Theorem \ref{th:impulsive} to a particular model
one would use a change of variables to satisfy \eqref{eq:impulsiveDFCond}.

To state Theorem \ref{th:impulsive} we use polar coordinates:
$x = r \cos(\theta)$, $y = r \sin(\theta)$.
The positive $y$-axis corresponds to $\theta = \frac{\pi}{2}$
and we assume $\theta \in \left( -\frac{3 \pi}{2}, \frac{\pi}{2} \right]$.
Since $\omega > 0$, orbits following the vector field $F$ rotate clockwise (at least near the origin).
Thus the value of $\theta$ decreases with time until $\theta = -\frac{3 \pi}{2}$
when we reset it to $\theta = \frac{\pi}{2}$ and apply the impulse law.
We write the impulse law as
\begin{equation}
\left( y, \frac{\pi}{2} \right) \mapsto \left( R(y;\mu), \Theta(y;\mu) \right).
\label{eq:impulsiveLaw}
\end{equation}
where $R$ and $\Theta$ are smooth functions.
We assume $R(0;\mu) = 0$, for all $\mu$,
so that the impulse is zero at the origin, and let
\begin{equation}
\gamma(\mu) = \frac{\partial R}{\partial y}(0;\mu).
\label{eq:impulsivegamma}
\end{equation}
We also assume $\Theta(y;\mu) \in \left( \frac{-3 \pi}{2}, \frac{\pi}{2} \right)$, and let
\begin{equation}
\phi(\mu) = \Theta(0;\mu).
\label{eq:impulsivephi}
\end{equation}
Thus the impulse law can be written as
\begin{equation}
\left( y, \frac{\pi}{2} \right) \mapsto \left( \gamma y + \cO \left( y^2 \right), \phi + \cO(y) \right).
\end{equation}
Upon subsequent evolution via \eqref{eq:impulsiveODE}, the orbit returns to $\theta = \frac{\pi}{2}$
at $r = \gamma y \re^{\lambda t} + \cO \left( y^2 \right)$,
where the evolution time is $t = \frac{1}{\omega} \left( \phi + \frac{3 \pi}{2} \right) + \cO(y)$.
Thus one revolution corresponds to $y \mapsto \re^\Lambda y + \cO \left( y^2 \right)$, where
\begin{equation}
\Lambda(\mu) = \ln(\gamma(\mu)) + \frac{\lambda(\mu)}{\omega(\mu)}
\left( \phi(\mu) + \frac{3 \pi}{2} \right).
\label{eq:impulsivexi}
\end{equation}
This shows that the stability of the origin is governed by the sign of $\Lambda$.
For HLB 14 we assume $\Lambda(0) = 0$ and $\beta \ne 0$ where
\begin{equation}
\beta = \frac{d \Lambda}{d \mu}(0).
\label{eq:impulsiveTransCond}
\end{equation}

Some work is required to state the non-degeneracy coefficient $\alpha$.
First we write
\begin{equation}
\begin{split}
f(x,y;\mu) &= \lambda(\mu) x + \omega(\mu) y \\
&\quad+ a_1 x^2 + a_2 x y + a_3 y^2 + \co \left( \left( |x| + |y| + |\mu| \right)^2 \right), \\
g(x,y;\mu) &= -\omega(\mu) x + \lambda(\mu) y \\
&\quad+ b_1 x^2 + b_2 x y + b_3 y^2 + \co \left( \left( |x| + |y| + |\mu| \right)^2 \right),
\end{split}
\label{eq:impulsivefg}
\end{equation}
for constants $a_1,\ldots,b_3 \in \mathbb{R}$.
In polar coordinates \eqref{eq:impulsiveODE} may be written as
\begin{equation}
\begin{split}
\dot{r} &= \lambda(\mu) r + r^2 \cF(\theta) + \co \left( \left( r + |\mu| \right)^2 \right), \\
\dot{\theta} &= -\omega(\mu) + r \cG(\theta) + \co \left( r + |\mu| \right),
\end{split}
\label{eq:impulsiveODEpolar}
\end{equation}
where
\begin{equation}
\begin{split}
\cF(\theta) &= a_1 \cos^3(\theta) + (a_2 + b_1) \cos^2(\theta) \sin(\theta) \\
&\quad+ (a_3 + b_2) \cos(\theta) \sin^2(\theta) + b_3 \sin^3(\theta), \\
\cG(\theta) &= b_1 \cos^3(\theta) + (b_2 - a_1) \cos^2(\theta) \sin(\theta) \\
&\quad+ (b_3 - a_2) \cos(\theta) \sin^2(\theta) - a_3 \sin^3(\theta).
\end{split}
\label{eq:impulsivecFcG}
\end{equation}
Then
\begin{align}
\alpha &= \frac{\frac{\partial^2 R}{\partial y^2}(0;0)}{2 \gamma(0)}
+ \frac{\lambda}{\omega} \,\frac{\partial \Theta}{\partial y}(0;0) \nonumber \\
&\quad+ \frac{1}{\omega} \,\re^{\frac{-3 \pi \lambda}{2 \omega}}
\int_{\frac{-3 \pi}{2}}^{\phi(0)}
\re^{\frac{-\lambda \theta}{\omega}} \left( \cF(\theta) + \frac{\lambda}{\omega} \,\cG(\theta) \right) \,d\theta.
\label{eq:impulsiveNondegCond}
\end{align}
The integral in \eqref{eq:impulsiveNondegCond}
can be evaluated explicitly but we have left it in integral form for brevity.

\begin{theorem}[HLB 14]
Consider \eqref{eq:impulsiveODE} with \eqref{eq:impulsiveLaw},
where $F$ is $C^2$, $R$ is $C^2$, and $\Theta$ is $C^1$.
Suppose \eqref{eq:impulsiveCond} and \eqref{eq:impulsiveDFCond} are satisfied,
$\Lambda(0) = 0$, and $\beta > 0$.
In a neighbourhood of $(x,y;\mu) = (0,0;0)$,
\begin{enumerate}
\item
the origin is the unique stationary solution
and is stable for $\mu < 0$ and unstable for $\mu > 0$,
\item
if $\alpha < 0$ [$\alpha > 0$] there exists a unique stable [unstable] limit cycle for $\mu > 0$ [$\mu < 0$],
and no limit cycle for $\mu < 0$ [$\mu > 0$].
\end{enumerate}
The amplitude of the limit cycle is asymptotically proportional to $|\mu|$,
and its period is
\begin{equation}
T = \frac{1}{\omega(0)} \left( \phi(0) + \frac{3 \pi}{2} \right) + \cO(\mu).
\label{eq:impulsivePeriod}
\end{equation}
\label{th:impulsive}
\end{theorem}

\begin{figure*}
\begin{center}
\includegraphics[width=16.6cm]{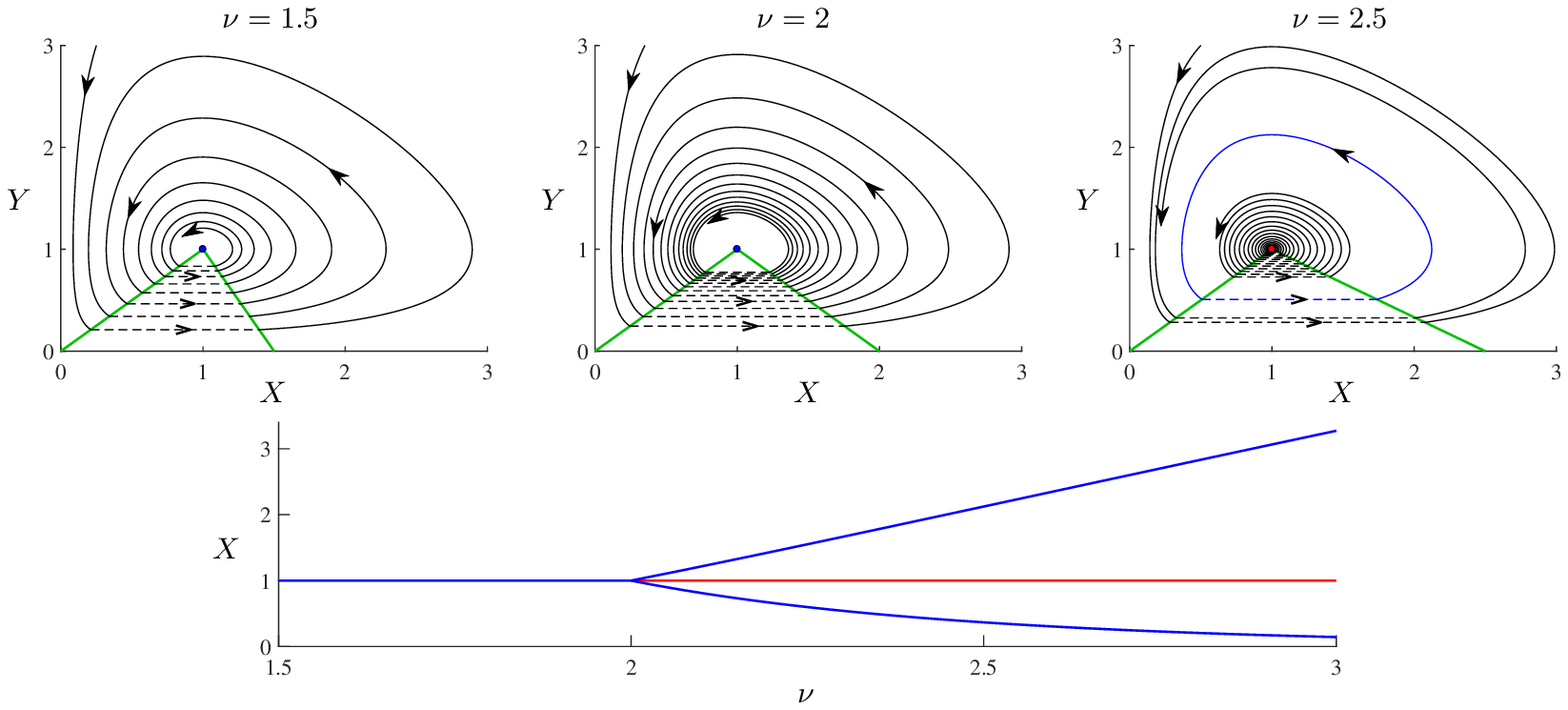}
\caption{
An illustration of HLB 14 for the Lotka-Volterra system \eqref{eq:LotkaVolterra}
with impulse \eqref{eq:LotkaVolterraImpulseLaw} using the parameter values \eqref{eq:paramLotkaVolterra}.
By increasing the value of $\nu$
the equilibrium $(X^*,Y^*) = \left( \frac{c}{d}, \frac{a}{b} \right)$ loses stability when $\nu = 2$
and a stable limit cycle is created.
\label{fig:allLotkaVolterra}
} 
\end{center}
\end{figure*}

To illustrate Theorem \ref{th:impulsive} we study a simple Lotka-Volterra system with impulse.
Such systems model the populations of competing species where impulses correspond to
poison drops, insect outbreaks, or other abrupt events such as the artificial addition
of species members for control purposes \cite{LiZh05,LiCh03,LiRo98,NiPe09}.

Consider
\begin{equation}
\begin{split}
\dot{X} &= a X - b X Y, \\
\dot{Y} &= -c Y + d X Y,
\end{split}
\label{eq:LotkaVolterra}
\end{equation}
where $X$ and $Y$ represent prey and predator populations respectively,
and $a, b, c, d > 0$.
This system has a saddle equilibrium at the origin,
a centre equilibrium at
$(X^*,Y^*) = \left( \frac{c}{d}, \frac{a}{b} \right)$,
and an uncountable family of periodic orbits encircling $(X^*,Y^*)$ throughout the first quadrant, $X, Y > 0$.

Now suppose that whenever a forward orbit of \eqref{eq:LotkaVolterra}
intersects the line segment connecting the two equilibria,
the value $\nu \left( X^* - X \right)$ is added to the prey population, where $\nu > 0$ is a constant.
That is,
\begin{equation}
(X,Y) \mapsto \left( X + \nu \left( \frac{c}{d} - X \right), Y \right), 
\label{eq:LotkaVolterraImpulseLaw}
\end{equation}
where $X \in \left( 0, \frac{c}{d} \right)$,
and $Y = \frac{a d}{b c} X$.
This particular impulse law provides a succinct example of HLB 14.
Fig.~\ref{fig:allLotkaVolterra} shows phase portraits and a bifurcation diagram using
\begin{equation}
a = b = c = d = 1,
\label{eq:paramLotkaVolterra}
\end{equation}
for which HLB 14 occurs at $\nu = 2$.

To apply Theorem \ref{th:impulsive} using the particular parameter values \eqref{eq:paramLotkaVolterra},
we perform the coordinate change
\begin{equation}
\begin{split}
x &= X - Y, \\
y &= -X - Y + 2.
\end{split}
\nonumber
\end{equation}
Then \eqref{eq:LotkaVolterra} with \eqref{eq:paramLotkaVolterra} becomes
\begin{equation}
\begin{split}
\dot{x} &= y + \frac{x^2 - y^2}{2}, \\
\dot{y} &= -x,
\end{split}
\nonumber
\end{equation}
and in polar coordinates the impulse law is
\begin{equation}
\begin{split}
R(y) &= y \sqrt{1 - \nu + \frac{\nu^2}{2}}, \\
\Theta(y) &= \tan^{-1} \left( \frac{2}{\nu} - 1 \right).
\end{split}
\nonumber
\end{equation}
Here $\Lambda = \frac{1}{2} \ln \left( 1 - \nu + \frac{\nu^2}{2} \right)$
and so $\Lambda = 0$ when $\nu = 2$. 
The transversality condition is satisfied because
$\beta = \frac{d \Lambda}{d \nu}(2) = \frac{1}{2}$.
To evaluate $\alpha$ observe that the first two terms in \eqref{eq:impulsiveNondegCond} are zero.
In the integral we have $\lambda = 0$, $\omega = 1$, and
\begin{equation}
\begin{aligned}
a_1 &= \frac{1}{2}, &
a_2 &= 0, &
a_3 &= -\frac{1}{2}, \\
b_1 &= 0, &
b_2 &= 0, &
b_3 &= 0,
\end{aligned}
\nonumber
\end{equation}
and so
\begin{align}
\alpha = \frac{1}{2} \int_{\frac{-3 \pi}{2}}^0 \cos^3(\theta) - \cos(\theta) \sin^2(\theta) \,d\theta = -\frac{1}{6}.
\nonumber
\end{align}
Thus Theorem \ref{th:impulsive} tells us that a stable limit cycle
is created at $\nu = 2$ and exists for $\nu > 2$, as seen in Fig.~\ref{fig:allLotkaVolterra}.

\section{Switching with hysteresis and time delay}
\setcounter{equation}{0}
\setcounter{figure}{0}
\setcounter{table}{0}
\label{sec:perturbations}

Here we suppose the Filippov system
\begin{equation}
\begin{bmatrix} \dot{x} \\ \dot{y} \end{bmatrix} =
\begin{cases}
F_L(x,y), & x < 0, \\
F_R(x,y), & x > 0, \end{cases}
\label{eq:FilippovPerturbation}
\end{equation}
has a stable pseudo-equilibrium
or invisible-invisible two-fold at the origin.
We analyse the limit cycle created by adding hysteresis or time delay to the switching condition
(called {\em border-splitting} in \cite{MaLa12}).
The hysteretic system is
\begin{equation}
\begin{bmatrix} \dot{x} \\ \dot{y} \end{bmatrix} =
\begin{cases} F_L(x,y), & {\rm until~} x = \mu, \\
F_R(x,y), & {\rm until~} x = -\mu,
\end{cases}
\label{eq:hysteresis}
\end{equation}
and the delayed system is
\begin{equation}
\begin{bmatrix} \dot{x}(t) \\ \dot{y}(t) \end{bmatrix} =
\begin{cases} F_L(x(t),y(t)), & x(t-\mu) < 0, \\
F_R(x(t),y(t)), & x(t-\mu) > 0,
\end{cases}
\label{eq:timeDelay}
\end{equation}
where in each case $\mu \ge 0$ is assumed to be small in the analysis below.
As usual we write
\begin{equation}
F_J(x,y) = \begin{bmatrix} f_J(x,y) \\ g_J(x,y) \end{bmatrix},
\label{eq:FJhysteresis}
\end{equation}
for each $J \in \{ L,R \}$.

When the origin is a two-fold it must be stable for a limit cycle to be created.
This is because hysteresis and delay produce an instability.
Similarly when the origin is a pseudo-equilibrium it must belong to an attracting sliding region for a limit cycle to be created.
In this case the stability of the limit cycle (with $\mu > 0$) matches that of the pseudo-equilibrium (with $\mu = 0$).

Such limit cycles are ubiquitous in relay control systems
that use switches to affect control tasks \cite{Ts84}.
{\em Sliding mode control} employs rapid switching to maintain
the system state near a desired manifold \cite{PeBa02,UtGu09,Ut92}.
The rapid switching motion tends to sliding motion
as the switching frequency tends to infinity.
In reality the switching frequency is finite
and the system state `chatters' about the switching manifold.
Chattering can be periodic forming a limit cycle of a size that directly correlates to
discrepancies of the system from its idealised (infinite frequency) limit \cite{Bo09,Ja93}.
Note that chattering typically causes undue strain and wear on mechanical components
so control algorithms that minimise chattering are usually preferred.

A stable pseudo-equilibrium is usually the desired state of a first-order sliding mode control system.
Bang-bang controllers use hysteresis (most simply with no backlash or deadzone)
resulting in a limit cycle \cite{MaHa17}
(see \cite{LiYu08,LoMi88} in the case of delay).
Invisible-invisible two-folds correspond to second-order sliding mode control \cite{BaPi03}.
Calculations of the resulting limit cycle can be found in \cite{Ma17} for hysteresis
and in \cite{Ko17,LePu06,LiYu13} for delay
(see also \cite{Si06} for large delay).
Below we treat a pseudo-equilibrium in \S\ref{sub:psEqPerturbation}
and a two-fold in \S\ref{sub:twoFoldPerturbation}.

\subsection{Perturbing a pseudo-equilibrium --- HLBs 15--16}
\label{sub:psEqPerturbation}

For a system of the form \eqref{eq:FilippovPerturbation}, let
\begin{align}
a_{0L} &= f_L(0,0), &
a_{0R} &= f_R(0,0).
\end{align}
We suppose $a_{0R} < 0 < a_{0L}$ so that the origin belongs to an attracting sliding region.
We also suppose that the origin is a pseudo-equilibrium, hence
\begin{equation}
\left( f_L g_R - f_R g_L \middle) \right|_{x = y = 0} = 0.
\label{eq:hysteresisPseudoEq}
\end{equation}
The stability of the origin is governed by the sign of
\begin{equation}
\alpha = \frac{\partial}{\partial y}
\left( f_L g_R - f_R g_L \middle) \right|_{x = y = 0} \,.
\label{eq:hysteresisNondeg}
\end{equation}
The following theorems describe a local limit cycle
for the hysteretic and delayed systems \eqref{eq:hysteresis}--\eqref{eq:timeDelay}.
As with other results in this paper, the results are proved by analysing a Poincar\'e map.
This is sufficient to completely describe the local dynamics of the hysteretic system.
The delayed system, however, is infinite dimensional \cite{Er09,HaLu93}.
The limit cycle obtained is only unique among orbits for which the time between
any two switches is no less than the delay $\mu$.
For Glass networks with delay, Edwards {\em et.~al.} \cite{EdVa07}
prove the limit cycle is unique under the assumption
that the initial condition history includes no switches.

\begin{theorem}[HLB 15]
Consider the hysteretic system \eqref{eq:hysteresis} where $F_L$ and $F_R$ are $C^2$.
Suppose $a_{0R} < 0 < a_{0L}$ and \eqref{eq:hysteresisPseudoEq} is satisfied.
In a neighbourhood of $(x,y;\mu) = (0,0;0)$,
if $\alpha < 0$ [$\alpha > 0$] there exists a unique stable [unstable] limit cycle for $\mu > 0$.
The minimum and maximum $x$-values of the limit cycle are $\pm \mu$
and its period is
\begin{equation}
T = \left( \frac{2}{a_{0L}} - \frac{2}{a_{0R}} \right) \mu + \cO \left( \mu^2 \right).
\label{eq:hysteresisPseudoEqPeriod}
\end{equation}
\label{th:hysteresisPseudoEq}
\end{theorem}

\begin{theorem}[HLB 16]
Consider the delayed system \eqref{eq:timeDelay} where $F_L$ and $F_R$ are $C^2$.
Suppose $a_{0R} < 0 < a_{0L}$ and \eqref{eq:hysteresisPseudoEq} is satisfied.
In a neighbourhood of $(x,y;\mu) = (0,0;0)$,
treating only orbits for which times between consecutive switches is not less than $\mu$,
if $\alpha < 0$ [$\alpha > 0$] there exists a
unique stable [unstable] limit cycle for $\mu > 0$.
The minimum and maximum $x$-values of the limit cycle are asymptotically proportional to $\mu$,
and its period is
\begin{equation}
T = \left( 2 - \frac{a_{0L}}{a_{0R}} - \frac{a_{0R}}{a_{0L}} \right) \mu + \cO \left( \mu^2 \right).
\label{eq:timeDelayPseudoEqPeriod}
\end{equation}
\label{th:timeDelayPseudoEq}
\end{theorem}

As an example we use the observer canonical form for a two-dimensional relay control system \cite{Bu01}:
\begin{equation}
\begin{bmatrix} \dot{x} \\ \dot{y} \end{bmatrix} =
\begin{bmatrix} \tau & 1 \\ -\delta & 0 \end{bmatrix}
\begin{bmatrix} x \\ y \end{bmatrix} - {\rm sgn}(x)
\begin{bmatrix} b_1 \\ b_2 \end{bmatrix}.
\label{eq:hysteresisExampleEq}
\end{equation}
This system is of the form \eqref{eq:FilippovPerturbation}.
If $b_1 > 0$, then the interval $-b_1 < y < b_1$ is an attracting sliding region
and the origin is a pseudo-equilibrium.
If also $b_2 > 0$, then the origin is stable.
In this case the inclusion of hysteresis or delay generates a stable limit cycle
in accordance with Theorems \ref{th:hysteresisPseudoEq} and \ref{th:timeDelayPseudoEq}
as shown in Figs.~\ref{fig:allHysteresisEq} and \ref{fig:allTimeDelayEq} for the parameter values
\begin{align}
\tau &= -0.5, & 
\delta &= 1, & 
b_1 &= 1, &
b_2 &= 1.
\label{eq:paramHysteresisEq}
\end{align}

\begin{figure*}
\begin{center}
\includegraphics[width=10.9cm]{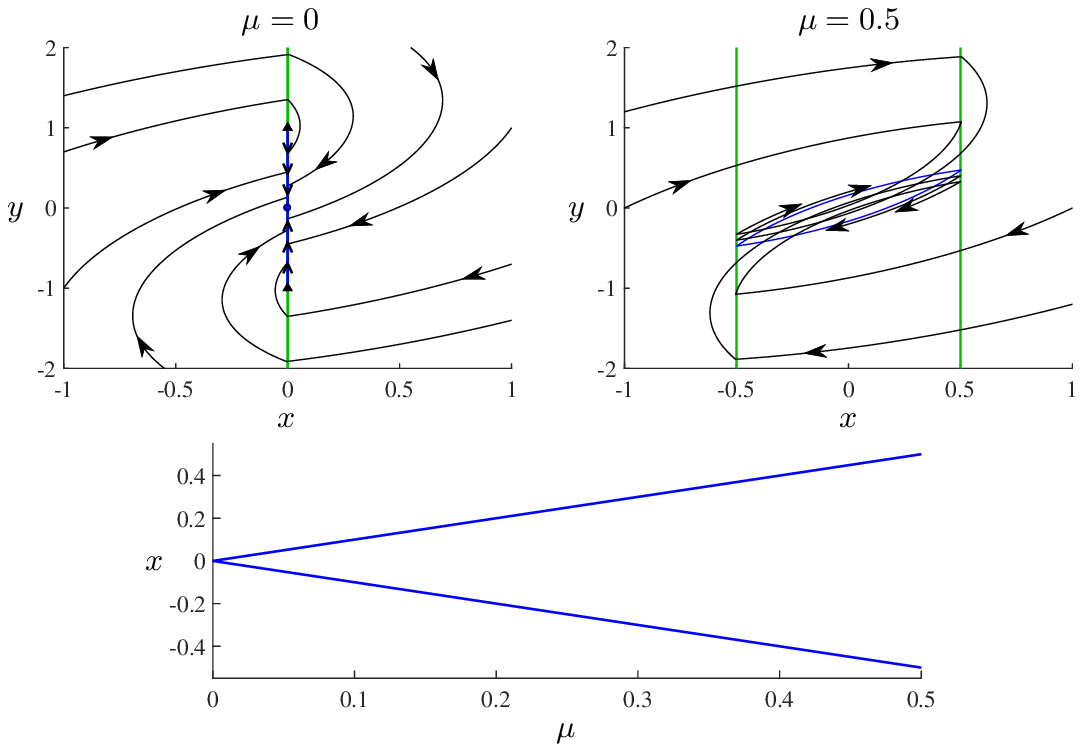}
\caption{
An illustration of HLB 15 for the canonical relay control system \eqref{eq:hysteresisExampleEq}
using hysteresis of the form \eqref{eq:hysteresis} and the parameter values \eqref{eq:paramHysteresisEq}.
By increasing the value of $\mu$ from $0$,
a stable pseudo-equilibrium is replaced by a limit cycle with period asymptotically proportional to $\mu$.
\label{fig:allHysteresisEq}
} 
\end{center}
\end{figure*}

\begin{figure*}
\begin{center}
\includegraphics[width=10.9cm]{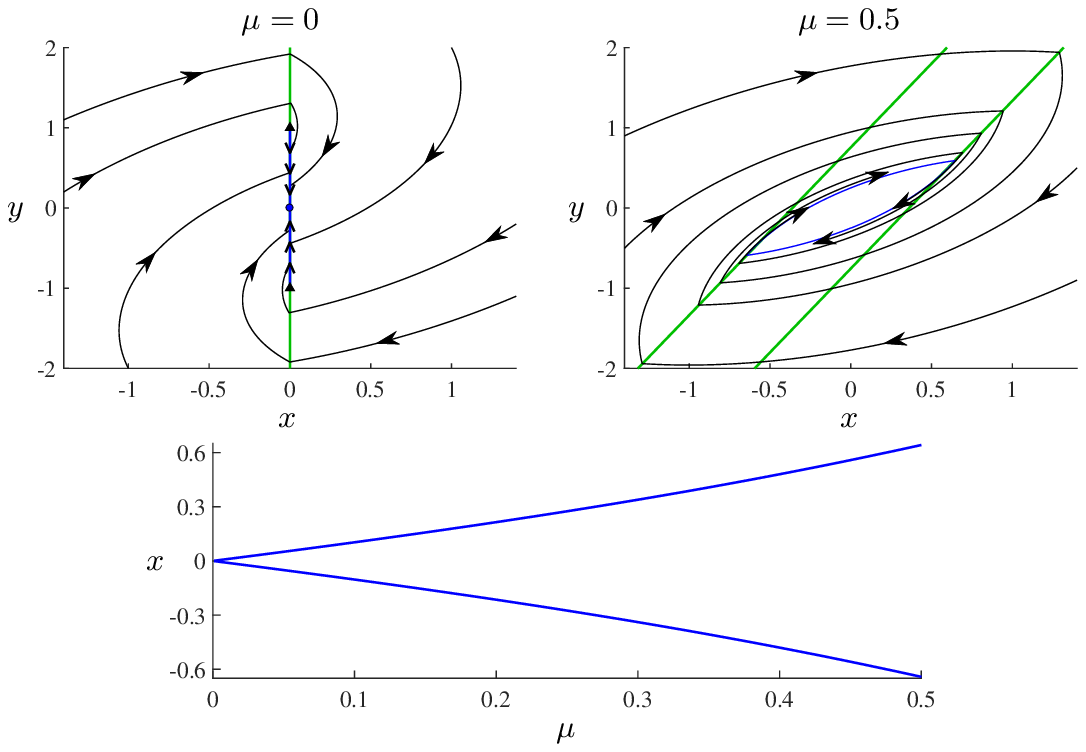}
\caption{
An illustration of HLB 16 for the canonical relay control system \eqref{eq:hysteresisExampleEq}
using delay of the form \eqref{eq:timeDelay} and the parameter values \eqref{eq:paramHysteresisEq}.
The green lines indicate where switches occur.
\label{fig:allTimeDelayEq}
} 
\end{center}
\end{figure*}

\subsection{Perturbing a two-fold --- HLBs 17--18}
\label{sub:twoFoldPerturbation}

Here we let
\begin{align}
a_{2J} &= \frac{\partial f_J}{\partial y}(0,0), &
b_{0J} &= g_J(0,0),
\end{align}
and
\begin{equation}
\gamma_J = a_{2J} b_{0J},
\label{eq:hysteresisTwoFoldInvisCond}
\end{equation}
for each $J \in \{ L, R \}$.
We also define
\begin{equation}
\kappa = \frac{b_{0L}}{a_{2L}} - \frac{b_{0R}}{a_{2R}}.
\label{eq:hysteresisTwoFoldkappa}
\end{equation}
We suppose the origin is an invisible fold for both the
left and right half-systems of \eqref{eq:FilippovPerturbation}.
That is
\begin{equation}
f_L(0,0) = f_R(0,0) = 0,
\label{eq:hysteresisTwoFoldCond}
\end{equation}
and $\gamma_L > 0$ and $\gamma_R < 0$.
We assume $a_{2L} a_{2R} > 0$, so that orbits
of both half-systems involve the same direction of rotation.
Then, as with HLB 7, the stability of the two-fold is governed by the sign of
\begin{equation}
\alpha = \sigma_{{\rm fold},L} - \sigma_{{\rm fold},R} \,,
\label{eq:hysteresisTwoFoldNondegCond}
\end{equation}
using the formula \eqref{eq:foldsigma}.

\begin{theorem}[HLB 17]
Consider the hysteretic system \eqref{eq:hysteresis} where $F_L$ and $F_R$ are $C^3$.
Suppose \eqref{eq:hysteresisTwoFoldCond} is satisfied,
$\gamma_L > 0$, $\gamma_R < 0$, and $a_{2L} a_{2R} > 0$.
In a neighbourhood of $(x,y;\mu) = (0,0;0)$,
if $\alpha < 0$ [$\alpha > 0$] there exists a unique stable limit cycle [no limit cycle] for $\mu > 0$.
The minimum and maximum $x$-values of the limit cycle are asymptotically proportional to $\mu^{\frac{2}{3}}$,
the minimum and maximum $y$-values are asymptotically proportional to $\mu^{\frac{1}{3}}$,
and its period is
\begin{equation}
T = \left| \frac{2}{b_{0L}} - \frac{2}{b_{0R}} \right|
\left( \frac{3 \kappa \mu}{|\alpha|} \right)^{\frac{1}{3}} + \co \left( \mu^{\frac{1}{3}} \right).
\label{eq:hysteresisTwoFoldPeriod}
\end{equation}
\label{th:hysteresisTwoFold}
\end{theorem}

\begin{figure}[b!]
\begin{center}
\includegraphics[width=5.6cm]{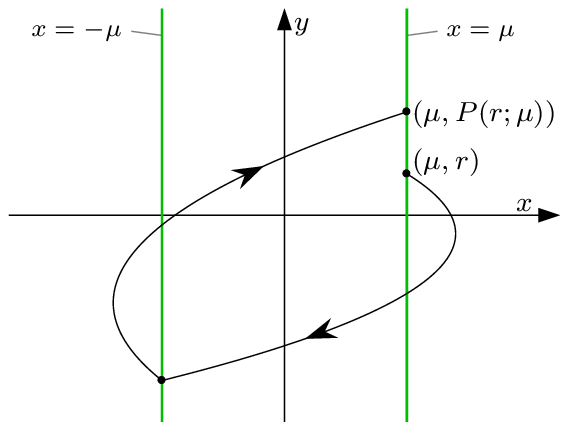}
\caption{
An illustration of the Poincar\'e map $P$ associated with HLB 17.
\label{fig:schemPoinHyst}
} 
\end{center}
\end{figure}

The interesting cube-root asymptotics are explained by a Poincar\'e map.
Given $r \in \mathbb{R}$, suppose the forward orbit of
$(x,y) = (\mu,r)$ undergoes a switch on $x = -\mu$, then a switch on $x = \mu$,
as in Fig.~\ref{fig:schemPoinHyst}.
Let $P(r;\mu)$ denote the $y$-value of the second switch.
In Appendix \ref{app:hysteresis} we show that
\begin{equation}
P(r;\mu) = \sqrt{r^2 + 4 \kappa \mu + \frac{4 \alpha}{3} \,r^3 + E},
\label{eq:hysteresisTwoFoldP0}
\end{equation}
where $E$ is an error term.
By solving the fixed point equation $P(r;\mu) = r$ we obtain
$r = \sqrt{\frac{3 \kappa}{|\alpha|}} \,\mu^{\frac{1}{3}}$, to leading order.
To deal with the unique technical difficulties
that arise in deriving \eqref{eq:hysteresisTwoFoldP0},
this equation is obtained in Appendix \ref{app:hysteresis} by non-constructive means.
A different form for $P(r;\mu)$ is obtained in \cite{Ma17}
by assuming {\em a priori} that the value of $r$ is not too small relative to $\mu$.
This form can be obtained from \eqref{eq:hysteresisTwoFoldP0} as follows.
Assume $r > 0$ and $r = \cO \big( \mu^{\frac{1}{3}} \big)$ (this sufficient to obtain the limit cycle).
Then a linear approximation to the square-root produces
\begin{equation}
P(r;\mu) = r + \frac{2 \kappa \mu}{r} + \frac{2 \alpha}{3} r^2
+ \co \left( r^2 \right),
\nonumber
\end{equation}
which is the form obtained in \cite{Ma17}.

\begin{theorem}[HLB 18]
Consider the delayed system \eqref{eq:timeDelay} where $F_L$ and $F_R$ are $C^3$.
Suppose \eqref{eq:hysteresisTwoFoldCond},
$\gamma_L > 0$, $\gamma_R < 0$, and $a_{2L} a_{2R} > 0$.
In a neighbourhood of $(x,y;\mu) = (0,0;0)$,
treating only orbits for which times between consecutive switches is not less than $\mu$,
if $\alpha < 0$ [$\alpha > 0$] there exists a unique stable limit cycle [no limit cycle] for $\mu > 0$.
The minimum and maximum $x$-values of the limit cycle are asymptotically proportional to $\mu$,
the minimum and maximum $y$-values are asymptotically proportional to $\sqrt{\mu}$,
and its period is
\begin{equation}
T = \left| \frac{2}{b_{0L}} - \frac{2}{b_{0R}} \right|
\sqrt{\frac{3 \left( a_{2L} + a_{2R} \right) \kappa}{2 |\alpha|}}
\,\sqrt{\mu} + \co \left( \sqrt{\mu} \right).
\label{eq:timeDelayTwoFoldPeriod}
\end{equation}
\label{th:timeDelayTwoFold}
\end{theorem}

To illustrate Theorems \ref{th:hysteresisTwoFold} and \ref{th:timeDelayTwoFold}
we consider a linear oscillator with discontinuous forcing:
\begin{equation}
m \ddot{x} + b \dot{x} + k x = F_{\rm apply} \,.
\label{eq:hysteresisExample}
\end{equation}
Let $F > 0$ be a constant and suppose the applied force is given by
$F_{\rm apply} = F$ until $x = \mu$, and $F_{\rm apply} = -F$ until $x = -\mu$.
By writing the system as
\begin{equation}
\begin{split}
\dot{x} &= y, \\
\dot{y} &= \begin{cases}
-\frac{k}{m} \,x - \frac{b}{m} \,y + \frac{F}{m}, & {\rm until~} x = \mu, \\
-\frac{k}{m} \,x - \frac{b}{m} \,y - \frac{F}{m}, & {\rm until~} x = -\mu,
\end{cases}
\end{split}
\label{eq:hysteresisExample2}
\end{equation}
we can apply Theorem \ref{th:hysteresisTwoFold} directly.
Here the non-degeneracy coefficient is $\alpha = \frac{-2 b}{F}$,
so a positive damping coefficient $b$ guarantees the creation of a stable limit cycle.
Fig.~\ref{fig:allHysteresis} shows a bifurcation diagram and phase portraits using
\begin{align}
m &= 1, &
b &= 0.5, &
k &= 1, &
F &= 1.
\label{eq:paramHysteresis}
\end{align}
Fig.~\ref{fig:allTimeDelay} shows a bifurcation diagram and phase portraits for the analogous delayed system
\begin{equation}
\begin{split}
\dot{x}(t) &= y(t), \\
\dot{y}(t) &= \begin{cases}
-\frac{k}{m} \,x(t) - \frac{b}{m} \,y(t) + \frac{F}{m}, & x(t-\mu) < 0, \\
-\frac{k}{m} \,x(t) - \frac{b}{m} \,y(t) - \frac{F}{m}, & x(t-\mu) > 0.
\end{cases}
\end{split}
\label{eq:timeDelayExample}
\end{equation}

\begin{figure*}
\begin{center}
\includegraphics[width=10.9cm]{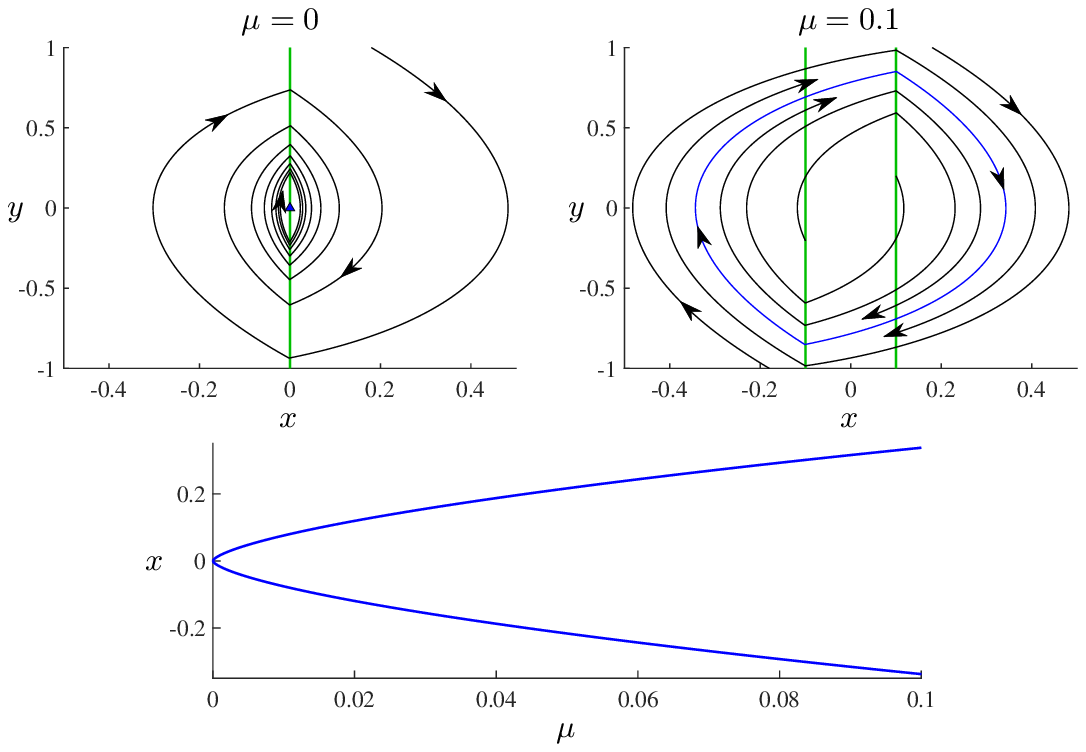}
\caption{
An illustration of HLB 17 for a linear oscillator with hysteretic discontinuous forcing,
\eqref{eq:hysteresisExample2} with \eqref{eq:paramHysteresis}.
By increasing the value of $\mu$ from $0$,
a stable invisible-invisible two-fold is replaced by a limit cycle
with period asymptotically proportional to $\mu^\frac{1}{3}$.
\label{fig:allHysteresis}
} 
\end{center}
\end{figure*}

\begin{figure*}
\begin{center}
\includegraphics[width=10.9cm]{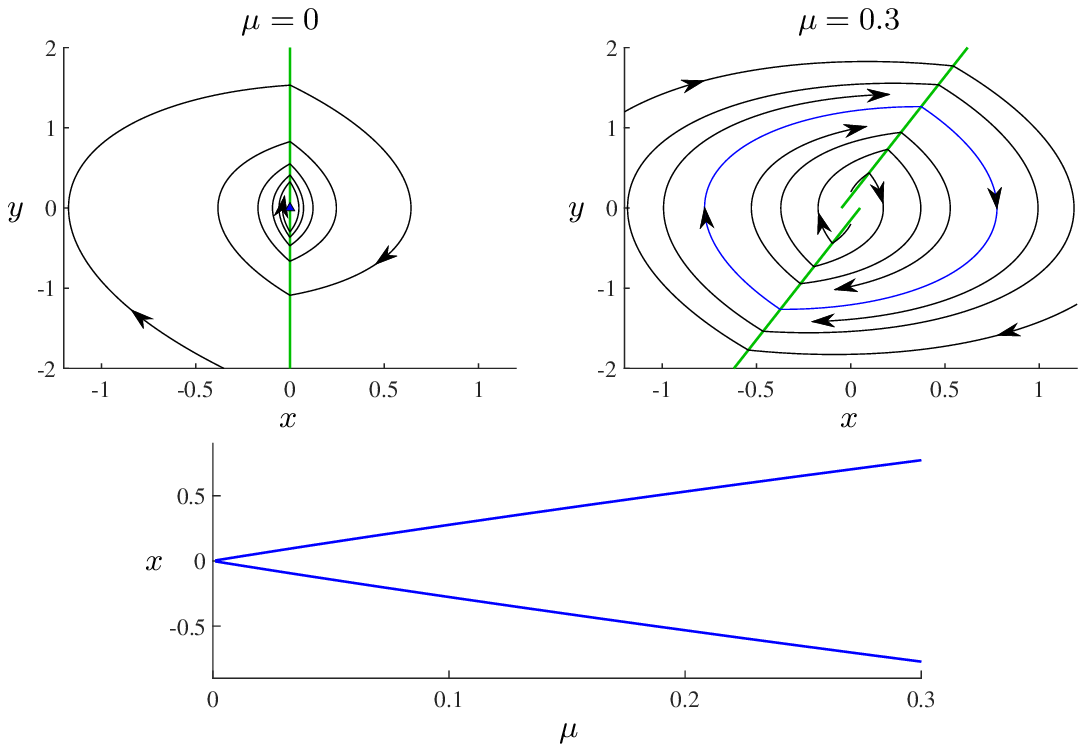}
\caption{
An illustration of HLB 18 for a linear oscillator with delayed discontinuous forcing,
\eqref{eq:timeDelayExample} with \eqref{eq:paramHysteresis}.
\label{fig:allTimeDelay}
} 
\end{center}
\end{figure*}

\section{Other mechanisms for limit cycle creation}
\setcounter{equation}{0}
\setcounter{figure}{0}
\setcounter{table}{0}
\label{sec:other}

\subsection{Intersecting switching manifolds --- HLB 19}
\label{sub:intersecting}

Many systems involve multiple switches;
examples include gene networks \cite{EdGl14,GoSa02} and relay control systems \cite{Ut92}.
In phase space, two switching manifolds generically intersect on a codimension-two surface.
Determining an appropriate sliding vector field on such a surface is surprisingly difficult in a general setting,
see \cite{DiDi17,Je14c,JeKa18} for recent developments,
however here we are only concerned with two-dimensional systems
for which two switching manifolds can be expected to intersect at a point.
The dynamics near two transversally intersecting switching manifolds,
depends on the direction of the vector field in each of the four regions
bounded by the switching manifolds \cite{GuHa15,Je16f}.
Here we suppose orbits spiral around the intersection point, as in Fig.~\ref{fig:schemPoinIntersecting}.
Then the intersection point is a stationary solution
and it may emit a limit cycle when its stability changes under parameter variation.

\begin{figure}[b!]
\begin{center}
\includegraphics[width=4.2cm]{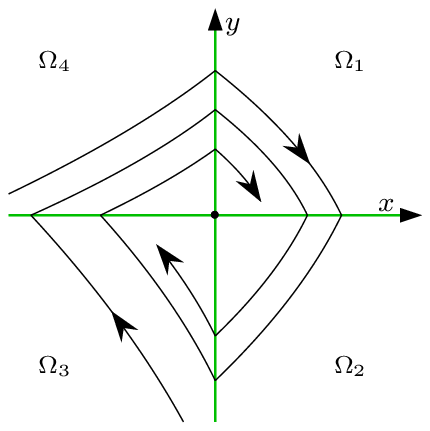}
\caption{
Parts of typical orbits of \eqref{eq:fourPieceODE} subject to \eqref{eq:fourPieceSpiralCond}.
\label{fig:schemPoinIntersecting}
} 
\end{center}
\end{figure}

We assume the switching manifolds coincide with the coordinate axes, and define
\begin{equation}
\begin{split}
\Omega_1 &= \{ (x,y) \,|\, x>0,\, y>0 \}, \\
\Omega_2 &= \{ (x,y) \,|\, x>0,\, y<0 \}, \\
\Omega_3 &= \{ (x,y) \,|\, x<0,\, y<0 \}, \\
\Omega_4 &= \{ (x,y) \,|\, x<0,\, y>0 \}.
\end{split}
\nonumber
\end{equation}
We write the system as
\begin{equation}
\begin{bmatrix} \dot{x} \\ \dot{y} \end{bmatrix} =
\begin{cases}
F_1(x,y;\mu), & (x,y) \in \Omega_1 \,, \\[-1.3mm]
\hspace{8mm} \vdots & \hspace{9.5mm} \vdots \\[-1.3mm]
F_4(x,y;\mu), & (x,y) \in \Omega_4 \,,
\end{cases}
\label{eq:fourPieceODE}
\end{equation}
where
\begin{equation}
F_j(x,y;\mu) = \begin{bmatrix} f_j(x,y;\mu) \\ g_j(x,y;\mu) \end{bmatrix},
\nonumber
\end{equation}
is smooth for each $j = 1,\ldots,4$.
We suppose
\begin{equation}
\begin{aligned}
f_1(0,0;0) &> 0, &
g_1(0,0;0) &< 0, \\
f_2(0,0;0) &< 0, &
g_2(0,0;0) &< 0, \\
f_3(0,0;0) &< 0, &
g_3(0,0;0) &> 0, \\
f_4(0,0;0) &> 0, &
g_4(0,0;0) &> 0,
\end{aligned}
\label{eq:fourPieceSpiralCond}
\end{equation}
so that nearby orbits spiral clockwise around the origin.

Suppose the forward orbit of $(x,y) = (0,r)$, with $r > 0$, next intersects
the positive $y$-axis at some point $(0,P(r;\mu))$.
It is a simple exercise to show that $P(r;\mu) = \Lambda(\mu) r + \cO \left( r^2 \right)$, where
\begin{equation}
\Lambda(\mu) = \frac{m_2(\mu) m_4(\mu)}{m_1(\mu) m_3(\mu)},
\label{eq:fourPieceLambda}
\end{equation}
and
\begin{equation}
m_j(\mu) = \frac{g_j(0,0;\mu)}{f_j(0,0;\mu)},
\label{eq:mj}
\end{equation}
represents the slope $\frac{d y}{d x}$ of $F_j$ evaluated at the origin.
Thus the stability of the origin is governed by the sign of $\ln(\Lambda)$.
In order for HLB 19 to occur at $\mu = 0$ we assume $\Lambda(0) = 1$.
The transversality condition is $\beta \ne 0$, where
\begin{equation}
\beta = \frac{d \Lambda}{d \mu}(0).
\label{eq:fourPieceTransCond}
\end{equation}
Also let
\begin{equation}
\xi_j(\mu) = \frac{1}{f_j} \left( \frac{g_j}{f_j} \frac{\partial f_j}{\partial y}
- \left( \frac{\partial f_j}{\partial x} + \frac{\partial g_j}{\partial y} \right)
+ \frac{f_j}{g_j} \frac{\partial g_j}{\partial x} \middle) \right|_{x = y = 0}, \\
\label{eq:fourPieceNondegCoeffs}
\end{equation}
for each $j$, and
\begin{equation}
\alpha = \left( \frac{\xi_1 - \xi_2}{m_1} + \frac{\xi_3-\xi_4}{m_4} \middle) \right|_{\mu = 0} \,.
\label{eq:fourPieceNondegCond}
\end{equation}

\begin{theorem}[HLB 19]
Consider \eqref{eq:fourPieceODE} where $F_j$ is $C^2$ for each $j = 1,\ldots,4$.
Suppose \eqref{eq:fourPieceSpiralCond} is satisfied, $\Lambda(0) = 1$, and $\beta > 0$.
In a neighbourhood of $(x,y;\mu) = (0,0;0)$,
\begin{enumerate}
\item
the origin is the unique stationary solution and is stable for $\mu < 0$ and unstable for $\mu > 0$,
\item
if $\alpha < 0$ [$\alpha > 0$] there exists a unique stable [unstable] limit cycle
for $\mu > 0$ [$\mu < 0$], and no limit cycle for $\mu < 0$ [$\mu > 0$].
\end{enumerate}
The minimum and maximum $x$ and $y$-values of the limit cycle are asymptotically proportional to $|\mu|$,
and its period is
\begin{equation}
T = \frac{2 \beta}{|\alpha|}
\left( \frac{-1}{g_1} + \frac{1}{f_2 m_1} - \frac{1}{f_3 m_4} + \frac{1}{g_4} \middle) \right|_{x = y = \mu = 0} |\mu|
+ \co(\mu).
\label{eq:periodFourPiece}
\end{equation}
\label{th:fourPiece}
\end{theorem}

HLB 19 occurs in a version of the Wilson-Cowan model
for the aggregated dynamics of a large network of excitatory and inhibitory neurons:
\begin{equation}
\begin{split}
\dot{u} &= -u + H(u - a v - b), \\
\tau \dot{v} &= -v + H(u - c v - d).
\end{split}
\label{eq:HaEr15}
\end{equation}
Here $u(t)$ represents the average activity (e.g.~voltage) of the neurons,
$v(t)$ is a recovery variable, and $a,b,c,d,\tau > 0$ are parameters.
As in \cite{HaEr15}, we take the firing rate function $H$ to be the Heaviside function.
This system has two switching manifolds, $u = a v + b$ and $u = c v + d$,
and for typical parameter values orbits spiral around their unique point of intersection.
With $\tau$ as the primary bifurcation parameter,
Fig.~\ref{fig:allHaEr15} shows a bifurcation diagram and representative phase portraits using
\begin{align}
a &= 2, &
b &= 0.05, &
c &= 0.25, &
d &= 0.3,
\label{eq:paramHaEr15}
\end{align}
as given in \cite{HaEr15}.

\begin{figure*}
\begin{center}
\includegraphics[width=16.6cm]{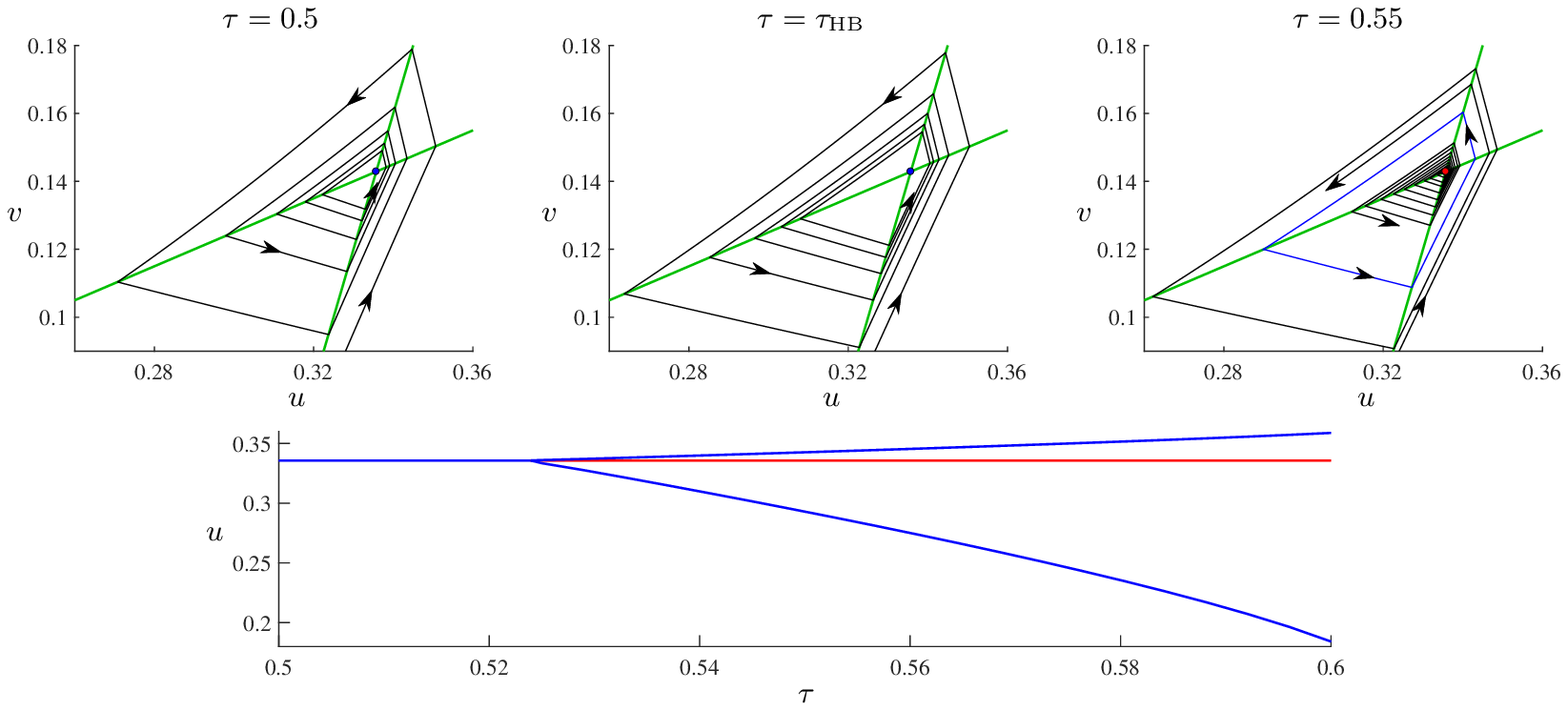}
\caption{
An illustration of HLB 19 for the Wilson-Cowan model \eqref{eq:HaEr15} with \eqref{eq:paramHaEr15}.
By increasing the value of $\tau$,
the stable stationary solution located at the intersection of the two switching manifolds
loses stability when $\tau = \tau_{\rm HB} \approx 0.5240$ and emits a stable limit cycle.
\label{fig:allHaEr15}
} 
\end{center}
\end{figure*}

In \cite{HaEr15} the authors use a Poincar\'e map to determine when HLB 19 occurs.
Here use Theorem \ref{th:fourPiece} to obtain the same result
and show that the limit cycle is stable by numerically evaluating \eqref{eq:fourPieceNondegCond}.
To do this we employ the change of variables 
\begin{equation}
\begin{split}
x &= u - c v - d, \\
y &= u - a v - b.
\end{split}
\label{eq:coordChangeHaEr15}
\end{equation}
and, assuming $a > c$, scale time by
\begin{equation}
t \mapsto \frac{t}{(a-c) \tau},
\nonumber
\end{equation}
to achieve additional simplification.
The system is then given by \eqref{eq:fourPieceODE} with
\begin{equation}
\begin{split}
F_1(x,y) &= F_3(x,y) + (a-c) \begin{bmatrix} \tau - c \\ \tau - a \end{bmatrix}, \\
F_2(x,y) &= F_3(x,y) - (a-c) \begin{bmatrix} c \\ a \end{bmatrix}, \\
F_3(x,y) &= \begin{bmatrix}
-(a \tau - c) & -c (1 - \tau) \\
a (1 - \tau) & -(a - c \tau)
\end{bmatrix} \begin{bmatrix} x \\ y \end{bmatrix} \\
&\quad- \begin{bmatrix} c (b - d) + (a d - b c) \tau \\ a (b - d) + (a d - b c) \tau \end{bmatrix}, \\
F_4(x,y) &= F_3(x,y) + (a-c) \begin{bmatrix} \tau \\ \tau \end{bmatrix}.
\end{split}
\label{eq:transformedHaEr15}
\end{equation}
By numerically evaluating \eqref{eq:fourPieceLambda}
we find that $\Lambda = 1$ when $\tau = \tau_{\rm HB} \approx 0.5240$ (to four decimal places).
Also $\alpha \approx -8.47$ and $\beta \approx 9.53$,
using $\mu = \tau - \tau_{\rm HB}$.
Since $\alpha < 0$ the limit cycle is stable (as shown in Fig.~\ref{fig:allHaEr15}).

\subsection{A square-root singularity --- HLB 20}
\label{sub:sqrt}

The leading-order dynamics near some {\em grazing bifurcations}
(where a limit cycle attains a tangential intersection with a switching manifold)
are described by piecewise-smooth Poincar\'e maps involving a square-root singularity \cite{ChOt94,No91}.
In this section we study the analogous class of ODEs.
Specifically we consider
\begin{equation}
\begin{bmatrix} \dot{x} \\ \dot{y} \end{bmatrix} =
F(x,y;S(x);\mu),
\label{eq:sqrtODE}
\end{equation}
where $F$ is a smooth function of its four components, and
\begin{equation}
S(x) = \begin{cases}
0, & x \le 0, \\
\sqrt{x}, & x \ge 0.
\end{cases}
\label{eq:sqrtS}
\end{equation}
This is motivated by the work \cite{NiCa16}
where an ODE system with a square-root singularity is introduced
to model the averaged behaviour of a large number of coupled neurons.

We suppose \eqref{eq:sqrtODE} exhibits a BEB at the origin when $\mu = 0$.
Thus
\begin{equation}
f(0,0;0;0) = g(0,0;0;0) = 0,
\label{eq:sqrtEqCond}
\end{equation}
where
\begin{equation}
F(x,y;z;\mu) = \begin{bmatrix} f(x,y;z;\mu) \\ g(x,y;z;\mu) \end{bmatrix}.
\label{eq:sqrtF}
\end{equation}
Since \eqref{eq:sqrtODE} is continuous on the switching manifold $x = 0$,
both half-systems generically have a unique equilibrium (located at the origin when $\mu = 0$)
that is admissible for one sign of $\mu$.
We let $(x_L^*(\mu),y_L^*(\mu))$ and $(x_R^*(\mu),y_R^*(\mu))$ denote these equilibria.
As in the usual continuous scenario (HLBs 1--2) a local limit cycle can only encircle an admissible focus.
But due to the square-root singularity, in generic situations $(x_R^*(\mu),y_R^*(\mu))$ is always a node.
Since this was stated incorrectly in \cite{Si18c}, let us pause to verify this.
We write
\begin{equation}
\begin{split}
f(x,y;z;\mu) &= a_1 x + a_2 y + a_3 \mu + a_4 z \\
&\quad+ \cO \left( \left( |x| + |y| + |z| + |\mu| \right)^2 \right), \\
g(x,y;z;\mu) &= b_1 x + b_2 y + b_3 \mu + b_4 z \\
&\quad+ \cO \left( \left( |x| + |y| + |z| + |\mu| \right)^2 \right),
\end{split}
\label{eq:sqrtfg}
\end{equation}
where $a_1,\ldots,b_4 \in \mathbb{R}$.
The Jacobian matrix of the right half-system is
\begin{align}
&\begin{bmatrix}
\frac{\partial}{\partial x} \,f(x,y;\sqrt{x};\mu) &
\frac{\partial}{\partial y} \,f(x,y;\sqrt{x};\mu) \\
\frac{\partial}{\partial x} \,g(x,y;\sqrt{x};\mu) &
\frac{\partial}{\partial y} \,g(x,y;\sqrt{x};\mu)
\end{bmatrix} \nonumber \\
&= \begin{bmatrix}
\frac{a_4}{2 \sqrt{x}} + \cdots &
a_2 + \cdots \\
\frac{b_4}{2 \sqrt{x}} + \cdots &
b_2 + \cdots
\end{bmatrix},
\label{eq:sqrtJacobianR}
\end{align}
where we have only retained the leading order terms.
Let $\tau$ and $\delta$ denote the trace and determinant of \eqref{eq:sqrtJacobianR}, respectively.
Then $\delta - \frac{\tau^2}{4} = -\frac{a_4^2}{16 x} + \cdots$ is negative,
assuming $a_4 \ne 0$ and $x > 0$ is sufficiently small,
in which case the eigenvalues of \eqref{eq:sqrtJacobianR} are real-valued.

It follows that to have a local limit cycle $(x_L^*(\mu),y_L^*(\mu))$ must be a focus.
For the system in \cite{NiCa16}, both equilibria are nodes
and for some parameter values a large amplitude limit cycle
is created in a BEB due to global features of the system.
A complete catagorisation of BEBs of \eqref{eq:sqrtODE},
analogous to \cite{FrPo98,SiMe12} in the usual continuous setting, remains for future work.

For clarity we assume $(x_L^*(0),y_L^*(0))$ is an unstable focus, i.e.
\begin{equation}
{\rm eig}(\rD F(0,0;0;0)) = \lambda \pm \ri \omega, \quad
{\rm with~} \lambda > 0, \omega > 0,
\label{eq:sqrtEigCond}
\end{equation}
and that $(x_R^*(0),y_R^*(0))$ is stable, thus $a_4 < 0$.
Let
\begin{align}
\beta &= a_3 b_2 - a_2 b_3 \,, \label{eq:sqrtbeta} \\
\gamma &= a_4 b_2 - a_2 b_4 \,. \label{eq:sqrtgamma}
\end{align}
As with HLBs 1--2, we have $x_L^*(\mu) = \frac{-\beta \mu}{\lambda^2 + \omega^2} + \cO \left( \mu^2 \right)$,
so $\beta \ne 0$ is the transversality condition for the BEB.
Also $\sqrt{x_R^*(\mu)} = -\frac{\beta \mu}{\gamma} + \cO \left( \mu^2 \right)$,
assuming $\frac{\beta \mu}{\gamma} \le 0$.
In Theorem \ref{th:sqrt} we assume $\beta > 0$ (as usual)
and $\gamma > 0$ so that $(x_L^*(\mu),y_L^*(\mu))$ and $(x_R^*(\mu),y_R^*(\mu))$
are admissible for different signs of $\mu$.

\begin{theorem}[HLB 20]
Consider \eqref{eq:sqrtODE} where $F$ is $C^2$.
Suppose \eqref{eq:sqrtEqCond} and \eqref{eq:sqrtEigCond} are satisfied,
$a_4 < 0$, $\beta > 0$, and $\gamma > 0$.
In a neighbourhood of $(x,y;\mu) = (0,0;0)$,
\begin{enumerate}
\item
there exists a unique stationary solution:
a stable node in $\Omega_R$ for $\mu < 0$, and an unstable focus in $\Omega_L$ for $\mu > 0$,
\item
there exists a unique stable limit cycle for $\mu > 0$, and no limit cycle for $\mu < 0$.
\end{enumerate}
The maximum $x$-value of the limit cycle is asymptotically proportional to $\mu^2$,
the minimum $x$-value and minimum and maximum $y$-values are asymptotically proportional to $\mu$,
and its period is $T = T_L + T_R + \cO(\mu)$ where
\begin{equation}
\begin{split}
T_L &= \frac{1}{\omega} \,\hat{s} \left( \frac{\lambda}{\omega} \right), \\
T_R &= \kappa \ln \left( 1 + \frac{1}{\kappa \omega} \,\re^{\frac{\lambda}{\omega}
\,\hat{s} \left( \frac{\lambda}{\omega} \right)} \sin \left( \hat{s} \left( \frac{\lambda}{\omega} \right) \right) \right),
\end{split}
\label{eq:periodsqrt}
\end{equation}
where $\kappa = \frac{|a_4|}{\gamma}$
and $\hat{s}$ is defined in \S\ref{sub:affine}.
\label{th:sqrt}
\end{theorem}

HLB 20 resembles HLB 2 (as it involves a focus and a node).
However the maximum $x$-value of the limit cycle is much smaller than its amplitude,
so it also resembles HLB 3 where the limit cycle involves a segment of sliding motion.
Indeed \eqref{eq:sqrtODE} may be viewed as a {\em regularisation} of a Filippov system \cite{BuDa06,Je18b,SoTe96}.
HLB 20 has attributes of both HLB 2 and HLB 3
and so is perhaps best viewed as an intermediary of these two bifurcations.

As an example consider
\begin{equation}
\begin{bmatrix} \dot{x} \\ \dot{y} \end{bmatrix} =
\begin{bmatrix} \lambda x + y + \eta S(x) \\ -x + \lambda y - \mu + \nu S(x) \end{bmatrix}.
\label{eq:sqrtExample}
\end{equation}
Here $\beta = 1$,
$\gamma = \lambda \eta - \nu$,
and $a_4 = \eta$.
Fig.~\ref{fig:allSqrt} shows a bifurcation diagram and phase portraits using
\begin{align}
\lambda &= 0.5, &
\eta &= -1, &
\nu &= -1.
\label{eq:paramSqrt}
\end{align}

\begin{figure*}
\begin{center}
\includegraphics[width=16.6cm]{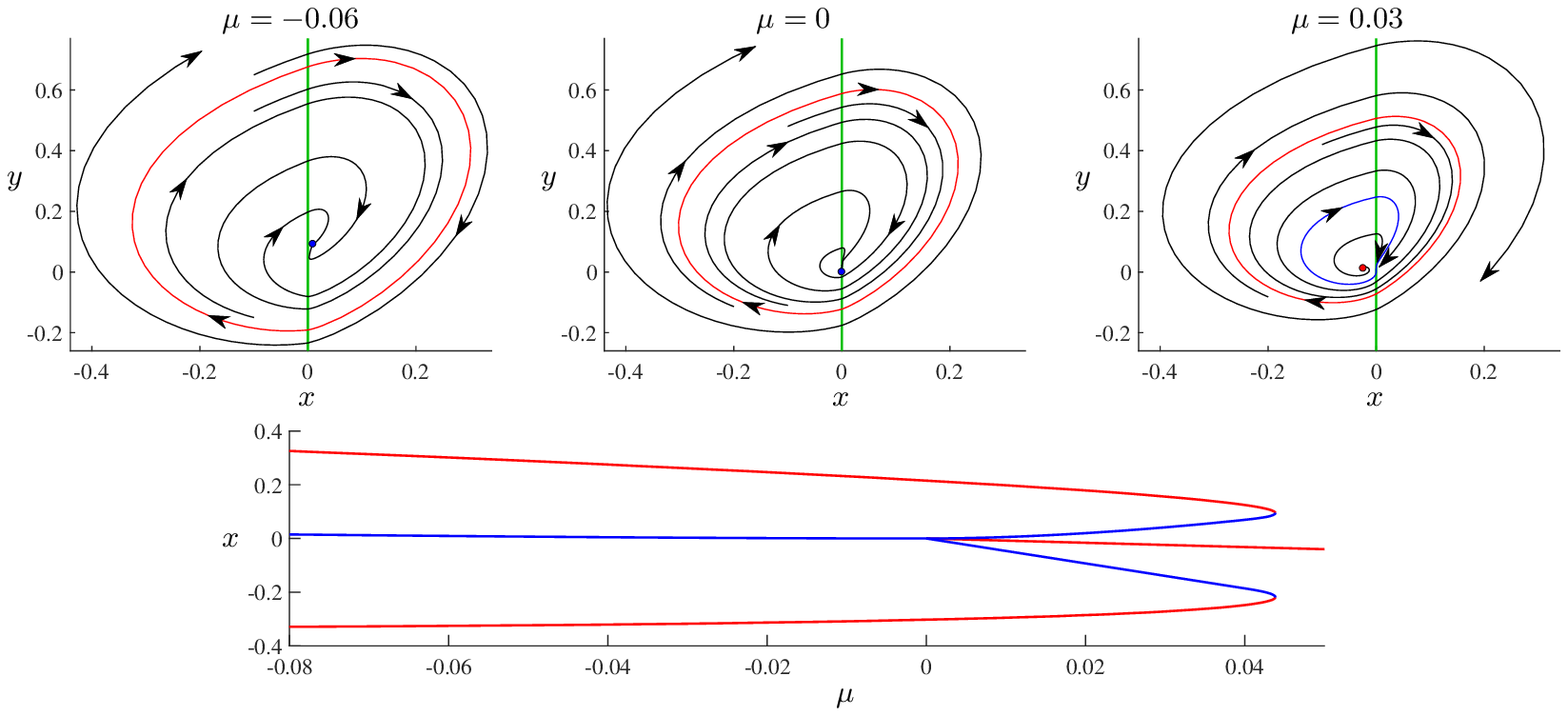}
\caption{
An illustration of HLB 20 for \eqref{eq:sqrtExample} with \eqref{eq:paramSqrt}.
By increasing the value of $\mu$ a stable node collides with the switching manifold when $\mu = 0$
at which it turns into an unstable focus and emits a stable limit cycle.
The limit cycle subsequently undergoes a saddle-node bifurcation and becomes unstable.
\label{fig:allSqrt}
} 
\end{center}
\end{figure*}

\section{Discussion}
\setcounter{equation}{0}
\setcounter{figure}{0}
\setcounter{table}{0}
\label{sec:conc}

In this paper we have detailed $20$ HLBs by which
a steady-state solution can transition to sustained small amplitude oscillations
in systems for which the equations of motion are piecewise-smooth.
In most cases an isolated steady-state solution changes stability at $\mu = 0$
and we identified a scalar quantity $\beta$
so that the steady-state solution is stable if $\beta \mu < 0$ and unstable if $\beta \mu > 0$.
We identified a second scalar quantity $\alpha$
so that when $\mu = 0$ the steady-state solution is stable if $\alpha < 0$ and unstable if $\alpha > 0$.
If $\alpha < 0$ and $\beta \mu > 0$,
then the competing effects of attraction and repulsion produce a stable limit cycle.
With these inequalities reversed the limit cycle is unstable.

For each HLB we have provided formulas for $\beta$ and $\alpha$
in terms of the coefficients of a (sometimes piecewise) power series of the ODEs
centred at $(x,y) = (0,0)$.
For a Hopf bifurcation, $\beta$ is determined by linear terms in the series
while $\alpha$ is determined by higher-order terms.
The same is true for some HLBs,
but for HLBs 1--6, 11--13, 15, 16 and 19, linear terms fully determine $\beta$ and $\alpha$.
It is not surprising then that these HLBs produce a limit cycle whose amplitude
is proportional to $|\mu|$, to leading order, see Table \ref{tb:bifs}.
In order to have square-root growth (like a Hopf bifurcation)
the bifurcation must either involve a two-fold or have degenerate higher-order terms (as in \cite{ZoKu06}, for example).

Nonlinear terms influence the value of $\alpha$ for each of the other HLBs,
although it is only for HLBs 8 and 19 that nonlinear terms are {\em necessary}
to construct an instance of the bifurcation that is generic.
For example, quartic terms affect the value of $\alpha$ for HLB 10,
yet the piecewise-linear system \eqref{eq:fixedTwoFoldODE} exhibits HLB 10 in a generic fashion.

Naturally we would like to extend the results to higher dimensions
(i.e.~systems with more than two variables).
This is possible for the Hopf bifurcation theorem
due to the presence of a two-dimensional centre manifold on which all the relevant dynamics occurs \cite{Ku04}.
The majority of the HLBs presented here will not generate a centre manifold.
For BEBs and other discontinuity induced bifurcations of piecewise-smooth systems
the possible complexity of the local dynamics increases rapidly with dimension \cite{GlJe15}.
This suggests that statements regarding limit cycles in higher dimensions
will need to be weaker than those presented here.
Attempts to make strong statements may lead to a proliferation of cases \cite{Gl17b} (if $20$ HLBs are not enough already!).
But certainly the basic geometric mechanisms described here for the creation of a limit cycle
(e.g.~BEBs, loss of stability, inclusion of delay)
should be identifiable for higher dimensional systems, if only qualitatively.

\appendix
\section{Derivations of the three return maps}
\setcounter{equation}{0}
\setcounter{figure}{0}
\setcounter{table}{0}
\label{app:lemmas}

\begin{proof}[Proof of Lemma \ref{le:focus}]
Since $f$ and $g$ are $C^2$ they may be Taylor expanded to second order; we write
\begin{equation}
\begin{split}
f(x,y) &= a_1 x + a_2 y + a_3 x^2 + a_4 x y + a_5 y^2 + \co \left( \left( |x| + |y| \right)^2 \right), \\
g(x,y) &= b_1 x + b_2 y + b_3 x^2 + b_4 x y + b_5 y^2 + \co \left( \left( |x| + |y| \right)^2 \right),
\end{split}
\nonumber
\end{equation}
which is consistent with \eqref{eq:focusDF}.
In order to expand asymptotically in $r$, we `blow up' a neighbourhood of the origin via
\begin{align}
\tilde{x} &= \frac{x}{r}, &
\tilde{y} &= \frac{y}{r},
\nonumber
\end{align}
where $r > 0$.
In these coordinates the system is given by
\begin{equation}
\begin{split}
\dot{\tilde{x}} &= a_1 \tilde{x} + a_2 \tilde{y} +
r \left( a_3 \tilde{x}^2 + a_4 \tilde{x} \tilde{y} + a_5 \tilde{y}^2 \right) + \co(r), \\
\dot{\tilde{y}} &= b_1 \tilde{x} + b_2 \tilde{y} +
r \left( b_3 \tilde{x}^2 + b_4 \tilde{x} \tilde{y} + b_5 \tilde{y}^2 \right) + \co(r).
\end{split}
\label{eq:focusscaledODE}
\end{equation}
We denote the $\tilde{x}$ and $\tilde{y}$ components of the flow of \eqref{eq:focusscaledODE}
by $\varphi_t(\tilde{x},\tilde{y})$ and $\psi_t(\tilde{x},\tilde{y})$, respectively.
The trajectory associated with $P_{\rm focus}$ is the forward orbit of $(\tilde{x},\tilde{y}) = (0,1)$.
Thus $P_{\rm focus}$ and $T_{\rm focus}$ satisfy
\begin{align}
\varphi_{T_{\rm focus}(r)}(0,1) &= 0,
\label{eq:focusTDefn} \\
\psi_{T_{\rm focus}(r)}(0,1) &= \frac{1}{r} \,P_{\rm focus}(r).
\label{eq:focusPDefn}
\end{align}
Notice \eqref{eq:focusscaledODE} is $C^1$ in $r$,
thus $T_{\rm focus}(r)$ and $\frac{1}{r} P_{\rm focus}(r)$ are $C^1$ by Lemma \ref{le:mapSmoothness},
which justifies the error terms in \eqref{eq:focusP}--\eqref{eq:focusT}.

The form \eqref{eq:focusscaledODE} is a regular expansion in powers of $r$,
thus the proof can be completed through direct asymptotic expansions in powers of $r$.
We write
\begin{equation}
\begin{split}
\varphi_t(0,1) &= \Phi^{(0)}_t(r) + r \Phi^{(1)}_t(r) + \co(r), \\
\psi_t(0,1) &= \Psi^{(0)}_t(r) + r \Psi^{(1)}_t(r) + \co(r),
\end{split}
\nonumber
\end{equation}
and by substituting these into \eqref{eq:focusscaledODE} we obtain
\begin{widetext}
\begin{equation}
\begin{split}
\dot{\Phi}^{(0)}_t + r \dot{\Phi}^{(1)}_t &=
a_1 \left( \Phi^{(0)}_t + r \Phi^{(1)}_t \right) +
a_2 \left( \Psi^{(0)}_t + r \Psi^{(1)}_t \right) +
r \left( a_3 \Phi^{(0)^{\scriptstyle 2}}_t + a_4 \Phi^{(0)}_t \Psi^{(0)}_t + a_5 \Psi^{(0)^{\scriptstyle 2}}_t \right) + \co(r), \\
\dot{\Psi}^{(0)}_t + r \dot{\Psi}^{(1)}_t &=
b_1 \left( \Phi^{(0)}_t + r \Phi^{(1)}_t \right) +
b_2 \left( \Psi^{(0)}_t + r \Psi^{(1)}_t \right) +
r \left( b_3 \Phi^{(0)^{\scriptstyle 2}}_t + b_4 \Phi^{(0)}_t \Psi^{(0)}_t + b_5 \Psi^{(0)^{\scriptstyle 2}}_t \right) + \co(r),
\end{split}
\label{eq:focusPhiPsiODE}
\end{equation}
\end{widetext}
with initial condition
\begin{equation}
\begin{split}
\Phi^{(0)}_0 + r \Phi^{(1)}_0 + \co(r) &= 0, \\
\Psi^{(0)}_0 + r \Psi^{(1)}_0 + \co(r) &= 1.
\end{split}
\label{eq:focusPhiPsiIC}
\end{equation}
To leading order
\begin{align}
\begin{bmatrix} \dot{\Phi}^{(0)}_t \\ \dot{\Psi}^{(0)}_t \end{bmatrix} &=
A \begin{bmatrix} \Phi^{(0)}_t \\ \Psi^{(0)}_t \end{bmatrix}, &
\begin{bmatrix} \Phi^{(0)}_0 \\ \Psi^{(0)}_0 \end{bmatrix} &=
\begin{bmatrix} 0 \\ 1 \end{bmatrix},
\nonumber
\end{align}
where $A$ is the Jacobian matrix \eqref{eq:focusDF}.
This has the solution
\begin{equation}
\begin{bmatrix} \Phi^{(0)}_t \\ \Psi^{(0)}_t \end{bmatrix} =
\re^{t A} \begin{bmatrix} 0 \\ 1 \end{bmatrix}.
\label{eq:focusPhiPsi0}
\end{equation}
The $r$-terms in \eqref{eq:focusPhiPsiODE}--\eqref{eq:focusPhiPsiIC} give
\begin{align}
\begin{bmatrix} \dot{\Phi}^{(1)}_t \\ \dot{\Psi}^{(1)}_t \end{bmatrix} &=
A \begin{bmatrix} \Phi^{(1)}_t \\ \Psi^{(1)}_t \end{bmatrix} + h(t), &
\begin{bmatrix} \Phi^{(1)}_0 \\ \Psi^{(1)}_0 \end{bmatrix} &=
\begin{bmatrix} 0 \\ 0 \end{bmatrix},
\nonumber
\end{align}
where
\begin{equation}
h(t) = \begin{bmatrix}
a_3 \Phi^{(0)^{\scriptstyle 2}}_t + a_4 \Phi^{(0)}_t \Psi^{(0)}_t + a_5 \Psi^{(0)^{\scriptstyle 2}}_t \\
b_3 \Phi^{(0)^{\scriptstyle 2}}_t + b_4 \Phi^{(0)}_t \Psi^{(0)}_t + b_5 \Psi^{(0)^{\scriptstyle 2}}_t
\end{bmatrix}.
\label{eq:focush}
\end{equation}
This has the solution
\begin{equation}
\begin{bmatrix} \Phi^{(1)}_t \\ \Psi^{(1)}_t \end{bmatrix} =
\re^{t A} \int_0^t \re^{-s A} \,h(s) \,ds.
\label{eq:focusPhiPsi1}
\end{equation}

We now calculate $T_{\rm focus}$.
To this end we write
\begin{equation}
T_{\rm focus}(r) = T^{(0)} + r T^{(1)} + \co(r),
\nonumber
\end{equation}
By \eqref{eq:focusPhiPsi0} we have
$\Phi^{(0)}_t = \frac{a_2}{\omega} \,\re^{\lambda t} \sin(\omega t)$,
where $\lambda \pm \ri \omega$ are the eigenvalues of $A$ \eqref{eq:focusEigCond}, and so
\begin{equation}
T^{(0)} = \frac{\pi}{\omega}.
\label{eq:focusT0}
\end{equation}
This is sufficient to verify \eqref{eq:focusT},
but to verify \eqref{eq:focusP} we require $T^{(1)}$.
Equation \eqref{eq:focusT0} implies 
$\Phi^{(0)}_{T_{\rm focus}(r)} = -a_2 \re^{\frac{\lambda \pi}{\omega}} T^{(1)} r + \co(r)$,
and so from the $r$-terms in \eqref{eq:focusTDefn} we obtain
\begin{equation}
T^{(1)} = \frac{1}{a_2} \,\re^{\frac{-\lambda \pi}{\omega}} \Phi^{(1)}_{\frac{\pi}{\omega}}.
\label{eq:focusT1}
\end{equation}

Next we calculate $P_{\rm focus}$ which is determined by \eqref{eq:focusPDefn}.
We write
\begin{equation}
\frac{1}{r} \,P_{\rm focus}(r) = P^{(0)} + r P^{(1)} + \co(r),
\label{eq:focusPproof}
\end{equation}
with which \eqref{eq:focusPDefn} is
\begin{equation}
\Psi^{(0)}_{T_{\rm focus}(r)} + r \Psi^{(1)}_{T_{\rm focus}(r)} = P^{(0)} + r P^{(1)} + \co(r).
\label{eq:focusPDefn2}
\end{equation}
By \eqref{eq:focusPhiPsi0} we have
$\Psi^{(0)}_t = \re^{\lambda t}
\left( \cos(\omega t) - \frac{a_1 - b_2}{2 \omega} \,\sin(\omega t) \right)$,
and so by \eqref{eq:focusT0} we obtain
\begin{equation}
P^{(0)} = -\re^{\frac{\lambda \pi}{\omega}}.
\label{eq:focusP0}
\end{equation}
From the $r$-terms in \eqref{eq:focusPDefn2} we obtain
(after simplification using $\lambda = \frac{a_1 + b_2}{2}$)
\begin{equation}
P^{(1)} = -\frac{b_2}{a_2} \,\Phi^{(1)}_{\frac{\pi}{\omega}} + \Psi^{(1)}_{\frac{\pi}{\omega}}.
\label{eq:focusP1}
\end{equation}

To complete the proof it remains for us to
evaluate \eqref{eq:focusPhiPsi1} at $t = \frac{\pi}{\omega}$.
Observe $\re^{\frac{\pi A}{\omega}} = -\re^{\frac{\lambda \pi}{\omega}} I$, thus
\begin{equation}
\begin{bmatrix} \Phi^{(1)}_{\frac{\pi}{\omega}} \\ \Psi^{(1)}_{\frac{\pi}{\omega}} \end{bmatrix} =
-\re^{\frac{\lambda \pi}{\omega}} \int_0^{\frac{\pi}{\omega}} \re^{-s A} \,h(s) \,ds.
\label{eq:focusPhiPsi12}
\end{equation}
By the laborious task of expanding and grouping terms using
\eqref{eq:focusPhiPsi0} and \eqref{eq:focush}, we obtain
\begin{equation}
\re^{-s A} \,h(s) = \re^{\lambda s}
\begin{bmatrix}
\sum\limits_{n=0}^3 \frac{c_n}{\omega^n} \,\cos^{3-n}(\omega s) \sin^n(\omega s) \\
\sum\limits_{n=0}^3 \frac{d_n}{\omega^n} \,\cos^{3-n}(\omega s) \sin^n(\omega s)
\end{bmatrix},
\label{eq:focush2}
\end{equation}
where
\begin{equation}
\begin{split}
c_0 &= a_5 \,, \\
c_1 &= a_2 (a_4 - b_5) - \frac{3 (a_1 - b_2) a_5}{2}, \\
c_2 &= a_2^2 (a_3 - b_4) - a_2 (a_1 - b_2) (a_4 - b_5) + \frac{3 a_5 (a_1 - b_2)^2}{4}, \\
c_3 &= -a_2^3 b_3 - \frac{a_2^2 (a_1 - b_2) (a_3 - b_4)}{2} \\
&\quad+ \frac{a_2 (a_1 - b_2)^2 (a_4 - b_5)}{4} - \frac{a_5 (a_1 - b_2)^3}{8}, \\
d_0 &= b_5 \,, \\
d_1 &= a_2 b_4 - a_5 b_1 - \frac{(a_1 - b_2) b_5}{2}, \\
d_2 &= a_2^2 b_3 - a_2 a_4 b_1 + (a_1 - b_2) a_5 b_1 - \frac{(a_1 - b_2)^2 b_5}{4}, \\
d_3 &= -a_2^2 a_3 b_1 + \frac{a_2^2 (a_1 - b_2) b_3}{2}
+ \frac{a_2 (a_1 - b_2) a_4 b_1}{2} \\
&\quad- \frac{a_2 (a_1 - b_2)^2 b_4}{4} - \frac{(a_1 - b_2)^2 a_5 b_1}{4} + \frac{(a_1 - b_2)^3 b_5}{8}.
\end{split}
\label{eq:focusc0tod3}
\end{equation}
Direct evaluation of
\begin{equation}
I_n = \int_0^{\frac{\pi}{\omega}} \re^{\lambda s} \cos^{3-n}(\omega s) \sin^n(\omega s) \,ds,
\label{eq:focusIn}
\end{equation}
where $n = 0,\ldots,3$, yields
\begin{equation}
I_n = \frac{\left( \re^{\frac{\lambda \pi}{\omega}} + 1 \right) q_n}
{\left( \lambda^2 + \omega^2 \right) \left( \lambda^2 + 9 \omega^2 \right)},
\label{eq:focusIn2}
\end{equation}
where
\begin{equation}
\begin{split}
q_0 &= -\lambda \left( \lambda^2 + 7 \omega^2 \right), \\
q_1 &= \omega \left( \lambda^2 + 3 \omega^2 \right), \\
q_2 &= -2 \lambda \omega^2, \\
q_3 &= 6 \omega^3.
\end{split}
\label{eq:focusw0tow3}
\end{equation}
Combining \eqref{eq:focusP1}, \eqref{eq:focusPhiPsi12}, \eqref{eq:focush2}, \eqref{eq:focusIn},
and \eqref{eq:focusIn2} produces
\begin{equation}
P^{(1)} = \frac{\re^{\frac{\lambda \pi}{\omega}} \left( \re^{\frac{\lambda \pi}{\omega}} + 1 \right)}
{\left( \lambda^2 + \omega^2 \right) \left( \lambda^2 + 9 \omega^2 \right)}
\sum_{n=0}^3 \left( \frac{b_2 c_n}{a_2} - d_n \right) \frac{q_n}{\omega^n}.
\nonumber
\end{equation}
Finally from \eqref{eq:focusc0tod3} and \eqref{eq:focusw0tow3} we obtain
\begin{align}
\sum_{n=0}^3 \left( \frac{b_2 c_n}{a_2} - d_n \right) \frac{q_n}{\omega^n} &=
2 a_3 k_1 + a_4 k_2 + 2 a_5 k_3 \nonumber \\
&\quad+ 2 b_3 \ell_1 + b_4 \ell_2 + 2 b_5 \ell_3 \,,
\nonumber
\end{align}
hence $P^{(1)} = \re^{\frac{\lambda \pi}{\omega}}
\left( \re^{\frac{\lambda \pi}{\omega}} + 1 \right) \chi_{\rm focus}$, as required.
\end{proof}

\begin{proof}[Proof of Lemma \ref{le:fold}]
Here we perform direct asymptotic expansions, but, for brevity, only include terms sufficient to obtain
$P_{\rm fold}$ and $T_{\rm fold}$ to second order.
One simply includes more terms to obtain \eqref{eq:foldP} to fourth order.

We denote the $x$ and $y$ components of the flow by $\varphi_t(x,y)$ and $\psi_t(x,y)$, respectively.
By substituting Taylor expansions for $\varphi_t(0,r)$ and $\psi_t(0,r)$
into the ODEs and matching terms using \eqref{eq:foldfg} we obtain
\begin{align}
\varphi_t(0,r) &= a_2 r t + \frac{a_2 b_0}{2} \,t^2 + a_5 r^2 t \nonumber \\
&\quad+ \left( \frac{a_1 a_2}{2} + \frac{a_2 b_2}{2} + a_5 b_0 \right) r t^2 \nonumber \\
&\quad+ \left( \frac{a_1 a_2 b_0}{6} + \frac{a_2 b_0 b_2}{6} + \frac{a_5 b_0^2}{3} \right) t^3 + \cO \left( (r+t)^4 \right),
\nonumber \\
\psi_t(0,r) &= r + b_0 t + b_2 r t + \frac{b_0 b_2}{2} \,t^2 + \cO \left( (r+t)^3 \right).
\nonumber
\end{align}
We then solve $\varphi_t(0,r) = 0$ for $t = T_{\rm fold}(r)$.
Since $t = 0$ is a solution to $\varphi_t(0,r) = 0$,
we use the IFT to solve $\frac{1}{t} \,\varphi_t(0,r) = 0$.
This is possible because $a_2 b_0 \ne 0$ and implies $T_{\rm fold}(r)$ is $C^4$ at $r = 0$.
By matching terms we obtain
\begin{equation}
T_{\rm fold}(r) = \frac{-2}{b_0} \,r + \frac{2 \sigma_{\rm fold}}{3 b_0} \,r^2 + \cO \left( r^3 \right),
\nonumber
\end{equation}
where $\sigma_{\rm fold}$ is given by \eqref{eq:foldsigma}.
Then
\begin{align}
P_{\rm fold}(r) &= \psi_{T_{\rm fold}}(0,r) \nonumber \\
&= -r + \frac{2 \sigma_{\rm fold}}{3} \,r^2 + \cO \left( r^3 \right),
\nonumber
\end{align}
and has the same degree of differentiability at $r = 0$ as $T_{\rm fold}$.
\end{proof}

\begin{proof}[Proof of Lemma \ref{le:affine}]
First we derive \eqref{eq:affineT2}--\eqref{eq:affineP2}.
Since \eqref{eq:affineODE} is affine, it admits the analytical solution
\begin{equation}
\begin{bmatrix} \varphi_t(x,y) \\ \psi_t(x,y) \end{bmatrix} =
\re^{t A} \left( \begin{bmatrix} x \\ y \end{bmatrix}
- \begin{bmatrix} x^* \\ y^* \end{bmatrix} \right)
+ \begin{bmatrix} x^* \\ y^* \end{bmatrix},
\label{eq:affinephipsi}
\end{equation}
where
\begin{equation}
\re^{t A} = \re^{\lambda t} \begin{bmatrix}
\cos(\omega t) + \frac{a_1-b_2}{2 \omega} \,\sin(\omega t) &
\frac{a_2}{\omega} \,\sin(\omega t) \\
\frac{b_1}{\omega} \,\sin(\omega t) &
\cos(\omega t) - \frac{a_1-b_2}{2 \omega} \,\sin(\omega t) \end{bmatrix},
\label{eq:affineexptA}
\end{equation}
and the equilibrium $(x^*,y^*)$ is given by \eqref{eq:affineEq}.
Also
\begin{align}
\lambda &= \frac{a_1 + b_2}{2}, &
\omega &= \sqrt{-a_2 b_1 - \frac{(a_1-b_2)^2}{4}}.
\nonumber
\end{align}
Upon substituting $(x,y) = (0,r)$ into \eqref{eq:affinephipsi},
we obtain, after much simplification,
\begin{align}
\varphi_t(0,r) &=
\frac{a_2}{\omega} \,\re^{\lambda t} \sin(\omega t) r
+ \frac{a_2 \kappa}{\omega} \,\varrho \left( \omega t; \tfrac{\lambda}{\omega} \right),
\label{eq:affinephi2} \\
\psi_t(0,r) &=
\left( \cos(\omega t) - \frac{a_1-b_2}{2 \omega} \,\sin(\omega t) \right)
\left( \re^{\lambda t} r + \frac{\kappa \varrho
\left( \omega t; \frac{\lambda}{\omega} \right)}{\sin(\omega t)} \right) \nonumber \\
&\quad+ \frac{\kappa \re^{\lambda t} \varrho \left( \omega t; -\frac{\lambda}{\omega} \right)}{\sin(\omega t)},
\label{eq:affinepsi2}
\end{align}
where $\kappa = \frac{b_0 \omega}{\lambda^2 + \omega^2}$, see \eqref{eq:affinekappa},
and $\varrho$ is the auxiliary function \eqref{eq:auxFunc}.
The evolution time $T$ is defined by $\varphi_T(0,r) = 0$,
and through \eqref{eq:affinephi2} we obtain \eqref{eq:affineT2}.
To evaluate $P = \psi_T(0,r)$,
we substitute \eqref{eq:affineT2} into \eqref{eq:affinepsi2},
with which the first part of \eqref{eq:affinepsi2} vanishes leaving \eqref{eq:affineP2}.

Next we analyse the evolution time $T$.
If $b_0 < 0$, then $x^* < 0$,
and so as the orbit travels from $(0,r)$ to $(0,P)$ it completes less than half a revolution around $(x^*,y^*)$,
Fig.~\ref{fig:schemPoinAffine}.
The time taken to complete half a revolution is $\frac{\pi}{\omega}$,
thus $\omega T \in (0,\pi)$ and $\sin(\omega T) > 0$.
If instead $b_0 > 0$,
then $\omega T \in (\pi,2 \pi)$ and $\sin(\omega T) < 0$.
By using the identity
\begin{equation}
\frac{\partial}{\partial s} \frac{\re^{-\nu s} \varrho(s;\nu)}{\sin(s)}
= \frac{\varrho(s;-\nu)}{\sin^2(s)}
\label{eq:auxFuncFuncDeriv}
\end{equation}
to differentiate \eqref{eq:affineT2}, we obtain
\begin{equation}
\frac{d T}{d r} = -\frac{\re^{\lambda T} \sin(\omega T)}{\omega P}.
\label{eq:affinedTdr}
\end{equation}
Since $\omega > 0$ and $P < 0$,
we conclude that $\frac{d T}{d r} > 0$ if $b_0 < 0$ (here $\sin(\omega T) > 0$)
and $\frac{d T}{d r} < 0$ if $b_0 > 0$ (here $\sin(\omega T) < 0$).
Moreover by \eqref{eq:affineT2} we have
$\lim_{r \to \infty} T = \frac{\pi}{\omega}$ (here $\sin(\omega T) \to 0$).

Now we analyse $P$.
By \eqref{eq:affineT2}--\eqref{eq:affineP2},
$\frac{P}{r} = \frac{-\re^{2 \lambda T} \varrho \left( \omega T; -\frac{\lambda}{\omega} \right)}
{\varrho \left( \omega T; \frac{\lambda}{\omega} \right)}$.
Substituting $T = \frac{\pi}{\omega}$ gives $\frac{P}{r} = -\re^{\frac{\lambda \pi}{\omega}}$.
Thus $P \sim -r \,\re^{\frac{\lambda \pi}{\omega}}$ as $r \to \infty$.
By applying \eqref{eq:auxFuncFuncDeriv} to \eqref{eq:affineT2}--\eqref{eq:affineP2} we obtain
\begin{equation}
\frac{d P}{d r} = -\frac{\varrho \left( \omega T; \frac{\lambda}{\omega} \right)}
{\varrho \left( \omega T; -\frac{\lambda}{\omega} \right)},
\label{eq:PaffineDeriv1}
\end{equation}
and a further application of \eqref{eq:affineT2}--\eqref{eq:affineP2} yields \eqref{eq:PaffineDeriv2}.
This implies $\frac{d P}{d r} < 0$ (because $r > 0$ and $P < 0$).

Finally we treat Types \ref{it:b0neg}--\ref{it:b0poslambdapos} in order. 
\begin{enumerate}[label=\Roman{*})]
\item
By \eqref{eq:affineT2}, as $r \to 0$, $\varrho \left( \omega T; \frac{\lambda}{\omega} \right) \to 0$.
With $b_0 < 0$, $\omega T \in (0,\pi)$,
so as $r \to 0$ we must have $T \to 0$ (see Fig.~\ref{fig:schemAuxFunc}).
Similarly $P \to 0$ by \eqref{eq:affineP2}.
\item
If $b_0 > 0$ and $\lambda < 0$, then $\omega T \in (\pi,2 \pi)$ and so
by \eqref{eq:affineP2} and the definition of $\hat{s}$, see \eqref{eq:shat},
we have $P = 0$ when $T = \frac{\hat{s}(\nu)}{\omega}$,
where $\nu = -\frac{\lambda}{\omega} > 0$.
To show that this occurs when $r = \hat{r}$,
we substitute $T = \frac{\hat{s}(\nu)}{\omega}$ into \eqref{eq:affineT2} and
use the identity
\begin{equation}
\varrho(\hat{s}(\nu);-\nu) = \left( 1 + \nu^2 \right) \sin^2(\hat{s}(\nu)),
\label{eq:auxFunchatsMinusnu}
\end{equation}
to obtain $r = -\kappa (1 + \nu^2) \,\re^{\nu \hat{s}(\nu)} \sin(\hat{s}) = \hat{r}$.
Since $P \in (-\infty,0)$ is a decreasing function of $r$,
$P$ is undefined for all $r < \hat{r}$.
\item
If $b_0 > 0$ and $\lambda > 0$, again $\omega T \in (\pi,2 \pi)$ and so
by \eqref{eq:affineT2} we have $r = 0$ when $T = \frac{\hat{s}(\nu)}{\omega}$,
where $\nu = \frac{\lambda}{\omega} > 0$.
Here $P = \frac{b_0}{\omega} \,\re^{\nu \hat{s}(\nu)} \sin(\hat{s}(\nu))$
by \eqref{eq:auxFunchatsMinusnu}.
\end{enumerate}
\end{proof}

\section{Proofs for boundary equilibrium bifurcations}
\setcounter{equation}{0}
\setcounter{figure}{0}
\setcounter{table}{0}
\label{app:beb}

We start by proving the theorem for HLB 4.
The theorem for HLB 1 is then a simple corollary.
The proof for HLB 4 involves a blow-up of phase space about the origin
so that the limit $\mu \to 0^+$ produces a piecewise-linear system
for which a Poincar\'e map can be described using Lemma \ref{le:affine}.
The Poincar\'e map is shown to have a unique fixed point by
generalising analytical arguments developed in \cite{FrPo98}.
After this we prove the theorems for HLBs 2, 3 and 20.

\begin{proof}[Proof of Theorem \ref{th:FilippovDeg} (HLB 4)]
Write
\begin{equation}
\begin{split}
f_J(x,y;\mu) &= a_{1J} x + a_{2J} y + a_{3J} \mu + \cO \left( \left( |x| + |y| + |\mu| \right)^2 \right), \\
g_J(x,y;\mu) &= b_{1J} x + b_{2J} y + b_{3J} \mu + \cO \left( \left( |x| + |y| + |\mu| \right)^2 \right),
\end{split}
\label{eq:FilippovDegfJgJ}
\end{equation}
for each $J \in \{ L, R \}$.
Then
\begin{equation}
\beta_J = a_{3J} b_{2J} - a_{2J} b_{3J} \,,
\label{eq:FilippovDegbetaJ}
\end{equation}
and
\begin{equation}
\gamma = a_{2L} a_{3R} - a_{3L} a_{2R} \,.
\label{eq:FilippovDeggamma}
\end{equation}
In view of the replacement $y \mapsto -y$ it suffices to assume
orbits rotate clockwise, that is $a_{2L} > 0$ and $a_{2R} > 0$.

Locally, the left half-system has a unique equilibrium:
an unstable focus with $x$-value
$\frac{-\beta_L \mu}{\lambda_L^2 + \omega_L^2} + \cO \left( \mu^2 \right)$.
Since $\beta_L > 0$, this equilibrium is admissible (in $\Omega_L$) for $\mu > 0$.
Similarly, locally, the right half-system has a unique equilibrium:
a stable focus with $x$-value
$\frac{-\beta_R \mu}{\lambda_R^2 + \omega_R^2} + \cO \left( \mu^2 \right)$.
Since $\beta_R > 0$, this equilibrium is admissible (in $\Omega_R$) for $\mu < 0$.

For each $J \in \{ L, R \}$,
$f_J$ is $C^2$ and $a_{2J} \ne 0$
thus, by the IFT, $f_J(0,\zeta_J(\mu);\mu) = 0$
for a unique $C^2$ function
\begin{equation}
\zeta_J(\mu) = -\frac{a_{3J}}{a_{2J}} \,\mu + \cO \left( \mu^2 \right),
\label{eq:FilippovDegzetaJ}
\end{equation}
defined in a neighbourhood of $\mu = 0$.
The points $(0,\zeta_L(\mu))$ and $(0,\zeta_R(\mu))$ are folds, and observe
\begin{equation}
\zeta_L(\mu) - \zeta_R(\mu) = \frac{\gamma}{a_{2L} a_{2R}} \,\mu + \cO \left( \mu^2 \right).
\label{eq:FilippovDegzetaLRdiff}
\end{equation}
The folds may be classified by Definition \ref{df:fold}
(recall $\beta_L, \beta_R > 0$).
At $(0,\zeta_L(\mu))$,
$\frac{\partial f_L}{\partial y} g_L = -\beta_L \mu + \cO \left( \mu^2 \right)$,
thus $(0,\zeta_L(\mu))$ is visible for $\mu > 0$ [invisible for $\mu < 0$].
Similarly at $(0,\zeta_R(\mu))$,
$\frac{\partial f_R}{\partial y} g_R = -\beta_R \mu + \cO \left( \mu^2 \right)$,
thus $(0,\zeta_R(\mu))$ is invisible for $\mu > 0$ [visible for $\mu < 0$].

For the remainder of the proof we consider $\mu \ge 0$.
This is because the result for $\mu < 0$ follows from the following symmetry property:
flipping the signs of $x$ and $\mu$ and reversing the direction of time
transforms \eqref{eq:FilippovBEB} to another system of this form
satisfying the conditions of Theorem \ref{th:FilippovDeg}.

In order to accommodate the nonlinear terms in \eqref{eq:FilippovDegfJgJ},
we work in scaled coordinates
\begin{align}
\tilde{x} &= \frac{x}{\mu}, &
\tilde{y} &= \frac{y - \zeta_R(\mu)}{\mu}.
\label{eq:FilippovDegScaling}
\end{align}
This blow-up shifts the vertical coordinate to put the right fold at the origin.
In these coordinates the right half-system is
\begin{equation}
\begin{split}
\dot{\tilde{x}} &= a_{1R} \tilde{x} + a_{2R} \tilde{y} + \cO(\mu), \\
\dot{\tilde{y}} &= -\frac{\beta_R}{a_{2R}} + b_{1R} \tilde{x} + b_{2R} \tilde{y} + \cO(\mu),
\end{split}
\label{eq:FilippovDegfRgRScaled}
\end{equation}
and the left half-system is
\begin{equation}
\begin{split}
\dot{\tilde{x}} &= -\frac{\gamma}{a_{2R}} + a_{1L} \tilde{x} + a_{2L} \tilde{y} + \cO(\mu), \\
\dot{\tilde{y}} &= -\frac{a_{2R} b_{3L} - a_{3R} b_{2L}}{a_{2R}} +
b_{1L} \tilde{x} + b_{2L} \tilde{y} + \cO(\mu).
\end{split}
\label{eq:FilippovDegfLgLScaled}
\end{equation}
Fig.~\ref{fig:schemFilippovDeg} shows a phase portrait.
Since $\gamma \ge 0$ by assumption,
$\zeta_L(\mu) \ge \zeta_R(\mu)$ by \eqref{eq:FilippovDegzetaLRdiff}.

\begin{figure}[b!]
\begin{center}
\includegraphics[width=8.4cm]{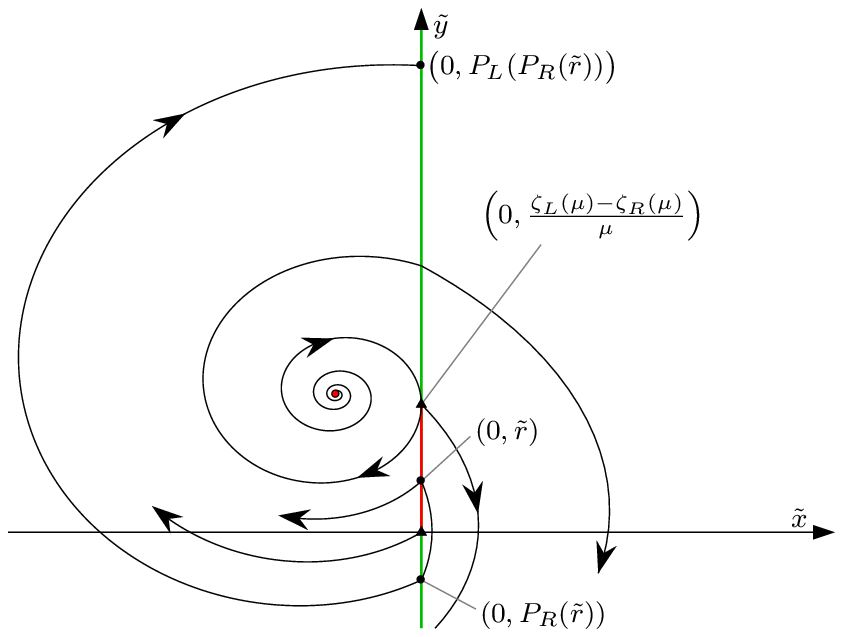}
\caption{
For HLB 4, a sketch of the local dynamics with $\mu > 0$ and using the coordinates \eqref{eq:FilippovDegScaling}.
There exists an admissible unstable focus in $\Omega_L$;
a visible fold of the left half-system
and an invisible fold of the right half-system bound a repelling sliding region.
On this region the direction of sliding motion is not indicated because
there may be pseudo-equilibria, see \cite{Si18f}.
\label{fig:schemFilippovDeg}
} 
\end{center}
\end{figure}

Given $\tilde{r} > 0$, consider the forward orbit of $(\tilde{x},\tilde{y}) = (0,\tilde{r})$
that immediately enters $\tilde{x} > 0$.
Let $P_R(\tilde{r};\mu)$ be the $\tilde{y}$-value of the next intersection of this orbit with $\tilde{x} = 0$
(leave $P_R$ undefined if this orbit does not return to $\tilde{x} = 0$)
and let $T_R(\tilde{r};\mu)$ be the corresponding evolution time.
Similarly for any $\tilde{r} < 0$,
let $P_L(\tilde{r};\mu)$ be the $\tilde{y}$-value of the next intersection of the forward orbit
of $(0,\tilde{r})$ with $\tilde{x} = 0$
(leave $P_L$ undefined if this orbit does not return to $\tilde{x} = 0$)
and let $T_L(\tilde{r};\mu)$ be the corresponding evolution time.
Let
\begin{equation}
P(\tilde{r};\mu) = P_L \left( P_R(\tilde{r};\mu); \mu \right),
\label{eq:FilippovDegP}
\end{equation}
be the induced Poincar\'{e} map.
This corresponds to an evolution time of
\begin{equation}
T(\tilde{r};\mu) = T_R(\tilde{r};\mu) + T_L \left( P_R(\tilde{r};\mu); \mu \right).
\label{eq:FilippovDegT}
\end{equation}

Fixed points of $P$ correspond to periodic orbits of
\eqref{eq:FilippovDegfRgRScaled}--\eqref{eq:FilippovDegfLgLScaled}
and thus also of the original system \eqref{eq:FilippovBEB}.
Locally, periodic orbits of \eqref{eq:FilippovBEB} must evolve on both
sides of $x = 0$, thus the fixed points of $P$ generate all local periodic orbits of \eqref{eq:FilippovBEB}.

Although the scaled system \eqref{eq:FilippovDegfRgRScaled}--\eqref{eq:FilippovDegfLgLScaled}
was constructed under the assumption $\mu > 0$,
in \eqref{eq:FilippovDegfRgRScaled}--\eqref{eq:FilippovDegfLgLScaled} we may set $\mu = 0$.
Next we analyse $P$ and $T$ with $\mu = 0$ and accommodate $\mu > 0$ at the end.

With $\mu = 0$, \eqref{eq:FilippovDegfRgRScaled} is affine
and we can apply Lemma \ref{le:affine} with $b_0 = -\frac{\beta_R}{a_{2R}}$.
Specifically,
\begin{equation}
\begin{split}
P_R(\tilde{r};0) &= P_{\rm affine} \left( \tilde{r}; \lambda_R, \omega_R, -\tfrac{\beta_R}{a_{2R}} \right), \\
T_R(\tilde{r};0) &= T_{\rm affine} \left( \tilde{r}; \lambda_R, \omega_R, -\tfrac{\beta_R}{a_{2R}} \right),
\end{split}
\end{equation}
which are of Type \ref{it:b0neg} because $b_0 < 0$.
In order to apply Lemma \ref{le:affine} to \eqref{eq:FilippovDegfRgRScaled},
we rotate coordinates by $180^\circ$ and translate so that the fold is at the origin.
Specifically, we let
\begin{equation}
\begin{split}
\tilde{\tilde{x}} &= -\tilde{x}, \\
\tilde{\tilde{y}} &= -\left( \tilde{y} - \tilde{\gamma} \right),
\end{split}
\label{eq:FilippovDegScaling2}
\end{equation}
where
\begin{equation}
\tilde{\gamma} = \frac{\gamma}{a_{2L} a_{2R}},
\label{eq:FilippovDegFoldDifference}
\end{equation}
with which \eqref{eq:FilippovDegfLgLScaled} with $\mu = 0$ becomes
\begin{equation}
\begin{split}
\dot{\tilde{\tilde{x}}} &= a_{1L} \tilde{\tilde{x}} + a_{2L} \tilde{\tilde{y}}, \\
\dot{\tilde{\tilde{y}}} &= \frac{\beta_L}{a_{2L}} + b_{1L} \tilde{\tilde{x}} + b_{2L} \tilde{\tilde{y}}.
\end{split}
\label{eq:FilippovDegfLgLScaled2}
\end{equation}
Then Lemma \ref{le:affine} with $b_0 = \frac{\beta_L}{a_{2L}}$ gives
\begin{equation}
\begin{split}
P_L(\tilde{r};0) &= \tilde{\gamma} -
P_{\rm affine} \left( \tilde{\gamma} - \tilde{r};
\lambda_L, \omega_L, \frac{\beta_L}{a_{2L}} \right), \\
T_L(\tilde{r};0) &= T_{\rm affine} \left( \tilde{\gamma} - \tilde{r};
\lambda_L, \omega_L, \frac{\beta_L}{a_{2L}} \right),
\end{split}
\end{equation}
which are of Type \ref{it:b0poslambdapos} because $b_0 > 0$ and $\lambda_L > 0$.
Notice $P(\tilde{r};0)$ is well-defined for all $\tilde{r} > 0$.

By Lemma \ref{le:affine},
$P_R(\tilde{r};0) \sim -\tilde{r} \,\re^{\frac{\lambda_R \pi}{\omega_R}}$ as $\tilde{r} \to \infty$, and
$P_L(\tilde{r};0) \sim -\tilde{r} \,\re^{\frac{\lambda_L \pi}{\omega_L}}$ as $\tilde{r} \to -\infty$.
Thus
\begin{equation}
P(\tilde{r};0) \sim \tilde{r} \,\re^{\alpha \pi},
\label{eq:FilippovDegPlimit}
\end{equation}
as $\tilde{r} \to \infty$,
where $\alpha = \frac{\lambda_L}{\omega_L} + \frac{\lambda_R}{\omega_R}$, \eqref{eq:FilippovDegNondegCond}.
By \eqref{eq:PaffineDeriv2} we obtain
\begin{equation}
\frac{\partial P}{\partial \tilde{r}}(\tilde{r};0) =
\frac{\tilde{r} \left( \tilde{\gamma} - P_R \right)}
{P_R \left( \tilde{\gamma} - P \right)} \,\re^{2 h(\tilde{r})},
\label{eq:FilippovDegPDeriv}
\end{equation}
where
\begin{equation}
h(\tilde{r}) = \lambda_R T_R(\tilde{r};0) + \lambda_L T_L \left( P_R(\tilde{r};0); 0 \right).
\label{eq:FilippovDegchi}
\end{equation}
Observe $\tilde{r} > 0$, and $P_R < 0 < \tilde{\gamma} < P$.
For all $\tilde{r} > 0$,
\begin{equation}
\frac{d h}{d \tilde{r}} < 0,
\label{eq:FilippovDegchiDeriv}
\end{equation}
because $\lambda_R < 0$ and $T_R(\tilde{r};0)$ is increasing and
$\lambda_L > 0$ and $T_L \left( P_R(\tilde{r};0); 0 \right)$ is decreasing.
The period of the limit cycle corresponding to a fixed point $\tilde{r}^*$ is $T = T_L + T_R$,
where $P_L = \tilde{r}^*$, $P_R$, $T_L$, and $T_R$ satisfy
\begin{align*}
\tilde{r}^* &= \frac{-\kappa_R \,\re^{-\lambda_R T_R} \varrho \big( \omega_R T_R; \frac{\lambda_R}{\omega_R} \big)}
{\sin(\omega_R T_R)}, \\
P_R &= \frac{\kappa_R \,\re^{\lambda_R T_R} \varrho \big( \omega_R T_R; -\frac{\lambda_R}{\omega_R} \big)}
{\sin(\omega_R T_R)}, \\
\tilde{\gamma} - P_R &= \frac{-\kappa_L \,\re^{-\lambda_L T_L} \varrho \big( \omega_L T_L; \frac{\lambda_L}{\omega_L} \big)}
{\sin(\omega_L T_L)}, \\
\tilde{\gamma} - P_L &= \frac{\kappa_L \,\re^{\lambda_L T_L} \varrho \big( \omega_L T_L; -\frac{\lambda_L}{\omega_L} \big)}
{\sin(\omega_L T_L)},
\end{align*}
where $\kappa_R = \frac{-\beta_R \omega_R}{a_{2R} \left( \lambda_R^2 + \omega_R^2 \right)}$
and $\kappa_L = \frac{\beta_L \omega_L}{a_{2L} \left( \lambda_L^2 + \omega_L^2 \right)}$.
By eliminating $P_L$ and $P_R$ in these equations we produce \eqref{eq:periodFilippovDeg}.

This completes our construction and basic description of the Poincar\'e map $P$
and the evolution time $T$.
We now use a series of analytical arguments to characterise fixed points of $P$ in the limit $\mu \to 0^+$.
Following the logical steps of \cite{FrPo98},
we suppose, for the moment,
that $P(\tilde{r};0)$ has a fixed point, and let $\tilde{r}^*$ be the smallest such point.
Since $P(0;0) > 0$, we must have $P(\tilde{r};0) > \tilde{r}$, for all $0 \le \tilde{r} < \tilde{r}^*$,
and so $\frac{\partial P}{\partial \tilde{r}} \left( \tilde{r}^*;0 \right) \le 1$.

Now suppose, for a contradiction, that
$\frac{\partial P}{\partial \tilde{r}} \left( \tilde{r}^*;0 \right) = 1$.
Then we need
$\frac{\partial^2 P}{\partial \tilde{r}^2} \left( \tilde{r}^*;0 \right) \ge 0$,
because $P(\tilde{r};0) > \tilde{r}$ for all $0 \le \tilde{r} < \tilde{r}^*$.
But by differentiating \eqref{eq:FilippovDegPDeriv}
and substituting $\tilde{r} = P = \tilde{r}^*$ and
$\frac{\partial P}{\partial \tilde{r}} = 1$, we get
\begin{equation}
\frac{\partial^2 P}{\partial \tilde{r}^2} \left( \tilde{r}^*;0 \right) =
\frac{\tilde{\gamma}}{\tilde{r}^* \left( \tilde{\gamma} - \tilde{r}^* \right)} -
\frac{\tilde{\gamma} \,\frac{\partial P_R}{\partial \tilde{r}}}
{P_R \left( \tilde{\gamma} - P_R \right)} +
2 \frac{d h}{d \tilde{r}}.
\label{eq:FilippovDegPDerivDeriv}
\end{equation}
The first two terms \eqref{eq:FilippovDegPDerivDeriv} are non-positive,
and the last term is negative,
thus $\frac{\partial^2 P}{\partial \tilde{r}^2} \left( \tilde{r}^*;0 \right) < 0$,
which is a contradiction.
Hence $\frac{\partial P}{\partial \tilde{r}} \left( \tilde{r}^*;0 \right) < 1$,
and so $\tilde{r}^*$ is an asymptotically stable fixed point of $P(\tilde{r};0)$.

Next suppose, for a contradiction, that
$P(\tilde{r};0)$ has other fixed points,
and let $\tilde{r}^{**}$ be the next smallest such point.
Then $P(\tilde{r};0) < \tilde{r}$ for all $\tilde{r}^* \le \tilde{r} < \tilde{r}^{**}$
and so $\frac{\partial P}{\partial \tilde{r}} \left( \tilde{r}^{**};0 \right) \ge 1$.
At a fixed point, equation \eqref{eq:FilippovDegPDeriv}
is satisfied with $P = \tilde{r}$, that is
\begin{equation}
\frac{\partial P}{\partial \tilde{r}}(\tilde{r};0) =
G(\tilde{r}) \,\re^{2 h(\tilde{r})},
\label{eq:FilippovDegPDerivAtFP}
\end{equation}
where
\begin{equation}
G(\tilde{r}) = \frac{\tilde{r} \left( \tilde{\gamma} - P_R \right)}
{P_R \left( \tilde{\gamma} - \tilde{r} \right)}.
\nonumber
\end{equation}
We have
\begin{equation}
\frac{d G}{d \tilde{r}} = \left(
\frac{\tilde{\gamma}}{\tilde{r}^* \left( \tilde{\gamma} - \tilde{r}^* \right)} -
\frac{\tilde{\gamma} \,\frac{\partial P_R}{\partial \tilde{r}}}
{P_R \left( \tilde{\gamma} - P_R \right)} \right) G(\tilde{r}),
\nonumber
\end{equation}
so $G$ is a decreasing function because
$G$ is positive and, as above, the term in large brackets is negative.
Since $h$ is decreasing,
we conclude from \eqref{eq:FilippovDegPDerivAtFP} that $\frac{\partial P}{\partial \tilde{r}}(\tilde{r};0)$ is decreasing.
But $\frac{\partial P}{\partial \tilde{r}} \left( \tilde{r}^*;0 \right) < 1$,
so we cannot have
$\frac{\partial P}{\partial \tilde{r}} \left( \tilde{r}^{**};0 \right) \ge 1$.
This is contradiction, hence $P(\tilde{r};0) < \tilde{r}$ for all $\tilde{r} > \tilde{r}^*$.

By \eqref{eq:FilippovDegPlimit}, this requires $\alpha \le 0$.
Thus $P(\tilde{r};0)$ has no fixed points if $\alpha > 0$.
If $\alpha < 0$, then $P(\tilde{r};0)$ has a fixed point by the intermediate value theorem,
and, as we have just shown, this fixed point is unique, call it $\tilde{r}^*$.
Moreover, as shown above,
$\frac{\partial P}{\partial \tilde{r}} \left( \tilde{r}^*;0 \right) < 1$, thus the fixed point is asymptotically stable.

Finally we consider $P(\tilde{r};\mu)$ with $\mu > 0$.
Notice \eqref{eq:FilippovDegfRgRScaled} and \eqref{eq:FilippovDegfLgLScaled}
are $C^1$ (with respect to both the variables $\tilde{x}$ and $\tilde{y}$ and the parameter $\mu$).
Thus the flow of \eqref{eq:FilippovDegfRgRScaled} and \eqref{eq:FilippovDegfLgLScaled}
are $C^1$ functions of the initial condition and of $\mu$.
For all $\tilde{r} > 0$, $P(\tilde{r};0)$ is well-defined and corresponds to transversal intersections with $\tilde{x} = 0$,
thus $P(\tilde{r};\mu)$ is well-defined and $C^1$, by Lemma \ref{le:mapSmoothness},
for sufficiently small $\mu > 0$.
By the IFT applied to $P$,
if $\alpha < 0$ then $P$ has a unique stable fixed point $\tilde{r}^*(\mu)$.
Hence, locally, the scaled system \eqref{eq:FilippovDegfRgRScaled}--\eqref{eq:FilippovDegfLgLScaled}
has a unique stable limit cycle, and the same is true for the original system \eqref{eq:FilippovBEB}.
In view of \eqref{eq:FilippovDegScaling},
in \eqref{eq:FilippovBEB} the size of the limit cycle is asymptotically proportional to $\mu$
(for example if $\tilde{x}_{\rm max}$ denotes the maximum $\tilde{x}$-value of the limit cycle
in \eqref{eq:FilippovDegfRgRScaled}--\eqref{eq:FilippovDegfLgLScaled} with $\mu = 0$,
then the maximum $x$-value of the limit cycle in \eqref{eq:FilippovBEB} is
$\mu \tilde{x}_{\rm max} + \cO \left( \mu^2 \right)$).
If instead $\alpha > 0$,
then, locally, $P$ has no fixed points and thus \eqref{eq:FilippovBEB} has no limit cycles.
As mentioned above, the dynamics for $\mu < 0$ follows by symmetry.
\end{proof}

\begin{proof}[Proof of Theorem \ref{th:pwscFocusFocus} (HLB 1)]
Unlike for Theorem \ref{th:FilippovDeg},
locally there must exist a unique stationary solution because \eqref{eq:pwscODE} is continuous;
the only stationary solutions are regular equilibria.
The left half-system has an unstable focus with $x$-value $\frac{-\beta \mu}{\lambda_L^2 + \omega_L^2} + \cO \left( \mu^2 \right)$,
so is admissible (in $\Omega_L$) for $\mu > 0$ because $\beta > 0$.
Similarly the right half-system has admissible stable focus for $\mu < 0$.

To complete the proof we apply Theorem \ref{th:FilippovDeg}.
In \eqref{eq:FilippovDegTransCond} we have $\beta_L = \beta_R = \beta > 0$.
In \eqref{eq:FilippovDegSlidingCond} we have $\gamma = 0$.
Also $a_{2L} a_{2R} = \left( \frac{\partial f_L}{\partial y}(0,0;0) \right)^2$
is non-zero (and hence positive) because $\rD F_L(0,0;0)$ has complex eigenvalues.
This shows that the conditions of Theorem \ref{th:FilippovDeg} are satisfied.
The conclusions of Theorem \ref{th:FilippovDeg} reduce to those Theorem \ref{th:pwscFocusFocus};
this completes the proof.
\end{proof}

\begin{proof}[Proof of Theorem \ref{th:pwscFocusNode} (HLB 2)]
This proof follows that of Theorem \ref{th:FilippovDeg} (given above)
except we need to accommodate the real eigenvalues of the stable node
and carefully consider the case $\mu < 0$
(to show that no limit cycle exists when the node is admissible).
Write
\begin{equation}
\begin{split}
f_J(x,y;\mu) &= a_{1J} x + a_2 y + a_3 \mu + \cO \left( \left( |x| + |y| + |\mu| \right)^2 \right), \\
g_J(x,y;\mu) &= b_{1J} x + b_2 y + b_3 \mu + \cO \left( \left( |x| + |y| + |\mu| \right)^2 \right),
\end{split}
\label{eq:pwscFocusNodefJgJ}
\end{equation}
for each $J \in \{ L, R \}$.
Notice $a_2 \ne 0$ by the eigenvalue condition \eqref{eq:pwscEigCondFocusNode}
and we can assume $a_2 > 0$ in view of the replacement $y \mapsto -y$.
Also by \eqref{eq:pwscEigCondFocusNode}, locally the left half-system has a unique equilibrium:
an unstable focus with $x$-value $\frac{-\beta \mu}{\lambda_L^2 + \omega_L^2} + \cO \left( \mu^2 \right)$,
that is admissible (in $\Omega_L$) for $\mu > 0$ (because $\beta > 0$).
Similarly, locally, the right half-system has a unique equilibrium:
a stable node with $x$-value $\frac{-\beta \mu}{\lambda_R^2 - \eta_R^2} + \cO \left( \mu^2 \right)$,
that is admissible (in $\Omega_R$) for $\mu < 0$.
Since $f_L$ is $C^2$ and $a_2 \ne 0$,
by the IFT we have $f_L(0,\zeta(\mu);\mu) = 0$ for a unique $C^2$ function
\begin{equation}
\zeta(\mu) = -\frac{a_3}{a_2} \,\mu + \cO \left( \mu^2 \right),
\nonumber
\end{equation}
defined in a neighbourhood of $\mu = 0$.
Notice $(x,y) = (0,\zeta(\mu))$ is a fold for both half-systems of \eqref{eq:pwscODE}.

First consider $\mu > 0$.
We introduce the scaled coordinates
\begin{align}
\tilde{x} &= \frac{x}{\mu}, &
\tilde{y} &= \frac{y - \zeta(\mu)}{\mu},
\label{eq:pwscFocusNodeScaling}
\end{align}
with which the system becomes
\begin{equation}
\begin{split}
\dot{\tilde{x}} &= a_{1J} \tilde{x} + a_2 \tilde{y} + \cO(\mu), \\
\dot{\tilde{y}} &= -\frac{\beta}{a_2} + b_{1J} \tilde{x} + b_2 \tilde{y} + \cO(\mu),
\end{split}
\label{eq:pwscFocusNodefJgJScaled}
\end{equation}
with $J = L$ for $\tilde{x} < 0$ and $J = R$ for $\tilde{x} > 0$.
Given $\tilde{r} > 0$, let $P_R(\tilde{r};\mu) < 0$ be the $\tilde{y}$-value
of the next intersection of the forward orbit of $(0,\tilde{r})$ with $\tilde{x} = 0$,
and let $T_R(\tilde{r};\mu)$ be the corresponding evolution time.
Similarly given $\tilde{r} < 0$, let $P_L(\tilde{r};\mu) > 0$ be the $\tilde{y}$-value
of the next intersection of the forward orbit of $(0,\tilde{r})$ with $\tilde{x} = 0$,
and let $T_L(\tilde{r};\mu)$ be the corresponding evolution time.
Let
\begin{equation}
P(\tilde{r};\mu) = P_L \left( P_R(\tilde{r};\mu); \mu \right),
\label{eq:pwscFocusNodeP}
\end{equation}
which corresponds to an evolution time of
\begin{equation}
T(\tilde{r};\mu) = T_R(\tilde{r};\mu) + T_L \left( P_R(\tilde{r};\mu); \mu \right).
\label{eq:pwscFocusNodeT}
\end{equation}
Locally, limit cycles must enter both $\tilde{x} < 0$ and $\tilde{x} > 0$,
thus all possible limit cycles are generated by the fixed points of $P$.

Next we analyse $P(\tilde{r};0)$ and $T(\tilde{r};0)$
(which are well-defined even though \eqref{eq:pwscFocusNodefJgJScaled}
was constructed under the assumption $\mu > 0$).
In the limit $\mu \to 0^+$, \eqref{eq:pwscFocusNodefJgJScaled} has the affine form \eqref{eq:affineODE}.
Thus for \eqref{eq:pwscFocusNodefJgJScaled} with $J = R$,
\begin{equation}
\begin{split}
P_R(\tilde{r};0) &= P_{\rm affine} \left( \tilde{r}; \lambda_R, \ri \eta_R, -\frac{\beta}{a_2} \right), \\
T_R(\tilde{r};0) &= T_{\rm affine} \left( \tilde{r}; \lambda_R, \ri \eta_R, -\frac{\beta}{a_2} \right),
\end{split}
\nonumber
\end{equation}
where the imaginary argument arises from putting
the real eigenvalues $\lambda_R \pm \eta_R$ \eqref{eq:pwscEigCondFocusNode}
in the complex form \eqref{eq:affineEigCond}.
Let
\begin{equation}
\varrho_{\rm node}(s;\nu) = 1 - \re^{\nu s} \left( \cosh(s) - \nu \sinh(s) \right).
\label{eq:auxFuncNode}
\end{equation}
Like $\varrho$ (see Fig.~\ref{fig:schemAuxFunc}),
the function $\varrho_{\rm node}$ has a double zero at $s = 0$,
but if $|\nu| > 1$ then $\varrho_{\rm node}(s;\nu) > 0$ for all $s \ne 0$.
Using the identity $\varrho(\ri s;-\ri \nu) = \varrho_{\rm node}(s;\nu)$, by Lemma \ref{le:affine}
\begin{align}
\tilde{r} &= \frac{-{\rm Im}(\kappa_R) \,\re^{-\lambda_R T_R} \varrho_{\rm node} \left( \eta_R T_R; \frac{\lambda_R}{\eta_R} \right)}
{\sinh \left( \eta_R T_R \right)},
\label{eq:pwscFocusNodeTR} \\
P_R &= \frac{{\rm Im}(\kappa_R) \,\re^{\lambda_R T_R} \varrho_{\rm node} \left( \eta_R T_R; \frac{-\lambda_R}{\eta_R} \right)}
{\sinh \left( \eta_R T_R \right)},
\label{eq:pwscFocusNodePR}
\end{align}
where
\begin{equation}
\kappa_R = \frac{-\ri \beta \eta_R}{a_2 \left( \lambda_R^2 - \eta_R^2 \right)}.
\label{eq:pwscFocusNodekappaR}
\end{equation}
Analogous to \eqref{eq:affinedTdr},
\begin{equation}
\frac{\partial T_R}{\partial \tilde{r}} = -\frac{\re^{\lambda_R T_R} \sinh(\eta_R T_R)}{\eta_R P_R},
\label{eq:pwscFocusNodedTdr}
\end{equation}
thus $T_R(\tilde{r};0)$ is an increasing function of $\tilde{r}$
(because $\eta_R > 0$ and $P_R < 0$).
Analogous to \eqref{eq:PaffineDeriv1},
\begin{equation}
\frac{\partial P_R}{\partial \tilde{r}} =
-\frac{\varrho_{\rm node} \left( \eta_R T_R; \frac{\lambda_R}{\eta_R} \right)}
{\varrho_{\rm node} \left( \eta_R T_R; -\frac{\lambda_R}{\eta_R} \right)},
\label{eq:pwscFocusNodedPdr}
\end{equation}
thus $P_R(\tilde{r};0)$ is a decreasing function of $\tilde{r}$
(because $\big| \frac{\lambda_R}{\eta_R} \big| > 1$).
By \eqref{eq:pwscFocusNodeTR}--\eqref{eq:pwscFocusNodePR},
$T_R(\tilde{r};0) \to 0$ and $P_R(\tilde{r};0) \to 0$ as $\tilde{r} \to 0$.
As $\tilde{r} \to \infty$,
\begin{equation}
\tilde{r} \sim -2 {\rm Im}(\kappa_R) \,\re^{-(\lambda_R + \eta_R) T_R},
\label{eq:pwscFocusNodeTRLimit}
\end{equation}
and\footnote{
The limiting value \eqref{eq:pwscFocusNodePRLimit} corresponds to the intersection
of the linear slow subspace of \eqref{eq:pwscFocusNodefJgJScaled} (with $J = R$ and $\mu = 0$)
which can be obtained more simply by calculating the eigenspace of $D F_R(0,0;0)$
corresponding to $\lambda_R + \eta_R$.
}
\begin{equation}
P_R(\tilde{r};0) \to \frac{-\beta}{a_2 \left( \eta_R - \lambda_R \right)}.
\label{eq:pwscFocusNodePRLimit}
\end{equation}

Lemma \ref{le:affine} can be applied to \eqref{eq:pwscFocusNodefJgJScaled} with $J = L$
(still in the limit $\mu \to 0^+$) by rotating coordinates by $180^\circ$.
In this way we obtain
\begin{equation}
\begin{split}
P_L(\tilde{r};0) &= -P_{\rm affine} \left( -\tilde{r}; \lambda_L, \omega_L, \frac{\beta}{a_2} \right), \\
T_L(\tilde{r};0) &= T_{\rm affine} \left( -\tilde{r}; \lambda_L, \omega_L, \frac{\beta}{a_2} \right),
\end{split}
\end{equation}
which are of Type \ref{it:b0poslambdapos} because $\frac{\beta}{a_2} > 0$ and $\lambda_L > 0$.
It follows that $P(\tilde{r};0) = P_L \left( P_R(\tilde{r};0); 0 \right)$
is an increasing function of $\tilde{r}$ with
$\lim_{\tilde{r} \to 0} P(\tilde{r};0) > 0$ and
$\lim_{\tilde{r} \to \infty} P(\tilde{r};0) < \infty$.
By the intermediate value theorem, $P(\tilde{r};0)$ has a fixed point $\tilde{r}^*$.
It can be shown that $\tilde{r}^*$ is asymptotically stable and unique
(we omit this step for brevity; it can be achieved by a minor adaptation of arguments
that appear in the proof of Theorem \ref{th:FilippovDeg}).
Also by repeating the arguments at the end of the proof of Theorem \ref{th:FilippovDeg}
we conclude that $P(\tilde{r};\mu)$ has a unique stable fixed point $\tilde{r}^*(\mu)$ for
sufficiently small $\mu > 0$ because \eqref{eq:pwscFocusNodefJgJScaled} is $C^1$.
Thus \eqref{eq:pwscODE} has a unique stable limit cycle with amplitude proportional to $\mu$
(due to the scaling \eqref{eq:pwscFocusNodeScaling}).

Second consider $\mu < 0$.
Let $\tilde{x} = \frac{x}{|\mu|}$ and $\tilde{y} = \frac{y - \zeta(\mu)}{|\mu|}$
and take $\mu \to 0^-$ with which the right half-system is affine with an admissible stable node
$\left( \tilde{x}_R^*, \tilde{y}_R^* \right)$, see Fig.~\ref{fig:schemFocusNode}.
Let $\ell_1$ be the line segment connecting $(0,0)$ and $\left( \tilde{x}_R^*, \tilde{y}_R^* \right)$.
Let $\ell_2$ be the linear slow subspace with $\tilde{x} \ge \tilde{x}_R^*$
(this has slope $\frac{\lambda_R + \eta_R - a_{1R}}{a_2}$).
For all $\tilde{r} > 0$, the forward orbit of $(0,\tilde{r})$ cannot cross $\ell_1$
(here $\dot{\tilde{x}} = 0$ and $\dot{\tilde{y}} > 0$) or $\ell_2$ (because $\ell_2$ is an invariant).
Thus the orbit converges to $\left( \tilde{x}_R^*, \tilde{y}_R^* \right)$
without reintersecting $\tilde{x} = 0$.
Now choose any $\tilde{r}_{\rm max} > 0$.
Since \eqref{eq:pwscFocusNodefJgJScaled} is $C^1$ there exists $\delta > 0$ such that
for all $-\delta < \mu < 0$, $\left( \tilde{x}_R^*, \tilde{y}_R^* \right)$ persists
as a stable node $\left( \tilde{x}_R^*(\mu), \tilde{y}_R^*(\mu) \right)$, and
for all $0 < \tilde{r} < \tilde{r}_{\rm max}$
the forward orbit of $(0,\tilde{r})$ converges to $\left( \tilde{x}_R^*(\mu), \tilde{y}_R^*(\mu) \right)$
without reintersecting $\tilde{x} = 0$.
Thus, locally, \eqref{eq:pwscFocusNodefJgJScaled}, and hence also \eqref{eq:pwscODE},
has no limit cycle for $\mu < 0$.
\end{proof}

\begin{figure}[b!]
\begin{center}
\includegraphics[width=5.6cm]{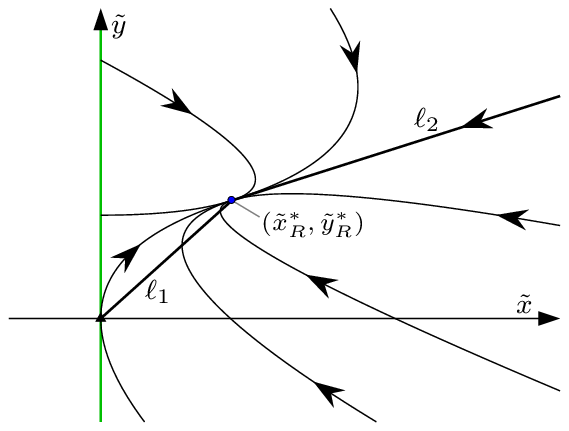}
\caption{
For HLB 2, a sketch of the local dynamics with $\mu < 0$ in scaled coordinates
indicating the line segment $\ell_1$ and the ray $\ell_2$ defined in the text.
\label{fig:schemFocusNode}
} 
\end{center}
\end{figure}

\begin{proof}[Proof of Theorem \ref{th:FilippovGeneric} (HLB 3)]
Write
\begin{equation}
\begin{split}
f_L(x,y;\mu) &= a_{1L} x + a_{2L} y + a_{3L} \mu + \cO \left( \left( |x| + |y| + |\mu| \right)^2 \right), \\
g_L(x,y;\mu) &= b_{1L} x + b_{2L} y + b_{3L} \mu + \cO \left( \left( |x| + |y| + |\mu| \right)^2 \right).
\end{split}
\label{eq:FilippovGenfLgL}
\end{equation}
Then
\begin{align}
\beta &= a_{3L} b_{2L} - a_{2L} b_{3L} \,,
\label{eq:FilippovGenbeta} \\
\gamma &= a_{2L} b_{0R} - a_{0R} b_{2L} \,,
\label{eq:FilippovGengamma}
\end{align}
where $b_{0R} = g_R(0,0;0)$.
Notice $a_{2L} \ne 0$ by \eqref{eq:FilippovEigCondFocus},
and it suffices to assume $a_{2L} > 0$ in view of the replacement $y \mapsto -y$.
Locally the left half-system has a unique equilibrium:
an unstable focus $\left( x_L^*(\mu), y_L^*(\mu) \right)$ with
$x_L^*(\mu) = \frac{-\beta \mu}{\lambda_L^2 + \omega_L^2} + \cO \left( \mu^2 \right)$.
Since $\beta > 0$ the equilibrium is admissible (in $\Omega_L$) for $\mu > 0$.

Since $f_L$ is $C^2$ and $a_{2L} \ne 0$,
by the IFT we have $f_L(0,\zeta_L(\mu);\mu) = 0$ for a unique $C^2$ function
\begin{equation}
\zeta_L(\mu) = -\frac{a_{3L}}{a_{2L}} \,\mu + \cO \left( \mu^2 \right),
\label{eq:FilippovGenzetaL}
\end{equation}
defined in a neighbourhood of $\mu = 0$.
Since $a_{2L} > 0$ and $a_{0R} < 0$,
locally $x=0$ is a crossing region for $y < \zeta_L(\mu)$
and an attracting sliding region for $y > \zeta_L(\mu)$.
On the attracting sliding region orbits evolve according to
$\dot{y} = g_{\rm slide}(y;\mu)$, where $g_{\rm slide}$ is given by \eqref{eq:gslide2}.
Here
\begin{equation}
g_{\rm slide}(y;\mu) = \frac{-\gamma}{a_{0R}} \left( y - \zeta_L(\mu) \right) -
\frac{\beta}{a_{2L}} \,\mu + \cO \left( \left( |y| + |\mu| \right)^2 \right).
\label{eq:FilippovGengslide}
\end{equation}
Pseudo-equilibria satisfy $g_{\rm slide}(y;\mu) = 0$.
Since \eqref{eq:FilippovGengslide} is $C^2$ and $\gamma \ne 0$,
by the IFT there exists a unique pseudo-equilibrium $\left( 0, y_{\rm ps}^*(\mu) \right)$ where
$y_{\rm ps}^*(\mu) = \zeta_L(\mu) - \kappa \mu + \cO \left( \mu^2 \right)$
and $\kappa = \frac{a_{0R} \beta}{a_{2L} \gamma}$.
The pseudo-equilibrium is admissible if $y_{\rm ps}^*(\mu) > \zeta_L(\mu)$.
Since $\kappa > 0$ this is the case when $\mu < 0$.
The pseudo-equilibrium is stable because it belongs to an attracting sliding region and
$\frac{\partial g_{\rm slide}}{\partial y} \left( y_{\rm ps}^*(\mu); \mu \right) =
\frac{-\gamma}{a_{0R}} + \cO(\mu) < 0$.

The following is true in a sufficiently small neighbourhood of $(x,y;\mu) = (0,0;0)$.
The forward orbit of any point,
other than $\left( x_L^*(\mu), y_L^*(\mu) \right)$ in the case $\mu \ge 0$,
eventually reaches the attracting sliding region.
This is because $a_{0R} < 0$, so orbits in $\Omega_R$ reach $x=0$,
and $\left( x_L^*(\mu), y_L^*(\mu) \right)$ is an unstable focus,
so orbits in $\Omega_L$ spiral out and must eventually reach the attracting sliding region.
With $\mu < 0$ orbits on the attracting sliding region
approach the pseudo-equilibrium, hence in this case there are no limit cycles.
With $\mu > 0$ we have $g_{\rm slide}(y;\mu) < 0$ on the attracting sliding region.
Thus sliding orbits reach $y = \zeta_L(\mu)$, then return to the attracting sliding region
after an excursion in $\Omega_L$ forming a unique stable limit cycle.
It remains for us to calculate this limit cycle to ascertain its amplitude and period.

We work in the scaled coordinates
\begin{align}
\tilde{x} &= \frac{x}{\mu}, &
\tilde{y} &= \frac{y - \zeta_L(\mu)}{\mu},
\label{eq:FilippovGenScaling}
\end{align}
assuming $\mu > 0$,
with which the left half-system becomes
\begin{equation}
\begin{split}
\dot{\tilde{x}} &= a_{1L} \tilde{x} + a_{2L} \tilde{y} + \cO(\mu), \\
\dot{\tilde{y}} &= -\frac{\beta}{a_{2L}} + b_{1L} \tilde{x} + b_{2L} \tilde{y} + \cO(\mu).
\end{split}
\label{eq:FilippovGenfLgLScaled}
\end{equation}
On $\tilde{x} = 0$ with $\tilde{y} > 0$,
orbits slide according to $\dot{\tilde{y}} = \tilde{g}_{\rm slide}(\tilde{y};\mu)$, where
\begin{equation}
\tilde{g}_{\rm slide}(\tilde{y};\mu) = \frac{-\gamma}{a_{0R}} \,\tilde{y} -
\frac{\beta}{a_{2L}} + \cO(\mu),
\label{eq:FilippovGengslideScaled}
\end{equation}
until reaching $\tilde{y} = 0$.
Let $P_L(\mu)$ be the $\tilde{y}$-value of the next intersection of the forward orbit
of $(\tilde{x},\tilde{y}) = (0,0)$ with $\tilde{x} = 0$.
That is, $0$ and $P_L(\mu)$ are the departure and arrival $\tilde{y}$-values
of the limit cycle on $\tilde{x} = 0$.
Let $T_L(\mu)$ and $T_{\rm slide}(\mu)$ be the evolution times in $\Omega_L$ and
on $\tilde{x} = 0$ respectively.

By applying Lemma \ref{le:affine} to \eqref{eq:FilippovGenfLgLScaled} with $\mu = 0$
(by rotating coordinates by $180^\circ$) we obtain
\begin{align}
P_L(0) &= \frac{-\beta}{a_{2L} \omega_L} \,\re^{\lambda_L \hat{s}}{\omega_L} \sin(\hat{s}),
\label{eq:FilippovGenPL} \\
T_L(0) &= \frac{\hat{s}}{\omega_L},
\label{eq:FilippovGenTL}
\end{align}
where $\hat{s} = \hat{s} \left( \frac{\lambda_L}{\omega_L} \right)$.
To determine $T_{\rm slide}(0)$, observe that when $\mu = 0$
equation \eqref{eq:FilippovGengslideScaled} is linear with solution
\begin{equation}
\psi_t(\tilde{y}) = \re^{\frac{-\gamma t}{a_{0R}}} \left( \tilde{y} + \kappa \right) - \kappa.
\nonumber
\end{equation}
By definition, $\psi_{T_{\rm slide}(0)} \left( P_L(0) \right) = 0$, thus
\begin{equation}
T_{\rm slide}(0) = \frac{a_{0R}}{\gamma} \,\ln \left( 1 + \frac{P_L(0)}{\kappa} \right).
\label{eq:FilippovGenTslide0}
\end{equation}
Equations \eqref{eq:FilippovGenPL}--\eqref{eq:FilippovGenTslide0} verify \eqref{eq:periodFilippovGen}.
The period of the limit cycle is an $\cO(\mu)$ perturbation of \eqref{eq:periodFilippovGen}
because \eqref{eq:FilippovGenfLgLScaled} and \eqref{eq:FilippovGengslideScaled}
are $\cO(\mu)$ perturbations from their values at $\mu = 0$.
The amplitude of the limit cycle is asymptotically proportional to $\mu$
due to the scaling \eqref{eq:FilippovGenScaling}.
\end{proof}

\begin{proof}[Proof of Theorem \ref{th:sqrt} (HLB 20)]
Without loss of generality assume $a_2 > 0$ (clockwise rotation).
By applying the IFT to
the equation $F(x,y;0;\mu) = (0,0)$, we find that locally
the left half-system has a unique equilibrium:
an unstable focus $\left( x_L^*(\mu), y_L^*(\mu) \right)$ with
$x_L^*(\mu) = \frac{-\beta \mu}{\lambda^2 + \omega^2} + \cO \left( \mu^2 \right)$.
Since $\beta > 0$, the equilibrium is admissible (in $\Omega_L$) for $\mu > 0$.

By applying the IFT to
the equation $F(z^2,y;z;\mu) = (0,0)$, we find that, since $\gamma \ne 0$,
locally the right half-system has a unique equilibrium
$\left( x_R^*(\mu), y_R^*(\mu) \right)$ with
$\sqrt{x_R^*(\mu)} = \frac{-\beta \mu}{\gamma} + \cO \left( \mu^2 \right)$,
valid for $\mu \le 0$ (because $\beta > 0$ and $\gamma > 0$).
By evaluating the Jacobian of \eqref{eq:sqrtODE} with $x > 0$
at $\left( x_R^*(\mu), y_R^*(\mu) \right)$, we find that this equilibrium is a stable node
(for sufficiently small $\mu < 0$).

Since $f$ is $C^2$ and $a_2 \ne 0$,
by the IFT we have $f(0,\zeta(\mu);0;\mu) = 0$ for a unique $C^2$ function
\begin{equation}
\zeta(\mu) = -\frac{a_3}{a_2} \,\mu + \cO \left( \mu^2 \right),
\label{eq:sqrtzeta}
\end{equation}
defined in a neighbourhood of $\mu = 0$.

First consider $\mu > 0$.
As with earlier proofs in this section, to analyse the left half-system we let
\begin{align}
\tilde{x} &= \frac{x}{\mu}, &
\tilde{y} &= \frac{y - \zeta(\mu)}{\mu},
\label{eq:sqrtScalingL}
\end{align}
with which the left half-system becomes
\begin{equation}
\begin{split}
\dot{\tilde{x}} &= a_1 \tilde{x} + a_2 \tilde{y} + \cO(\mu), \\
\dot{\tilde{y}} &= -\frac{\beta}{a_2} + b_1 \tilde{x} + b_2 \tilde{y} + \cO(\mu).
\end{split}
\label{eq:sqrtfLgLScaled}
\end{equation}
Given $\tilde{r} < 0$, let $P_L(\tilde{r};\mu)$ be the $\tilde{y}$-value of the next intersection
of the forward orbit of $(0,\tilde{r})$ with $\tilde{x} = 0$,
and let $T_L(\tilde{r};\mu)$ be the corresponding evolution time.
By Lemma \ref{le:affine} with $b_0 = \frac{\beta}{a_2}$,
\begin{equation}
\begin{split}
P_L(\tilde{r};0) &= -P_{\rm affine} \left( -\tilde{r};
\lambda, \omega, \frac{\beta}{a_2} \right), \\
T_L(\tilde{r};0) &= T_{\rm affine} \left( -\tilde{r};
\lambda, \omega, \frac{\beta}{a_2} \right),
\end{split}
\end{equation}
which are of Type \ref{it:b0poslambdapos} because $b_0 > 0$ and $\lambda > 0$.

To analyse the right half-system we employ the alternate scaling
\begin{align}
\tilde{x} &= \frac{x}{\mu^2}, &
\tilde{y} &= \frac{y - \zeta(\mu)}{\mu},
\label{eq:sqrtScalingR}
\end{align}
with which the right half-system becomes
\begin{equation}
\begin{split}
\mu \dot{\tilde{x}} &= a_4 \sqrt{\tilde{x}} + a_2 \tilde{y} + \cO(\mu), \\
\dot{\tilde{y}} &= -\frac{\beta}{a_2} + b_4 \sqrt{\tilde{x}} + b_2 \tilde{y} + \cO(\mu),
\end{split}
\label{eq:sqrtfRgRScaled}
\end{equation}
Given $\tilde{r} > 0$, let $P_R(\tilde{r};\mu)$ be the $\tilde{y}$-value of the next intersection
of the forward orbit of $(0,\tilde{r})$ with $\tilde{x} = 0$,
and let $T_R(\tilde{r};\mu)$ be the corresponding evolution time.
Since $\tilde{y}$ has the same definition in \eqref{eq:sqrtScalingL} and \eqref{eq:sqrtScalingR},
we can simply let
\begin{equation}
P(\tilde{y};\mu) = P_L \left( P_R(\tilde{y};\mu); \mu \right),
\label{eq:sqrtP}
\end{equation}
be the Poincar\'{e} map, which corresponds to an evolution time of
\begin{equation}
T(\tilde{r};\mu) = T_R(\tilde{r};\mu) + T_L \left( P_R(\tilde{r};\mu); \mu \right).
\label{eq:sqrtT}
\end{equation}

The system \eqref{eq:sqrtfRgRScaled} is {\em slow-fast} \cite{Fe79,Jo95b,Ku15};
$\tilde{x}$ is the fast variable,
$\tilde{y}$ is the slow variable,
and $\mu$ is the time-scale separation parameter.
Taking $\mu \to 0$ produces the {\em reduced system}
\begin{equation}
\begin{split}
0 &= a_4 \sqrt{\tilde{x}} + a_2 \tilde{y}, \\
\dot{\tilde{y}} &= -\frac{\beta}{a_2} + b_4 \sqrt{\tilde{x}} + b_2 \tilde{y}.
\end{split}
\label{eq:reducedSystem}
\end{equation}
Alternatively on the fast time-scale $\tau = \frac{t}{\mu}$
the limit $\mu \to 0$ produces the {\em layer equations}
\begin{equation}
\begin{split}
\frac{d \tilde{x}}{d \tau} &= a_4 \sqrt{\tilde{x}} + a_2 \tilde{y}, \\
\frac{d \tilde{y}}{d \tau} &= 0.
\end{split}
\label{eq:layerEquations}
\end{equation}

The algebraic constraint in \eqref{eq:reducedSystem} defines the critical manifold $\cM_0$:
$\tilde{y} = -\frac{a_4}{a_2} \,\sqrt{\tilde{x}}$,
which belongs to the first quadrant ($\tilde{x}, \tilde{y} > 0$) because $a_4 < 0 < a_2$,
see Fig.~\ref{fig:schemSqrt}.
By \eqref{eq:reducedSystem}, on $\cM_0$ we have
\begin{equation}
\dot{\tilde{y}} = -\frac{\beta}{a_2} + \frac{\gamma}{a_4} \,\tilde{y},
\label{eq:M0ode}
\end{equation}
which has the flow
\begin{equation}
\psi_t(\tilde{y}) = \re^{\frac{\gamma t}{a_4}}
\left( \tilde{y} - \frac{a_4 \beta}{a_2 \gamma} \right) +
\frac{a_4 \beta}{a_2 \gamma}.
\label{eq:M0soln}
\end{equation}

\begin{figure}[b!]
\begin{center}
\includegraphics[width=5.6cm]{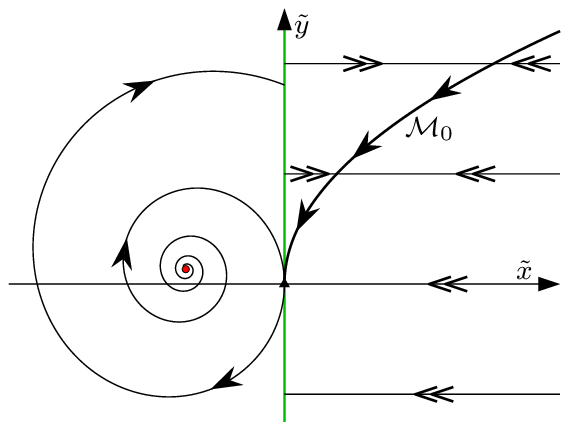}
\caption{
For HLB 20, a sketch of the local dynamics in the limit $\mu \to 0^+$.
Note, $\tilde{x}$ is defined by \eqref{eq:sqrtScalingL} for the left half-system
and by \eqref{eq:sqrtScalingR} for the right half-system.
\label{fig:schemSqrt}
} 
\end{center}
\end{figure}

With small $\mu > 0$ and given $\tilde{r} > 0$, the forward orbit of $(0,\tilde{r})$ approaches $\cM_0$
on the fast time-scale, then evolves near $\cM_0$ on the slow time-scale
until intersecting $\tilde{x} = 0$ near the origin.
In the limit $\mu \to 0$, the orbit arrives at the origin,
thus $\lim_{\mu \to 0} P_R(\tilde{r};\mu) = 0$,
with evolution time $T_R$ satisfying $\psi_{T_R}(\tilde{r}) = 0$.
From \eqref{eq:M0soln} we obtain
\begin{equation}
\lim_{\mu \to 0} T_R(\tilde{r};\mu) = \kappa \ln \left( 1 + \frac{a_2}{\kappa \beta} \,\tilde{r} \right).
\label{eq:sqrtTR0}
\end{equation}
Then
\begin{align}
\lim_{\mu \to 0} P(\tilde{r};\mu)
&= P_L(0;0) \nonumber \\
&= -P_{\rm affine} \left( 0;\lambda,\omega,\frac{\beta}{a_2} \right), \nonumber \\
&= \frac{\beta}{a_2 \omega} \,\re^{\nu \hat{s}(\nu)} \sin \left( \hat{s}(\nu) \right),
\label{eq:sqrtP0}
\end{align}
where $\nu = \frac{\lambda}{\omega}$, see Lemma \ref{le:affine}.
Also by Lemma \ref{le:affine},
\begin{equation}
T_L(0;0) = \frac{\hat{s}(\nu)}{\omega}.
\label{eq:sqrtT0}
\end{equation}
Thus in this limit $P$ has the unique super-stable fixed point \eqref{eq:sqrtP0},
and \eqref{eq:sqrtTR0}--\eqref{eq:sqrtT0} explain the limiting period \eqref{eq:periodsqrt}.

$P_R(\tilde{r};\mu)$ is $C^1$ because we can perform
a multiple time-scale analysis \cite{Ku15} on \eqref{eq:sqrtfRgRScaled}
to calculate its next order contribution.
It follows that $P(\tilde{r};\mu)$ is $C^1$
and thus has a unique stable fixed point $\tilde{r}^*(\mu)$
(equal to \eqref{eq:sqrtP0} in the limit $\mu \to 0$) by the IFT.
Hence, in a neighbourhood of the origin,
\eqref{eq:sqrtODE} has a unique stable limit cycle for sufficiently small $\mu > 0$.
The order of its minimum and maximum $x$ and $y$-values
are explained by the scalings \eqref{eq:sqrtScalingL} and \eqref{eq:sqrtScalingR}.

Second consider $\mu < 0$.
With the scaling
\begin{align}
\tilde{x} &= \frac{x}{\mu^2}, &
\tilde{y} &= \frac{y - \zeta(\mu)}{|\mu|},
\label{eq:sqrtScalingR2}
\end{align}
the right half-system becomes
\begin{equation}
\begin{split}
|\mu| \dot{\tilde{x}} &= a_4 \sqrt{\tilde{x}} + a_2 \tilde{y} + \cO(\mu), \\
\dot{\tilde{y}} &= \frac{\beta}{a_2} + b_4 \sqrt{\tilde{x}} + b_2 \tilde{y} + \cO(\mu).
\end{split}
\label{eq:sqrtfRgRScaled2}
\end{equation}
The system \eqref{eq:sqrtfRgRScaled2}
has a unique equilibrium: a stable node
$\left( \tilde{x}_R^*, \tilde{y}_R^* \right)
= \left( \frac{\beta^2}{\gamma^2}, -\frac{a_4 \beta}{a_2 \gamma} \right) + \cO(\mu)$.
It is straight-forward to see that for any $\tilde{r} > 0$,
the forward orbit of $\left( \tilde{x}, \tilde{y} \right) = \left( 0, \tilde{r} \right)$
rapidly approaches a slow manifold, then slowly approaches $\left( \tilde{x}_R^*, \tilde{y}_R^* \right)$.
Consequently, locally, \eqref{eq:sqrtODE} has no limit cycle when $\mu < 0$.
\end{proof}

\section{Proofs for slipping foci and folds}
\setcounter{equation}{0}
\setcounter{figure}{0}
\setcounter{table}{0}
\label{app:slipping}

\begin{proof}[Proof of Theorem \ref{th:slippingFocusFocus} (HLB 5)]
We first compute regular equilibria and sliding motion to verify part (i) of Theorem \ref{th:slippingFocusFocus},
and then construct and analyse a Poincar\'e map to verify part (ii).

Write
\begin{equation}
\begin{split}
f_J(x,y;\mu) &= a_{1J} x + a_{2J} y + a_{3J} \mu + \co \left( |x| + |y| + |\mu| \right), \\
g_J(x,y;\mu) &= b_{1J} x + b_{2J} y + b_{3J} \mu + \co \left( |x| + |y| + |\mu| \right),
\end{split}
\label{eq:slippingFocusFocusfJgJ}
\end{equation}
for each $J \in \{ L,R \}$.
Notice $a_{2L} > 0$, $a_{2R} > 0$, and
\begin{equation}
\gamma = a_{2L} b_{2R} - a_{2R} b_{2L} \,.
\label{eq:slippingFocusFocusgamma2}
\end{equation}

The boundary equilibria $(0,\xi_L(\mu))$ and $(0,\xi_R(\mu))$
are clockwise-rotating foci because $a_{2L} > 0$ and $a_{2R} > 0$.
Thus when $\xi_L(\mu) < \xi_R(\mu)$, as is the case for $\mu < 0$ because $\beta > 0$ by assumption,
$\Gamma$ is an attracting sliding region, see Fig.~\ref{fig:schemSlippingFocusFocus2}.
Conversely when $\xi_L(\mu) < \xi_R(\mu)$, as is the case for $\mu > 0$,
$\Gamma$ is a repelling sliding region.
By evaluating \eqref{eq:gslide2} we find we can write the sliding vector field as
\begin{equation}
g_{\rm slide}(y;\mu) = \frac{\left( y - \xi_L(\mu) \right) \left( y - \xi_R(\mu) \right) h(y;\mu)}
{f_L(0,y;\mu) - f_R(0,y;\mu)},
\label{eq:slippingFocusFocusydotslide2}
\end{equation}
where $\lim_{(y;\mu) \to (0,0)} h(y;\mu) = \gamma$.
With $\mu < 0$, in $\Gamma$ we have
$f_L > 0$, $f_R < 0$, and $\xi_L < y < \xi_R$.
Thus for sufficiently small values of $y$ and $\mu$,
if $\gamma < 0$ then $g_{\rm slide} > 0$ and so $(0,\xi_R(\mu))$ is stable,
while if $\gamma > 0$ then $g_{\rm slide} < 0$ and $(0,\xi_L(\mu))$ is stable.
This verifies part (i).

Now given $r > \xi_R(\mu)$, 
consider the forward orbit of $(x,y) = (0,r)$ that immediately enters $x > 0$.
Let $P_R(r;\mu)$ be the $y$-value of the next intersection of
this orbit with $x = 0$, and let $T_R(r;\mu)$ be the corresponding evolution time.
Similarly given $r < \xi_L(\mu)$,
consider the forward orbit of $(x,y) = (0,r)$ that immediately enters $x < 0$.
Let $P_L(r;\mu)$ be the $y$-value of the next intersection of
this orbit with $x = 0$, and let $T_L(r;\mu)$ be the corresponding evolution time.

Via a change of coordinates that translates $(0,\xi_R(\mu))$ to the origin,
Lemma \ref{le:focus} gives
\begin{equation}
P_R(r;\mu) = \xi_R(\mu) - \re^{\frac{\lambda_R \pi}{\omega_R}} \left( r - \xi_R(\mu) \right) + \co \left( |r| + |\mu| \right),
\label{eq:slippingFocusFocusPR}
\end{equation}
and $T_R(r;\mu) \to \frac{\pi}{\omega_R}$ as $r \to 0$
(here the ODEs are only $C^1$ so we only get the leading order terms of \eqref{eq:focusP}--\eqref{eq:focusT}).
Similarly
\begin{equation}
P_L(r;\mu) = \xi_L(\mu) - \re^{\frac{\lambda_L \pi}{\omega_L}} \left( r - \xi_L(\mu) \right) + \co \left( |r| + |\mu| \right),
\label{eq:slippingFocusFocusPL}
\end{equation}
and $T_L(r;\mu) \to \frac{\pi}{\omega_L}$ as $r \to 0$.

To accommodate the case $\mu < 0$, let
$\check{r}(\mu) = P_R^{-1}(\xi_L(\mu);\mu) = \xi_R(\mu) - \beta \,\re^{\frac{-\lambda_R \pi}{\omega_R}} \mu + \co(\mu)$.
Then $P = P_L \circ P_R$ is well-defined and $C^1$ (by Lemma \ref{le:mapSmoothness})
for small $r > \check{r}(\mu)$ in the case $\mu < 0$,
and small $r > \xi_R(\mu)$ in the case $\mu > 0$.
By composing \eqref{eq:slippingFocusFocusPR} and \eqref{eq:slippingFocusFocusPL} we obtain
\begin{align}
P(r;\mu) &= \xi_R(\mu) + \left( 1 + \re^{\frac{\lambda_L \pi}{\omega_L}} \right) \beta \mu \nonumber \\
&\quad+ \re^{\alpha \pi} \left( r - \xi_R(\mu) \right) + \co \left( |r| + |\mu| \right).
\label{eq:slippingFocusFocusP}
\end{align}
The Poincar\'e map $P$ is not defined for all $(r;\mu)$ near $(0;0)$,
but in view of \eqref{eq:slippingFocusFocusP} we can extend $P$ so that it is $C^1$
throughout a neighbourhood of $(r;\mu) = (0;0)$.
This allows to formally apply the IFT to solve $P(r;\mu) = r$,
and from this we obtain the unique fixed point
\begin{equation}
r^*(\mu) = \xi_R(\mu) + \frac{\Big( 1 + \re^{\frac{\lambda_L \pi}{\omega_L}} \Big) \beta}
{1 - \re^{\alpha \pi}} \,\mu + \co(\mu),
\label{eq:slippingFocusFocusyStar}
\end{equation}
assuming $\alpha \ne 0$.

First suppose $\alpha < 0$.
If $\mu < 0$ then $r^*(\mu) < \xi_R(\mu)$, which is outside the domain of $P$.
Hence $P$ has no fixed points and, locally, \eqref{eq:FilippovBEB} has no limit cycle.
If instead $\mu > 0$ then $r^*(\mu) > \xi_R(\mu)$, so $P$ has a unique fixed point
and \eqref{eq:FilippovBEB} has a unique limit cycle.
The limit cycle is stable because
$\lim_{\mu \to 0} \frac{\partial P}{\partial r}\left( r^*(\mu); \mu \right) = \re^{\alpha \pi} \in (0,1)$.
The limit cycle intersects $x=0$ at $r^*(\mu)$ and
\begin{equation}
P_R \left( r^*(\mu); \mu \right) = \xi_L(\mu)
- \frac{\Big( 1 + \re^{\frac{\lambda_R \pi}{\omega_R}} \Big) \beta}
{1 - \re^{\alpha \pi}} \,\mu + \co(\mu),
\nonumber
\end{equation}
and consequently its minimum and maximum $x$ and $y$-values are asymptotically proportional to $|\mu|$.
Its period is $T_L \left( P_R \left( r^*(\mu); \mu \right); \mu \right) + T_R \left( r^*(\mu); \mu \right)$
which limits to $\frac{\pi}{\omega_L} + \frac{\pi}{\omega_R}$ as $\mu \to 0$.
 
Finally suppose $\alpha > 0$.
If $\mu < 0$ then \eqref{eq:slippingFocusFocusyStar} implies $r^*(\mu) > \check{r}(\mu)$.
Thus $P$ has a unique fixed point and \eqref{eq:FilippovBEB} has a unique limit cycle.
This limit cycle is unstable because $\re^{\alpha \pi} > 1$
and the above properties hold for its amplitude and period.
If instead $\mu > 0$ then $r^*(\mu) < \xi_R(\mu)$ and so \eqref{eq:FilippovBEB} has no limit cycle.
\end{proof}

\begin{proof}[Proof of Theorem \ref{th:slippingFocusFold} (HLB 6)]
Here we write
\begin{equation}
\begin{split}
f_L(x,y;\mu) &= a_{1L} x + a_{2L} y + \co \left( |x| + |y| + |\mu| \right), \\
g_L(x,y;\mu) &= b_{1L} x + b_{2L} y + \co \left( |x| + |y| + |\mu| \right),
\end{split}
\label{eq:slippingFocusFoldfLgL}
\end{equation}
where $a_{2L} > 0$, and
\begin{equation}
\begin{split}
f_R(x,y;\mu) &= a_{1R} x + a_{2R} y + a_{3R} \mu \\
&\quad+ \co \left( |x| + |y| + |\mu| \right), \\
g_R(x,y;\mu) &= b_{0R} + b_{1R} x + b_{2R} y + b_{3R} \mu \\
&\quad+ \co \left( |x| + |y| + |\mu| \right),
\end{split}
\label{eq:slippingFocusFoldfRgR}
\end{equation}
where $a_{2R} > 0$, $a_{3R} = \beta$, and $b_{0R} < 0$.

The $y$-value of the fold is $\zeta_R(\mu) = -\frac{\beta}{a_{2R}} \,\mu + \co(\mu)$.
Since $\beta > 0$, for $\mu < 0$ we have $\zeta_R(\mu) > 0$
and consequently $\Gamma$ is an attracting sliding region.
Conversely for $\mu > 0$ we have $\zeta_R(\mu) < 0$
and $\Gamma$ is a repelling sliding region.
Through \eqref{eq:gslide2} we obtain
\begin{equation}
g_{\rm slide}(y;\mu) = \frac{y h(y;\mu)}{f_L(0,y;\mu) - f_R(0,y;\mu)},
\label{eq:slippingFocusFoldydotslide}
\end{equation}
where $\lim_{(y;\mu) \to (0,0)} h(y;\mu) = a_{2L} b_{0R}$.
Thus with $\mu < 0$, on $\Gamma$ we have $y > 0$
and so $g_{\rm slide}(y;\mu) < 0$ (because $a_{2L} b_{0R} < 0$).
Therefore sliding orbits approach the origin and so the origin is a stable stationary solution.
With instead $\mu > 0$ the origin is unstable because $f_R(0,0;\mu) > 0$.

Define the return maps $P_R$ and $P_L$ and return times $T_R$ and $T_L$
as in the proof of Theorem \ref{th:slippingFocusFocus}, see Fig.~\ref{fig:schemSlippingFocusFold}.
Via a change of coordinates that translates $(0,\zeta_R(\mu))$ to the origin,
Lemma \ref{le:fold} gives
\begin{equation}
P_R(r;\mu) = 2 \zeta_R(\mu) - r + \co \left( |r| + |\mu| \right),
\label{eq:slippingFocusFoldPR}
\end{equation}
and $T_R(r;\mu) = \cO(1)$.
Also Lemma \ref{le:focus} gives
\begin{equation}
P_L(r;\mu) = -\re^{\frac{\lambda_L \pi}{\omega_L}} \,r + \co \left( |r| + |\mu| \right),
\label{eq:slippingFocusFoldPL}
\end{equation}
and $T_L(r;\mu) \to \frac{\pi}{\omega_L}$ as $r \to 0$.

\begin{figure}[t!]
\begin{center}
\includegraphics[width=8.4cm]{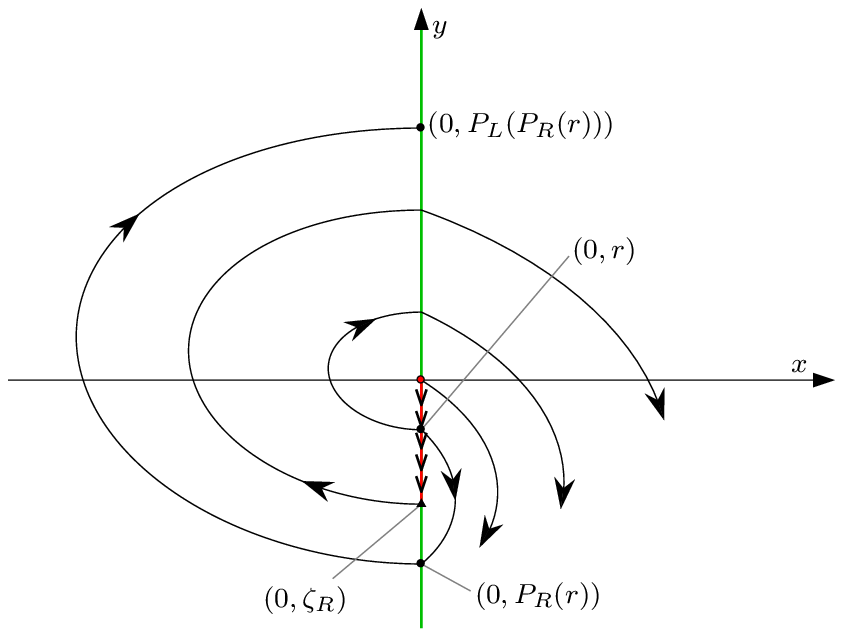}
\caption{
A phase portrait illustrating the Poincar\'e map $P = P_L \circ P_R$ for HLB 6.
\label{fig:schemSlippingFocusFold}
} 
\end{center}
\end{figure}

Let $\check{r}(\mu) = P_R^{-1}(0;\mu) = 2 \zeta_R(\mu) + \co(\mu)$.
The Poincar\'e map $P = P_L \circ P_R$ is well-defined and $C^1$ for $r > \check{r}(\mu)$ in the case $\mu < 0$,
and for $r > \zeta_R(\mu)$ in the case $\mu > 0$.
By composing \eqref{eq:slippingFocusFoldPR} and \eqref{eq:slippingFocusFoldPL} we obtain
\begin{equation}
P(r;\mu) = \re^{\frac{\alpha \pi}{\omega_L}} \left( r + \frac{2 \beta}{a_{2R}} \,\mu \right) + \co \left( |r| + |\mu| \right).
\label{eq:slippingFocusFoldP}
\end{equation}
By using the IFT to solve for fixed points of $P$
(as in the proof of Theorem \ref{th:slippingFocusFocus}) we obtain the locally unique value
\begin{equation}
r^*(\mu) = \frac{2 \beta}{a_{2R} \left( \re^{\frac{-\alpha \pi}{\omega_L}} - 1 \right)} \,\mu + \co(\mu),
\label{eq:slippingFocusFoldyStar}
\end{equation}
assuming $\alpha \ne 0$.
The proof is completed by determining the admissibility and stability $r^*(\mu)$
as in the proof of Theorem \ref{th:slippingFocusFocus} (we omit these details for brevity).
Note that the maximum $x$-value attained by the orbit segment associated with $P_R(r;\mu)$
is asymptotically proportional to $\left( r - \zeta_R(\mu) \right)^2$
and for this reason the maximum $x$-value of the limit cycle is
asymptotically proportional to $\mu^2$.
\end{proof}

\begin{proof}[Proof of Theorem \ref{th:slippingTwoFold} (HLB 7)]
We write
\begin{equation}
\begin{split}
f_J(x,y;\mu) &= a_{1J} x + a_{2J} y + a_{3J} \mu \\
&\quad+ \cO \left( \left( |x| + |y| + |\mu| \right)^2 \right), \\
g_J(x,y;\mu) &= b_{0J} + b_{1J} x + b_{2J} y + b_{3J} \mu \\
&\quad+ \cO \left( \left( |x| + |y| + |\mu| \right)^2 \right),
\end{split}
\label{eq:slippingTwoFoldfJgJ}
\end{equation}
where
\begin{align}
a_{2L} &> 0, &
a_{2R} &> 0, &
b_{0L} &> 0, &
b_{0R} &< 0.
\label{eq:slippingTwoFoldSignConditions}
\end{align}
By the IFT, locally we have
$f_J(0,\zeta_J(\mu);\mu) = 0$ for a unique $C^2$ function
$\zeta_J(\mu) = -\frac{a_{3J}}{a_{2J}} \,\mu + \cO \left( \mu^2 \right)$.
The folds $(0,\zeta_L(\mu))$ and $(0,\zeta_R(\mu))$ are the endpoints of $\Gamma$.
Notice
\begin{equation}
\zeta_L(\mu) - \zeta_R(\mu) = \beta \mu + \cO \left( \mu^2 \right).
\label{eq:slippingTwoFoldxiLminusxiR}
\end{equation}
Since $\beta > 0$, for $\mu < 0$ we have $\zeta_L(\mu) < \zeta_R(\mu)$
and so $\Gamma$ is an attracting sliding region.
With instead $\mu > 0$ we have $\zeta_L(\mu) > \zeta_R(\mu)$
and $\Gamma$ is a repelling sliding region.

From \eqref{eq:gslide2} we obtain
\begin{widetext}
\begin{equation}
g_{\rm slide}(y;\mu) =
\frac{-\left( a_{2R} b_{0L} - a_{2L} b_{0R} \right) y
+ \left( a_{3L} b_{0R} - a_{3R} b_{0L} \right) \mu + \cO \left( \left( |y| + |\mu| \right)^2 \right)}
{f_L(0,y;\mu) - f_R(0,y;\mu)}.
\label{eq:slippingTwoFoldydotslide}
\end{equation}
\end{widetext}
Solving $\dot{y}_{\rm slide}(y;\mu) = 0$ reveals (by the IFT)
a unique pseudo-equilibrium at $\left( 0, y_{\rm eq}(\mu) \right)$ where
\begin{equation}
y_{\rm eq}(\mu) = \frac{a_{3L} b_{0R} - a_{3R} b_{0L}}{a_{2R} b_{0L} - a_{2L} b_{0R}} \,\mu + \cO \left( \mu^2 \right).
\label{eq:slippingTwoFoldyEq}
\end{equation}
We can write \eqref{eq:slippingTwoFoldyEq} as the convex combination
\begin{equation}
y_{\rm eq}(\mu) = \left( 1 - s(\mu) \right) \zeta_L(\mu) + s(\mu) \zeta_R(\mu),
\nonumber
\end{equation}
where
\begin{equation}
s(\mu) = \frac{a_{2R} b_{0L}}{a_{2R} b_{0L} - a_{2L} b_{0R}}.
\nonumber
\end{equation}
Notice \eqref{eq:slippingTwoFoldSignConditions} implies $s(\mu) \in (0,1)$,
thus $\left( 0, y_{\rm eq}(\mu) \right)$ lies within the sliding region $\Gamma$.
Therefore $\left( 0, y_{\rm eq}(\mu) \right)$ is admissible and is
the unique stationary solution of \eqref{eq:FilippovBEB}.
The pseudo-equilibrium $\left( 0, y_{\rm eq}(\mu) \right)$ is unstable when $\mu > 0$ because here $\Gamma$ is repelling;
it is stable when $\mu < 0$ because
here $\Gamma$ is attracting and
\begin{equation}
\frac{\partial g_{\rm slide}}{\partial y} \left( y_{\rm eq}(\mu); \mu \right)
= -\frac{a_{2R} b_{0L} - a_{2L} b_{0R}}{f_L \left( 0, y_{\rm eq}(\mu); \mu \right) - f_R \left( 0, y_{\rm eq}(\mu); \mu \right)}
\nonumber
\end{equation}
is negative.

Define the return maps $P_R$ and $P_L$ and return times $T_R$ and $T_L$
as in the proof of Theorem \ref{th:slippingFocusFocus}.
Via a change of coordinates that translates $(0,\zeta_R(\mu))$ to the origin,
Lemma \ref{le:fold} gives
\begin{equation}
P_R(r;\mu) = 2 \zeta_R(\mu) - r + \frac{2 \sigma_R}{3} \left( r - \zeta_R(\mu) \right)^2 + \co \left( \left( |r| + |\mu| \right)^2 \right),
\label{eq:slippingTwoFoldPR}
\end{equation}
where $\sigma_R = \sigma_{\rm fold}[f_R(x,y;0),g_R(x,y;0)]$
(here we are evaluating \eqref{eq:foldsigma} using $f_R(x,y;0)$ and $g_R(x,y;0)$
in place of the functions $f(x,y)$ and $g(x,y)$ of \S\ref{sec:lemmas}).
Also $T_R(r;\mu) = \frac{-2}{b_{0R}} \,r + \cO(2)$.
Similarly
\begin{equation}
P_L(r;\mu) = 2 \zeta_L(\mu) - r + \frac{2 \sigma_L}{3} \left( r - \zeta_L(\mu) \right)^2 + \co \left( \left( |r| + |\mu| \right)^2 \right),
\label{eq:slippingTwoFoldPL}
\end{equation}
where $\sigma_L = \sigma_{\rm fold}[f_L(x,y;0),g_L(x,y;0)]$
and $T_L(r;\mu) = \frac{-2}{b_{0L}} \,r + \cO(2)$.

Locally the Poincar\'e map $P = P_L \circ P_R$
is well-defined and $C^2$ for $r > \check{r}(\mu) = P_R^{-1}(\zeta_L(\mu);\mu)$ in the case $\mu < 0$
and for $r > \zeta_R(\mu)$ in the case $\mu > 0$.
By composing \eqref{eq:slippingTwoFoldPR} and \eqref{eq:slippingTwoFoldPL} we obtain
\begin{equation}
P(r;\mu) = r + 2 \beta \mu + \frac{2 \alpha}{3} \,r^2 + \kappa_1 \mu r + \kappa_2 \mu^2 + \co \left( \left( |r| + |\mu| \right)^2 \right),
\label{eq:slippingTwoFoldP}
\end{equation}
where we have substituted \eqref{eq:slippingTwoFoldxiLminusxiR} and $\alpha = \sigma_L - \sigma_R$.
Formulas for $\kappa_1, \kappa_2 \in \mathbb{R}$ will not be needed.
We can extend $P$ so that it is $C^2$ throughout a neighbourhood of $(r;\mu) = (0;0)$,
then apply the IFT to solve $P(r;\mu) = r$, not for $r$,
but for $\mu$ (this is possible because $\frac{\partial P}{\partial \mu}(0;0) = 2 \beta \ne 0$).
This yields a unique $C^2$ function $\mu = h(r) = -\frac{\alpha}{3 \beta} \,r^2 + \co \left( r^2 \right)$.

Now suppose $\alpha < 0$.
Then $h(r) \ge 0$, so if $\mu < 0$ then $P$ has no fixed point and \eqref{eq:FilippovBEB} has no limit cycle,
while if $\mu > 0$ then $P$ has a unique fixed point
\begin{equation}
r^*(\mu) = \sqrt{\frac{-3 \beta \mu}{\alpha}} + \co \left( \sqrt{\mu} \right),
\label{eq:slippingTwoFoldyStar}
\end{equation}
and \eqref{eq:FilippovBEB} has a unique limit cycle.
This limit cycle is stable because
$\frac{\partial P}{\partial r} \left( r^*(\mu); \mu \right) = 1 + \frac{2 \alpha r^*(\mu)}{3} + \cO(\mu)$
has modulus less than $1$.
Its maximum $y$-value is $r^*(\mu)$,
while its maximum $x$-value is asymptotically proportional to $\mu$ because
the maximum $x$-value attained by the orbit segment associated with $P_R(r;\mu)$
is asymptotically proportional to $\left( r - \zeta_R(\mu) \right)^2$.
The minimum $x$ and $y$-values of the limit cycle evidently have the same asymptotic behaviour.
The period of the limit cycle is
$T_L \left( P_R \left( r^*(\mu); \mu \right); \mu \right) + T_R \left( r^*(\mu); \mu \right)
= \left( \frac{2}{b_{0L}} - \frac{2}{b_{0R}} \right) r^*(\mu) + \cO(\mu)$
which verifies \eqref{eq:slippingTwoFoldPeriod}.
The result in the case $\alpha > 0$ follows similarly.
\end{proof}

\section{Proofs for fixed foci and folds}
\setcounter{equation}{0}
\setcounter{figure}{0}
\setcounter{table}{0}
\label{app:fixed}

For the proofs of Theorems \ref{th:fixedFocusFocus}--\ref{th:fixedTwoFold} we can assume without loss of generality
that both half-systems involve clockwise rotation and use the following definition for the components
of the Poincar\'e map $P = P_L \circ P_R$, see Fig.~\ref{fig:schemPoinComposition}.
Given $r > 0$, let $P_R(r;\mu)$ denote the $y$-value at which
the forward orbit of $(x,y) = (0,r)$ next intersects $x = 0$, and let $T_R(r;\mu)$ be the corresponding evolution time.
Similarly given $r < 0$, let $P_L(r;\mu)$ denote the $y$-value at which
the forward orbit of $(x,y) = (0,r)$ next intersects $x = 0$, and let $T_L(r;\mu)$ be the corresponding evolution time.

\begin{proof}[Proof of Theorem \ref{th:fixedFocusFocus} (HLB 8)]
From Lemma \ref{le:focus} we obtain
\begin{align}
P_R(r;\mu) &= -\re^{\frac{\lambda_R \pi}{\omega_R}} r
+ \re^{\frac{\lambda_R \pi}{\omega_R}} \left( \re^{\frac{\lambda_R \pi}{\omega_R}} + 1 \right) \chi_R r^2 + \co \left( r^2 \right),
\nonumber \\
P_L(r;\mu) &= -\re^{\frac{\lambda_L \pi}{\omega_L}} r
+ \re^{\frac{\lambda_L \pi}{\omega_L}} \left( \re^{\frac{\lambda_L \pi}{\omega_L}} + 1 \right) \chi_L r^2 + \co \left( r^2 \right),
\nonumber
\end{align}
where we have omitted `focus' from the subscripts of $\chi$ for brevity.
By composing these
and substituting $\Lambda(\mu) = \beta \mu + \cO \left( \mu^2 \right)$
and $\chi_L(0) - \chi_R(0) = \alpha$, we obtain
\begin{equation}
P(r;\mu) = r + \beta \pi \mu r + \alpha \left( \re^{\frac{\lambda_R \pi}{\omega_R}} + 1 \right) r^2
+ \co \left( \left( |r| + |\mu| \right)^2 \right).
\label{eq:fixedFocusFocusP3}
\end{equation}
By construction, $r=0$ is a fixed point of $P$ for all sufficiently small values of $\mu$.
This corresponds to the boundary equilibrium at the origin.
We have $\frac{\partial P}{\partial r}(0;\mu) = 1 + \beta \pi \mu + \co(\mu)$,
thus if $\mu < 0$ then $\frac{\partial P}{\partial r}(0;\mu) < 1$ in a neighbourhood of $r=0$ and so the origin is stable,
while if $\mu > 0$ then $\frac{\partial P}{\partial r}(0;\mu) > 1$ in a neighbourhood of $r=0$ and the origin is unstable.

Since $P$ is $C^2$ and $P(0;\mu) = 0$ for all $\mu$,
the function $D(r;\mu) = \frac{1}{r} \left( P(r;\mu) - r \right)$ is $C^1$.
Moreover we can extend its domain to a neighbourhood of $(r;\mu) = (0;0)$.
This allows us to use the IFT to solve $D(r;\mu) = 0$ for $r$
(because $\frac{\partial D}{\partial r}(0;0) \ne 0$ assuming $\alpha \ne 0$) to obtain
\begin{equation}
r^*(\mu) = \frac{-\beta \pi}{\alpha \Big( \re^{\frac{\lambda_R \pi}{\omega_R}} + 1 \Big)} \,\mu + \co(\mu).
\label{eq:fixedFocusFocusyStar}
\end{equation}
This is valid (i.e.~$r^*(\mu) > 0$) when $\alpha \mu < 0$ and corresponds to a limit cycle of \eqref{eq:FilippovBEB}.
Thus, locally, \eqref{eq:FilippovBEB} has no limit cycle if $\alpha \mu > 0$
and a unique limit cycle if $\alpha \mu < 0$.
Since $\frac{\partial P}{\partial r} \left( r^*(\mu);\mu \right) = 1 - \beta \pi \mu + \co(\mu)$,
the stability of the limit cycle is opposite to that of the origin.
The statements regarding the size of the limit cycle follow from \eqref{eq:fixedFocusFocusyStar}.
Finally, Lemma \ref{le:focus} also gives
$T_R(r;\mu) = \frac{\pi}{\omega_R(\mu)} + \cO(r)$
and $T_L(r;\mu) = \frac{\pi}{\omega_L(\mu)} + \cO(r)$
which explain \eqref{eq:fixedFocusFocusPeriod}.
\end{proof}

\begin{proof}[Proof of Theorem \ref{th:fixedFocusFold} (HLB 9)]
From Lemmas \ref{le:focus} and \ref{le:fold} we obtain
\begin{align}
P_R(r;\mu) &= -r + \frac{2 \sigma_{{\rm fold},R}}{3} \,r^2 + \co \left( r^2 \right), \label{eq:fixedFocusFoldPR} \\
P_L(r;\mu) &= -\re^{\frac{\lambda_L \pi}{\omega_L}} r
+ \re^{\frac{\lambda_L \pi}{\omega_L}} \left( \re^{\frac{\lambda_L \pi}{\omega_L}} + 1 \right) \chi_{{\rm focus},L} \,r^2 + \co \left( r^2 \right),
\label{eq:fixedFocusFoldPL}
\end{align}
Since $\lambda_L(\mu) = \beta \mu + \co(\mu)$,
we can replace each instance of $\re^{\frac{\lambda_L \pi}{\omega_L}}$ in \eqref{eq:fixedFocusFoldPL} with
$1 + \frac{\beta \pi}{\omega_L(0)} \,\mu + \co(\mu)$.
Then by composing \eqref{eq:fixedFocusFoldPR} and \eqref{eq:fixedFocusFoldPL} we obtain
\begin{equation}
P(r;\mu) = r + \frac{\beta \pi}{\omega_L(0)} \,\mu r + 2 \alpha r^2 + \co \left( \left( |r| + |\mu| \right)^2 \right),
\label{eq:fixedFocusFoldP3}
\end{equation}
where we have also used $\chi_{{\rm focus},L}(0) - \frac{\sigma_{{\rm fold},R}(0)}{3} = \alpha$.

As in the proof of Theorem \ref{th:fixedFocusFocus},
$P$ has at most two fixed points: $r = 0$, corresponding to the origin, and
\begin{equation}
r^*(\mu) = \frac{-\beta \pi}{2 \alpha \omega_L(0)} \,\mu + \co(\mu),
\label{eq:fixedFocusFoldyStar}
\end{equation}
which is valid when $\alpha \mu < 0$ and corresponds to a limit cycle.
The stability of the origin and limit cycle can be determined from
the sign of $\frac{\partial P}{\partial r} - 1$, as in the previous proof.
The extremal $x$ and $y$ values of the limit cycle are asymptotically proportional to $|\mu|$ by \eqref{eq:fixedFocusFoldyStar},
except the maximum $x$-value is asymptotically proportional to $\mu^2$
due to the fold of the right half-system.
Finally, Lemmas \ref{le:focus} and \ref{le:fold} also give
$T_R(r;\mu) = \frac{-2}{g_R(0,0;\mu)} \,r + \cO \left( r^2 \right)$
and $T_L(r;\mu) = \frac{\pi}{\omega_L(\mu)} + \cO(r)$,
and the sum of these produces \eqref{eq:fixedFocusFoldPeriod}.
\end{proof}

\begin{proof}[Proof of Theorem \ref{th:fixedTwoFold} (HLB 10)]
Lemma \ref{le:fold} gives
\begin{align}
P_R(r;\mu) &= -r + \frac{2 \sigma_R}{3} \,r^2 - \frac{4 \sigma_R^2}{9} r^3
+ \frac{2 \chi_R}{15} r^4 + \co \left( r^4 \right), \nonumber \\
P_L(r;\mu) &= -r + \frac{2 \sigma_L}{3} \,r^2 - \frac{4 \sigma_L^2}{9} r^3
+ \frac{2 \chi_L}{15} r^4 + \co \left( r^4 \right), \nonumber
\end{align}
where we have omitted `fold' from the subscripts of $\sigma$ and $\chi$ for brevity.
By composing these we obtain
\begin{widetext}
\begin{equation}
P(r;\mu) = r + \frac{2}{3}(\sigma_L - \sigma_R) r^2 + \frac{4}{9}(\sigma_L - \sigma_R)^2 r^3
+ \left( \frac{2}{15}(\chi_L - \chi_R) - \frac{8}{9} \sigma_L \sigma_R (\sigma_L - \sigma_R) \right) r^4
+ \co \left( r^4 \right).
\label{eq:fixedTwoFoldP2}
\end{equation}
\end{widetext}
By then substituting
$\sigma_L(\mu) - \sigma_R(\mu) = \beta \mu + \cO \left( \mu^2 \right)$
and $\chi_L(0) - \chi_R(0) = \alpha$ we arrive at
\begin{equation}
P(r;\mu) = r + \left( \frac{2 \beta \mu}{3} + \cO \left( \mu^2 \right) \right) r^2
+ \frac{2 \alpha}{15} r^4 + \co \left( \left( |r| + |\mu| \right)^4 \right).
\label{eq:fixedTwoFoldP3}
\end{equation}
This map has the fixed point $r=0$ (corresponding to the two-fold at the origin)
which from $\frac{\partial P}{\partial r}$ we see is stable for $\mu < 0$ and unstable for $\mu > 0$ (because $\beta > 0$).
To formally obtain a second fixed point we first observe that
$P$ is $C^4$ for small $r > 0$ because $F_L$ and $F_R$ are $C^4$ (see Lemma \ref{le:mapSmoothness}).
It is a simple exercise to show that $D(r;\mu) = \frac{1}{r^2} \left( P(r;\mu) - r \right)$ is $C^2$.
Moreover, $D$ can be extended in a $C^2$ fashion to a neighbourhood of $(r;\mu) = (0,0)$.
We can then use the IFT to solve $D(r;\mu) = 0$ for $\mu$
(possible because $\frac{\partial D}{\partial \mu} = \frac{2 \beta}{3} \ne 0$),
and by inverting the result we obtain the fixed point
\begin{equation}
r^*(\mu) = \sqrt{\frac{-5 \beta \mu}{\alpha}} + \co \left( \sqrt{|\mu|} \right).
\label{eq:fixedTwoFoldyStar}
\end{equation}
Under the restriction $r > 0$, this point exists and is unique
when $\alpha \mu < 0$ and corresponds to a limit cycle of \eqref{eq:FilippovBEB}.
The period of the limit cycle is
$T_L \left( P_R \left( r^*(\mu); \mu \right); \mu \right) + T_R \left( r^*(\mu); \mu \right)$,
where $T_R(r;\mu) = \frac{-2}{g_R(0,0;\mu)} \,r + \cO \left( r^2 \right)$
and $T_L(r;\mu) = \frac{-2}{g_L(0,0;\mu)} \,r + \cO \left( r^2 \right)$
by Lemma \ref{le:fold}, and these combine to produce \eqref{eq:fixedTwoFoldPeriod}.
\end{proof}

\section{Proofs for impacting systems}
\setcounter{equation}{0}
\setcounter{figure}{0}
\setcounter{table}{0}
\label{app:impact}

We may write the vector field of \eqref{eq:impactingODE}
subject to \eqref{eq:FLimpactingCond} and \eqref{eq:impactingEqCond} as
\begin{equation}
\begin{split}
f(x,y;\mu) &= a_1 x + a_2 y + \co \left( |x| + |y| + |\mu| \right), \\
g(x,y;\mu) &= b_1 x + b_2 y + b_3 \mu + \co \left( |x| + |y| + |\mu| \right).
\end{split}
\label{eq:impactingfg}
\end{equation}
Notice $a_2 > 0$ and $\beta = -a_2 b_3$.
Thus $\beta > 0$ implies $b_3 < 0$.
Let
\begin{equation}
A = \rD F(0,0;0) = \begin{bmatrix} a_1 & a_2 \\ b_1 & b_2 \end{bmatrix}.
\nonumber
\end{equation}
Assuming $\det(A) \ne 0$, locally there exists a unique regular equilibrium
$\left( x^*(\mu), y^*(\mu) \right)$ with
\begin{equation}
x^*(\mu) = \frac{-\beta}{\det(A)} \,\mu + \co(\mu).
\label{eq:impactingxStar}
\end{equation}

\begin{proof}[Proof of Lemma \ref{le:impacting}]
By \eqref{eq:impactingxStar} we have $\left( x^*(\mu), y^*(\mu) \right) \in \Omega_L$ when $\mu > 0$.
Since $\det(A) > 0$, if $\lambda < 0$ then $\left( x^*(\mu), y^*(\mu) \right)$ is a stable node or focus,
while if $\lambda > 0$ then $\left( x^*(\mu), y^*(\mu) \right)$ is a unstable node or focus.

Now suppose $\mu < 0$ and the consider the forward orbit of $(x,y) = (0,r)$ where $r > 0$.
This point is mapped to $(0,-\gamma r + \co(r))$,
then, since the origin is an invisible fold,
evolves in $\Omega_L$ and next intersects $x=0$ at
$y = \gamma r + \co(r)$ by Lemma \ref{le:fold}.
Thus the sign of $\ln(\gamma)$ determines the stability of the pseudo-equilibrium.
\end{proof}

\begin{proof}[Proof of Theorems \ref{th:impactingA}--\ref{th:impactingC} (HLBs 11--13)]
Here we derive and analyse two Poincar\'e maps, one for $\mu < 0$ and one for $\mu > 0$,
assuming $\beta > 0$ and $\lambda \ln(\gamma) < 0$.
From the fixed points of these maps, subject to different assumptions on the values of $\alpha$ and $\omega$,
we are able to verify all statements in Theorems \ref{th:impactingA}--\ref{th:impactingC}.
In view of the substitution $(x,y,t) \mapsto (x,-y,-t)$, we assume $\lambda > 0$ which implies $\gamma < 1$.

First suppose $\mu > 0$.
By Lemma \ref{le:impacting} the system has an unstable regular equilibrium in $\Omega_L$.
Let
\begin{align}
\tilde{x} &= \frac{x}{\mu}, &
\tilde{y} &= \frac{y}{\mu},
\label{eq:impactingScalingPos}
\end{align}
with which the ODEs become
\begin{equation}
\begin{split}
\dot{\tilde{x}} &= a_1 \tilde{x} + a_2 \tilde{y} + \cO(\mu), \\
\dot{\tilde{y}} &= b_3 + b_1 \tilde{x} + b_2 \tilde{y} + \cO(\mu).
\end{split}
\label{eq:impactingfgScaledPos}
\end{equation}
The impact law becomes
$\tilde{y} \mapsto \tilde{\phi}(\tilde{y}) =
\frac{1}{\mu} \phi \left( \mu \tilde{y} \right)$ and observe
\begin{equation}
\tilde{\phi}(\tilde{y}) = -\gamma \tilde{y} + \cO(\mu).
\label{eq:impactingphiScaled}
\end{equation}

For any $\tilde{r} < 0$, let $P_L(\tilde{r};\mu)$ be the $\tilde{y}$-component
of the next intersection of the forward orbit of $(0,\tilde{r})$ with $\tilde{x} = 0$
(leave $P_L(\tilde{r};\mu)$ undefined if this orbit does not return to $\tilde{x} = 0$).
Also let $T_L(\tilde{r};\mu)$ be the corresponding evolution time.
Then, given $\tilde{r} > 0$, our Poincar\'e map is
\begin{equation}
P(\tilde{r};\mu) = P_L \big( \tilde{\phi}(\tilde{r}); \mu \big),
\label{eq:impactingP}
\end{equation}
see Fig.~\ref{fig:schemPoinImpacting}, corresponding to an evolution time of
\begin{equation}
T(\tilde{r};\mu) = T_L \big( \tilde{\phi}(\tilde{r}); \mu \big),
\label{eq:impactingT}
\end{equation}
because the impact law is assumed to be instantaneous.

\begin{figure}[b!]
\begin{center}
\includegraphics[width=5.6cm]{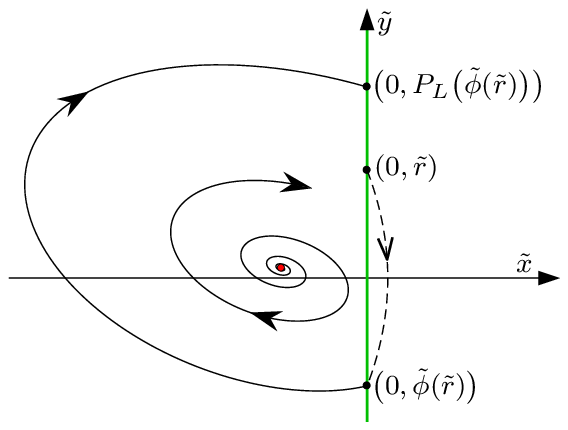}
\caption{
An illustration of the Poincar\'e map \eqref{eq:impactingP} for HLBs 11--12 with $\mu > 0$.
\label{fig:schemPoinImpacting}
} 
\end{center}
\end{figure}

As in the proof of Theorem \ref{th:FilippovDeg},
we now study the scaled system in the limit $\mu \to 0$.
With $\mu = 0$ the ODEs \eqref{eq:impactingfgScaledPos} are affine
and the impact law \eqref{eq:impactingphiScaled} is linear.
In the case that the regular equilibrium is a node (HLB 13),
as in the proof of Theorem \ref{th:pwscFocusNode} we can
use the slow invariant subspace of the node to demonstrate that the system has no limit cycle
(we omit these arguments for brevity).
Now suppose the regular equilibrium is a focus (HLBs 11--12).
By Lemma \ref{le:affine},
\begin{equation}
\begin{split}
P_L(\tilde{r};0) &= -P_{\rm affine} \left( -\tilde{r}; \lambda, \omega, -b_3 \right), \\
T_L(\tilde{r};0) &= T_{\rm affine} \left( -\tilde{r}; \lambda, \omega, -b_3 \right),
\end{split}
\end{equation}
which are of Type III because $\lambda > 0$ and $b_3 < 0$.

By \eqref{eq:affineT2}--\eqref{eq:affineP2}, any fixed point of $P(\tilde{r};0)$,
call it $\tilde{r}^*$, satisfies
\begin{equation}
\begin{split}
\gamma \tilde{r}^* &= \frac{-\kappa \re^{-\lambda T_0}
\varrho \left( \omega T_0; \frac{\lambda}{\omega} \right)}
{\sin(\omega T_0)}, \\
-\tilde{r}^* &= \frac{\kappa \re^{\lambda T_0}
\varrho \left( \omega T_0; -\frac{\lambda}{\omega} \right)}
{\sin(\omega T_0)}
\end{split}
\label{eq:impactingT0}
\end{equation}
where $T_0 = T_L(\tilde{r}^*;0)$
and $\kappa = \frac{-b_3 \omega}{\lambda^2 + \omega^2}$.
Note that these combine to produce the desired formula \eqref{eq:periodImpactingA}.
Using \eqref{eq:PaffineDeriv2} we obtain
\begin{equation}
\frac{\partial P}{\partial \tilde{r}}(\tilde{r};0) =
\frac{\gamma^2 \tilde{r}}{P_L(-\gamma \tilde{r};0)} \,\re^{2 \lambda T_L(-\gamma \tilde{r};0)},
\label{eq:impactingPDeriv}
\end{equation}
and with $\tilde{r} = \tilde{r}^*$ this simplifies to
\begin{equation}
\frac{\partial P}{\partial \tilde{r}}(\tilde{r}^*;0) =
\gamma^2 \re^{2 \lambda T_L(-\gamma \tilde{r}^*;0)}.
\label{eq:impactingPDeriv2}
\end{equation}
In the next few steps we use the fact that $T_L(-\gamma \tilde{r};0)$
is a decreasing function of $\tilde{r}$ (by Lemma \ref{le:affine}).

As in the proof of Theorem \ref{th:FilippovDeg}, suppose for the moment
that $P(\tilde{r};0)$ has a fixed point and let $\tilde{r}^*$ be the smallest such point.
Since $P(0;0) > 0$, we must have $P(\tilde{r};0) > \tilde{r}$
for all $0 \le \tilde{r} < \tilde{r}^*$,
and so $\frac{\partial P}{\partial \tilde{r}}(\tilde{r}^*;0) \le 1$.
Actually $\frac{\partial P}{\partial \tilde{r}}(\tilde{r}^*;0) < 1$
because $\frac{\partial P}{\partial \tilde{r}} \left( \tilde{r}^*;0 \right) = 1$ implies
$\frac{\partial^2 P}{\partial \tilde{r}^2} \left( \tilde{r}^*;0 \right)
= 2 \lambda \frac{d}{d \tilde{r}} T_L(-\gamma \tilde{r}^*;0) < 0$ which is not possible.
Since $\tilde{r}^*$ is the smallest fixed point,
by \eqref{eq:impactingPDeriv2} we must have
$\frac{\partial P}{\partial \tilde{r}} < 1$ at every fixed point.
Hence $P(\tilde{r};0)$ has no other fixed points
and so $P(\tilde{r};0) < \tilde{r}$ for all $\tilde{r} > \tilde{r}^*$.

This situation requires $\alpha \le 0$ because as $\tilde{r} \to \infty$ we have
$P(\tilde{r};0) \sim \gamma \,\re^{\frac{\lambda \pi}{\omega}} \tilde{r}
= \re^\alpha \tilde{r}$ (see Lemma \ref{le:affine}).
Thus in the case $\alpha > 0$ (HLB 12), $P(\tilde{r};0)$ has no fixed points
(and thus \eqref{eq:impactingODE} has no limit cycle),
while if $\alpha < 0$ (HLB 11) then $P(\tilde{r};0)$ has a fixed point by the intermediate value theorem
and, as we have just shown, this point is unique and asymptotically stable.
Since $P$ is $C^1$ the fixed point persists for small $\mu > 0$
and corresponds to a stable limit cycle of \eqref{eq:impactingODE} of size asymptotically proportional to $\mu$
(see the final part of the proof of Theorem \ref{th:FilippovDeg} for details).

Second suppose $\mu < 0$.
Let
\begin{align}
\tilde{x} &= -\frac{x}{\mu}, &
\tilde{y} &= -\frac{y}{\mu},
\label{eq:impactingScalingNeg}
\end{align}
with which the ODEs become
\begin{equation}
\begin{split}
\dot{\tilde{x}} &= a_1 \tilde{x} + a_2 \tilde{y} + \cO(\mu), \\
\dot{\tilde{y}} &= -b_3 + b_1 \tilde{x} + b_2 \tilde{y} + \cO(\mu)
\end{split}
\label{eq:impactingfgScaledNeg}
\end{equation}
We define $P_L$ and $T_L$ as above but now have
\begin{equation}
\begin{split}
P_L(\tilde{r};0) &= -P_{\rm affine} \left( -\tilde{r}; \lambda, \omega, b_3 \right), \\
T_L(\tilde{r};0) &= T_{\rm affine} \left( -\tilde{r}; \lambda, \omega, b_3 \right),
\end{split}
\end{equation}
which are of Type I because $b_3 < 0$
and $\omega = \ri \eta$ in the case of HLB 13.
Again any fixed point $\tilde{r}^*$ of $P(\tilde{r};0)$, defined by \eqref{eq:impactingP},
satisfies \eqref{eq:impactingT0} where $T_0 = T_L(\tilde{r}^*;0)$,
but now $\kappa = \frac{b_3 \omega}{\lambda^2 + \omega^2}$.
Thus again we have \eqref{eq:periodImpactingA}.

Since $P_{\rm affine}$ is of Type I we have $P(\tilde{r};0) \to 0$ as $\tilde{r} \to 0$.
Moreover $\frac{\partial P}{\partial \tilde{r}}(\tilde{r};0) \to \gamma$ as $\tilde{r} \to 0$
(as in the proof of Lemma \ref{le:impacting}).
Thus if $\tilde{r}^*$ is the smallest positive fixed point of $P(\tilde{r};0)$
then $\frac{\partial P}{\partial \tilde{r}}(\tilde{r}^*;0) \ge 1$.
By repeating the above arguments we can show that
$\frac{\partial P}{\partial \tilde{r}}(\tilde{r}^*;0) < 1$
and, if it exists, $\tilde{r}^*$ is the only positive fixed point of $P(\tilde{r};0)$.

If the regular equilibrium is a focus (HLBs 11--12)
then again $P(\tilde{r};0) \sim \re^\alpha \tilde{r}$ as $\tilde{r} \to \infty$.
Thus if $\alpha < 0$ (HLB 11), $P(\tilde{r};0)$ has no fixed points
(and thus \eqref{eq:impactingODE} has no limit cycle),
while if $\alpha > 0$ (HLB 12) then $P(\tilde{r};0)$ has a fixed point by the intermediate value theorem
and this point is unique and unstable.
Since $P$ is $C^1$ it persists for small $\mu < 0$
and corresponds to an unstable limit cycle of \eqref{eq:impactingODE}.

If the regular equilibrium is a node (HLB 13)
then $P(\tilde{r};0) \to \infty$ as $\tilde{r} \to \frac{-b_3}{\lambda + \eta}$,
thus $P(\tilde{r};0)$ has a fixed point by the intermediate value theorem.
As above this point is unique, unstable, persists for small $\mu < 0$,
and corresponds to an unstable limit cycle of \eqref{eq:impactingODE}.
\end{proof}

\section{Proof for impulsive systems (HLB 14)}
\setcounter{equation}{0}
\setcounter{figure}{0}
\setcounter{table}{0}
\label{app:impulse}

For sufficiently small $y > 0$, the forward orbit of $(r,\theta) = \left( y, \frac{\pi}{2} \right)$
undergoes an impulse then evolves via \eqref{eq:impulsiveODEpolar}
until $\theta = -\frac{3 \pi}{2} \mapsto \frac{\pi}{2}$ at some $r = P(y;\mu)$
(at which point another impulse is applied).
Let $T(y;\mu)$ denote the corresponding evolution time.
In order to 'blow up' the dynamics about the origin we write
\begin{equation}
\tilde{r} = \frac{r}{y},
\nonumber
\end{equation}
with which
\begin{equation}
\begin{split}
\dot{\tilde{r}} &= \lambda \tilde{r} + y \tilde{r}^2 \cF(\theta) + \co(y), \\
\dot{\theta} &= -\omega + y \tilde{r} \cG(\theta) + \co(y).
\end{split}
\label{eq:impulsiveODEpolar2}
\end{equation}
Let $\varphi_t(\tilde{r},\theta)$ and $\psi_t(\tilde{r},\theta)$, respectively,
denote the $r$ and $\theta$-components of the flow of \eqref{eq:impulsiveODEpolar2}.
Then $P(y;\mu)$ and $T(y;\mu)$ satisfy
\begin{equation}
\begin{split}
\varphi_{T(y;\mu)} \left( \tfrac{1}{y} R(y;\mu), \Theta(y;\mu) \right) &= \frac{1}{y} \,P(y;\mu), \\
\psi_{T(y;\mu)} \left( \tfrac{1}{y} R(y;\mu), \Theta(y;\mu) \right) &= -\frac{3 \pi}{2},
\end{split}
\label{eq:impulsivePT}
\end{equation}
and are $C^2$ because the flow intersects the positive $y$-axis transversally.

We write the flow as a series expansion in $y$:
\begin{equation}
\begin{split}
\varphi_t \left( \tfrac{1}{y} R(y;\mu), \Theta(y;\mu) \right) &=
\Phi^{(0)}_t(y;\mu) + y \Phi^{(1)}_t(y;\mu) + \co(y), \\
\psi_t \left( \tfrac{1}{y} R(y;\mu), \Theta(y;\mu) \right) &=
\Psi^{(0)}_t(y;\mu) + y \Psi^{(1)}_t(y;\mu) + \co(y).
\end{split}
\label{eq:impulsiveFlowExpansion}
\end{equation}
We also write
\begin{equation}
\begin{split}
R(y;\mu) &= \gamma y + \zeta y^2 + \co \left( y^2 \right), \\
\Theta(y;\mu) &= \phi + \xi y + \co(y),
\end{split}
\label{eq:impulsiveRTheta}
\end{equation}
for some constants $\zeta, \xi \in \mathbb{R}$
and we have neglected the $\mu$-dependency for brevity.
By substituting \eqref{eq:impulsiveFlowExpansion} into the first equation of
\eqref{eq:impulsiveODEpolar2} we obtain
\begin{equation}
\dot{\Phi}^{(0)}_t + y \dot{\Phi}^{(1)}_t = \lambda \Phi^{(0)}_t
+ y \Big( \lambda \Phi^{(1)}_t + \Phi^{(0)^{\scriptstyle 2}}_t \cF \left( \Psi^{(0)}_t \right) \Big)
+ \co(y),
\label{eq:impulsivephidotExpansion}
\end{equation}
with initial condition
\begin{equation}
\Phi^{(0)}_0 + y \Phi^{(1)}_0 = \gamma + \zeta y + \co(y).
\nonumber
\end{equation}
To first order, $\dot{\Phi}^{(0)}_t = \lambda \Phi^{(0)}_t$ and $\Phi^{(0)}_0 = \gamma$, thus
\begin{equation}
\Phi^{(0)}_t = \re^{\lambda t} \gamma.
\label{eq:impulsivePhi0}
\end{equation}
Similarly the second equation in \eqref{eq:impulsiveODEpolar2} gives
\begin{equation}
\dot{\Psi}^{(0)}_t + y \dot{\Psi}^{(1)}_t = -\omega + y \Phi^{(0)}_t \cG \left( \Psi^{(0)}_t \right) + \co(y),
\label{eq:impulsivepsidotExpansion}
\end{equation}
with initial condition
\begin{equation}
\Psi^{(0)}_0 + y \Psi^{(1)}_0 = \phi + \xi y + \co(y).
\nonumber
\end{equation}
To first order, $\dot{\Psi}^{(0)}_t = -\omega$ and $\Psi^{(0)}_0 = \phi$, thus
\begin{equation}
\Psi^{(0)}_t = \phi - \omega t.
\label{eq:impulsivePsi0}
\end{equation}

Next we write $T$ and $P$ as series expansions in $y$:
\begin{equation}
\begin{split}
T &= T^{(0)} + y T^{(1)} + \co(y), \\
P &= y P^{(0)} + y^2 P^{(1)} + \co \left( y^2 \right).
\end{split}
\label{eq:impulsivePTExpansion}
\end{equation}
By substituting \eqref{eq:impulsivePsi0} and \eqref{eq:impulsivePTExpansion}
into the second equation of \eqref{eq:impulsivePT} we obtain
\begin{equation}
\phi - \omega T^{(0)} - y \omega T^{(1)} + y \Psi^{(1)}_{T^{(0)}} = -\frac{3 \pi}{2} + \co(y),
\label{eq:impulsiveTExpansion}
\end{equation}
and from the constant terms in \eqref{eq:impulsiveTExpansion} we obtain
\begin{equation}
T^{(0)} = \frac{1}{\omega} \left( \phi + \frac{3 \pi}{2} \right).
\label{eq:impulsiveT0}
\end{equation}
By similarly substituting \eqref{eq:impulsivePhi0} and \eqref{eq:impulsivePTExpansion}
into the first equation of \eqref{eq:impulsivePT} we obtain
\begin{equation}
\re^{\lambda T^{(0)}} \,\gamma + y \lambda T^{(1)} \re^{\lambda T^{(0)}} \,\gamma + y \Phi^{(1)}_{T^{(0)}} =
P^{(0)} + y P^{(1)} + \co(y),
\label{eq:impulsivePExpansion}
\end{equation}
and from the constant terms in \eqref{eq:impulsivePExpansion} we obtain
\begin{equation}
P^{(0)} = \gamma \re^{\lambda T^{(0)}} = \re^{\Lambda}.
\nonumber
\end{equation}
Thus
\begin{equation}
P^{(0)}(\mu) = 1 + \beta \mu + \co(\mu).
\label{eq:impulsiveP02}
\end{equation}

In order to complete the proof it suffices to evaluate $P^{(1)}$ at $\mu = 0$.
Thus we now set $\mu = 0$ with which $\Lambda = 0$.
Then the $y$-terms in \eqref{eq:impulsiveTExpansion} and \eqref{eq:impulsivePExpansion} imply
\begin{align}
T^{(1)} &= \frac{1}{\omega} \,\Psi^{(1)}_{T^{(0)}}, \nonumber \\
P^{(1)} &= \Phi^{(1)}_{T^{(0)}} + \frac{\lambda}{\omega} \,\Psi^{(1)}_{T^{(0)}}.
\label{eq:impulsiveP1}
\end{align}
Next we show that $P^{(1)}(0) = \alpha$ \eqref{eq:impulsiveNondegCond}.

The $y$-terms in \eqref{eq:impulsivephidotExpansion} give
$\dot{\Phi}^{(1)}_t = \lambda \Phi^{(1)}_t + \re^{2 \lambda t} \,\gamma^2 \cF(\phi - \omega t)$ and
$\Phi^{(1)}_0 = \zeta$, thus
\begin{equation}
\Phi^{(1)}_t = \re^{\lambda t} \left( \zeta + \gamma^2
\int_0^t \re^{\lambda s} \,\cF(\phi  - \omega s) \,ds \right).
\nonumber
\end{equation}
By substituting \eqref{eq:impulsiveT0} (and using $\Lambda = 0$
and $\theta = \phi - \omega s$) we obtain
\begin{equation}
\Phi^{(1)}_{T^{(0)}} = \frac{\zeta}{\gamma} + \frac{1}{\omega} \,\re^{\frac{-3 \pi \lambda}{2 \omega}}
\int_{\frac{-3 \pi}{2}}^\phi \re^{\frac{-\lambda \theta}{\omega}} \,\cF(\theta) \,d\theta.
\label{eq:impulsivePhi1}
\end{equation}
Similarly the $y$-terms in \eqref{eq:impulsivepsidotExpansion} give
$\dot{\Psi}^{(1)}_t = \re^{\lambda t} \,\gamma \cG(\phi - \omega t)$ and
$\Psi^{(1)}_0 = \xi$, thus
\begin{equation}
\Psi^{(1)}_t = \xi + \int_0^t \re^{\lambda s} \,\gamma \cG(\phi - \omega s) \,ds,
\nonumber
\end{equation}
and by substituting \eqref{eq:impulsiveT0} we arrive at
\begin{equation}
\Psi^{(1)}_{T^{(0)}} = \xi + \frac{1}{\omega} \,\re^{\frac{-3 \pi \lambda}{2 \omega}}
\int_{\frac{-3 \pi}{2}}^\phi \re^{\frac{-\lambda \theta}{\omega}} \,\cG(\theta) \,d\theta.
\label{eq:impulsivePsi1}
\end{equation}

In summary we have shown
\begin{equation}
P(y;\mu) = y + \beta \mu y + \alpha y^2 + \co \left( \left( |y| + |\mu| \right)^2 \right),
\label{eq:impulsiveP}
\end{equation}
for sufficiently small values of $y \ge 0$ and $\mu \in \mathbb{R}$.
Locally $P$ has at most two fixed points: $y = 0$,
corresponding to the boundary equilibrium at the origin, and
\begin{equation}
y^*(\mu) = -\frac{\beta}{\alpha} \,\mu + \co(\mu),
\label{eq:impulsiveyStar}
\end{equation}
which is valid when $y^*(\mu) > 0$ and corresponds to a limit cycle
(formally \eqref{eq:impulsiveyStar} can be obtained via the IFT
by solving for zeros of $\frac{1}{y} \left( P(y;\mu) - y \right)$ as in the proof of Theorem \ref{th:fixedFocusFocus}).
Since $\beta > 0$, the system has a unique limit cycle if $\alpha \mu < 0$
and no limit cycle if $\alpha \mu > 0$.
The limit cycle intersects the positive $y$-axis at $y^*(\mu)$,
thus has amplitude asymptotically proportional to $|\mu|$.
Its period is given by \eqref{eq:impulsiveT0}, to leading order.

Stability is determined from
\begin{equation}
\frac{\partial P}{\partial y}(y;\mu) = 1 + \beta \mu + 2 \alpha y + \co \left( |y| + |\mu| \right).
\nonumber
\end{equation}
Observe $\frac{\partial P}{\partial y}(0;\mu) = 1 + \beta \mu + \co(\mu)$,
thus the origin is stable if $\mu < 0$ and is unstable if $\mu > 0$.
Also $\frac{\partial P}{\partial y} \left( y^*(\mu);\mu \right) = 1-\beta \mu + \co(\mu)$,
thus the stability of the limit cycle is opposite to that of the origin.
\manualEndProof

\section{Proofs for hysteresis}
\setcounter{equation}{0}
\setcounter{figure}{0}
\setcounter{table}{0}
\label{app:hysteresis}

\begin{proof}[Proof of Theorem \ref{th:hysteresisPseudoEq} (HLB 15)]
For each $J \in \{ L, R \}$, write
\begin{equation}
\begin{split}
f_J(x,y) &= a_{0J} + a_{1J} x + a_{2J} y + \cO \left( \left( |x| + |y| \right)^2 \right), \\
g_J(x,y) &= b_{0J} + b_{1J} x + b_{2J} y + \cO \left( \left( |x| + |y| \right)^2 \right),
\end{split}
\label{eq:hysteresisfJgJ}
\end{equation}
where $a_{0R} < 0 < a_{0L}$.
The assumption \eqref{eq:hysteresisPseudoEq} implies
\begin{equation}
\frac{b_{0L}}{a_{0L}} = \frac{b_{0R}}{a_{0R}}.
\label{eq:hysteresisSameSlopes}
\end{equation}
Also
\begin{equation}
\alpha = a_{0L} b_{2R} + a_{2L} b_{0R} - a_{0R} b_{2L} - a_{2R} b_{0L} \,.
\label{eq:hysteresisalpha}
\end{equation}

Let $\varphi^R_t(x,y)$ and $\psi^R_t(x,y)$, respectively,
denote the $x$ and $y$-components of the flow of $F_R$.
By asymptotic matching we obtain
\begin{equation}
\begin{split}
\varphi^R_t(\mu,y) &= \mu + a_{0R} t + a_{1R} \mu t + a_{2R} y t \\
&\quad+ \frac{a_{0R} a_{1R} + a_{2R} b_{0R}}{2} \,t^2 + \co \left( \left( |y| + |\mu| + |t| \right)^2 \right), \\
\psi^R_t(\mu,y) &= y + b_{0R} t + b_{1R} \mu t + b_{2R} y t \\
&\quad+ \frac{a_{0R} b_{1R} + b_{0R} b_{2R}}{2} \,t^2 + \co \left( \left( |y| + |\mu| + |t| \right)^2 \right).
\end{split}
\label{eq:hysteresisEqphipsi}
\end{equation}
Define $P_R(r;\mu)$ and $T_R(r;\mu)$ by
\begin{align}
\varphi^R_{T_R(r;\mu)}\left( \mu, r \right) &= -\mu
\label{eq:hysteresisEqTRdef} \\
\psi^R_{T_R(r;\mu)}\left( \mu, r \right) &= P_R(r;\mu),
\label{eq:hysteresisEqPRdef}
\end{align}
see Fig.~\ref{fig:schemPoinHysteresis}.
Using \eqref{eq:hysteresisEqphipsi} we obtain
\begin{align}
T_R(r;\mu) &= \frac{-2}{a_{0R}} \,\mu + \frac{2 a_{2R}}{a_{0R}^2} \,\mu r
- \frac{2 a_{2R} b_{0R}}{a_{0R}^3} \,\mu^2 \nonumber \\
&\quad+ \co \left( \left( |r| + |\mu| \right)^2 \right), \label{eq:hysteresisEqTR} \\
P_R(r;\mu) &= r - \frac{2 b_{0R}}{a_{0R}} \,\mu
+ c_{1R} \mu r
+ c_{2R} \mu^2 \nonumber \\
&\quad+ \co \left( \left( |r| + |\mu| \right)^2 \right), \label{eq:hysteresisEqPR}
\end{align}
where we define
\begin{equation}
c_{1J} = \frac{2 \left( a_{2J} b_{0J} - a_{0J} b_{2J} \right)}{a_{0J}^2},
\label{eq:hysteresisc1J}
\end{equation}
for each $J \in \{ L, R \}$
(a formula for $c_{2R}$ will not be needed).

Define $P_L(r;\mu)$ and $T_L(r;\mu)$ analogously, Fig.~\ref{fig:schemPoinHysteresis}.
Via the replacement $(R,\mu) \mapsto (L,-\mu)$ in \eqref{eq:hysteresisEqTR}--\eqref{eq:hysteresisEqPR} we obtain
\begin{align}
T_L(r;\mu) &= \frac{2}{a_{0L}} \,\mu - \frac{2 a_{2L}}{a_{0L}^2} \,\mu r
- \frac{2 a_{2L} b_{0L}}{a_{0L}^3} \,\mu^2 \nonumber \\
&\quad+ \co \left( \left( |r| + |\mu| \right)^2 \right), \label{eq:hysteresisEqTL} \\
P_L(r;\mu) &= r + \frac{2 b_{0L}}{a_{0L}} \,\mu
- c_{1L} \mu r
+ c_{2L} \mu^2 \nonumber \\
&\quad+ \co \left( \left( |r| + |\mu| \right)^2 \right), \label{eq:hysteresisEqPL}
\end{align}
where $c_{2L} \in \mathbb{R}$.

\begin{figure}[b!]
\begin{center}
\includegraphics[width=8.4cm]{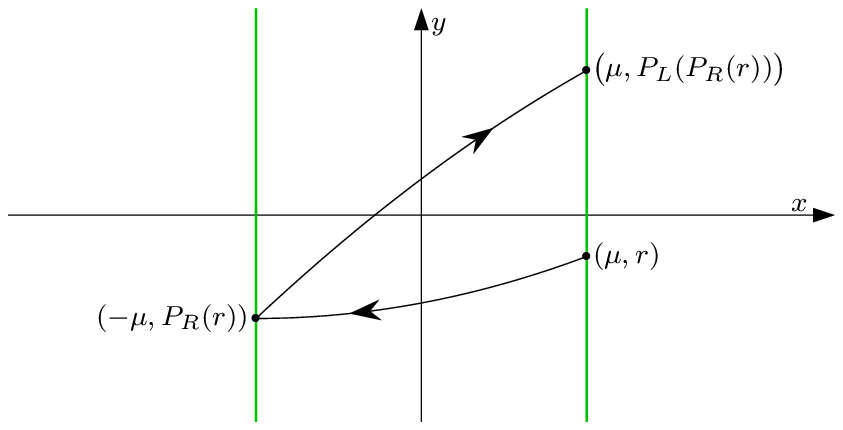}
\caption{
An illustration of the Poincar\'e map $P = P_L \circ P_R$ for HLB 15.
\label{fig:schemPoinHysteresis}
} 
\end{center}
\end{figure}

Let $P = P_L \circ P_R$ be our Poincar\'e map.
This corresponds to an evolution time of
\begin{equation}
T(r;\mu) = T_R(r;\mu) + T_L \left( P_R(r;\mu); \mu \right).
\label{eq:hysteresisEqT}
\end{equation}
By composing \eqref{eq:hysteresisEqPR} and \eqref{eq:hysteresisEqPL} we obtain
\begin{align}
P(r;\mu) &= r + 2 \left( \frac{b_{0L}}{a_{0L}} - \frac{b_{0R}}{a_{0R}} \right) \mu
- \left( c_{1L} - c_{1R} \right) \mu r \nonumber \\
&\quad+ \left( c_{2L} + c_{2R} \right) \mu^2 + \co \left( \left( |r| + |\mu| \right)^2 \right), \nonumber
\end{align}
and via careful use of \eqref{eq:hysteresisSameSlopes}, \eqref{eq:hysteresisalpha},
and \eqref{eq:hysteresisc1J} we find this reduces to
\begin{equation}
P(r;\mu) = r + \frac{2 \alpha}{a_{0L} a_{0R}} \, \mu r + \left( c_{2L} + c_{2R} \right) \mu^2 + \co \left( \left( |r| + |\mu| \right)^2 \right).
\label{eq:hysteresisEqP2}
\end{equation}

Since $P$ is $C^2$ and $P(r;0) = r$ for all $r$,
the function $D(r;\mu) = \frac{1}{\mu} \left( P(r;\mu) - r \right)$ is $C^1$.
Its domain can be extended so that it is $C^1$ throughout a neighbourhood of $(r;\mu) = (0;0)$.
If $\alpha \ne 0$ the IFT then implies the existence of a unique $C^1$ function
$r^*(\mu)$ such that $D \left( r^*(\mu); \mu \right) = 0$ for all values of $\mu$ in a neighbourhood of $0$, and
\begin{equation}
r^*(\mu) = \frac{a_{0L} a_{0R} \left( c_{2L} + c_{2R} \right)}{2 \alpha} \,\mu + \co \left( \mu \right).
\label{eq:hystereEqyStar}
\end{equation}
Notice $r^*(\mu)$ is a fixed point of $P(r;\mu)$ and corresponds to a limit cycle of \eqref{eq:hysteresis}.
Since $\frac{\partial P}{\partial r}(r;\mu) = 1 - \frac{2 \alpha \mu}{a_{0L} a_{0R}} + \co \left( |r| + |\mu| \right)$,
the limit cycle is stable if $\alpha < 0$ and unstable if $\alpha > 0$.
From \eqref{eq:hysteresisEqTR}, \eqref{eq:hysteresisEqTL}, and \eqref{eq:hysteresisEqT}
we see that its period is given by \eqref{eq:hysteresisPseudoEqPeriod}.
\end{proof}

\begin{proof}[Proof of Theorem \ref{th:hysteresisTwoFold} (HLB 17)]
For each $J \in \{ L,R \}$, write
\begin{equation}
\begin{split}
f_J(x,y) &= a_{1J} x + a_{2J} y + a_{3J} x^2 + a_{4J} x y + a_{5J} y^2 \\
&\quad+ \cO \left( \left( |x| + |y| \right)^3 \right), \\
g_J(x,y) &= b_{0J} + b_{1J} x + b_{2J} y + b_{3J} x^2 + b_{4J} x y + b_{5J} y^2 \\
&\quad+ \cO \left( \left( |x| + |y| \right)^3 \right).
\end{split}
\label{eq:hysteresisTwoFoldfJgJ}
\end{equation}
Without loss of generality we assume clockwise rotation, thus
\begin{equation}
\begin{aligned}
a_{2L} &> 0, &
a_{2R} &> 0, \\
b_{0L} &> 0, &
b_{0R} &< 0.
\end{aligned}
\label{eq:hysteresisTwoFolda2Jb0J}
\end{equation}
Also
\begin{equation}
\sigma_{{\rm fold},J} = \frac{a_{1J}}{b_{0J}} + \frac{b_{2J}}{b_{0J}} - \frac{a_{5J}}{a_{2J}}.
\label{eq:hysteresisTwoFoldalphaJ}
\end{equation}

As in the previous proof, let $\varphi^R_t(x,y)$ and $\psi^R_t(x,y)$, respectively,
denote the $x$ and $y$-components of the flow of $F_R$
and define $P_R(r;\mu)$ and $T_R(r;\mu)$ by \eqref{eq:hysteresisEqTRdef}--\eqref{eq:hysteresisEqPRdef}.
By asymptotic matching we obtain
\begin{widetext}
\begin{align}
\varphi^R_t \left( \mu, r \right) &= \mu + a_{1R} \mu t + a_{2R} r t + \frac{a_{2R} b_{0R}}{2} \,t^2
+ a_{3R} \mu^2 t + a_{4R} \mu r t + a_{5R} r^2 t
+ \left( \frac{a_{1R}^2}{2} + \frac{a_{2R} b_{1R}}{2} + \frac{a_{4R} b_{0R}}{2} \right) \mu t^2 \nonumber \\
&\quad+ \left( \frac{a_{1R} a_{2R}}{2} + \frac{a_{2R} b_{2R}}{2} + a_{5R} b_{0R} \right) r t^2
+ \left( \frac{a_{1R} a_{2R} b_{0R}}{6} + \frac{a_{2R} b_{0R} b_{2R}}{6} + \frac{a_{5R} b_{0R}^2}{3} \right) t^3
+ \co \left( \left( |\mu| + |r| + |t| \right)^3 \right),
\label{eq:hysteresisTwoFoldphi} \\
\psi^R_t(\mu,r) &= r + b_{0R} t + b_{1R} \mu t + b_{2R} r t + \frac{b_{0R} b_{2R}}{2} \,t^2
+ \cO \left( \left( |\mu| + |r| + |t| \right)^3 \right),
\label{eq:hysteresisTwoFoldpsi}
\end{align}
\end{widetext}
where we have not stated the cubic terms of $\psi^R_t$ as we will not need them.
Away from $(r;\mu) = (0;0)$ the forward orbit of $(x,y) = (\mu,r)$
intersects $x = -\mu$ transversally and so $P_R$ and $T_R$ are $C^3$ functions
of $r$ and $\mu$ by Lemma \ref{le:mapSmoothness}.

We wish to solve \eqref{eq:hysteresisEqTRdef} for $t = T_R(r;\mu)$.
Since $\varphi^R_t \left( \mu, r \right)$ has no $t$-term
we cannot use the IFT to solve \eqref{eq:hysteresisEqTRdef} for $t$.
But \eqref{eq:hysteresisEqTRdef} does have a $\mu$-term
so we can solve for $\mu$.
First observe that if $\mu = 0$ then \eqref{eq:hysteresisEqTRdef} has two solutions: $t = 0$ and
$t = \frac{-2}{b_{0R}} \,r + \cO \left( r^2 \right)$.
We are only interested in extending the second solution to $\mu > 0$
because the first solution corresponds to backwards evolution for $\mu > 0$.
Hence we substitute the form
\begin{equation}
t = \frac{-2}{b_{0R}} \,r + S,
\nonumber
\end{equation}
into \eqref{eq:hysteresisTwoFoldphi}.
By the IFT there exists a unique
$C^3$ function $h$ such that $\varphi^R_t(h(r,S),r) = -\mu$
throughout a neighbourhood of $(r,S) = (0,0)$.
Moreover, we can write this function as
\begin{equation}
h(r,S) = \frac{a_{2R}}{2} \,r S - \frac{a_{2R} \sigma_{{\rm fold},R}}{3 b_{0R}} r^3
+ \co \left( \left( |r| + \sqrt{|S|} \right)^3 \right),
\label{eq:hysteresisTwoFoldmu}
\end{equation}
where we have substituted \eqref{eq:hysteresisTwoFoldalphaJ}
and only explicitly stated the terms that will be needed below
(e.g.~the $S^2$-term in \eqref{eq:hysteresisTwoFoldmu} is not needed).
In summary, $T_R(r;\mu) = \frac{-2}{b_{0R}} r + S$,
where $S$ is given implicitly by $\mu = h(r,S)$.

Recall, $P_R$ is defined by \eqref{eq:hysteresisEqPRdef}.
Here we show that
\begin{equation}
P_R(r;\mu) = -\sqrt{r^2 + c_{1R} \mu + c_{2R} \,r^3
+ \co \left( \big( \big| \mu^{\frac{1}{3}} \big| + |r| \big)^3 \right)},
\label{eq:hysteresisTwoFoldPR}
\end{equation}
and obtain formulas for $c_{1R}, c_{2R} \in \mathbb{R}$.
By replacing $\mu$ in \eqref{eq:hysteresisTwoFoldPR} with \eqref{eq:hysteresisTwoFoldmu}, we obtain
\begin{align}
P_R(r;\mu)^2 &= r^2 + \frac{a_{2R} c_{1R}}{2} \,r S
+ \left( c_{2R} - \frac{a_{2R} c_{1R} \sigma_{{\rm fold},R}}{3 b_{0R}} \right) r^3 \nonumber \\
&\quad+ \co \left( \left( |r| + \sqrt{|S|} \right)^3 \right).
\label{eq:hysteresisTwoFoldPRrhs}
\end{align}
By putting $t = \frac{-2}{b_{0R}} \,r + S$ into \eqref{eq:hysteresisTwoFoldpsi} we obtain
\begin{equation}
\psi^R_{T_R}(\mu,r) = -r + b_{0R} S
+ \cO \left( \left( |r| + \sqrt{|S|} \right)^3 \right),
\nonumber
\end{equation}
and so
\begin{equation}
\psi^R_{T_R}(\mu,r)^2 = r^2 - 2 b_{0R} r S
+ \co \left( \left( |r| + \sqrt{|S|} \right)^3 \right).
\label{eq:hysteresisTwoFoldPRlhs}
\end{equation}
Asymptotic matching of \eqref{eq:hysteresisTwoFoldPRrhs} and \eqref{eq:hysteresisTwoFoldPRlhs} yields
\begin{align}
c_{1R} &= \frac{-4 b_{0R}}{a_{2R}}, &
c_{2R} &= \frac{-4 \sigma_{{\rm fold},R}}{3},
\end{align}
and justifies the form \eqref{eq:hysteresisTwoFoldPR}.
This completes our derivation of $P_R$.

We define $P_L$ and $T_L$ analogously and (repeating the above calculations) obtain
\begin{equation}
P_L(r;\mu) = \sqrt{r^2 + \frac{4 b_{0L}}{a_{2L}} \,\mu - \frac{4 \sigma_{{\rm fold},L}}{3} \,r^3
+ \co \left( \big( \big| \mu^{\frac{1}{3}} \big| + |r| \big)^3 \right)},
\label{eq:hysteresisTwoFoldPL}
\end{equation}
and $T_L(r;\mu) = \frac{2}{b_{0L}} r + \cO \left( \big( \big| \mu^{\frac{1}{3}} \big| + |r| \big)^2 \right)$.
As usual let $P = P_L \circ P_R$.
Then by composing \eqref{eq:hysteresisTwoFoldPR} and \eqref{eq:hysteresisTwoFoldPL} we arrive at
\begin{equation}
P(r;\mu) = \sqrt{r^2 + 4 \kappa \mu + \frac{4 \alpha}{3} \,r^3
+ \co \left( \big( \big| \mu^{\frac{1}{3}} \big| + |r| \big)^3 \right)},
\label{eq:hysteresisTwoFoldP2}
\end{equation}
where $\kappa$ and $\alpha$ are given by 
\eqref{eq:hysteresisTwoFoldkappa} and \eqref{eq:hysteresisTwoFoldNondegCond}.

To compute fixed points of $P$ in a formal manner
we search for zeros of the modified displacement function
\begin{align}
D(r;\mu) &= P(r;\mu)^2 - r^2 \\
&= 4 \kappa \mu + \frac{4 \alpha}{3} \,r^3
+ \co \left( \big( \big| \mu^{\frac{1}{3}} \big| + |r| \big)^3 \right).
\label{eq:hysteresisTwoFold}
\end{align}
Since $\kappa \ne 0$, the equation $D(r;\mu) = 0$ has a locally unique $C^3$ solution $\mu$
(by the IFT)
which we can invert to produce
\begin{equation}
r^*(\mu) = \left( \frac{-3 \kappa}{\alpha} \right)^{\frac{1}{3}} \mu^{\frac{1}{3}}
+ \co \left( \mu^{\frac{1}{3}} \right),
\label{eq:hysteresisTwoFoldyStar}
\end{equation}
assuming $\alpha \ne 0$.
If $r^*(\mu) < 0$, this is not a valid fixed point of $P$ because $P \ge 0$
in view of \eqref{eq:hysteresisTwoFoldP2}.
We are only interested in $\mu > 0$,
thus if $\alpha > 0$, $P$ has no fixed points (because $\kappa > 0$),
while if $\alpha < 0$, $P$ has the unique fixed point $r^*(\mu)$.
This value corresponds to a stable limit cycle because $D$ is a decreasing function of $r$.

The maximum $y$-value of the limit cycle is $r^*(\mu)$;
similarly its minimum $y$-value is asymptotically proportional to $\mu^{\frac{1}{3}}$.
Since both half-systems have an invisible fold (quadratic tangency) at the origin,
the minimum and maximum $x$-values of the limit cycle are asymptotically proportional
to the square of the minimum and maximum $y$-values
(this can be established formally via \eqref{eq:hysteresisTwoFoldphi}).

The evolution time associated with $P(r;\mu)$ is
\begin{equation}
T(r;\mu) = T_R(r;\mu) + T_L \left( P_R(r;\mu); \mu \right),
\label{eq:hysteresisTwoFoldT}
\end{equation}
hence the period of the limit cycle is
$T = \left( \frac{2}{b_{0L}} - \frac{2}{b_{0R}} \right) r^*(\mu)
+ \co \left( \mu^{\frac{1}{3}} \right)$.
\end{proof}

\section{Proofs for time delay}
\setcounter{equation}{0}
\setcounter{figure}{0}
\setcounter{table}{0}
\label{app:timeDelay}

Let us first describe the behaviour of a typical orbit of \eqref{eq:timeDelay} crossing $x=0$.
Let $p(t) \in \mathbb{R}^2$ denote the location of the orbit at time $t$.
Suppose $p(0) = (0,r)$ for some $r \in \mathbb{R}$.
Further suppose there exists $J \in \{ L,R \}$ such that $p(t) \in \Omega_J$ for all $t \in (-\mu,0)$
and $p(t) \notin \Omega_J$ for arbitrarily small $t > 0$.
Then the orbit next experiences a switch at $p(\mu) = \big( G_J(r;\mu), H_J(r;\mu) \big)$
where we define $G_J(r;\mu) = \varphi^J_\mu(0,r)$ and $H_J(r;\mu) = \psi^J_\mu(0,r)$
and where $\varphi^J_t(x,y)$ and $\psi^J_t(x,y)$ denote the
$x$ and $y$-components of the flow of $F_J$.
By varying $r$ (and keeping $\mu > 0$ fixed) the points $\big( G_L(r;\mu), H_L(r;\mu) \big)$
and $\big( G_R(r;\mu), H_R(r;\mu) \big)$ trace out curves
where orbits of \eqref{eq:timeDelay} experience a switch;
these are shown green in Fig.~\ref{fig:schemPoinTimeDelay}.

Next we define a Poincar\'e map $P$ (this will be well-defined locally
subject to the assumptions of Theorems \ref{th:timeDelayPseudoEq} and \ref{th:timeDelayTwoFold}).
Let $P_R(r;\mu)$ denote the $y$-value of the first intersection of the forward orbit of
$\left( G_L(r;\mu), H_L(r;\mu) \right)$ with $x = 0$ following the right half-system
and let $T_R(r;\mu)$ denote the corresponding evolution time, Fig.~\ref{fig:schemPoinTimeDelay}.
These are defined implicitly by
\begin{align}
0 &= \varphi^R_{T_R(r;\mu)}\left( G_L(r;\mu), H_L(r;\mu) \right), 
\label{eq:timeDelayEqTRdef} \\
P_R(r;\mu) &= \psi^R_{T_R(r;\mu)}\left( G_L(r;\mu), H_L(r;\mu) \right).
\label{eq:timeDelayEqPRdef}
\end{align}
Similarly let $P_L(r;\mu)$ denote the $y$-value of the first intersection of the forward orbit of
$\left( G_R(r;\mu), H_R(r;\mu) \right)$ with $x = 0$ following the left half-system,
and let $T_L(r;\mu)$ denote the corresponding evolution time.
Then let $P = P_L \circ P_R$.
This corresponds to an evolution time of
\begin{equation}
T(y;\mu) = 2 \mu + T_R(y;\mu) + T_L \left( P_R(y;\mu); \mu \right).
\label{eq:timeDelayEqT}
\end{equation}

\begin{figure}[b!]
\begin{center}
\includegraphics[width=8.4cm]{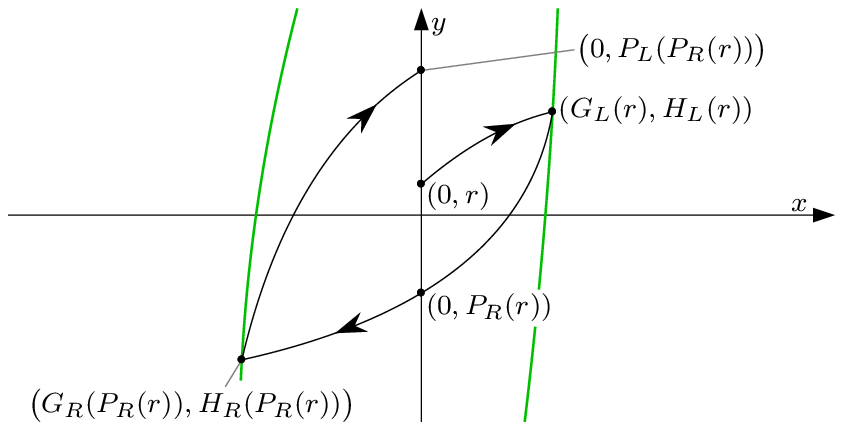}
\caption{
An illustration of the Poincar\'e map $P = P_L \circ P_R$ for HLB 16.
\label{fig:schemPoinTimeDelay}
} 
\end{center}
\end{figure}

\vspace{2mm} 			
\begin{proof}[Proof of Theorem \ref{th:timeDelayPseudoEq} (HLB 16)]
As in the proof of Theorem \ref{th:hysteresisPseudoEq}, we write
\begin{equation}
\begin{split}
f_J(x,y) &= a_{0J} + a_{1J} x + a_{2J} y + \cO \left( \left( |x| + |y| \right)^2 \right), \\
g_J(x,y) &= b_{0J} + b_{1J} x + b_{2J} y + \cO \left( \left( |x| + |y| \right)^2 \right),
\end{split}
\label{eq:timeDelayfJgJ}
\end{equation}
where $a_{0R} < 0 < a_{0L}$ and we have
\eqref{eq:hysteresisSameSlopes} and \eqref{eq:hysteresisalpha}.

By using asymptotic matching to calculate the first few terms of the flow
of the left half-system, we obtain
\begin{equation}
\begin{split}
G_L(r;\mu) &= a_{0L} \mu + a_{2L} r \mu + \frac{a_{0L} a_{1L} + a_{2L} b_{0L}}{2} \,\mu^2 \\
&\quad+ \co \left( \left( |r| + |\mu| \right)^2 \right), \\
H_L(r;\mu) &= r + b_{0L} \mu + b_{2L} r \mu + \frac{a_{0L} b_{1L} + b_{0L} b_{2L}}{2} \,\mu^2 \\
&\quad+ \co \left( \left( |r| + |\mu| \right)^2 \right).
\end{split}
\label{eq:timeDelayEqGLHL}
\end{equation}
Similarly via asymptotic matching of the right half-system,
from \eqref{eq:timeDelayEqTRdef}--\eqref{eq:timeDelayEqPRdef} we obtain
\begin{align}
T_R(r;\mu) &= \frac{-1}{a_{0R}} \,G_L
+ \frac{a_{0R} a_{1R} - a_{2R} b_{0R}}{2 a_{0R}^3} \,G_L^2 \nonumber \\
&\quad+ \frac{a_{2R}}{a_{0R}^2} \,G_L H_L + \co \left( \left( |r| + |\mu| \right)^2 \right), 
\label{eq:timeDelayEqTR} \\
P_R(r;\mu) &= H_L - \frac{b_{0R}}{a_{0R}} \,G_L \nonumber \\
&\quad- \frac{a_{2R} b_{0R}^2 - a_{0R} b_{0R}
\left( a_{1R} + b_{2R} \right) + a_{0R}^2 b_{1R}}{2 a_{0R}^3} \,G_L^2 \nonumber \\
&\quad+ \frac{a_{2R} b_{0R} - a_{0R} b_{2R}}{a_{0R}^2} \,G_L H_L + \co \left( \left( |r| + |\mu| \right)^2 \right).
\label{eq:timeDelayEqPR}
\end{align}
By combining \eqref{eq:timeDelayEqGLHL}--\eqref{eq:timeDelayEqPR} we obtain
\begin{align}
T_R(r;\mu) &= -\frac{a_{0L}}{a_{0R}} \,\mu + \cO \left( \left( |r| + |\mu| \right)^2 \right),
\label{eq:timeDelayEqTR2} \\
P_R(r;\mu) &= r + c_{1R} r \mu + c_{2R} \mu^2 + \co \left( \left( |r| + |\mu| \right)^2 \right),
\label{eq:timeDelayEqPR2}
\end{align}
where
\begin{equation}
c_{1R} = \frac{a_{0R} b_{2L} - a_{2L} b_{0R}}{a_{0R}}
+ \frac{a_{0L} \left( a_{2R} b_{0R} - a_{0R} b_{2R} \right)}{a_{0R}^2},
\label{eq:timeDelayEqc1R}
\end{equation}
while a formula for $c_{2R} \in \mathbb{R}$ will not be needed.

By symmetry we can simply swap the $L$'s and $R$'s
in \eqref{eq:timeDelayEqTR2}--\eqref{eq:timeDelayEqPR2} to obtain
\begin{align}
T_L(r;\mu) &= \frac{-a_{0R}}{a_{0L}} \,\mu + \cO \left( \left( |r| + |\mu| \right)^2 \right),
\label{eq:timeDelayEqTL} \\
P_L(r;\mu) &= r + c_{1L} r \mu + c_{2L} \mu^2 + \co \left( \left( |r| + |\mu| \right)^2 \right),
\label{eq:timeDelayEqPL}
\end{align}
where
\begin{equation}
c_{1L} = \frac{a_{0L} b_{2R} - a_{2R} b_{0L}}{a_{0L}}
+ \frac{a_{0R} \left( a_{2L} b_{0L} - a_{0L} b_{2L} \right)}{a_{0L}^2},
\label{eq:timeDelayEqc1L}
\end{equation}
and $c_{2L} \in \mathbb{R}$.
By then combining the above formulas we arrive at
\begin{align}
T(r;\mu) &= \left( 2 - \frac{a_{0L}}{a_{0R}} - \frac{a_{0R}}{a_{0L}} \right) \mu + \cO \left( \left( |r| + |\mu| \right)^2 \right).
\label{eq:timeDelayEqT2} \\
P(r;\mu) &= r + \kappa \alpha r \mu + \left( c_{2L} + c_{2R} \right) \mu^2 + \co \left( \left( |r| + |\mu| \right)^2 \right),
\label{eq:timeDelayEqP2}
\end{align}
where $\kappa = \frac{a_{0R} - a_{0L}}{a_{0L} a_{0R}} > 0$
and $\alpha$ is given by \eqref{eq:hysteresisalpha}.

By repeating steps in the proof of Theorem \ref{th:hysteresisPseudoEq}, we obtain a unique fixed point
\begin{equation}
r^*(\mu) = -\frac{c_{2L} + c_{2R}}{\kappa \alpha} \,\mu + \co \left( \mu \right),
\label{eq:timeDelayEqyStar}
\end{equation}
assuming $\alpha \ne 0$.
This fixed point corresponds to a limit cycle of \eqref{eq:timeDelay}
with minimum and maximum $x$-values asymptotically proportional to $\mu$.
Since $\frac{\partial P}{\partial y}(y;\mu) = 1 + \kappa \alpha \mu + \co(1)$,
the limit cycle is stable if $\alpha < 0$ and unstable if $\alpha > 0$.
Its period is given by \eqref{eq:timeDelayEqT2} evaluated at $r = r^*(\mu)$.
\end{proof}

\begin{proof}[Proof of Theorem \ref{th:timeDelayTwoFold} (HLB 18)]
\begin{widetext}
As in the proof of Theorem \ref{th:hysteresisTwoFold},
we write the components of $F_J$ as \eqref{eq:hysteresisTwoFoldfJgJ}
and assume clockwise rotation (without loss of generality) so that we have
\eqref{eq:hysteresisTwoFolda2Jb0J}; also \eqref{eq:hysteresisTwoFoldalphaJ}.
Here we obtain
\begin{equation}
\begin{split}
G_L(r;\mu) &= a_{2L} \mu r + \frac{a_{2L} b_{0L}}{2} \,\mu^2 + \cO \left( \left( |r| + |\mu| \right)^3 \right), \\
H_L(r;\mu) &= r + b_{0L} \mu + b_{2L} \mu r + \frac{b_{0L} b_{2L}}{2} \,\mu^2 + \cO \left( \left( |r| + |\mu| \right)^3 \right),
\end{split}
\label{eq:timeDelayTwoFoldGLHL}
\end{equation}
and
\begin{align}
\varphi^R_t \left( G_L(r;\mu), H_L(r;\mu) \right) &=
a_{2L} \mu r + a_{2R} b_{0L} \mu t + a_{2R} r t + \frac{a_{2R} b_{0R}}{2} \,t^2
+ \frac{a_{2L} b_{0L}}{2} \,\mu^2 + a_{5R} r^2 t \nonumber \\
&\quad+ \left( \frac{a_{1R} a_{2R}}{2} + \frac{a_{2R} b_{2R}}{2} + a_{5R} b_{0R} \right) r t^2 + \left( \frac{a_{1R} a_{2R} b_{0R}}{6}
+ \frac{a_{2R} b_{0R} b_{2R}}{6} + \frac{a_{5R} b_{0R}^2}{3} \right) t^3 \nonumber \\
&\quad+ U + \co \left( \left( |r| + |\mu| + |t| \right)^3 \right),
\label{eq:timeDelayTwoFoldphi} \\
\psi^R_t \left( G_L(r;\mu), H_L(r;\mu) \right) &=
r + b_{0L} \mu + b_{0R} t + b_{2R} r t + \frac{b_{0R} b_{2R}}{2} \,t^2
+ b_{2L} \mu r + b_{0L} b_{2R} \mu t + \frac{b_{0L} b_{2L}}{2} \,\mu^2 \nonumber \\
&\quad+ \cO \left( \left( |r| + |\mu| + |t| \right)^3 \right),
\label{eq:timeDelayTwoFoldpsi}
\end{align}
where $U$ represents all third order terms involving $\mu$
(these will not be needed below).

By using \eqref{eq:timeDelayTwoFoldphi} to solve \eqref{eq:timeDelayEqTRdef} we obtain
\begin{align}
T_R(r;\mu) &= \frac{-2}{b_{0R}} \,r
+ \left( \frac{a_{2L}}{a_{2R}} - \frac{2 b_{0L}}{b_{0R}} \right) \mu + \frac{2 \sigma_{{\rm fold},R}}{3 b_{0R}} \,r^2
+ c_{1R} \mu r + c_{2R} \mu^2 + \co \left( \left( |r| + |\mu| \right)^2 \right),
\label{eq:timeDelayTwoFoldTR}
\end{align}
for some $c_{1R}, c_{2R} \in \mathbb{R}$.
By substituting \eqref{eq:timeDelayTwoFoldTR} into \eqref{eq:timeDelayTwoFoldpsi} we obtain
\begin{align}
P_R(r;\mu) &= -r - a_{2L} \kappa \mu
+ \frac{2 \sigma_{{\rm fold},R}}{3} \,r^2 + d_{1R} \mu r + d_{2R} \mu^2 + \co \left( \left( |r| + |\mu| \right)^2 \right),
\label{eq:timeDelayTwoFoldPR}
\end{align}
for some $d_{1R}, d_{2R} \in \mathbb{R}$.
Swapping $L$'s and $R$'s gives
\begin{align}
P_L(r;\mu) &= -r + a_{2R} \kappa \mu
+ \frac{2 \sigma_{{\rm fold},L}}{3} \,r^2 + d_{1L} \mu r + d_{2L} \mu^2 + \co \left( \left( |r| + |\mu| \right)^2 \right),
\label{eq:timeDelayTwoFoldPL}
\end{align}
where $d_{1L}, d_{2L} \in \mathbb{R}$.
By composing \eqref{eq:timeDelayTwoFoldPR} and \eqref{eq:timeDelayTwoFoldPL} we obtain
\begin{align}
P(r;\mu) &= r + \left( a_{2L} + a_{2R} \right) \kappa \mu + \frac{2 \alpha}{3} \,r^2
+ \ell_1 \mu r + \ell_2 \mu^2 + \co \left( \left( |r| + |\mu| \right)^2 \right),
\label{eq:timeDelayTwoFoldP2}
\end{align}
for some $\ell_1, \ell_2 \in \mathbb{R}$.
\end{widetext}

As in the proof of Theorem \ref{th:slippingTwoFold}
we can use the IFT to solve $P(r;\mu) = r$ for $\mu$ and invert to obtain
the unique fixed point
\begin{equation}
r^*(\mu) = \sqrt{\frac{-3 \left( a_{2L} + a_{2R} \right) \kappa}{2 \alpha}} \sqrt{\mu} + \co \left( \sqrt{\mu} \right),
\label{eq:timeDelayTwoFoldyStar}
\end{equation}
assuming $\alpha < 0$.
This corresponds to the desired limit cycle and is stable because
$\frac{\partial P}{\partial r} < 1$ at $r = r^*(\mu)$.

The maximum $y$-value of the limit cycle is $H_L \left( y^*(\mu); \mu \right) = y^*(\mu) + \co \left( \sqrt{\mu} \right)$,
and the minimum $y$-value is similarly asymptotically proportional to $\sqrt{\mu}$.
From \eqref{eq:timeDelayTwoFoldphi} we find that maximum $x$-value
is $-\frac{a_{2R}}{2 b_{0R}} \,r^*(\mu)^2 + \cO(\mu)$,
and the minimum $x$-value is similarly asymptotically proportional to $\mu$.
Finally, by \eqref{eq:timeDelayEqT}, \eqref{eq:timeDelayTwoFoldTR}
(and the analogous formula for $T_L$), and \eqref{eq:timeDelayTwoFoldyStar},
we see that the period of the limit cycle is given by \eqref{eq:timeDelayTwoFoldPeriod}.
\end{proof}

\section{Proof for intersecting switching manifolds (HLB 19)}
\setcounter{equation}{0}
\setcounter{figure}{0}
\setcounter{table}{0}
\label{app:fourPiece}

Given a point $(0,r)$ on the positive $y$-axis,
let $P_1(r;\mu)$ denote the $x$-value of the first intersection of the forward orbit of $(0,r)$ with the $x$-axis
and let $T_1(r;\mu)$ denote the corresponding evolution time.
Define $P_2(r;\mu)$, $P_3(r;\mu)$, and $P_4(r;\mu)$ similarly, see Fig.~\ref{fig:schemPoinFourPiece},
and let $T_2(r;\mu)$, $T_3(r;\mu)$, and $T_4(r;\mu)$ denote the corresponding evolution times.
By \eqref{eq:fourPieceSpiralCond} these quantities are well-defined
for small values of $r > 0$ and $\mu \in \mathbb{R}$.
Let
\begin{equation}
P = P_4 \circ P_3 \circ P_2 \circ P_1 \,,
\label{eq:fourPieceP}
\end{equation}
denote the Poincar\'e map on the positive $y$-axis.
The corresponding evolution time is
\begin{equation}
T = T_1 + T_2 \circ P_1 + T_3 \circ P_2 \circ P_1 + T_4 \circ P_3 \circ P_2 \circ P_1 \,.
\label{eq:fourPieceT}
\end{equation}

\begin{figure}[t!]
\begin{center}
\includegraphics[width=8.4cm]{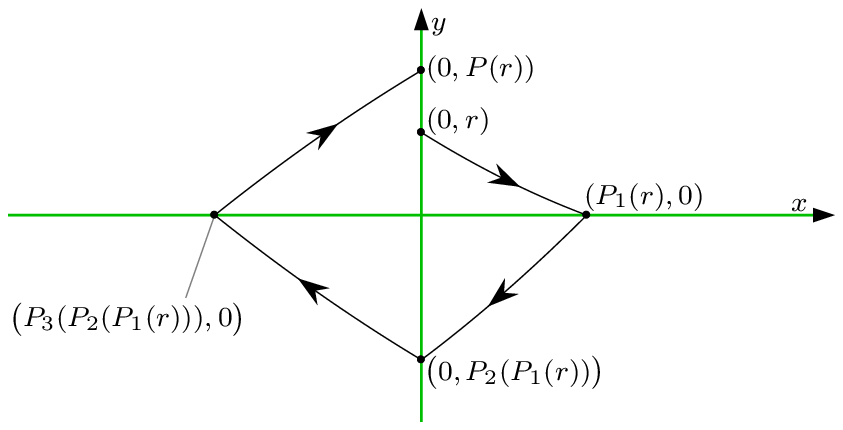}
\caption{
An illustration of the Poincar\'e map \eqref{eq:fourPieceP} for HLB 19.
\label{fig:schemPoinFourPiece}
} 
\end{center}
\end{figure}

For clarity now we suppress the $\mu$-dependency for a moment.
We write
\begin{equation}
\begin{split}
f_1(x,y) &= a_0 + a_1 x + a_2 y + \cO \left( \left( |x| + |y| \right)^2 \right), \\
g_1(x,y) &= b_0 + b_1 x + b_2 y + \cO \left( \left( |x| + |y| \right)^2 \right),
\end{split}
\label{eq:fourPiecef1g1}
\end{equation}
where $a_0 > 0$ and $b_0 < 0$.
Also
\begin{align}
m_1 &= \frac{b_0}{a_0},
\label{eq:fourPiecem1} \\
\xi_1 &= \frac{1}{a_0} \left( \frac{a_2 b_0}{a_0} - (a_1 + b_2) + \frac{a_0 b_1}{b_0} \right).
\label{eq:fourPiecexi1}
\end{align}
Let $\varphi_t(x,y)$ and $\psi_t(x,y)$ denote the $x$ and $y$-components of the flow of $F_1$.
By asymptotic matching we obtain
\begin{equation}
\begin{split}
\varphi_t(0,r) &= a_0 t + a_2 r t + \frac{a_0 a_1 + a_2 b_0}{2} \,t^2 + \co \left( \left( |r| + |t| \right)^2 \right), \\
\psi_t(0,r) &= r + b_0 t + b_2 r t + \frac{a_0 b_1 + b_0 b_2}{2} \,t^2 + \co \left( \left( |r| + |t| \right)^2 \right).
\end{split}
\label{eq:fourPiecevarphipsi}
\end{equation}
Solving $\psi_t(0,r) = 0$ for $t = T_1$ gives
\begin{equation}
T_1(r) = \frac{-1}{b_0} \,r + \left( \frac{b_2}{2 b_0^2} - \frac{a_0 b_1}{2 b_0^3} \right) r^2 + \co \left( r^2 \right).
\label{eq:fourPiecet1}
\end{equation}
By then using this to evaluate $\varphi_t(0,r)$ we obtain
\begin{equation}
P_1(r) = \frac{-1}{m_1} \,r - \frac{\xi_1}{2 m_1^2} \,r^2 + \co \left( r^2 \right).
\label{eq:fourPieceP1}
\end{equation}
Similarly
\begin{equation}
\begin{split}
P_2(r) &= m_2 r + \frac{m_2 \xi_2}{2} \,r^2 + \co \left( r^2 \right), \\
P_3(r) &= \frac{-1}{m_3} \,r + \frac{\xi_3}{2 m_3^2} \,r^2 + \co \left( r^2 \right), \\
P_4(r) &= m_4 r - \frac{m_4 \xi_4}{2} \,r^2 + \co \left( r^2 \right),
\end{split}
\label{eq:fourPieceP2P3P4}
\end{equation}
and, to leading order,
\begin{equation}
\begin{split}
T_2(r) &= \frac{-1}{f_2(0,0)} \,r + \cO \left( r^2 \right), \\
T_3(r) &= \frac{1}{g_3(0,0)} \,r + \cO \left( r^2 \right), \\
T_4(r) &= \frac{1}{f_4(0,0)} \,r + \cO \left( r^2 \right).
\end{split}
\label{eq:fourPiecet2t3t4}
\end{equation}
By composing according to \eqref{eq:fourPieceP} we obtain
\begin{equation}
P(r) = \Lambda r + \left( \frac{(\xi_1 - \xi_2) \Lambda}{2 m_1}
+ \frac{(\xi_3 - \xi_4) \Lambda^2}{2 m_4} \right) r^2 + \co \left( r^2 \right),
\label{eq:fourPieceP2}
\end{equation}
where $\Lambda$ is given by \eqref{eq:fourPieceLambda}.
By evaluating \eqref{eq:fourPieceT} we obtain
\begin{widetext}
\begin{equation}
T(r) = \left( \frac{-1}{g_1(0,0)} + \frac{1}{m_1 f_2(0,0)} - \frac{\Lambda}{m_4 f_3(0,0)} + \frac{\Lambda}{g_4(0,0)} \right) r
+ \cO \left( r^2 \right).
\label{eq:fourPieceT2}
\end{equation}
\end{widetext}

Now we examine the $\mu$-dependency of $P$.
By substituting $\Lambda(\mu) = 1 + \beta \mu + \cO \left( \mu^2 \right)$ 
into \eqref{eq:fourPieceP2} we obtain
\begin{equation}
P(r;\mu) = r + \beta r \mu + \frac{\alpha}{2} \,r^2 + \co \left( \left( |r| + |\mu| \right)^2 \right).
\label{eq:fourPieceP3}
\end{equation}
The map $P$ has at most two fixed points: $r = 0$, corresponding to the pseudo-equilibrium at the origin, and
\begin{equation}
r^*(\mu) = -\frac{2 \beta}{\alpha} \,\mu + \co \left( \mu \right),
\label{eq:fourPiecerStar}
\end{equation}
which is valid when $r^*(\mu) > 0$ and in this case corresponds to a limit cycle
(formally \eqref{eq:fourPiecerStar} can be obtained via the IFT
by solving for zeros of $\frac{1}{r} \left( P(r;\mu) - r \right)$ as in the proof of Theorem \ref{th:fixedFocusFocus}).
Since $\beta > 0$ we conclude that, locally, \eqref{eq:fourPieceODE}
has a unique limit cycle if $\alpha \mu < 0$ and no limit cycle if $\alpha \mu > 0$.
In view of \eqref{eq:fourPiecerStar}, the maximum $y$-value of the limit cycle
is asymptotically proportional to $|\mu|$,
and by symmetry this is also true for the other extremal values.
From \eqref{eq:fourPieceT2} and \eqref{eq:fourPiecerStar} we obtain
\eqref{eq:periodFourPiece} for the period.

Finally, since $\frac{\partial P}{\partial r}(0;\mu) = 1 + \beta \mu + \co(\mu)$,
the origin is stable if $\mu < 0$ and unstable if $\mu > 0$.
Since $\frac{\partial P}{\partial r}(r^*(\mu);\mu) = 1 - \beta \mu + \co(\mu)$,
the stability of the limit cycle is opposite to that of the origin.
\manualEndProof

\onecolumngrid
~\\
\noindent
\renewcommand{\bibsection}{}
{\bf References}

\end{document}